\documentclass[11pt, a4paper]{amsart}
\usepackage{hyperref}
\usepackage{microtype}
\usepackage{amsmath,amssymb}

\usepackage{tikz}
\usetikzlibrary{cd,patterns}

\usepackage{enumitem}
\setenumerate{label=(\arabic*)}

\usepackage[bb=ams,scr=euler]{mathalpha}

\usepackage{thmtools}
\declaretheoremstyle[headfont=\normalsize\normalfont\bfseries,notefont=\mdseries,
notebraces={(}{)},bodyfont=\normalfont\itshape,postheadspace=0.5em]{italstyle}

\declaretheorem[style=italstyle,name=Theorem]{theorem}
\declaretheorem[style=italstyle,name=Lemma,sibling=theorem]{lemma}
\declaretheorem[style=italstyle,name=Proposition,sibling=theorem]{proposition}
\declaretheorem[style=italstyle,name=Corollary,sibling=theorem]{corollary}

\makeatletter
\newcommand{\abs}[1]{|#1|}
\newcommand{\bd}{\partial}
\newcommand{\C}{\mathbb{C}}
\renewcommand{\d}{\mathrm{d}}
\newcommand{\id}{\mathrm{id}}
\newcommand{\Hom}{\mathrm{Hom}}
\newcommand{\intprod}{\mathbin{{\tikz{\draw(-0.1,0)--(0.1,0)--(0.1,0.2)}\hspace{0.5mm}}}}
\newcommand{\ip}[1]{\left\langle#1\right\rangle}
\newcommand{\norm}[1]{\left\lVert#1\right\rVert}
\newcommand{\pd}[2]{\frac{\partial #1}{\partial #2}}
\newcommand{\R}{\mathbb{R}}
\newcommand{\set}[1]{\{#1\}}
\newcommand{\Z}{\mathbb{Z}}

\def\@secnumfont{\itshape}
\def\section{\@startsection{section}{1}{0pt}%
  {-3.5ex \@plus -1ex \@minus -.2ex}{2.3ex \@plus.2ex}%
  {\centering\itshape}}
\def\subsection{\@startsection{subsection}{2}%
  \z@{.7\linespacing\@plus\linespacing}{-.5em}%
  {\normalfont\itshape}}
\def\subsubsection{\@startsection{subsubsection}{3}%
  \z@{.5\linespacing\@plus.7\linespacing}{-.5em}%
  {\normalfont\itshape}}
\makeatother

\title[Parametric Gromov width of Liouville domains]{Parametric Gromov width of Liouville domains}
\author{Filip Bro\'ci\'c}
\email{filipvbrocic@gmail.com}
\address{Lehrstuhl f{\"u}r Analysis und Geometrie, Universit{\"a}t Augsburg, Universitätsstrasse 14, 86159 Augsburg, Germany}

\author{Dylan Cant}
\email{dylan@dylancant.ca}
\address{Institut de mathématique d'Orsay, Université Paris-Saclay, Bâtiment 307, rue Michel Magat, F-91405 Orsay Cedex, France}

\begin{document}

\begin{abstract}
  The classical Gromov width measures the largest symplectic ball embeddable into a symplectic manifold; inspired by the symplectic camel problem, we generalize this to ask how large a symplectic ball can be embedded as a family over a parameter space $N$. Given a smooth map $f: N \to \Omega$, where $\Omega$ is a symplectic manifold, we define the \emph{parametric Gromov width} $\mathrm{Gr}(f,\Omega)$ as the supremum of capacities $a>0$ for which there exists a family of balls, parametrized by $N$, of capacity $a$ whose centers trace out the map $f$. For Liouville domains $\Omega$, we establish upper bounds on $\mathrm{Gr}(f,\Omega)$ using the Floer cohomology persistence module associated to $\Omega$. Specializing to fiberwise starshaped domains in the cotangent bundle $T^*M$, we derive computable bounds via filtered string topology. Specific examples of $\Omega$ -- including disk cotangent bundles of thin ellipsoids, open books, and tori -- demonstrate our bounds, and reveal constraints on parameterized symplectic embeddings beyond the classical Gromov width.
\end{abstract}
\maketitle

\section{Introduction}
\label{sec:introduction}

The \emph{symplectic camel theorem} (see \cite[\S3.4.B]{eliashberg-gromov-AMS-1991} and \cite{viterbo-mathann-1992,mcduff-traynor-LMS-1994}) produced the first example of a connected symplectic manifold for which the space of symplectic embeddings of a ball is disconnected. Let us explain the result in a slightly non-standard way: consider $W=\R/\Z\times \R^{2n-1}$, with coordinates $x_{1},y_{1},x_{2},y_{2},\dots,x_{n},y_{n}$ and the standard symplectic structure. Define, for $n>1$, the space:
\begin{equation*}
  X=(W\setminus \set{x_{1}=0})\cup \set{x_{n}^{2}+y_{n}^{2}\le\pi^{-1}\epsilon}.
\end{equation*}
In words, any loop in $X$ with winding number $1$ (relative the $\R/\Z$ factor) must pass through the ``hole'' of capacity $\epsilon$. Then the classical camel theorem can be stated as:
\begin{proposition}\label{proposition:camel}
  If $a>\epsilon$, there does not exist a map $F:\R/\Z\times B(a)\to X$ such that:
  \begin{enumerate}
  \item $t\mapsto F(t,0)$ has winding number $1$ relative the $\R/\Z$-factor,
  \item $z\mapsto F(t,z)$ is a symplectic embedding $B(a)\to X$ for each $t$,
  \end{enumerate}
  where $B(a)$ is the ball of capacity $a$.
\end{proposition}
The results in this paper provide a framework for detecting such phenomena in Liouville domains. Our methods are based on Floer cohomology, and recover Proposition \ref{proposition:camel}.

Let $\Omega$ be a symplectic manifold. Denote by $B(a)\subset \C^{n}$ the compact ball of symplectic capacity $a$, and denote by $\mathfrak{B}(a,\Omega)$ the space of symplectic embeddings $B(a)\to \Omega$.

A smooth map $f:N\to \Omega$ is said to \emph{lift} to $\mathfrak{B}(a,\Omega)$ provided there exists a smooth map $F:N\times B(a)\to \Omega$ such that the restriction $F(\eta,-):B(a)\to \Omega$ is a symplectic embedding for each $\eta\in N$, and such that $f(\eta)=F(\eta,0)$. It is necessary that $f^{*}T\Omega$ is symplectically trivializable for such $F$ to exist.

To handle the case when $f^{*}T\Omega$ is not symplectically trivializable, we refine the notion of lifting. We say $f$ lifts to $\mathfrak{B}(a,\Omega)/U(n)$ provided that:
\begin{enumerate}
\item\label{glift-1} $N$ admits an open cover by sets $U_{\alpha}$,
\item\label{glift-2} there are maps $F_{\alpha}:U_{\alpha}\times B(a)\to \Omega$ such that $F_{\alpha}(\eta,0)=f(\eta)$ for $\eta \in U_{\alpha}$, and which are embeddings on each $B(a)$ factor.
\item\label{glift-3} there are maps $g_{\alpha\beta}:U_{\alpha}\cap U_{\beta}\to U(n)$ such that:
  \begin{equation*}
    F_{\beta}(\eta,z)=F_{\alpha}(\eta,g_{\alpha\beta}(\eta)z),
  \end{equation*}
  for $\eta\in U_{\alpha}\cap U_{\beta}$ and $z\in B(a)$.
\end{enumerate}
These conditions implies that the
$g_{\alpha\beta}$ form a cocycle, and hence define a unitary bundle over $N$. The derivatives of the $F_{\alpha}$ establish a symplectic isomorphism between this unitary bundle and $f^{*}T\Omega$.

The \emph{parametric Gromov width} of $f$ in $\Omega$ is defined by the formula:
\begin{equation}
  \label{eq:1}
  \mathrm{Gr}(f,\Omega):=\sup\set{a:f\text{ admits a lift to }\mathfrak{B}(a,\Omega)/U(n)}.
\end{equation}
It is not hard to see that $\mathrm{Gr}(f,\Omega)$ is invariant under homotopies of $f$. Moreover, a simple Moser-type argument proves $\mathrm{Gr}(f,\Omega)>0$ holds for all maps $f$. Furthermore, if $f^{*}T\Omega$ is trivializable, then $f$ lifts to $\mathfrak{B}(a,\Omega)/U(n)$ if and only if it lifts to $\mathfrak{B}(a,\Omega)$.

One obvious choice of $f$ is the inclusion of a point $f=[\mathrm{pt}]$, in which case $\mathrm{Gr}([\mathrm{pt}],\Omega)$ is simply the classical Gromov width of $\Omega$.

The goal of this paper is to provide Floer theoretic upper bounds on $\mathrm{Gr}(f,\Omega)$ when $\Omega$ is the interior of a Liouville domain. Recall that a \emph{Liouville domain} $(\bar{\Omega},\omega=\d\lambda)$ is a compact connected $2n$-dimensional exact symplectic manifold with boundary $\bd\Omega$, such that the Liouville vector field $Z$ defined by $Z\intprod \d\lambda=\lambda$ is outwardly transverse to $\bd\Omega$.

Henceforth we reserve the symbol $\Omega$ for the interior of a Liouville domain.

The main examples of $\Omega$ we will consider are fiberwise starshaped domains in cotangent bundles, and the upper bounds we state below ultimately come from the relationship between string topology and the BV-algebra structure on Floer cohomology.

\subsection{Examples}
\label{sec:examples}

In this section we state some applications of our methods. The proofs are contained in \S\ref{sec:proofs_for_examples}.

\subsubsection{Thin ellipsoids 1}
\label{sec:thin-ellipsoids-1}

Let $\Omega_{a}$ be the unit codisk bundle in $T^{*}S^{n}$ associated to the metric obtained by embedding $S^{n}$ into $\R^{n+1}$ as the level set:
\begin{equation*}
  \set{x_{0}^2+x_{1}^{2}+\dots+a^{-2}(x_{n-1}^{2}+x_{n}^{2})= 1},
\end{equation*}
where $a\le 1$ and $n\ge 2$. Let $[S^{n}]$ be the inclusion of the zero section $S^{n}\to \Omega$. We will show:
\begin{theorem}\label{theorem:ellipsoid-1}
  $\mathrm{Gr}([S^{n}],\Omega_{a})= 2\pi a$. If $a\le 1$, then: $$\mathrm{Gr}([\mathrm{pt}],\Omega_{a})\le 4\pi a,$$ while, when $a=1$, $\mathrm{Gr}([S^{n}],\Omega_{1})=\mathrm{Gr}([\mathrm{pt}],\Omega_{1})=2\pi$.
\end{theorem}
The fact that $\mathrm{Gr}([\mathrm{pt}],\Omega_{1})=2\pi$ is known; see \cite[\S6.3]{kislev-shelukhin-GT-2021}.

\emph{Remark}. It can be shown, by comparison with the unit cotangent bundle of the cylinder $a^{-2}(x_{n-1}^{2}+x_{n}^{2})=1$, that, as $a\to0$, the ratio $\mathrm{Gr}([\mathrm{pt}],\Omega_{a})/4\pi a$ converges to $1$. Indeed, in dimension $n=2$, it has been shown in \cite{ferreira-ramos-vicente-arXiv-2023} that the Gromov width of $\Omega_{a}$ eventually equals $4\pi a$. This remark illustrates that the parametric Gromov width can be non-zero and strictly less than the usual Gromov width.

\subsubsection{Thin ellipsoids 2}
\label{sec:thin-ellipsoids-2}

Let $\Omega_{a}$ be the unit codisk bundle in $T^{*}S^{n}$ associated to the metric obtained by embedding $S^{n}$ into $\R^{n+1}$ as the level set:
\begin{equation*}
  \set{x_{0}^2+x_{1}^{2}+\dots+a^{-2}(x_{n-3}^{2}+x_{n-2}^{2}+x_{n-1}^{2}+x_{n}^{2})= 1},
\end{equation*}
with $n\ge 3$. In this case our methods give:
\begin{theorem}\label{theorem:ellipsoid-2}
  $\mathrm{Gr}([\mathrm{pt}],\Omega_{a})=\mathrm{Gr}([S^{n}],\Omega_{a})=2\pi a$.
\end{theorem}

\subsubsection{Open books with trivial monodromy}
\label{sec:open-books-with}

Let $(V,\bd V)$ be a compact and connected manifold with boundary and let: $$M=(V\times \R/\Z) \cup (\bd V\times D(1))$$ be considered as a smooth open book with page $(V,\bd V)$ and trivial monodromy. Let $\mathscr{L}_{+}$ be the set of oriented loops of the form:
\begin{enumerate}
\item $\set{v}\times \R/\Z$, $v\in V$,
\item $\set{v}\times \bd D(r)$, $v\in \bd V$ and $r\le 1$,
\end{enumerate}
which form a singular foliation of $M$ (the singularities occur along the binding, where the loops are constant). Similarly let $\mathscr{L}_{-}$ be the same set of loops but with the reverse orientation. For each loop $q\in \mathscr{L}_{\pm}$, pick a parametrization and define:
\begin{equation}\label{eq:length-function}
  \ell_{\Omega}(q)=\int_{0}^{1}\max\set{\ip{p,q'(t)}:p\in \Omega\cap T^{*}M_{q(t)}}dt.
\end{equation}
This quantity is independent of the choice of parametrization, and should be considered as the length measured using $\Omega$. Define:
\begin{equation}\label{eq:E-max-length}
  E_{\pm}=\sup \set{\ell_{\Omega}(q):q\in \mathscr{L}_{\pm}}\text{ and }e_{\pm}=\inf\set{\ell_{\Omega}(q):q\in \mathscr{L}_{\pm}}.
\end{equation}
Our methods give the upper bounds:
\begin{theorem}\label{theorem:OB-1}
  Let $[M]$ be the inclusion of the zero section, and let $[V]$ be the inclusion of the page $V\times \set{0}$. Then:
  \begin{enumerate}
  \item\label{OB1} $\mathrm{Gr}([\mathrm{pt}],\Omega)\le E_{+}+E_{-}$,
  \item\label{OB2} if $\bd V\ne \emptyset$ then $\mathrm{Gr}([M],\Omega)\le \min\set{E_{+},E_{-}}$, and,
  \item\label{OB3} if $\bd V=\emptyset$, then $\mathrm{Gr}([V],\Omega)\le \min\set{E_{+}+e_{-},E_{-}+e_{+}}$.
  \end{enumerate}
  There are examples with $\bd V\ne \emptyset$ for which both \ref{OB1} and \ref{OB2} are equalities, and examples with $\bd V=\emptyset$ for which both \ref{OB1} and \ref{OB3} are equalities.
\end{theorem}

\subsubsection{Product with a torus}
\label{sec:product-with-torus}

The case of $M=V\times \R/\Z$ where $\bd V=\emptyset$, one can use the non-contractibility of the orbits $\set{x}\times \R/\Z$ to bound the Gromov width without appealing to the structure of the BV-operator (we will present an argument which uses only classical displacement energy ideas in \S\ref{sec:string-topology-non}). However, our methods still give interesting bounds on the parametric Gromov width which do not seem to be accessible with the more classical methods of \S\ref{sec:string-topology-non}. Indeed, part \ref{OB3} of Theorem \ref{theorem:OB-1} is already such a result. In this section we will state additional results for manifolds of the form $M=V\times T^{d}$, where $V$ is a closed manifold.

Fix a fiberwise starshaped domain $\Omega\subset T^{*}(V\times T^{d})$. In the following, the class $[V\times T^{k}]$ represents the inclusion $V\times T^{k}\to V\times T^{d}$, where $T^{k}\subset T^{d}$ is the subset of points of the form $(x_{1},\dots,x_{k},0,\dots,0)$. We also denote by:
\begin{equation*}
  \begin{aligned}
    \mathscr{L}_{-}&=\set{\text{loops of the form }t\in \R/\Z\mapsto (v,x_{1},\dots,x_{d-1},-t)},\\
    \mathscr{L}^{k}_{+}&=\set{\text{loops of the form }t\in \R/\Z\mapsto (v,0,\dots,0,x_{k+1},\dots,x_{d-1},t)},
  \end{aligned}
\end{equation*}
where $v\in V$, and where we require that $k<d$. Let $E_{-}$ be the maximum $\ell_{\Omega}$-length of loops in $\mathscr{L}_{-}$ and $E_{+}^{k}$ the maximum $\ell_{\Omega}$-length of loops in class $\mathscr{L}^{k}_{+}$, similarly to \eqref{eq:E-max-length}.

\begin{theorem}\label{theorem:prod-torus}
  With the notation set in the preceding paragraph, we have:
  \begin{equation*}
    \mathrm{Gr}([T^{k}],\Omega)\le E_{-}+E^{k}_{+},
  \end{equation*}
  where $[T^{k}]$ is represented by $x\mapsto (v_{0},x_{1},\dots,x_{k},0,\dots,0)$ for some basepoint $v_{0}$ (the homotopy class is independent of $v_{0}$).
\end{theorem}
This result may seem abstruse; however, in the case when $V=\mathrm{pt}$ and $k=1$ we should note that it obstructs a loop of symplectic balls similarly to the classical camel theorem Proposition \ref{proposition:camel}. In fact, this example will be used to prove Proposition \ref{proposition:camel} and the argument is given in \S\ref{sec:proof-prod-torus}.

\subsubsection{Non-orientable surfaces}
\label{sec:non-orient-surf}

Let $\Sigma$ be a compact non-orientable surface, and let $\Omega\subset T^{*}\Sigma$ be a fiberwise starshaped domain. Define:
\begin{equation*}
  \mathscr{L}=\set{\text{set of loops $q:\R/\Z\to \Sigma$ such that $q^{*}T\Sigma$ is non-orientable}},
\end{equation*}
and let:
\begin{equation*}
  E=\inf\set{\ell_{\Omega}(q)+\ell_{\Omega}(\bar{q}):q\in \mathscr{L}},
\end{equation*}
where $\bar{q}$ denotes the loop $q$ traversed in reverse. Then our methods bound the parametric Gromov width for the inclusion of the zero section $[\Sigma]$:
\begin{theorem}\label{theorem:non-orientable-surface}
  With the above notation, $\mathrm{Gr}([\Sigma],\Omega)\le E$.
\end{theorem}
This theorem is interesting because it applies to surfaces with negative curvature, where it is generally hard to bound symplectic capacities. It is also interesting to ponder the role of non-orientability. For instance, since certain non-orientable surfaces $\Sigma$ embed as Lagrangians in $\C^{2}$, the Hofer-Zehnder capacity of any disk cotangent bundle over such $\Sigma$ is finite. However, it seems to be an open question whether the same fact is true for orientable surfaces of genus at least two (see, e.g., \cite[pp.\,105]{bimmermann-arch-math-2024}).

\subsection{Floer cohomology persistence module}
\label{sec:floer-cohom-pers-1}

Let $W$ be the completion of $\bar{\Omega}$ obtained by attaching the symplectization end $\bd \Omega\times [1,\infty)$ to $\bar{\Omega}$ in such a way that $Z$ extends to $r\bd_{r}$, where $r$ is the projection to $[1,\infty)$, and such that the extension of $Z$ is a Liouville vector field.

Let $\mathscr{H}$ be the space of all Hamiltonian functions $H$ such that $H=cr$ holds outside of a compact set, for some $c\in \R$.

Fix an almost complex structure $J$ which is invariant under the flow by $Z$ in the end. For a time-dependent family of Hamiltonian functions $H_{t}\in \mathscr{H}$, whose flow $\varphi_{t}$ has a non-degenerate time-1 map, and for which the Floer cohomology chain complex $\mathrm{CF}(H_{t})$ is well-defined. Here the Floer complex is the $\Z/2\Z$ vector space generated by the $1$-periodic orbits of $\varphi_{t}$, and the differential uses the almost complex structure $J$. Denote by $\mathrm{HF}(H_{t},J)$ the resulting homology.

Using continuation maps, the resulting homology group depends only on the \emph{slope}, namely the average value of $H_{t}/r$ in the end; we denote by $V_{c}$ the homology group for a system with slope $c$. It is well-known that continuation maps endow $c\mapsto V_{c}$ with the structure of a persistence module, namely, a functor from $(\R,\le)\to \mathrm{Vect}(\Z/2\Z)$. Precise details of this construction are recalled in \S\ref{sec:floer-cohom-pers-defin}. The colimit of $V_{c}$ as $c\to\infty$ is the so-called \emph{symplectic cohomology} $\mathrm{SH}(W)$.

This persistence module $V_{c}$ has three structures relevant to our paper:
\begin{enumerate}
\item the product structure $\ast:V_{c_{1}}\otimes V_{c_{2}}\to V_{c_{1}+c_{2}}$ induced by the Floer cohomology pair-of-pants operation,
\item the BV-operator $\Delta:V_{c}\to V_{c}$ induced by counting $\R/\Z$-families of Floer cylinders,
\item the PSS morphism $\mathrm{PSS}:H^{*}(W)\to V_{c}$ for $c>0$;
\end{enumerate}
we refer the reader to \cite{abbondandolo-schwarz-GT-2010,abouzaid-EMS-2015} for background on these structures.

Our main result is the following:
\begin{theorem}\label{theorem:main-floer}
  Suppose that $\zeta_{i}\in V_{c_{i}}$, $i=1,2$, with $c_{i}>0$, are such that:
  \begin{equation*}
    \mathrm{PSS}(\beta)=\Delta(\zeta_{1})\ast \zeta_{2}\text{ holds in }V_{c_{1}+c_{2}}.
  \end{equation*}
  If a map $f:N\to \Omega$ has a non-zero mod 2 homological intersection number with the cohomology class $\beta\in H^{*}(W)$, then $\mathrm{Gr}(f,\Omega)\le c_{1}+c_{2}$.
\end{theorem}

The idea in the proof of Theorem \ref{theorem:main-floer} is to define a sort of \emph{evaluation map} $V_{c}\to \Z/2\Z$ using the family of ball embeddings $N\times B(a)\to \Omega$. The map is defined and non-trivial on $\mathrm{PSS}(\beta)$ provided that the slope $c$ is smaller than the capacity $a$. Arguing using a special action filtration defined using family Floer cohomology (in the sense of \cite{hutchings-agt-2008}), we will show that the non-triviality of this map obstructs the existence of a solution to the equation appearing in Theorem \ref{theorem:main-floer}. Thus, if the equation can be solved, the slope $c$ must be larger than the capacity $a$, as desired.

The details of this argument are given in \S\ref{sec:family-floer-cohom}.

\subsubsection{Relative Hofer-Zehnder capacities}
\label{sec:relat-hofer-zehnd}

In fact, the argument we give easily generalizes to the following statement concerning relative Hofer-Zehnder capacities. Let us define a \emph{Hofer-Zehnder admissible function} $K:\Omega\to \R$ with \emph{depth} $A$ and \emph{well} $U$ (a non-empty open set) to be a function such that:
\begin{enumerate}
\item $K$ is compactly supported in $\Omega$,
\item $-A=\min K$, and the interior of $\set{K=-A}$ is $U$,
\item $X_{K}$ has no nonconstant 1-periodic orbits.
\end{enumerate}
Then we have:
\begin{theorem}\label{theorem:hofer-zehnder}
  Suppose that the hypotheses of Theorem \ref{theorem:main-floer} are satisfied for slope $c>0$ and map $f:N\to \Omega$. For each Hofer-Zehnder admissible function $K$ whose well $U$ contains the image of $f$, we have:
  \begin{equation*}
    A + \mathrm{Gr}(f,U)\le c,
  \end{equation*}
  where $A$ is the depth of $K$. In particular, the so-called relative Hofer Zehnder capacity of $f(N)\subset \Omega$ is bounded from above by $c$.
\end{theorem}
This is proved in \S\ref{sec:proof-theorem-hofer-zehnder}. Here we recall \emph{the relative Hofer Zehnder capacity} of a pair $Y\subset X$ (where $X$ is a symplectic manifold) is simply the deepest depth $A$ one can achieve amongst Hofer-Zehnder admissible functions $K:X\to \R$ subject to the constraint that the well $U$ contains $Y$.

\subsubsection{Dilation classes}
\label{sec:dilation-classes}

The equation appearing in Theorem \ref{theorem:main-floer} can be considered as a generalization of the \emph{dilation class equation} introduced in \cite{seidel-solomon-GAFA-2012}. Recall that a class $\zeta$ is called a \emph{dilation class} if $\mathrm{PSS}(1)=\Delta \zeta$. Thus our theorem implies:
\begin{corollary}
The existence of a dilation class in the Floer cohomology group $V_{c}$ bounds the usual Gromov width by $c$.
\end{corollary}
This result is probably not so surprising to experts, as the slopes $c$ for which dilation classes appear are already used in quantitative symplectic geometry; see, e.g., \cite{seidel-inventiones-2014, zhou-topology-2021}.

The result implied by Theorem \ref{theorem:ellipsoid-2} on the Gromov width on the cotangent bundle of the round $S^{3}$ then follows from the existence of a dilation class in the Floer cohomology group of the appropriate slope, and the above corollary. Indeed, in Theorem \ref{theorem:ellipsoid-2}, our proof essentially shows that the dilation class equation $\mathrm{PSS}(1)=\Delta\zeta$ can be solved in the appropriate group $V_{c}$.

It is interesting to compare the situation with our results on tori, as the cotangent bundles of tori (or more generally $K(\pi,1)$ spaces) never have dilation classes. However, the more general equation in Theorem \ref{theorem:main-floer} does admit solutions (as we exploit in Theorem \ref{theorem:prod-torus}).

There are other classes of manifolds which are known to admit dilation classes. For instance, \cite{seidel-solomon-GAFA-2012} show that the total space of certain Lefschetz fibrations admit dilation classes. To apply the above corollary one would need present the total space as a completion of a Liouville domain, and then estimate the precise slopes for which the dilation class equation can be solved in total spaces of Lefshetz fibrations; we do not analyze this question in this paper.

\subsubsection{Subcritical handle attachment}
\label{sec:subcrit-handle-attachment}

Another interesting class of Liouville domains where the equation in Theorem \ref{theorem:main-floer} can be solved are domains $\Omega$ obtained by a subcritical handle attachment to a Liouville domain $\Omega_{0}$. Recall that $\Omega$ is obtained by attaching a handle along an isotropic sphere $S^{k-1}$ in the contact boundary $\bd\Omega_{0}$, and if $k<n$ we say the handle attachment is subcritical. For instance, attaching a $1$-handle to a $4$-dimensional Liouville domain is an example of a subcritical handle attachment. For more details on handle attachments see \cite{weinstein-hokkaido-91, cieliebak-JEMS-2002, fauck-IJM-2020}.

If the attaching sphere is null-homologous in $\Omega_{0}$, then the \emph{cocore disk} (which is a properly embedded open disk $D^{2n-k}\to \Omega$ which intersects the attaching disk in a single point) defines a cohomology class in $\Omega$. It is known that:
\begin{proposition}
  The cohomology class of the cocore disk $\beta$ satisfies:
  \begin{equation*}
    \mathrm{PSS}(\beta)=0
  \end{equation*}
  in $V_{c}$ for $c$ sufficiently large.
\end{proposition}
In particular, $\beta$ satisfies the equation in Theorem \ref{theorem:main-floer}.
\begin{proof}
  This follows from \cite{cieliebak-JEMS-2002} (see also \cite[Theorem 1.3]{fauck-IJM-2020}) which proves the so-called Viterbo restriction map from \cite{viterbo-GAFA-1999} is an isomorphism from the symplectic cohomology of $W_{0}$ to the symplectic cohomology of $W$ (the completions of $\Omega_{0}$ and $\Omega$, respectively).

  It is also known that the Viterbo restriction map ($\text{V.R.}$) commutes with PSS and the pullback map on cohomology associated to the inclusion $i:\Omega_{0}\to \Omega$:
  \begin{equation}\label{eq:viterbo_restriction}
    \begin{tikzcd}
      {H^{*}(W)}\arrow[d,"{\mathrm{PSS}}"]\arrow[r,"{i^*}"] &{H^{*}(W_{0})}\arrow[d,"{\mathrm{PSS}}"]\\
      {\mathrm{SH}(W)}\arrow[r,"\text{V.R.}"] &{\mathrm{SH}(W_{0})},
    \end{tikzcd}
  \end{equation}
  Since $i^{*}$ takes $\beta$ to $0$, and $\text{V.R.}$ is an isomorphism, it follows that $\mathrm{PSS}(\beta)=0$ in $\mathrm{SH}(W)$, as desired.
\end{proof}

Theorem \ref{theorem:main-floer} then bounds the parametric Gromov width of any map $f:N\to \Omega$ which has a non-zero intersection number with the cocore disk. This brings us close to the original camel problem Proposition \ref{proposition:camel}; indeed, in one formulation of the camel problem, the relevant space is the domain obtained by attaching a $1$-handle to a ball, see \cite[Figure 1.3]{mcduffsalamon-alt}.

To apply our methods to obtain explicit bounds on the parametric Gromov width, one would need to set up a careful model for the handle attachment $\Omega_{0}\to \Omega$, and determine at which slopes $c$ the equation $\mathrm{PSS}(\beta)=0$ can be solved in $V_{c}$. We do not perform such an analysis in this paper. We will instead recover Proposition \ref{proposition:camel} using a string topological computation in \S\ref{sec:proofs_for_examples}.

\subsection{Comparison between string topology and Floer cohomology}
\label{sec:comp-betw-string}

For our applications to fiberwise starshaped domains in cotangent bundles, we will use the well-established strategy of comparing string topology and Floer cohomology. The relationship between string topology and Floer cohomology of cotangent bundles is developed in \cite{viterbo-GAFA-1999,abbondandolo-schwarz-CPAM-2006,salamon-weber-GAFA-2006,abbondandolo-schwarz-GT-2010,abouzaid-EMS-2015}.

Our approach is inspired by Shelukhin's proof of the Viterbo conjecture for certain manifolds in \cite{shelukhin-GAFA-2022,shelukhin-invetiones-2022}. In particular, his approach in \cite{shelukhin-GAFA-2022} exploits the relationship between the string topology and symplectic cohomology for cotangent bundles of manifolds which are \emph{string point-invertible}; the notion of string point-invertibility is defined in terms the string bracket which measures the failure of the BV operator $\Delta$ to satisfy the Leibniz rule.

Another inspiration is Irie's bound on the Hofer-Zehnder capacity of disk cotangent bundles of manifolds admitting circle actions with non-contractible fibers \cite{irie-JEMS-2014}; his argument appeals to the aforementioned product structures.

Our framework for string topology is as follows. Let $\Omega$ be a fiberwise starshaped domain in $T^{*}M$. For $c>0$ and $k=0,1,\dots$, introduce $Z_{k}(\Lambda_{c})$ as the monoid of smooth maps $A:P\times \R/\Z\to M$ where:
\begin{enumerate}
\item $\dim P=k$,
\item $\ell_{\Omega}(A(x,-))<c$, for all $x\in P$, where $\ell_{\Omega}$ is the length function defined in \eqref{eq:length-function};
\end{enumerate}
here $A_{1}+A_{2}$ is the coproduct $(P_{1}\sqcup P_{2})\times \R/\Z\to M$.

A \emph{cobordism} between $A_{1},A_{2}\in Z_{k}(\Lambda_{c})$ is a cobordism $Q$ between $P_{1},P_{2}$ and a smooth map $C:Q\times \R/\Z\to M$ satisfying:
\begin{enumerate}
\item[($2'$)] $\ell_{\Omega}(C(q,-))<c$ for all $q\in Q$,
\end{enumerate}
and such that $C$ restricts to $A_{1},A_{2}$ on the boundary. Define $H_{k}(\Lambda_{c})$ be the quotient of $Z_{k}(\Lambda_{c})$ by the cobordism relation. The monoid structure on $Z_{k}(\Lambda_{c})$ induces on $H_{k}(\Lambda_{c})$ the structure of a vector space over $\Z/2\Z$. Write $Z(\Lambda_{c}),H(\Lambda_{c})$ for the direct sum of the graded pieces $Z_{k}(\Lambda_{c}),H_{k}(\Lambda_{c})$, respectively. If $A\in Z(\Lambda_{c})$, then we write its image $\alpha\in H(\Lambda_{c})$ using the symbol $\alpha=[A]$. One thinks of $H(\Lambda_{c})$ as a convenient proxy for the homology of the loop space.

The obvious inclusion morphisms $H(\Lambda_{c_{1}})\to H(\Lambda_{c_{2}})$ endow $c\mapsto H(\Lambda_{c})$ with the structure of a persistence module defined for $c>0$.

As in the Floer cohomology case, this persistence module has three natural structures:
\begin{enumerate}
\item the \emph{Chas-Sullivan product} $\ast:H(\Lambda_{c_{1}})\otimes H(\Lambda_{c_{2}})\to H(\Lambda_{c_{1}+c_{2}})$,
\item the \emph{BV-operator} $\Delta:H(\Lambda_{c})\to H(\Lambda_{c})$,
\item an \emph{inclusion of constant loops} morphism $\mathfrak{i}:H^{*}(W)\to H(\Lambda_{c})$ for $c>0$, which sends a degree $d$ cohomology class to an element in $H_{n-d}(\Lambda_{c})$.
\end{enumerate}
The technical result relating string topology and the Floer cohomology persistence module is:
\begin{theorem}\label{theorem:main-comparison}
  There is a morphism of persistence modules $\Theta_{c}:H(\Lambda_{c})\to V_{s}$ such that:
  \begin{enumerate}
  \item\label{product-compatible} $\ast\circ (\Theta_{c_{1}},\Theta_{c_{2}})=\Theta_{c_{1}+c_{2}}\circ \ast$ holds on $H(\Lambda_{c_{1}})\otimes H(\Lambda_{c_{2}})$,
  \item $\Delta\circ \Theta_{c}=\Theta_{c} \circ \Delta$,
  \item $\mathrm{PSS}=\Theta_{c}\circ \mathfrak{i}$.
  \end{enumerate}
\end{theorem}
The proof of this theorem occupies \S\ref{sec:string-topol-floer-cohom}. Most of the arguments are similar to those in \cite{abbondandolo-schwarz-CPAM-2006,abbondandolo-schwarz-GT-2010,abouzaid-EMS-2015}, with small modifications due to the fact we work with bordism classes of loops (rather than chains in the Morse homology associated to an energy functional). However, working directly with bordism classes makes the existing proof that the product structures are identified difficult to implement (the existing proof is not a direct argument, and instead is based on a homological algebra argument). For this reason, we provide a new proof that the product structures are identified. Interestingly enough, this leads to an \emph{adiabatic gluing problem}, similar to those considered in \cite{fukaya-oh-AJM-1997,ekholm-GT-2007,oh-zhu-JSG-2011}. This argument is given in \S\ref{sec:product-structures}.

Combining Theorem \ref{theorem:main-floer} and Theorem \ref{theorem:main-comparison} yields the following upper bound on the relative Gromov width in terms of string topology:
\begin{theorem}\label{theorem:corollary}
  Suppose there are classes $\alpha_{i}\in H(\Lambda_{c_{i}})$, $i=1,2$, with $c_{i}>0$, such that:
  \begin{equation*}
    \mathfrak{i}(\beta)=\Delta(\alpha_{1})\ast\alpha_{2}\text{ holds in }H(\Lambda_{c_{1}+c_{2}}).
  \end{equation*}
  If a map $f:N\to \Omega$ has a non-zero mod 2 homological intersection number with the cohomology class $\beta\in H^{*}(W)$, then $\mathrm{Gr}(f,\Omega)\le c_{1}+c_{2}$.\hfill$\square$
\end{theorem}
This corollary will be the result used in the demonstration of our examples. One of course has a similar corollary concerning the relative Hofer-Zehnder capacity by combining Theorem \ref{theorem:main-comparison} with Theorem \ref{theorem:hofer-zehnder}.

\subsection{String topology in non-trivial free homotopy classes}
\label{sec:string-topology-non}

In the case when there is sufficiently rich string topology in non-trivial free homotopy classes, one can establish bounds on the Gromov width without appealing to the BV-operator. For instance, a main result of \cite{irie-JEMS-2014} uses the existence of two classes $\alpha_{i}\in H(\Lambda_{c_{i}})$, $i=1,2$, such that:
\begin{enumerate}
\item $\alpha_{1}\ast \alpha_{2}=\mathfrak{i}([M])$,
\item $\alpha_{i}$ lies in a non-trivial free homotopy class,
\end{enumerate}
to prove the Hofer-Zehnder capacity of $\Omega$ is finite. Typically the way one obtains such classes $\alpha_{1},\alpha_{2}$ is via $\R/\Z$-actions whose orbits are non-contractible.

Such considerations yield the following variation of Theorem \ref{theorem:main-floer}:
\begin{theorem}\label{theorem:non-contractible}
  Suppose that $\zeta_{i}\in V_{c_{i}}(\Omega)$, $i=1,2$, with $c_{i}>0$, are such that:
  \begin{equation*}
    \mathrm{PSS}(\beta)=\zeta_{1}\ast \zeta_{2}\text{ holds in }V_{c_{1}+c_{2}},
  \end{equation*}
  and suppose that $\zeta_{i}$ can be represented by a cycle whose orbits are all non-contractible.

  If a map $f:N\to \Omega$ has a non-zero mod 2 homological intersection number with the cohomology class $\beta\in H^{*}(W)$, then $\mathrm{Gr}(f,\Omega)\le c_{1}+c_{2}$.
\end{theorem}
The proof is given in \S\ref{sec:proof-theorem-non-contractible}. There is also an obvious variation of Theorem~\ref{theorem:hofer-zehnder} based on Theorem \ref{theorem:non-contractible}, whose statement we omit.

In certain cases where orbits are non-contractible, one can prove bounds on the Gromov width using more classical ideas of Hamiltonian displacement. Indeed, we have:
\begin{proposition}
  Suppose $f_{t}$ is a smooth isotopy of a closed manifold $M$ such that $f_{0}=f_{1}=\id$ and such that the orbits $x\mapsto f_{t}(x)$ are non-contractible; denote by $\kappa$ the free homotopy class containing these orbits. There exists a constant $\mathrm{const}(f_{t},\Omega)$ such that, for any covering space $M'\to M$ to which loops in $\kappa$ do not lift, the following holds: if $K\subset \Omega$ is a compact set which does admit a lift to $K'\subset T^{*}M'$, then $K'$ has Hofer displacement energy at most $\mathrm{const}(f_{t},\Omega)$.
\end{proposition}
\emph{Remark}. This implies that the Gromov width of $\Omega$ is bounded by $\mathrm{const}(f_{t},\Omega)$. It also implies that Lagrangians which lift to $T^{*}M'$ bound holomorphic disks with symplectic area at most $\mathrm{const}(f_{t},\Omega)$.
\begin{proof}
  Let $\Phi_{t}$ the canonical lift of $f_{t}$, which lifts to an isotopy $\Phi_{t}'$ of $T^{*}M'$ such that $\Phi_{1}'$ is a canonical lift of a deck transformation of $M'\to M$. Then $\Phi_{1}'$ displaces any set $K'$ as in the statement. One can cut-off $\Phi_{t}'$ in a uniform way without changing how it acts on the inverse image of $\Omega$. The cut-off can be constructed so that it has a finite Hofer length. The desired result follows.
\end{proof}

\subsection{Conventions}
\label{sec:conventions}

\subsubsection{On the cohomology of $W$}
\label{sec:remark-on-cohomology}

For convenience, we define $H^{*}(W)$ to be the group of smooth proper maps $C:S\to W$, where $S$ is a smooth manifold, modulo proper cobordisms. This is graded by the codimension of $F$. This is a proxy for the cohomology of $W$. Prototypical examples are the fundamental class $\id:W\to W$ and, in the case when $W=T^{*}M$, the inclusion of the fiber $T^{*}M_{q}\to T^{*}M$ for some $q$.

\subsection{Acknowledgements}
\label{sec:acknowledgements}

First of all, the authors wish to thank Egor Shelukhin for many interesting discussions and helpful comments. The first author is thankful to Kai Cieliebak for useful suggestions. The second author also wishes to thank Georgios Dimitroglou-Rizell, Tobias Ekholm, and Yin Li, for key discussions during his visit to Uppsala University. The authors also wish to thank useful discussions with Alberto Abbondandolo and Johanna Bimmermann after the first version of this text was posted. F.B.\@ is supported by the Deutsche Forschungsgemeinschaft (DFG, German Research Foundation) – 517480394. D.C.\@ is supported by the ANR project CoSy.

\section{String topology}
\label{sec:string-topology}

In this section, we continue the discussion of string topology from where the introduction left off.

\subsection{Three structures on the string topology persistence module}
\label{sec:three-struct-string}

In order to apply Theorem \ref{theorem:corollary} it is necessary to explain the three structures $\ast,\Delta,\mathfrak{i}$.

\subsubsection{The Chas-Sullivan product}
\label{sec:chas-sull-prod}

The product $\alpha_{1}\ast \alpha_{2}\in H_{k_{1}+k_{2}-n}(\Lambda_{c_{1}+c_{2}})$ of two classes $\alpha_{i}\in H_{k_{i}}(\Lambda_{c_{i}})$, $i=1,2$, originally defined in \cite{chas-sullivan-arXiv-1999}, can be thought of as a combination of the intersection product and the Pontrjagin (concatenation) product.

It is defined as follows: write $\alpha_{i}=[A_{i}]$ where $A_{1},A_{2}$ are in \emph{general position}, which means that the \emph{evaluation-at-zero} maps: $$e_{i}:x\in P_{i}\mapsto A_{i}(x,0),$$ $i=0,1$, are transverse to one-another. Such representatives always exist. Define $P_{3}$ to be the transverse fiber product of these two maps:
\begin{equation*}
  \begin{tikzcd}
    {P_{3}}\arrow[d,"{}"]\arrow[r,"{}"] &{P_{2}}\arrow[d,"{e_{2}}"]\\
    {P_{1}}\arrow[r,"{e_{1}}"] &{M},
  \end{tikzcd}
\end{equation*}
so that $P_{3}$ is a compact manifold of dimension $k_{1}+k_{2}-n$, where $n=\dim M$. Concretely $P_{3}$ is the submanifold of pairs $(x_{1},x_{2})\in P_{1}\times P_{2}$ satisfying the incidence $e_{1}(x_{1})=e_{2}(x_{2})$.
Define $A_{3}:P_{3}\times \R/\Z\to M$ by the formula:
\begin{equation*}
  A_{3}(x_{1},x_{2},t):=\left\{
    \begin{aligned}
      &A_{1}(x_{1},\beta(2t))&&\text{ for }t\in [0,1/2],\\
      &A_{2}(x_{2},\beta(2t-1))&&\text{ for }t\in [1/2,1],
    \end{aligned}
  \right.
\end{equation*}
and extended by $1$-periodicity; here $\beta$ is a standard smooth cut-off function which equals $0$ for $t\le 0$ and equals $1$ for $t\ge 1$. The cut-offs ensure $A_{3}$ is a smooth map. It is important to note that $A_{3}\in Z(\Lambda_{c_{1}+c_{2}})$, because the length function is additive under concatenation and invariant under time reparametrizations.

\begin{lemma}
  The class $[A_{3}]$ in $H_{k_{1}+k_{2}-n}(\Lambda_{c_{1}+c_{2}})$ is independent of the choice of representatives $A_{1},A_{2}$ in general position, and depends only on $\alpha_{1},\alpha_{2}$.
\end{lemma}
\begin{proof}
  This is a straightforward argument in differential topology, similar to the arguments in \cite{milnor-book-1965}, and is left to the reader.
\end{proof}

Observe also that $(x_{3},t)\mapsto A_{3}(x_{3},t+1/2)$ represents the class of $\alpha_{2}\ast \alpha_{1}$, and hence $\alpha_{1}\ast \alpha_{2}=\alpha_{2}\ast\alpha_{2}$, since we can homotope from one to the other via the formula $(x_{3},t,s)\mapsto A_{3}(x_{3},t+s)$.

Similarly, one can prove that $\ast$ is an associative product, although we leave the details of this to the reader.

\subsubsection{The BV-operator}
\label{sec:bv-operator}

Like the Chas-Sullivan product, this string topology operation is introduced in \cite{chas-sullivan-arXiv-1999}. It is defined as follows; given $\alpha\in H_{k}(\Lambda_{c})$, write $\alpha=[A]$, where $A:P\times \R/\Z\to M$, and introduce:
\begin{equation*}
  \Delta(A):(\R/\Z\times P)\times \R/\Z\to M\text{ given by }\Delta(A)(\theta,x)(t)=A(x,t-\theta).
\end{equation*}
Then $\Delta(A)\in Z_{k+1}(\Lambda_{c})$ and we define $\Delta(\alpha)=[\Delta(A)]$. It is trivial to check that $\Delta(\alpha)$ is independent of the choice of representative $A$.

\subsubsection{Inclusion of the constant loops}
\label{sec:inclusion-of-constants}

Let $\beta\in H^{d}(T^{*}M)$ be represented by a smooth proper map $C:S\to T^{*}M$ which is transverse to the zero section. The class $\mathfrak{i}(\beta)\in H_{n-d}(\Lambda_{c})$ is the image of $\beta$ under a sort of Thom isomorphism.

Define $f:P\to M$ via the fiber product:
\begin{equation*}
  \begin{tikzcd}
    {P}\arrow[d,"{f}"]\arrow[r,"{}"] &{S}\arrow[d,"{C}"]\\
    {M}\arrow[r,"{}"] &{T^{*}M},
  \end{tikzcd}
\end{equation*}
and define $A:P\times \R/\Z\to M$ by $A(x,t)=f(x)$. Set: $$\mathfrak{i}(\beta):=[A]\in H_{n-d}(\Lambda_{c}).$$
It is straightforward to prove that $\mathfrak{i}(\beta)$ in $H_{n-d}(\Lambda_{c})$ is independent of the representative $C:S\to T^{*}M$ (since, by our proxy definition of $H^{*}(T^{*}M)$, any two representatives are properly cobordant).

The following discussion sheds light on the map $\mathfrak{i}$, and will be used in \S\ref{sec:incl-const-loops} in the comparison between string topology and Floer cohomology; it shows that $\mathfrak{i}$ is essentially a Thom isomorphism between cohomology of $T^{*}M$ and homology of $M$.
\begin{lemma}\label{lemma:isomorphism_constant}
  Any class $C:S\to M$ is cohomologous to the class $C':S'\to M$ determined by the fiber product:
  \begin{equation*}
    \begin{tikzcd}
      {S'}\arrow[d,"{}"]\arrow[r,"{C'}"] &{T^{*}M}\arrow[d,"{}"]\\
      {P}\arrow[r,"{f}"] &{M},
    \end{tikzcd}
  \end{equation*}
  where $f$ is defined above.
\end{lemma}
\begin{proof}
  Observe that $C$ and $C'$ have the same transverse intersection with the zero section. It follows from a straighforward surgery operation that $C+C'$ (the sum is taken in the cohomology cobordism group) is properly cobordant to a proper map $G:T\to T^{*}M$ whose image is disjoint from the zero section.

  Then the map $T\times [0,\infty)\to T^{*}M$ given by $x,t\mapsto \rho_{t}(G(x,t))$, where $\rho_{t}$ is the Liouville flow, is a proper cobordism from $G$ to $\emptyset$ (it is proper because the image of $G$ is disjoint from the zero section). Thus $[C+C']=0$, and hence the desired relation $[C]=[C']$ holds.
\end{proof}

\subsection{Proofs for \S \ref{sec:examples}}
\label{sec:proofs_for_examples}

In this section, we provide the proofs of the Theorems stated in \S \ref{sec:examples}. The upper bounds in Theorem \ref{theorem:ellipsoid-1} will ultimately follow from Theorem \ref{theorem:OB-1}. The proof of Theorem \ref{theorem:OB-1} is based on a geometric construction of classes solving the equation appearing in Theorem \ref{theorem:corollary}. The proof of Theorem \ref{theorem:ellipsoid-2} will follow from similar considerations, but with a different classes in the string topology of open books. We prove Theorem \ref{theorem:prod-torus} and explain how to recover Proposition \ref{proposition:camel} in \S\ref{sec:proof-prod-torus}. Finally we prove Theorem \ref{theorem:non-orientable-surface} in \S\ref{sec:proof-non-orientable-surface}.

\subsubsection{Action classes}
\label{sec:action-classes}

One recurring idea is the notion of an \emph{action class}. Let $\zeta_{t}:M\to M$, $t\in \R/\Z$, be a circle action. The \emph{action class} is the class:
\begin{equation*}
  A:M\times \R/\Z\to M\text{ given by }A(x,t)=\zeta_{t}(x).
\end{equation*}
More generally, for any bordism class $g:P\to M$, one defines:
\begin{equation*}
  A_{g,\pm}:P\times \R/\Z\to M\text{ given by }A_{g,\pm}(x,t)=\zeta_{\pm t}(g(x)).
\end{equation*}
It is clear that $[A_{g,\pm}]$ depends only on the bordism class of $g$ in $M$.

Denote by $E_{g,\pm}$ the maximal $\ell_{\Omega}$-length of the orbits appearing in $A_{g,\pm}$.

Such classes behave well with respect to the Chas-Sullivan product.
\begin{lemma}\label{lemma:chas-sullivan-action-class-1}
  For two transverse maps $g_{i}:P_{i}\to M$, $i=1,2$, one has:
  \begin{equation*}
    [A_{g_{1},+}]\ast [A_{g_{2},-}]=[g_{1}\cap g_{2}]\text{ in }\Lambda_{E_{g_{1},+}+E_{g_{2},-}},
  \end{equation*}
  where $[g_{1}\cap g_{2}]$ is the class of constant loops: $$(x_{1},x_{2})\in P_{3}\times \R/\Z\mapsto g_{1}(x_{1})=g_{2}(x_{2}),$$ where $P_{3}$ is the fiber product of $g_{1}$ and $g_{2}$.
\end{lemma}
\begin{proof}
  The evaluation at $0$ map of $A_{g_{i},\pm}$ is $g_{i}$. Thus $A_{g_{1},+}\ast A_{g_{2},+}$ is represented by the map:
  \begin{equation}\label{eq:chas-sullivan-product}
    (x_{1},x_{2},t)\in P_{3}\times \R/\Z\mapsto \left\{
      \begin{aligned}
        &\zeta_{\beta(2t)}(g_{1}(x_{1}))&&\text{ for }t\in [0,1/2],\\
        &\zeta_{-\beta(2t-1)}(g_{2}(x_{2}))&&\text{ for }t\in [1/2,1].
      \end{aligned}
    \right.
  \end{equation}
  We claim this class is homotopic to the class of $[g_{1}\cap g_{2}]$ described in the statement of the lemma; furthermore, we claim the homotopy can be made inside of $\Lambda_{E_{g_{1},+} + E_{g_{2},-}}$. To show this, define $G : P_{3} \times [0,1]\times \R/\Z \to M$ by:
  \begin{equation*}
    G(x_{1},x_{2},s,t):=\left\{
      \begin{aligned}
        &\zeta_{\beta(2st)}(g_{1}(x_{1}))&&\text{ for }t\in [0,1/2],\\
        &\zeta_{\beta(s)-\beta(2st-s)}(g_{2}(x_{2}))&&\text{ for }t\in [1/2,1].
      \end{aligned}
    \right.
  \end{equation*}
  This defines a smooth homotopy, within $\Lambda_{E_{g_{1},+}+E_{g_{2},-}}$, between the class of constant loops $g_{1}\cap g_{2}$, at $s=0$, and \eqref{eq:chas-sullivan-product}, at $s=1$.
\end{proof}

Another lemma which will be used in the proof of Theorem \ref{theorem:OB-1} is:
\begin{lemma}\label{lemma:chas-sullivan-action-class-2}
  For two transverse maps $g_{i}:P_{i}\to M$, $i=1,2$, then:
  \begin{equation*}
    [A_{g_{1},\pm}]\ast [g_{2}]=[A_{g_{1}\cap g_{2},\pm}]\text{ in }\Lambda_{E_{g_{1},\pm}},
  \end{equation*}
  where $[g_{2}]$ is considered as a class of constant loops, and $g_{1}\cap g_{2}:P_{3}\to M$ is as in Lemma \ref{lemma:chas-sullivan-action-class-1}.
\end{lemma}
\begin{proof}
  The proof is straightforward, and easier than Lemma \ref{lemma:chas-sullivan-action-class-1}.
\end{proof}

Given a smooth map $g:P\to M$, denote by $\zeta g:\R/\Z \times P\to M$ the class defined by $\zeta g(\theta,x)=\zeta_{\theta}(g(x))$. The final result we require concerning action classes is:
\begin{lemma}\label{lemma:chas-sullivan-action-class-3}
  Given a smooth map $g:P\to M$, it holds that:
  \begin{equation*}
    \Delta[A_{g,\pm}]=[A_{\zeta g,\pm}]
  \end{equation*}
  in $H(\Lambda_{E_{g,\pm}})$.
\end{lemma}
\begin{proof}
  This is a straightforward calculation: by definition, $\Delta[A_{g,\pm}]$ is the class represented by $\R/\Z\times P\times \R/\Z\to M$ given by:
  \begin{equation*}
    (\theta,x,t)\mapsto \zeta_{\pm(t-\theta)}(g(x)).
  \end{equation*}
  On the other hand, $A_{\zeta g,\pm}$ is given by:
  \begin{equation*}
    (\theta,x,t)\mapsto \zeta_{\pm t+\theta}(g(x)).
  \end{equation*}
  In the case of $-$ sign, these classes are literally the same. In the case of $+$ sign, the classes differ by the diffeomorphism $\R/\Z\times P\mapsto \R/\Z\times P$ given by $(\theta,x)\mapsto (-\theta,x)$, and so represent the same class in cobordism.
\end{proof}

\subsubsection{Proof of Theorem \ref{theorem:OB-1}}
\label{sec:proof-of-theorem-OB-1}
Let $M=(V\times \R/\Z) \cup (\bd V\times D(1))$ be an open book with trivial monodromy as in \S\ref{sec:open-books-with}. There is an obvious circle action $\zeta$ which acts as:
\begin{equation*}
  \zeta_{t}(x,\theta)=\left\{
    \begin{aligned}
      &(x,\theta+t), &&(x, \theta) \in V \times \R/\Z,\\
      &(x,r, \theta + t), &&(x,r,\theta) \in \partial V \times D(1),
    \end{aligned}
  \right.
\end{equation*}
where we use polar coordinates on $D(1)$ factor near the binding.

Theorem \ref{theorem:OB-1} will follow from Theorem \ref{theorem:corollary} applied to the action classes associated to this circle action. In order to apply Theorem \ref{theorem:corollary}, we need to show that the relevant action classes lie in the image of the BV-operator, among other things.

Abbreviate $A_{\pm}=A_{\id,\pm}$ and $E_{\pm}=E_{\id,\pm}$, using the notation in \S\ref{sec:action-classes}. Then $E_{\pm}$ agree with the numbers denoted using the same symbols in the statement of Theorem \ref{theorem:OB-1}.
\begin{lemma}\label{lemma:action_is_BV}
  There exist $[B_{\pm}]\in H(\Lambda_{E_{\pm}})$ such that $\Delta[B_{\pm}]=[A_{\pm}]$ in $H(\Lambda_{E_{\pm}})$.
\end{lemma}
\begin{proof}
  This is obvious in the case in $\bd V=\emptyset$ (one can simply appeal to Lemma \ref{lemma:chas-sullivan-action-class-3}). The argument is harder when $\bd V\ne \emptyset$.

  The class $B_{\pm}$ is defined as a map $D(V)\times \R/\Z\to M$ where here the \emph{double} $D(V)$ is the closed manifold obtained by gluing $V$ to another copy of $V$, say $\bar{V}$, along the boundary using the identity map.

  Then $A_{\pm}$ is defined on $M\times \R/\Z$, and $\Delta(B_{\pm})$ is defined on $\R/\Z\times D(V)\times \R/\Z$. There is a cobordism $X$ between $\R/\Z\times D(V)$ and $M$, and there is map $C_{\pm}:Q\times\R/\Z\to M$ which extends $A_{\pm},B_{\pm}$, and which remains in $\Lambda_{E_{\pm}}$; see Figure \ref{fig:cobordism}.

  In fact, the construction of $B_{\pm}$ and $C_{\pm}$ follows from a more general argument which we present in \S\ref{sec:diag-acti-open}. For this reason, we omit the details of the present proof.
\end{proof}
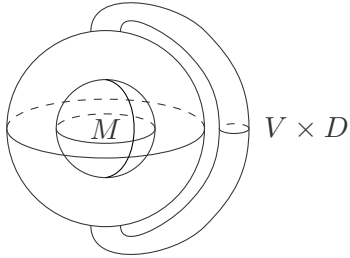
\begin{figure}[h]
  \centering
  \begin{tikzpicture}[scale=.65]
    \draw[black!80!white] (0,0) circle (2cm);
    \draw[black!80!white] (-2,0) arc (180:360:2 and 0.6);
    \draw[black!80!white,dashed] (2,0) arc (0:180:2 and 0.6);
    \draw[black!80!white] (0,0) circle (1cm);
    \draw[black!80!white] (-1,0) arc (180:360:1 and 0.3);
    \draw[black!80!white, dashed] (1,0) arc (0:180:1 and 0.3);
    \node[black!80!white] at (0,0) {$M$};

    \draw[black!80!white] (-0.3, -{sqrt(4-0.09)}) to [out=270, in=225] (2,-2) to [out=45, in=270] (2.9,0) to [out=90, in=315] (2,2) to [out = 135, in =90] (-0.3, {sqrt(4-0.09)});
    \draw[black!80!white] (0.3, -{sqrt(4-0.09)}) to [out=270, in=225] (1.6,-1.6) to [out=45, in=270] (2.3,0) to [out=90, in=315] (1.6,1.6) to [out = 135, in =90] (0.3, {sqrt(4-0.09)});
    \draw[black!80!white] (2.3,0) arc (180:360:0.3 and 0.1);
    \draw[black!80!white, dashed] (2.9,0) arc (0:180:0.3 and 0.1);
    \node[black!80!white] at (3,0) [right] {$V \times D$};

    \draw[black] (0,1) to [out=350, in=10] (0,-1);
  \end{tikzpicture}
  \caption{The cobordism between $M$ and $\R/\Z\times D(V)$ can be visualized as attaching a generalized handle to $M$. The simplest example is when $V=[0,1]$ where $M\simeq S^{2}$ and $\R/\Z\times D(V)\simeq T^{2}$.}
  \label{fig:cobordism}
\end{figure}

We now prove Theorem \ref{theorem:OB-1} using this lemma and the results in \S\ref{sec:action-classes}.

First of all, part \ref{OB1} follows easily from Theorem \ref{theorem:corollary}, and the fact that:
\begin{equation*}
  \mathfrak{i}(\mathrm{PD}(T^{*}M))=[A_{+}]\ast [A_{-}]=\Delta[B_{+}]\ast \Delta[B_{-}]\text{ in }H(\Lambda_{E_{+}+E_{-}}),
\end{equation*}
where we have used Lemma \ref{lemma:chas-sullivan-action-class-1} in the first equality, and the fact that $\mathfrak{i}(\mathrm{PD}(T^{*}M))$ is represented by the constant loops $(x,t)\in M\times \R/\Z\mapsto x$. The second equality follows from the previous lemma. By Theorem \ref{theorem:corollary}, this equation bounds the regular Gromov width from above by $E_{+}+E_{-}$, since $f:\mathrm{pt}\to \Omega$ is dual to the fundamental class $\mathrm{PD}(T^{*}M)$.

To prove part \ref{OB2}, we use Lemma \ref{lemma:chas-sullivan-action-class-2}, with $g_{1}=\id$ and $g_{2}$ the inclusion of a point $\mathrm{pt}$. We then have:
\begin{equation}\label{eq:equation-pt-OB-1-part-2}
  \Delta[B_{\pm}]\ast \mathfrak{i}(T^{*}M_{\mathrm{pt}})=[A_{\pm}]\ast [g_{2}]=[A_{g_{2},\pm}]\text{ in }H(E_{\pm}),
\end{equation}
where we use that $\mathfrak{i}(T^{*}M_{\mathrm{pt}})$ is represented by a single constant loop based at $\mathrm{pt}$.

The class $A_{g_{2},\pm}$ is represented by a single loop going around the open book. Since $\bd V\ne \emptyset$, this single loop can be homotoped to a constant loop in the binding; moreover, this homotopy can be taken in $H(E_{\pm})$.

Since the class of $f:M\to M$ is dual to $\mathrm{PD}(T^{*}M_{\mathrm{pt}})$, Theorem \ref{theorem:corollary} applied with equation \eqref{eq:equation-pt-OB-1-part-2} bounds the parametric Gromov width $\mathrm{Gr}([M],\Omega)$ from above by $\min\set{E_{+},E_{-}}$.

To prove part \ref{OB3}, we first appeal to Lemma \ref{lemma:chas-sullivan-action-class-3} where the map $g:P\to M$ is the inclusion of a point $\mathrm{pt}$; this shows
\begin{equation*}
  \Delta[A_{g,\pm}]=[A_{\zeta g,\pm}].
\end{equation*}
Then we apply Lemma \ref{lemma:chas-sullivan-action-class-1} to conclude:
\begin{equation*}
  \Delta[A_{g,\pm}]\ast \Delta[B_{\mp}]=\Delta[A_{\zeta g,\pm}]\ast [A_{\mp}]=[(\zeta g)\cap \id]\text{ in }\Lambda_{E_{g,\pm}+E_{\mp}},
\end{equation*}
A moment's thought reveals that the class of constant loops $(\zeta g)\cap \id$ (based at points which travel around a single orbit) equals $\mathfrak{i}(\beta)$ where $\beta$ is dual to the inclusion of $V\to M$. Thus we conclude from Theorem \ref{theorem:corollary} that $\mathrm{Gr}([V],\Omega)$ is bounded by $\min\set{E_{g,+}+E_{-},E_{g,-}+E_{+}}$. Since we can take $g$ to be an arbitrary point, we can make $E_{g,+}=e_{+}$ or $E_{g,-}=e_{-}$ (the constants appearing in Theorem \ref{theorem:OB-1}). This gives the desired result for \ref{OB3}.

To complete the proof of Theorem \ref{theorem:OB-1}, it remains for us to explain why the estimates are sharp in some cases.

One example with $\bd V\ne \emptyset$ where both \ref{OB1} and \ref{OB2} are sharp is $M=S^{2}$ with the ellipsoidal metric induced by $\set{x_{0}^{2}+a^{-2}(x_{1}^{2}+x_{2}^{2})=1}\subset \R^{3}$ for $a$ sufficiently small; that the bound $\mathrm{Gr}([S^{2}],\Omega_{a})\le 2\pi a$ is sharp is the content of Theorem \ref{theorem:ellipsoid-1}, and further details are given in \S\ref{sec:proof_thm_ell_1}. That the bound $\mathrm{Gr}([\mathrm{pt}],\Omega_{a})\le 4\pi a$ is sharp is proved in \cite{ferreira-ramos-vicente-arXiv-2023} for $a$ small enough.

One example with $\bd V=\emptyset$ where both \ref{OB1} and \ref{OB3} are sharp is $M=T^{2}$ with the flat metric obtained by $T^{2}=\R/a\Z\times \R\Z$ with $a\le 1$. That the Gromov width of the unit codisk bundle $\Omega_{a}$ is $4\pi a$ is proved in \cite{brocic-CCM-2025}; using canonical translations, one can transport this ball in a coherent manner to be based at all the points in $V=\set{0}\times \R/\Z$, proving that $\mathrm{Gr}([V],\Omega_{a})$ is also $4\pi a$, as desired (in general, for the codisk bundle $\Omega$ of a flat metric on a torus, or, more generally, the codisk bundle of any Lie group with a bi-invariant metric, it holds that $\mathrm{Gr}(\mathrm{pt},\Omega)=\mathrm{Gr}(f,\Omega)$).

\subsubsection{Proof of Theorem \ref{theorem:ellipsoid-1}}
\label{sec:proof_thm_ell_1}
As we will show momentarily, the upper bounds in Theorem \ref{theorem:ellipsoid-1} follow from Theorem \ref{theorem:OB-1}. The lower bounds will follow from explicit constructions.

There are two metrics we will consider on $S^n$ in the proof. Set:
\begin{equation*}
  S^{n}=\set{(x_{0},\dots,x_{n}):x_{0}^{2}+\dots+x_{n}^{2}=1}\subset \R^{n+1}.
\end{equation*}
The first metric $g_{a}$ we consider is the pullback of the Euclidean metric under the ellipsoid embedding $\varphi_{\mathrm{ell}}:S^{n}\to \R^{n+1}$ given by: $$\varphi_{\mathrm{ell}}(x_{0},\dots,x_{n-1},x_{n})=(x_{0},\dots,ax_{n-1},ax_{n}).$$ The unit codisk bundle of $g_{a}$ is the domain $\Omega_{a}$ under consideration.

The second metric is the round metric $g_{a}^{\mathrm{round}}$ is the pullback of the Euclidean metric under the embedding $\varphi_{\mathrm{round}}:S^{n}\to \R^{n+1}$ given by:
\begin{equation*}
  \varphi_{\mathrm{round}}(x)=ax;
\end{equation*}
we denote by $\Omega_{a}^{\mathrm{round}}$ the associated unit codisk bundle.

Observe that $S^{n}$ is identified with the open book with page $V=D^{n-1}$, via the parametrization:
$$V \times  \R / \Z \ni (x, \theta) \to (x, \sqrt{1 - \|x\|^2} \cos(2\pi \theta), \sqrt{1 - \|x\|^2} \sin(2 \pi \theta)) \in S^{n}.$$

Since the domain $\Omega_{a}$ is fiberwise symmetric, we have $E_{+} = E_{-}$. An easy computation shows that every loop in $\mathscr{L}_+$ has the length at most $2 \pi a$ as measured using $\Omega_{a}$; this maximum is achieved for the loop $(0,\cos(2\pi \theta),\sin(2 \pi \theta))$. Hence, from Theorem \ref{theorem:OB-1} we get
$$
\mathrm{Gr}([S^{n}],\Omega_{a})\le 2 \pi a, \text{ and } \mathrm{Gr}([\mathrm{pt}],\Omega_{a})\le 4\pi a.
$$

To obtain the equality $\mathrm{Gr}([S^{n}],\Omega_{a})= 2 \pi a$, we will first show that $\Omega_{a}^{\mathrm{round}}\subset \Omega_{a}$. Thus it follows that $\mathrm{Gr}([S^{n}],\Omega_{a}^{\mathrm{round}})\le \mathrm{Gr}([S^{n}],\Omega_{a})$. Then, we will show that $\mathrm{Gr}([S^n], \Omega_{a}^{\mathrm{round}}) = 2\pi a$.

It is a general fact that, if there is an inequality between Riemannian metric $g_{0}\le g_{1}$, then the unit codisk bundle determined by $g_{1}$ contains the unit codisk bundle determined by $g_{0}$. An easy computation shows that $g_{a}^{\mathrm{round}}\le g_{a}$.

Now we argue that $\mathrm{Gr}([S^{n}],\Omega_{a}^{\mathrm{round}})= 2\pi a$.
It is known that for the class of the point we have $Gr([\mathrm{pt}]; \Omega_{a}^{\mathrm{round}}) = 2\pi a$, i.e., for every $\epsilon >0$ there is a symplectic embedding $e:B^{2n}(2\pi a - \epsilon) \to \Omega_{a}^{\mathrm{round}}$. We can further assume that $e(0) = (e_1,0)$, where $e_1 = (1,0,...,0) \in S^n$. To estimate the parametric Gromov width from below, it is enough to find a family $A:S^n \to \mathrm{Symp}(\Omega_{a}^{\mathrm{round}})$ which satisfies $A(q) (e_1, 0) = (q, 0)$. Indeed, such a familly, together with an embedding $e$ induces $A_e: S^n \to \mathfrak{B}(2\pi a - \epsilon,\Omega_a)$ by
$$
A_e(q) := A(q) \circ e \in \mathfrak{B}(2 \pi a - \epsilon,\Omega_a).
$$
One way to think about $\Omega_{a}^{\mathrm{round}}$ is $$\Omega_{a}^{\mathrm{round}}\cong \{q+ i p \in \C^{n+1} \mid \|q\|=1, \langle q, p \rangle =0, \|p\| \leq a \}.$$
Each $A \in U(n+1)$ maps $\Omega_{a}^{\mathrm{round}}$ to itself, and it is a symplectomorphism. Hence, it is enough to find a family $A: S^n \to U(n+1)$ of unitary matrices such that $A(q) (e_1, 0) = (q, 0)$. This is equivalent to finding $n$ families of vectores $z_i (q)\in \C^{n+1}$ such that $\langle q, z_1(q), ..., z_n(q) \rangle$ is a unitary basis of $\C^{n+1}$ for every $q \in S^n$. Such a family of vectors can be found since the complexified tangent bundle $T S^n \otimes \C$ is trivial\footnote{This is a well-known fact, a fairly simple way to see this is from the existence of a Lagrangian immersion $S^n = \{(x,y) \in \R^n \times \R \mid \|x\|^2 + y^2=1\}  \mapsto (1+iy) x \in \C^n$ (see \cite[Example 13.2.4]{mcduffsalamon-alt}).}. This completes the proof of Theorem \ref{theorem:ellipsoid-1}.

\subsubsection{Diagonal actions on open books}
\label{sec:diag-acti-open}

If one has an $\R/\Z$-action on the page $V$, we have the induced diagonal $\R/\Z$ action on the open book $M$. The action class we used in \S\ref{sec:proof-of-theorem-OB-1} is obtained from the trivial action on $V$.

The goal in this section is to prove that the action classes $A_{\pm}:M \times \R /\Z \to M$ lie in the image of the BV operator $\Delta$, in $H(\Lambda_{c})$ for specific $c$. To this end, it is useful to give a description of the open book $M$ as a subset of $V \times \C$.

Let $(V,\bd V)$ be a smooth manifold with boundary, and let $f:V\to [0,1]$ be a smooth function such that $f(1)=\bd V$ is a regular level set. Define:
\begin{equation*}
  \mathrm{OB}(V)=\set{(v,z)\in V\times \C:f(v)+\abs{z}^{2}=1}.
\end{equation*}
The bordism classes we consider are related to a circle action of the following form: suppose that $(t,v)\mapsto \zeta_{t}(v)$ is a circle action on $V$ which preserves the level sets of $f$, and consider the resulting circle action:
\begin{equation*}
  (t,v,z)\mapsto (\zeta_{t}(v),e^{2\pi i t}z)
\end{equation*}
which acts on $\mathrm{OB}(V)$.

First, we deform the construction of $\mathrm{OB}(V)$ as follows. Define:
\begin{equation*}
  M=\set{(v,z)\in V\times \C:f(v)+\rho(\abs{z})^{2}=1},
\end{equation*}
where $\rho$ is non-decreasing, $\rho(x)=0$ for $x\le 1$, $\rho(x)=2$ for $x\ge 2$, and $\rho'(x)>0$ whenever $\rho(x)\in (0,1]$. Then:
\begin{lemma}\label{lemma:first-cobordism-step}
  The identity map $\id:\mathrm{OB}(V)\to \mathrm{OB}(V)$ is cobordant to the smooth map $R:M\to \mathrm{OB}(V)$ given by:
  \begin{equation*}
    R(v,z)=(v,g(z)z)\text{ where }g(z)=\abs{z}^{-1}\rho(\abs{z}).
  \end{equation*}
  which is well-defined; i.e., $f(v)+\abs{g(z)z}^{2}=1$.
\end{lemma}
\begin{proof}
  We define the cobordism explicitly: $$Q=\set{(v,z,s)\in V\times \C\times [0,1]:((1-s)\abs{z}+s\rho(\abs{z}))^{2}+f(v)=1}.$$
  A short computation shows that $Q$ is cut transversally. Moreover, the boundary $\bd Q=\set{s=0,1}$ is identified with $\mathrm{OB}(V)\sqcup M$ in the obvious way. Define a map:
  \begin{equation*}
    S:Q\to \mathrm{OB}(V)
  \end{equation*}
  by the formula:
  \begin{equation*}
    S(v,z,s)=(v,g_{s}(z)z)\text{ where }g_{s}(z)=(1-s)+s\abs{z}^{-1}\rho(\abs{z}),
  \end{equation*}
  which is a smooth function. This provides the desired cobordism, since $S(v,z,0)$ is identified with the identity map and $S(v,z,1)$ with $R(v,z)$.
\end{proof}

Next, we consider the action class:
\begin{equation*}
  A:\R/\Z\times \mathrm{OB}(V)\to \mathrm{OB}(V)\text{ given by }(t,v,z)\mapsto (\zeta_{t}(v),e^{2\pi i t}z).
\end{equation*}
Our goal is prove that $[A]$ lies in the image of $\Delta$.

It follows easily from Lemma \ref{lemma:first-cobordism-step} that $A$ is cobordant to the class:
\begin{equation*}
  B:\R/\Z\times M\to \mathrm{OB}(V)\text{ given by }(t,v,z)\mapsto (\zeta_{t}(v),e^{2\pi i t}g(z)z),
\end{equation*}
indeed, one can re-use the same cobordism $Q$, and $A,B$ extend to the cobordism.

\emph{Remark}. During the cobordism, the lengths of loops never exceeds the maximum length of loops in the class of $A$; indeed, each loop appearing is exactly one of the loops appearing in the family $A$.

Thus, in order to prove that $[A]$ lies in the image of $\Delta$, it is sufficient to prove that $[B]$ lies in the image of $\Delta$.

The next step is to pick a smooth non-increasing cut-off function $\psi:\R\to [0,2]$ such that $\psi(x)=2$ for $x\le 0$, $\psi(x)=0$ for $x\ge 1$, and $\psi'(x)\ne 0$ if $\psi(x)\in (0,1]$. Then we define:
\begin{equation*}
  N=\set{(v,z)\in V\times \C:f(v)+\psi(\abs{z})^{2}+\rho(\abs{z})^{2}=1}.
\end{equation*}
It is important to note that $\psi(\abs{z})$ and $\rho(\abs{z})$ are supported in different regions (the former in $\abs{z}\le 1$ and the latter in $\abs{z}\ge 1$). A quick computation shows that $N$ is cut transversally.

\begin{lemma}
  The map $B:\R/\Z\times M\to \mathrm{OB}(V)$ is cobordant to the map:
  \begin{equation*}
    C_{\vartheta}:\R/\Z\times N\to \mathrm{OB}(V)
  \end{equation*}
  given by:
  \begin{equation*}
    C_{\vartheta}(t,v,z)=\left\{
      \begin{aligned}
        &(\zeta_{t}(v),e^{2\pi i \vartheta}\psi(\abs{z}))\text{ if }\abs{z}\le 1,\\
        &(\zeta_{t}(v),e^{2\pi i t}g(\abs{z})z)\text{ if }\abs{z}\ge 1,
      \end{aligned}
    \right.
  \end{equation*}
  where $g(\abs{z})=\abs{z}^{-1}\rho(\abs{z})$.
\end{lemma}
\begin{proof}
  It is clear from the construction that the map is smooth and agrees on the overlap (since $\psi(x)$ and $\rho(x)$ vanish to all orders when $x=1$). Moreover, the map is well-defined and valued in $\mathrm{OB}(V)$, as can be checked by a direct computation. It remains to prove the map is cobordant to $B$.

  To see this, we define the cobordism explicitly:
  \begin{equation*}
    X=\set{(v,z,s)\in V\times \C\times [0,1]:f(v)+s^{2}\psi(\abs{z})^{2}+\rho(\abs{z})^{2}=1}.
  \end{equation*}
  It is not hard to see that this is cut transversally; when $s=0$, it is already known to be transverse, and when $s\ne 0$, the derivative with respect to $s$ is non-zero where $\psi(\abs{z})\ne 0$; elsewhere it is known to be transverse by the same argument used to prove $M$ was cut transversally.

  Now we we define the map $T_{\vartheta}:\R/\Z\times X\to \mathrm{OB}(V)$ by the formula:
  \begin{equation*}
    T_{\vartheta}(t,v,z,s)=\left\{
      \begin{aligned}
        &(\zeta_{t}(v),e^{2\pi i \vartheta}s\psi(\abs{z}))\text{ if }\abs{z}\le 1,\\
        &(\zeta_{t}(v),e^{2\pi i t}g(\abs{z})z)\text{ if }\abs{z}\ge 1,
      \end{aligned}
    \right.
  \end{equation*}
  where $g(\abs{z})=\abs{z}^{-1}\rho(\abs{z})$. This map is easily seen to be a cobordism between $B$ and $C_{\vartheta}$, as desired.
\end{proof}
\emph{Remark}. Unlike the cobordism between $A$ and $B$, the cobordism between $B$ and $C_{\vartheta}$ introduces new loops; namely, it introduces the loops
\begin{equation}\label{eq:new-loops}
  t\mapsto (\zeta_{t}(v),e^{2\pi i\vartheta}r)\text{ where }f(v)=1-r^{2}
\end{equation}
Thus, if $c$ is greater than the maximum length of loops which appear in the class $A$, and in the family \eqref{eq:new-loops} (for any chosen $\vartheta$), then the cobordism from $A$ to $C_{\vartheta}$ occurs within $H(\Lambda_{c})$. However, if the action on the page $V$ is trivial, these loops are constant and therefore do not increase the length threshold; this observation is relevant to obtaining the stated bound in Lemma \ref{lemma:action_is_BV}.

Our final result in this subsection is that $[C_{\vartheta}]=\Delta([D_{\vartheta}])$ for a family of loops $D_{\vartheta}$. To show this, define:
\begin{equation*}
  P=\set{(v,r)\in V\times (0,\infty):f(v)+\psi(r)^{2}+\rho(r)^{2}=1}\subset N;
\end{equation*}
essentially since $\psi(0)=2$ is larger than $1$, it is easy to see that $P$ is a closed manifold. Then we define: $D_{\vartheta}:\R/\Z\times P\to \mathrm{OB}(V)$ by:
\begin{equation*}
  D_{\vartheta}(t,v,r)=\left\{
    \begin{aligned}
      &(\zeta_{t}(v),e^{2\pi i \vartheta}s\psi(r))\text{ if }\abs{z}\le 1,\\
      &(\zeta_{t}(v),e^{2\pi i t}g(r)r)\text{ if }\abs{z}\ge 1.
    \end{aligned}
  \right.
\end{equation*}
\begin{lemma}\label{lemma:open_books_bv_2}
  It holds that $\Delta([D_{\vartheta}])=[C_{\vartheta}]$ within $H(\Lambda_{c})$ provided $C_{\vartheta}\in Z(\Lambda_{c})$.
\end{lemma}
\begin{proof}
  Observe that $\Delta(D_{\vartheta})$ is represented by $D_{\vartheta}':\R/\Z\times \R/\Z\times P\to \mathrm{OB}(V)$ given by:
  \begin{equation*}
    D'(t,\theta,v,r)=\left\{
      \begin{aligned}
        &(\zeta_{t-\theta}(v),e^{2\pi i \vartheta}s\psi(r))\text{ if }\abs{z}\le 1,\\
        &(\zeta_{t-\theta}(v),e^{2\pi i (t-\theta)}g(r)r)\text{ if }\abs{z}\ge 1.
      \end{aligned}
    \right.
  \end{equation*}
  Now consider the diffeomorphism $h:\R/\Z\times P\to N$ given by:
  \begin{equation*}
    h(\theta,v,r)=(\zeta_{-\theta}(v),e^{-2\pi i\theta}r).
  \end{equation*}
  Under this diffeomorphism we have:
  \begin{equation*}
    C_{\vartheta}(t,h(\theta,v,r))=\left\{
      \begin{aligned}
        &(\zeta_{t-\theta}(v),e^{2\pi i \vartheta}\psi(r))\text{ if }\abs{z}\le 1,\\
        &(\zeta_{t-\theta}(v),e^{2\pi i (t-\theta)}g(r)r)\text{ if }\abs{z}\ge 1,
      \end{aligned}
    \right.
  \end{equation*}
  so $[C_{\vartheta}]=[D_{\vartheta}']=\Delta([D_{\vartheta}])$. Note that two maps $\R/\Z\times P_{i}\to \mathrm{OB}(V)$, $i=0,1$, differing by a diffeomorphism $P_{0}\to P_{1}$ are cobordant in a trivial sense; Such a cobordism does not change any lengths of loops, so the cobordism from $C_{\vartheta}$ to $\Delta(D_{\vartheta})$ occurs within $H(\Lambda_{c})$.
\end{proof}

\subsubsection{Proof of Theorem \ref{theorem:ellipsoid-2}}
\label{sec:proof_thm_ell_2}

We begin with an auxiliary lemma about the Hopf flow. Consider the Hopf $\R /\Z$-action on $S^3\subset \C^{2}$ given by:
$$
\zeta_{t}(z_1, z_2) = (e^{2\pi i t} z_1, e^{2\pi i t} z_2).
$$
Then:
\begin{lemma}\label{lemma:hopf_flow}
The Hopf action is homotopic to the trivial $\R/\Z$-action through the loops of length $\leq 2\pi$.
\end{lemma}
\begin{proof}
After conjugating the Hopf action with $\psi(z_1,z_2) = (z_1, \Bar{z}_2)$ we get:
$$
t \cdot (z_1, z_2) = (e^{2\pi i t} z_1, e^{-2\pi i t} z_2).
$$
This $\R/\Z$-action is the restriction of the $\mathrm{SU}(2) \cong S^3$ action to the unit circle $\set{(z,0):z\in S^{1}}\subset S^{3}$. Consider a homotopy $h: \R/ \Z \times [0,1] \to S^3$ through loops of length $\leq 2\pi$ such that $h(t,0) = (e^{2\pi i t}, 0)$, $h(t,1) = (1,0)$. The homotopy $h$ induces a homotopy:
$$
H(z_1, z_2, t, s) = h(t,s) \cdot (z_1,z_2),
$$
which satisfies the conclusion of the Lemma.
\end{proof}

Similarly to \S\ref{sec:proof_thm_ell_1} one gets that the radius $a$ codisk bundle $\Omega_{a}^{\mathrm{round}}$ of the round metric is a subset of $\Omega_a$, where $\Omega_a$ is the unit codisk bundle determined by the embedding of: $$ \set{x_{0}^{2}+x_{1}^{2}+\dots+a^{-2}(x_{n-3}^{2}+x_{n-2}^{2} + x_{n-1}^{2}+x_{n}^{2})= 1}\subset \R^{n+1}.$$ Since $Gr([S^n]; \Omega_{a}^{\mathrm{round}}) = 2\pi a$ we get:
$$
\mathrm{Gr}([\mathrm{pt}]; \Omega_a) \geq Gr([S^n]; \Omega_a) \geq 2\pi a.
$$
To prove the upper bound, one appeals to the results of \S\ref{sec:diag-acti-open} with the diagonal $\R/\Z$-action given by:
\begin{equation}\label{eq:action-ellipsoid-2}
  \zeta_{t}(x_0, ...., x_{n-4}, z_1, z_2) = (x_0, ..., x_{n-4},e^{2\pi i t} z_1, e^{2\pi i t} z_2).
\end{equation}
One sees $S^n$ as an open book with page $V \subset \mathbb{D}^{n-3} \times \C$ given by:
$$
V= \{x_0^2 + \dots + x_{n-4}^2 + a^{-2} |z_1|^2 \leq 1\}.
$$
Then we have:
$$
S^n = \mathrm{OB}(V) = \{ (x, z_1, z_2) \in V \times \C \mid  f(x, z_1) + a^{-2} |z_2|^2 =1\},
$$
where $f(x,z) = x_0^2 + \dots + x_{n-4}^2 + a^{-2} |z|^2$. An $\R / \Z$ action on $V$ is given by the rotation of the $\C$ coordinate, and it preserves the levels of $f$, hence from Lemma \ref{lemma:open_books_bv_2}, the induced action class $A$ is in the image of $\Delta$; in brief, there exists a class $B$ such that $[A] = \Delta [B ] \in H(\Lambda_{2 \pi a})$. The stated length threshold of $2\pi a$ is obtained by inspection of the length thresholds in \ref{sec:diag-acti-open}.

Finally, the contractibility of the Hopf flow on $S^3$ implies that the $\R/ \Z$-action \eqref{eq:action-ellipsoid-2} is homotopic to the trivial $\R / \Z$ action, hence $[A] = [S^n] \in H(\Lambda_{2 \pi a})$. Then Theorem \ref{theorem:corollary} applied to the class $\beta =[T^*S^n]$, $\alpha_1 = [B]$ (and $\alpha_2 = [S^n]$) produces $\mathrm{Gr}([\mathrm{pt}]; \Omega_a) \leq 2\pi a$, as desired.

\subsubsection{Proof of Theorem \ref{theorem:prod-torus} and Proposition \ref{proposition:camel}}
\label{sec:proof-prod-torus}
The proof of Theorem \ref{theorem:prod-torus} follows very similar lines to the proofs of Theorem \ref{theorem:OB-1} in the case when $\bd V=\emptyset$. Indeed, one can define $B_{-}^{d}:V\times T^{d-1}\times \R/\Z\to V\times T^{d}$ by:
\begin{equation*}
  B_{-}(v,x,t)=(v,x_{1},\dots,x_{d-1},-t),
\end{equation*}
and, for $k<d$, define $B_{+}^{k}:V\times T^{d-k-1}\times \R/\Z\to V\times T^{d}$ by
\begin{equation*}
  B_{-}^{k}(x,t)=(v,0,\dots,0,x_{1},\dots,x_{d-k-1},t).
\end{equation*}
By Lemma \ref{lemma:chas-sullivan-action-class-1} and \ref{lemma:chas-sullivan-action-class-3}, it holds that:
\begin{equation*}
  \Delta [B_{-}^{d}]\ast \Delta [B_{+}^{k}]=[V\times T^{d-k}]\text{ in }\Lambda_{E_{-}+E_{+}^{k}}
\end{equation*}
where $[V\times T^{d-k}]$ is the class of $(v,x)\mapsto (v,0,\dots,0,x_{1},\dots,x_{d-k})$.

Since $[V\times T^{d-k}]$ has non-zero homological intersection number with the class $[T^{k}]$ represented by the map $f:T^{k}\to V\times T^{d}$ given by $x\mapsto (v_{0},x_{1},\dots,x_{k})$, we conclude from Theorem \ref{theorem:corollary} that $\mathrm{Gr}([T^{k}],\Omega)\le E_{-}+E_{+}^{k}$, as desired.

In the remainder of this subsection, we will explain how Theorem \ref{theorem:prod-torus} can be used to recover the camel theorem (Proposition \ref{proposition:camel}).

Recall the set-up of Proposition \ref{proposition:camel}; we had defined the symplectic manifold:
\begin{equation*}
  X=(W\setminus \set{x_{1}=0})\cup \set{x_{n}^{2}+y_{n}^{2}\le\pi^{-1}\epsilon},
\end{equation*}
where $W=\R/\Z\times \R^{2n-1}$, for $n>1$, and claimed that $\mathrm{Gr}(f,X)\le \epsilon$ provided that $f:\R/\Z\to X$ has non-trivial winding number.

We argue by contradiction: suppose there is a family of balls: $$F:\R/\Z\times B(a)\to X$$ with $a>\epsilon$, such that $t\mapsto F(t,0)=(t,0,\dots,0)$ (here we appeal to the fact that the winding number classifies the free homotopy class of a loop). It is convenient to fix some small parameter $\delta>0$. We begin with an auxiliary lemma, which converts $X$ into a space more compatible with Theorem \ref{theorem:prod-torus}.
\begin{lemma}
  There exists a symplectic isotopy $\psi_{s}:X\to W$ so that, abbreviating $y_{i}=y_{i}\circ \psi_{1}\circ F(t,z)$ and $x_{i}=x_{i}\circ \psi_{1}\circ F(t,z)$, we have:
  \begin{enumerate}
  \item\label{camel-transformation-1} $\psi_{1}\circ F(t,0)=(t,0,\dots,0)$,
  \item\label{camel-transformation-2} $x_{i}\in (-1/2,1/2)$ for $i=2,3,\dots,n$,
  \item\label{camel-transformation-3} $y_{n}>-\epsilon/2-2\delta$,
  \item\label{camel-transformation-4} $x_{1}=0\implies y_{n}<\epsilon/2+\delta$.
  \end{enumerate}
\end{lemma}
\begin{proof}
  Condition \ref{camel-transformation-1} is already satisfied. Condition \ref{camel-transformation-2} is easily satisfied using the transformation $(x_{i},y_{i})\mapsto (a^{-1}x_{i},ay_{i})$ where $a>\max 2\abs{x_{i}}$, for each $i=2,\dots,n$, where the maximum is taken over the image of the family of balls. Let us denote $\rho_{a}$ the result of this transformation. It is convenient to also take $\pi a^{2}>4\epsilon$.

  Now we focus on attaining \ref{camel-transformation-4}. The transformation used to achieve \ref{camel-transformation-2} implies that:
  \begin{equation*}
    x_{1}\circ \rho_{a}(F(t,z))=0\implies \rho_{a}(F(t,z))\subset \set{a^{2}x_{n}^{2}+a^{-2}y_{n}^{2}\le \pi^{-1}\epsilon}=R.
  \end{equation*}
  Since the projection $R_n$ of $R$ to $x_{n}, y_{n}$-plane is an ellipse of area $\epsilon$ and is contained in $S=\set{x_{n}\in (-1/2,1/2)}\subset \R^2$, one can find a symplectic isotopy $\varphi_{s}$ of the strip $S$ so $\varphi_{1}$ sends $R_n$ into the rectangle: $$(-1/2,1/2)\times (-\epsilon/2-\delta,\epsilon/2+\delta),$$ since the rectangle has strictly larger area than the ellipse. One can arrange that $\varphi_{1}(0)=0$. The isotopy $\varphi_{s}$ extends as a product isotopy $\id\times \varphi_{s}$ to all of $W$. Then $\varphi_{1}\circ \rho_{a}\circ F$ satisfies \ref{camel-transformation-1}, \ref{camel-transformation-2}, \ref{camel-transformation-4}. It remains only to ensure \ref{camel-transformation-3} holds.

  For this step, we consider the isotopy generated by $H=f(x_{1})x_{n}$, where $f(0)=0$ and $f'(0)=0$, and $f(x_{1})=1$ outside a small neighborhood $U$ of $x_{1}=0$. The Hamiltonian vector field generated by this function is:
  \begin{equation*}
    f(x_{1})\partial_{y_{n}}+f'(x_{1})x_{n}\partial_{y_{1}}
  \end{equation*}
  This vector field equals $\bd_{y_{n}}$ outside of $U$, and equals $0$ when $x_{1}=0$. We pick the neighborhood $U$ small enough that:
  \begin{equation*}
    \varphi_{1}\circ \rho_{a}\circ x_{1}\in U\implies \varphi_{1}\circ \rho_{a}\circ y_{n}\ge -\epsilon/2-2\delta,
  \end{equation*}
  which can be achieved by a simple compactness argument, since we already know \ref{camel-transformation-4} holds.

  The flow by $X_{H}$ increases the $y_{n}$ coordinates when $x_{1}\not\in U$;  thus flowing long enough will therefore satisfy condition \ref{camel-transformation-3} everywhere. Denoting by $\eta_{1}$ the long time-flow by $X_{H}$, we set $\psi_{1}=\eta_{1}\circ \varphi_{1}\circ \rho_{a}$ to complete the proof.
\end{proof}
Henceforth replace $F$ by $\psi_{1}\circ F$.

\begin{proof}[Proof of Proposition \ref{proposition:camel}]
Introduce the domain $\Omega\subset T^{*}T^{n}$ determined by the conditions:
\begin{enumerate}[label=(\alph*)]
\item\label{domain-camel-1} $p_{n}\ge -\epsilon/2-\delta$,
\item\label{domain-camel-2} $q_{1}=0\implies p_{n}\le \epsilon/2+2\delta$,
\end{enumerate}
using canonical coordinates $(p_{1},q_{1},\dots,p_{n},q_{n})$. Define the map:
\begin{equation*}
  (x_{1},y_{1},\dots,x_{n},y_{n})\mapsto T^{*}T^{N}\text{ by }x_{i}=q_{i}\text{ and }y_{i}=-p_{i}.
\end{equation*}
The family of balls $F$ projects to a family $F:\R/\Z\times B(a)\to T^{*}T^{n}$ which restricts to embeddings $z\mapsto F(t,z)$, since we have arranged that $F$ is valued in the region $x_{i}\in (-1/2,1/2)$ for $i=2,\dots,n$. Moreover, by construction, $F$ is actually a family of balls in $\Omega$. Thus the existence of $F$ implies $\mathrm{Gr}(f,\Omega)\ge a$, where $q\circ f(t)=(t,0,\dots,0)$.

On the other hand, Theorem \ref{theorem:prod-torus} with $k=1$ and $d=n$ yields:
\begin{equation*}
  \mathrm{Gr}(f,\Omega)\le \epsilon+3\delta,
\end{equation*}
as can be seen by approximating the domain $\Omega$ satisfying \ref{domain-camel-1} and \ref{domain-camel-2} from within by smooth fiberwise starshaped domains. Thus we conclude $a\le \epsilon+3\delta$. Since $\delta$ could be taken arbitrarily small, we have $a\le \epsilon$, as desired.
\end{proof}

\subsubsection{Proof of Theorem \ref{theorem:non-orientable-surface}}
\label{sec:proof-non-orientable-surface}

The proof is straightforward application of Theorem \ref{theorem:corollary} using a simple string topology computation: if $q$ is a loop such that $q^{*}T\Sigma$ is non-orientable, then $\Delta(q)\ast \Delta(\bar{q})=\mathrm{pt}$ in $H(\Lambda_{E})$, where $\mathrm{pt}$ denotes the class represented by a single constant loop. Since $\mathrm{pt}=\mathfrak{i}(T^{*}M_{\mathrm{pt}})$, and $T^{*}M_{\mathrm{pt}}$ has non-zero homological intersection number with the zero section, the desired result $\mathrm{Gr}([\Sigma],\Omega)\le E$ follows. \hfill$\square$

\section{The Floer cohomology persistence module}
\label{sec:floer-cohom-pers}

In this section we explain the aspects of our paper pertaining to the Floer cohomology persistence module, with the goal of proving Theorem \ref{theorem:main-floer}.

\subsection{Hamiltonian functions and isotopies}
\label{sec:hamilt-functions-isotopies}

\subsubsection{Class of Hamiltonian functions}
\label{sec:class-hamilt-funct}

As in \S\ref{sec:introduction}, throughout we fix a Liouville domain $\bar{\Omega}$, and denote by $W$ its completion. The structure of $W$ as a completion yields a distinguished ``convex end'' which is symplectomorphic to the positive half of the symplectization of $\bd\Omega$ with the contact structure $\ker(\lambda|_{\bd\Omega})$. This yields a function $r$ which is one-homogeneous with respect to the Liouville flow in the convex end, and which satisfies $\bd\Omega=\set{r=1}$. It is well-known that the Hamiltonian vector field $X_{r}$ is equivariant with respect to the Liouville flow, and restricts to $\bd\Omega$ as the Reeb vector field for the contact form $\lambda|_{\bd\Omega}$.

Let us introduce the notation $\Omega(r_{0})=\set{r< r_{0}}$, so $\Omega=\Omega(1)$.

As in \S\ref{sec:floer-cohom-pers-1}, this distinguishes a class of Hamiltonian functions: define $\mathscr{H}$ to be those smooth functions $H:W\to \R$ such that $H=ar$ holds when $r\ge r_{0}$, for some $r_{0}$ (which depends on $H$).

\subsubsection{Smooth families of Hamiltonians}
\label{sec:smooth-famil-hamilt}

One says that a family $H_{\tau}\in \mathscr{H}$, where $\tau$ is valued in a smooth manifold $T$ (potentially with boundary and corners) is \emph{smooth} provided:
\begin{enumerate}
\item the map $(\tau,w)\in T\times W\mapsto H_{\tau}(w)$ is smooth,
\item $H_{\tau}(w)=c_{\tau}r$ for $r\ge r_{0}$, for some $r_{0}$ which depends continuously on $\tau$, and where the slope $c_{\tau}$ varies smoothly with $\tau$.
\end{enumerate}

\subsection{Hamiltonian connections on surfaces}
\label{sec:hamilt-conn-surf}

The Floer theory PDEs used in this text are based on the notion of a \emph{Hamiltonian connection} as described in \cite[\S8]{mcduff-salamon-book-2012}; see also \cite[\S2.2.3]{alizadeh-atallah-cant-arXiv-2023} which works in a similar context to the present paper.

\subsubsection{Domains}
\label{sec:domains}

Let us agree that a \emph{domain} is a compact smooth surface $\Sigma$ with boundary $\bd\Sigma$ and the data of:
\begin{enumerate}
\item\label{domain1} a complex structure $j$, namely a smooth section of $\mathrm{End}(T\Sigma)$ whose square is $-1$, and,
\item\label{domain2} two collections $\Gamma_{\pm}$ of $j$-holomorphic embeddings of a closed disk $D(1)\to \Sigma$ whose images are mutually disjoint, and are also disjoint from $\bd\Sigma$.
\end{enumerate}
A \emph{smooth family of domains}, parametrized by $\sigma\in S$, is simply the parametric version of \ref{domain1} and \ref{domain2}; more prosaically, this means that one considers the product $S\times \Sigma$, with projection $\pi$ onto $\Sigma$, and picks:
\begin{enumerate}[label=(\arabic*$^{\prime}$)]
\item a smooth section $j$ of $\mathrm{End}(\pi^{*}T\Sigma)$ whose square is $-1$, and
\item two collections $\Gamma_{\pm}$ of smooth embeddings $S\times D(1)\to S\times \Sigma$ which sends $\set{\sigma}\times D(1)$ into $\set{\sigma}\times \Sigma$ and which satisfies \ref{domain2} when restricted to the disk $\set{\sigma}\times D(1)$, for the appropriate complex structure.
\end{enumerate}
A prototypical example of such a family is given by $\Sigma=\mathbb{CP}^{1}$, and $S$ the product of $k$-copies of $T\mathbb{CP}^{1}$ with an appropriate closed set removed. To each collection of $k$ non-zero tangent vectors, there exist $k$ uniquely determined biholomorphisms of $\mathbb{CP}^{1}$ which send $(0,1)\in T\mathbb{CP}^{1}$ to the chosen tangent vectors. These biholomorphisms take $D(1)\subset \mathbb{CP}^{1}$ onto $k$ biholomorphically embedded disks. There is a closed subset of $S$ for which these $k$ disks intersect; after removing this closed subset, one has a smooth family.

\subsubsection{Remark on notation and terminology}
\label{sec:remark-notation-terminology}

If either $j$ or $\Gamma_{\pm}$ is not germane to the discussion, we will suppress it from the notation, and simply refer to $S\times \Sigma$ as a family of domains.

It is also convenient to sometimes refer to $\Gamma_{\pm}$ as punctures, in which case we forget the holomorphic embedding and consider only the evaluation at the center of the disk. This produces a submanifold $\Gamma_{\pm}\subset S\times \Sigma$ of codimension $2$ which is a disjoint union of sections of $S\times \Sigma\to S$; we call a connected component $\zeta\subset \Gamma_{\pm}$ a puncture.

Frequently it is expedient to define $\Sigma$ as the already punctured surface: $\Sigma=\R\times \R/\Z$, $\Sigma=\C$, or $\Sigma=\C\setminus \set{z_{1},z_{2}}$. In these cases, we should understand the underlying closed surface to be the Riemann sphere $\mathbb{C}P^{1}$.

At other times, we will refer to $\Gamma_{\pm}$ as cylindrical ends, in which case we precompose the holomorphic embedding of the disk $D(1)\to \Sigma$ with the conformal reparametrization $(s,t)\mapsto e^{\mp 2\pi(s+it)}$. Note that the ends in $\Gamma_{+}$ are modelled on $[0,\infty)\times \R/\Z$ while those in $\Gamma_{-}$ are modelled on $(-\infty,0]\times \R/\Z$. We will sometimes appeal to cylindrical coordinates near a puncture $\zeta\subset \Gamma_{\pm}$, and this is to be understood in the context of this remark.

\subsubsection{Connection 1-forms}
\label{sec:connection-1-forms}

Let $(S\times \Sigma,\Gamma_{\pm})$ be a smooth family of domains. A connection 1-form is a singular 1-form $\mathfrak{a}$ on $S\times \Sigma \times W$ such that:
\begin{enumerate}
\item\label{c1f1} above a compact coordinate chart $(\sigma, z=x+iy)$ on $S\times \Sigma$ disjoint from the punctures, $\mathfrak{a}=H_{\sigma,x,y}\d x+K_{\sigma,x,y}\d y$, where $H,K$ are smooth families in $\mathscr{H}$;
\item\label{c1f2} for each puncture $\zeta\subset \Gamma_{\pm}$, one has $\mathfrak{a}=H_{\zeta,t}\d t$ near $\zeta$, where $H_{\zeta,t}$ is a smooth family in $\mathscr{H}$ on $\R/\Z$; this uses the appropriate cylindrical coordinates near $\zeta$. Moreover, the Hamiltonian isotopy generated by $H_{\zeta,t}$ has non-degenerate 1-periodic orbits.
\end{enumerate}
One should think that $\mathfrak{a}$ has prescribed singularities near the punctures, with a requirement on the dynamics of the induced asymptotic Hamiltonian system. Note that, because the asymptotic Hamiltonian system commutes with the Liouville flow outside of a compact set, the requirement that the orbits are non-degenerate forces all of the orbits to remain in a compact set, and hence there are only finitely many orbits.

\emph{Remark}. One slightly subtle requirement which plays a role in our proof of the maximum principle, Proposition \ref{proposition:maximum-principle}, is the following: the connection one form should appear in the form $\mathfrak{a}=H_{\zeta,t}\d t$, where $H_{\zeta,t}$ is non-degenerate, on cylindrical ends, \emph{and on some number of finite length cylinders}, such that the domain obtained by removing the cylindrical ends and the finite length cylinders is contained in a fixed compact subset of $\Sigma$. This is because our maximum principle is based on confining Floer cylinders, and then bounding the distance of other points to a region where $\mathfrak{a}=H_{\zeta,t}\d t$. This set-up is relevant when, for instance, gluing together two continuation cylinders.

\subsubsection{Perturbation 1-forms}
\label{sec:perturbation-1-forms}

Connection 1-forms are used to define a certain PDE. It is well-understood that Floer theory relies on the moduli spaces of solutions to this PDE being transversally cut out. We introduce in this section a perturbation term $\mathfrak{p}$ to obtain the requisite transversality.

Given a smooth family of domains $(S\times \Sigma,\Gamma_{\pm})$, a perturbation 1-form is a smooth $1$-form $\mathfrak{p}$ on $S\times \Sigma\times W$ such that:
\begin{enumerate}
\item above a compact coordinate chart $(\sigma,z=x+iy)$ on $S\times \Sigma$ disjoint from the punctures, $\mathfrak{p}=h_{\sigma,x,y}\d x+k_{\sigma,x,y}\d y$, where $h,k$ are smooth and uniformly bounded in $C^{1}$ families of functions $W\to \R$, and,
\item $\mathfrak{p}$ vanishes above a neighborhood of the punctures $\Gamma_{\pm}$.
\end{enumerate}
The boundedness of the functions appearing in $\mathfrak{p}$ is used in an essential way when establishing a priori estimates on the energy integrals in \S\ref{sec:energy-integral}; see the comment at the end of \S\ref{sec:priori-energy-estim}.

\subsubsection{Families of connection and perturbation 1-forms}
\label{sec:famil-conn-pert}

Let $S\times \Sigma$ be a smooth family of domains. One can pull this back to $T\times S\times \Sigma$ for any manifold $T$ to obtain a new family of domains. A connection or perturbation 1-form on $T\times S\times \Sigma\times W$ is, by definition, a smooth family of connection or perturbation 1-forms on $S\times \Sigma\times W$ parametrized by $T$.

In this sense, we may speak of homotopies of 1-forms by setting $T=[0,1]$.

\subsubsection{Connection associated to a 1-form}
\label{sec:conn-assoc-1}

Let $\mathfrak{a},\mathfrak{p}$ be connection and perturbation 1-forms on $S\times \Sigma\times W$. Define:
\begin{equation*}
  \mathfrak{H}=TW^{\perp \Omega}\text{ where }\Omega=\mathrm{pr}_{W}^{*}\omega-\d\mathfrak{a}-\d\mathfrak{p}.
\end{equation*}
Then $\mathfrak{H}$ is an Ehresmann connection on $S\times \Sigma\times W\to S\times \Sigma$.

The coordinate distribution $TS$ is contained in $\mathfrak{H}$. On the other hand, the horizontal lifts of $\bd_{x},\bd_{y}\subset T\Sigma$ are given by:
\begin{equation*}
  \bd_{x}^{\mathfrak{H}}=\bd_{x}+X_{H+h}\text{ and }\bd_{y}^{\mathfrak{H}}=\bd_{y}+X_{K+k},
\end{equation*}
where $\mathfrak{a}+\mathfrak{p}=(H+h)\d x+(K+k)\d y$.

The pullback of $\mathfrak{H}$ to $\set{\sigma}\times \Sigma\times W\to \set{\sigma}\times \Sigma$ will be denoted by $\mathfrak{H}_{\sigma}$ in the sequel, and should be considered as a family of connections on $\Sigma\times W\to \Sigma$.

\subsection{A general form of Floer's equation}
\label{sec:general-form-floers}

In this section, we will explain a general form of Floer's equation which will specialize to the various cases used in the proofs of our main results.

\subsubsection{Almost complex structures}
\label{sec:almost-compl-struct}

Let $S\times \Sigma$ be a family of domains, with punctures $\Gamma_{\pm}$. An almost complex structure on $S\times \Sigma\times W$ is a section $J$ of the pullback bundle $\mathrm{pr}_{W}^{*}\mathrm{End}(TW)\to S\times \Sigma\times W$ satisfying $J^{2}=-1$.

It is convenient to think of this as a family $J_{\sigma,z}$ of almost complex structures on $W$ parametrized by $(\sigma,z)\in S\times \Sigma$. We require that:
\begin{enumerate}
\item For $(\sigma,z)$ near a puncture $\zeta\subset \Gamma_{\pm}$, we have $J_{\sigma,z}=J_{\zeta}$ for an almost complex structure $J_{\zeta}$ on $W$.
\item $J$ is $\omega$-tame, i.e., $\omega(v,Jv)$ is a positive quadratic form for $v\in TW$.
\item Each $J_{\sigma,z}$ is invariant under the Liouville flow outside of $\Omega\subset W$.
\end{enumerate}

\subsubsection{Lagrangian boundary conditions}
\label{sec:lagr-bound-cond}

Suppose that the surface $\Sigma$ has boundary $\bd\Sigma$. For the purposes of this text, a \emph{Lagrangian boundary condition} is a smoothly varying family of Lagrangian submanifolds $L_{\sigma,z}\subset W$ where $(\sigma, z)\in S\times \bd \Sigma$. Here ``smoothly varying'' means that:
\begin{equation*}
  L=\set{(\sigma,z,p):p\in L_{\sigma,z}}\subset S\times \bd\Sigma \times W
\end{equation*}
should be a smooth submanifold. A submanifold is required to be properly embedded and without boundary. We require two additional properties:
\begin{enumerate}
\item $L_{\sigma,z}$ is weakly exact, i.e., $\omega$ vanishes on disks with boundary on $L_{\sigma,z}$.
\item outside of $\Omega$, the Liouville vector field is tangent to $L_{\sigma,z}$.
\end{enumerate}

In this text, Lagrangian boundary conditions will only be used in the comparison between string topology of $M$ and Floer cohomology of $T^{*}M$, and we will have that $L_{\sigma,z}=T^{*}M_{q(\sigma,z)}$ where $q(\sigma,z)$ depends smoothly on $\sigma,z$.

\subsubsection{Floer's equation for a Hamiltonian connection}
\label{sec:floers-equation-general}

The data required to formulate Floer's equation is:
\begin{enumerate}
\item\label{FEdata1} a family of domains $(S\times \Sigma,\Gamma_{\pm},j)$ as in \S\ref{sec:domains},
\item\label{FEdata2} a connection $1$-form $\mathfrak{a}$ and perturbation $1$-form $\mathfrak{p}$ on $S\times \Sigma \times W$ as in \S\ref{sec:connection-1-forms} and \S\ref{sec:perturbation-1-forms},
\item\label{FEdata3} an almost complex structure $J$ on $S\times \Sigma\times W$ as in \S\ref{sec:almost-compl-struct},
\item\label{FEdata4} a Lagrangian boundary condition $L$ for $S\times \Sigma\times W$ as in \S\ref{sec:lagr-bound-cond}.
\end{enumerate}

The data of \ref{FEdata2} produces an Ehresmann connection $\mathfrak{H}$ for the fiber bundle $S\times \Sigma\times W\to S\times \Sigma$. For each $\sigma\in S$, we consider the Ehresmann connection $\mathfrak{H}_{\sigma}$ obtained by pullback to $\set{\sigma}\times \Sigma \times W\to \set{\sigma}\times \Sigma$. Then $\mathfrak{H}_{\sigma}$ is identified with $T\Sigma$ via the projection map, and we use \ref{FEdata1} and \ref{FEdata3} to define:
\begin{equation*}
  \tilde{J}_{\sigma,z,w}=\left[
    \begin{matrix}
      J_{\sigma,z,w}&0\\
      0&j_{\sigma,z}
    \end{matrix}  \right]\text{ on }TW_{w}\oplus (\mathfrak{H}_{\sigma})_{z,w}=T(\Sigma\times W)_{z,w}.
\end{equation*}
This is considered as a family of almost complex structures on $\Sigma\times W$.

A pair $(\sigma,u)$ solves Floer's equation for this data if:
\begin{equation*}
  \left\{
    \begin{aligned}
      &u:\Sigma\to W\text{ is a smooth map},\\
      &\text{the section }z\mapsto (z,u(z))\text{ is }(j_{\sigma},\tilde{J}_{\sigma})\text{-holomorphic},\\
      &u(z)\in L_{\sigma,z}\text{ for all }z\in \bd \Sigma.
    \end{aligned}
  \right.
\end{equation*}
In the sequel, we will say that such $(\sigma,u)$ solves \S\ref{sec:floers-equation-general}.

\subsubsection{Regularity}
\label{sec:regularity}

We assume the reader is familiar with the standard elliptic regularity results, in particular, the following fact: if $(\sigma_{n},u_{n})$ is a sequence of solutions to \S\ref{sec:floers-equation-general} and:
\begin{enumerate}
\item $\sigma_{n}$ converges to $\sigma_{\infty}\in S$,
\item\label{reg2} the image $u_{n}(\Sigma)$ remains in a compact set,
\item\label{reg3} the first derivatives of $u_{n}:\Sigma\to W$ are bounded on compact subsets,
\end{enumerate}
then a subsequence of $u_{n}$ converges in the $C^{\infty}_{\mathrm{loc}}$-topology to a smooth map $u_{\infty}$ such that $(\sigma_{\infty},u_{\infty})$ solves \S\ref{sec:floers-equation-general}.

This result is true because of the general theory of pseudo-holomorphic curves in $\Sigma\times W$, and because the submanifold $L_{\sigma}\subset \Sigma\times W$ in \S\ref{sec:lagr-bound-cond} is totally real. We refer the reader to \cite[\S B]{mcduff-salamon-book-2012} for further details.

To establish \ref{reg2} and \ref{reg3} we rely on a priori estimates on the energy integral.

\subsubsection{Energy integral}
\label{sec:energy-integral}

Let $u:\Sigma\to W$ be a smooth map, and let $\mathfrak{H}$ be an Ehresmann connection on $\Sigma\times W$. Consider the 2-form on $\Sigma$ defined by:
\begin{equation*}
  \omega(\Pi_{\mathfrak{H}}\d v(-),\Pi_{\mathfrak{H}}\d v(-)),
\end{equation*}
where $v(z)=(z,u(z))$ and $\Pi_{\mathfrak{H}}:T(\Sigma\times W)\to TW$ is the projection whose kernel is $\mathfrak{H}$. The integral of this 2-form is called the \emph{energy} of the map $u$ relative the connection $\mathfrak{H}$.

If $(\sigma,u)$ solves \S\ref{sec:floers-equation-general}, and $\mathfrak{H}=\mathfrak{H}_{\sigma}$, then the above 2-form is non-negative, since $J$ is assumed to be $\omega$-tame. We say that $(\sigma,u)$ has \emph{finite energy} provided the energy of $u$ relative $\mathfrak{H}_{\sigma}$ is finite. The moduli space of all finite energy solutions of \ref{sec:floers-equation-general} plays a central role in this text, and might be denoted:
\begin{equation*}
  \mathscr{M}(S\times \Sigma\times W,\Gamma_{\pm},j,\mathfrak{a},\mathfrak{p},J,L);
\end{equation*}
however, we will typically suppress arguments from the notation and simply use the symbol $\mathscr{M}$ as the required data can be inferred from the context.

The finite energy condition implies the following results:

\begin{proposition}[Asymptotic convergence]\label{proposition:asymptotic}
  If $(\sigma,u)$ solves \S\ref{sec:floers-equation-general} and has finite energy, then, in cylindrical coordinates $z=s+it$ near the puncture $\zeta$, $u$ is asymptotic to a solution of the $s$-independent equation:
  \begin{equation*}
    \bd_{s}u+J_{\zeta}(u)(\bd_{t}u-X_{\zeta,t}(u))=0,
  \end{equation*}
  where $J_{\sigma,z}=J_{\zeta}$ and $\mathfrak{a}=H_{\zeta,t}\d t$ near $\zeta$. Moreover, this asymptotic solution satisfies $\bd_{s}u=0$, and is thus tracing out an orbit of the vector field $X_{\zeta,t}$.
\end{proposition}
\begin{proof}
  See, e.g., \cite{salamon-notes-1997}.
\end{proof}

\begin{proposition}[Gradient bound]\label{proposition:gradient-bound}
  If $(\sigma_{n},u_{n})$ is a sequence of solutions of \S\ref{sec:floers-equation-general}, with uniformly bounded energy, and $\sigma_{n}$ converges, then $u_{n}$ has bounded first derivatives with respect to a Riemannian metric $g$ on $W$ which is invariant under the Liouville flow in the end, and a metric on $\Sigma$ which is cylindrical near the punctures.
\end{proposition}
\begin{proof}
  Because $W$ is symplectically aspherical, and the Lagrangians appearing as boundary conditions are weakly exact, the stated result follows from standard bubbling analysis. The argument is simplified by the fact that the almost complex structures $J$ are invariant under the Liouville flow in the end. For details, we refer the reader to, e.g., \cite{brocic-cant-JFPTA-2024,alizadeh-atallah-cant-arXiv-2023}.
\end{proof}

\begin{proposition}[Maximum principle]\label{proposition:maximum-principle}
  If $(\sigma_{n},u_{n})$ is a sequence of solutions of \S\ref{sec:floers-equation-general} with uniformly bounded energy, and $\sigma_{n}$ converges, then $u_{n}(\Sigma)$ remains in a compact subset of $W$.
\end{proposition}
\begin{proof}
  This is proved in \cite{brocic-cant-JFPTA-2024} by a soft argument (using the Liouville-invariance of $J$ and the Hamiltonian systems in \ref{sec:hamilt-functions-isotopies}); see also \cite{alizadeh-atallah-cant-arXiv-2023}. We note that the proof does use the assumption in \S\ref{sec:connection-1-forms} that the Hamiltonian isotopy generated by the asymptotic system $H_{\zeta,t}$ has no 1-periodic orbits outside of a compact set.
\end{proof}

\subsubsection{Compactness and Floer differential cylinders}
\label{sec:compactness}

By combining Propositions \ref{proposition:gradient-bound} and \ref{proposition:maximum-principle} with the regularity result in \S\ref{sec:regularity}, one sees that, in order to conclude that a sequence $(\sigma_{n},u_{n})$ of solutions to \S\ref{sec:floers-equation-general} has a convergent subsequence in $C^{\infty}_{\mathrm{loc}}$, it suffices to ensure that $\sigma_{n}$ has a convergent subsequence and that $u_{n}$ has uniformly bounded energy. Thus the main consideration for ensuring compactness is a priori energy estimates.

Even if $(\sigma_{n},u_{n})$ has a subsequence such that $u_{n}\to u_{\infty}$ in $C^{\infty}_{\mathrm{loc}}$, it is not true in general that $u_{n}$ will converge uniformly to $u_{\infty}$; it is possible that solutions of the $s$-independent equation from Proposition \ref{proposition:asymptotic} break-off at the punctures. This phenomenon is a cornerstone of Floer theory, and we will not explain it any further except via the illustration in Figure~\ref{fig:compactness-up-to-breaking}.

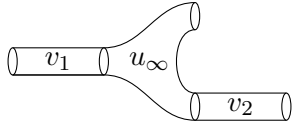
\begin{figure}[h]
  \centering
  \begin{tikzpicture}[scale=.6]
    \path (-2,0) coordinate(E) +(0,-0.6)coordinate(EM)-- (0,0) coordinate(A)+(0,-0.6)coordinate(AM) -- (1,-0.3)node{$u_{\infty}$} -- (2,1)coordinate(B)+(0,-0.6)coordinate(BM) -- (2,-1) coordinate(C)+(0,-0.6)coordinate(CM)--(4,-1) coordinate(D)+(0,-0.6)coordinate(DM);
    \draw (A)+(0,-0.3) circle (0.1 and 0.3) (B)+(0,-0.3) circle (0.1 and 0.3) (C)+(0,-0.3) circle (0.1 and 0.3) (D)+(0,-0.3) circle (0.1 and 0.3) (E)+(0,-0.3) circle (0.1 and 0.3);
    \draw (E)--(A)to[out=0,in=180](B) (BM)to[out=180,in=180](C)--(D) (EM)--(AM)to[out=0,in=180](CM)--(DM);
    \draw (EM)--(AM);
    \path (E)--node{$v_{1}$}(AM);
    \path (C)--node{$v_{2}$}(DM);
  \end{tikzpicture}
  \caption{Compactness up-to-breaking of Floer differential cylinders at the punctures; a sequence of solutions $u_{n}$ defined on the pair-of-pants surface converging on compact subsets to a limit $u_{\infty}$, with the breaking of two solutions $v_{0},v_{1}$ of the $s$-independent equation at the punctures.}
  \label{fig:compactness-up-to-breaking}
\end{figure}

Because of the breaking phenomenon, the $s$-independent equation in Proposition \ref{proposition:asymptotic} plays a central role in Floer theory; we refer to this equation as \emph{Floer's differential equation}, and it depends on the choice of an almost complex structure $J_{\zeta}$ and a time-dependent Hamiltonian vector field $X_{\eta,t}$ generated by a smooth family $H_{t}\in \mathscr{H}$ satisfying $H_{t+1}=H_{t}$.

\subsubsection{A priori energy estimates}
\label{sec:priori-energy-estim}

Consider a connection one-form $\mathfrak{a}$ and perturbation one-form $\mathfrak{p}$ on $S\times \Sigma\times W$. This induces an Ehresmann connection $\mathfrak{H}$ on $S\times \Sigma\times W\to S\times \Sigma$. It is well-understood that the curvature of the connection $\mathfrak{H}$ plays a role in a priori energy estimates for the energy integral for solutions to \S\ref{sec:floers-equation-general}; see, e.g., \cite[\S2.3.3]{alizadeh-atallah-cant-arXiv-2023}. In fact, there is an identity for the energy integral involving the two-form $\mathfrak{r}$ on $S\times \Sigma\times W$ called the curvature potential characterized by the property that, above a local coordinate patch $z=x+iy$ on $\Sigma$ where:
\begin{equation*}
  \mathfrak{a}=H_{\sigma,x,y}\d x+K_{\sigma,x,y}\d y\text{ and }\mathfrak{p}=h_{\sigma,x,y}\d x+k_{\sigma,x,y}\d y
\end{equation*}
we have:
\begin{equation*}
  \mathfrak{r}=(\bd_{x}(K_{\sigma,x,y}+k_{\sigma,x,y})-\bd_{y}(H_{\sigma,x,y}+h_{\sigma,x,y})+\omega(X^{H+h}_{\sigma,s,t},X^{K+k}_{\sigma,x,y})) \d x\wedge \d y,
\end{equation*}
where $X^{H+h}_{\sigma,x,y},X^{k+k}_{\sigma,x,y}$ are the domain dependent Hamiltonian vector fields tangent to the fibers of $S\times \Sigma \times W\to S\times \Sigma$.

We note that if $\mathfrak{p}=0$ and $\mathfrak{a}=H_{t}\d t$ then $\mathfrak{r}=0$, and so the curvature potential $\mathfrak{r}$ vanishes above a neighbourhood of the punctures $\Gamma\subset S\times \Sigma$.

It is convenient to phrase the energy identity in terms of the amount of energy a solution has on a compact subdomain $\Sigma'\subset \Sigma$. A priori energy estimates for the full energy are then obtained by taking a limit over larger and larger subdomains. The energy identity we will use is:
\begin{lemma}\label{lemma:energy-identity}
  If $(\sigma,u)$ solves \S\ref{sec:floers-equation-general} and $\Sigma'\subset \Sigma$ is a compact subdomain with boundary containing the support of $\mathfrak{p}_{\sigma}$, then the energy of $u|_{\Sigma'}$ is equal to:
  \begin{equation*}
    E(u|_{\Sigma'})=\int_{\Sigma'}u^{*}\omega+v^{*}\mathfrak{r}_{\sigma}-\int_{\bd\Sigma'}v^{*}\mathfrak{a}_{\sigma},
  \end{equation*}
  where $\mathfrak{a}_{\sigma},\mathfrak{p}_{\sigma},\mathfrak{r}_{\sigma}$ denote the pullbacks to $\set{\sigma}\times \Sigma\times W$ and $v(z)=(z,u(z))$ is considered as a section of $\Sigma\times W\to \Sigma$.
\end{lemma}
\begin{proof}
  This is a standard computation; see, e.g., \cite[Lemma 2.3]{alizadeh-atallah-cant-arXiv-2023}.
\end{proof}

Let us say that the data $\mathfrak{a},\mathfrak{p}$ has \emph{curvature bounded from above} provided:
\begin{equation*}
  \sup\set{\int_{\Sigma}v^{*}\mathfrak{r} : v\text{ is a smooth section }\Sigma\to \set{\sigma}\times \Sigma\times W\text{ for some $\sigma$}}<\infty,
\end{equation*}
where $\mathfrak{r}$ is the curvature potential associated to $\mathfrak{a},\mathfrak{p}$. We denote the quantity on the left by $\mathrm{const}(\mathfrak{r})$. This leads to the following a priori energy estimate:
\begin{lemma}\label{lemma:energy-estimate}
  If $\mathfrak{a},\mathfrak{p}$ has curvature bounded from above, then any finite energy solution $(\sigma,u)$ of \S\ref{sec:floers-equation-general} satisfies:
  \begin{equation*}
    E(u)\le \omega(u)+\int_{0}^{1}\sum_{\Gamma_{-}}H_{\zeta,t}(\gamma_{\zeta}(t))-\sum_{\Gamma_{+}}H_{\zeta,t}(\gamma_{\zeta}(t))dt-\int_{\bd\Sigma}v^{*}\mathfrak{a}_{\sigma}+\mathrm{const}(\mathfrak{r}),
  \end{equation*}
  where $u$ is asymptotic to the orbit $\gamma_{\zeta}(t)$ at the puncture $\zeta$, and $\mathfrak{a}=H_{\zeta,t}\d t$ holds near the puncture $\zeta$. In particular, the energy of finite energy solutions is uniformly bounded provided:
  \begin{enumerate}
  \item\label{AP-req-1} the symplectic area $\omega(u)$ is bounded above for solutions $(\sigma,u)$,
  \item\label{AP-req-2} the 1-periodic orbits of $H_{\zeta,t}$ are contained in a compact set,
  \item\label{AP-req-3} the integrals of $\mathfrak{a}_{\sigma}$ over sections $v:\bd\Sigma\to W$ valued in the Lagrangian boundary conditions $L_{\sigma}$ are bounded below as $\sigma$ varies in $S$.
  \end{enumerate}
\end{lemma}
\begin{proof}
  This follows from the energy identity; see \cite[Lemma 2.4]{alizadeh-atallah-cant-arXiv-2023}.
\end{proof}
We briefly explain why these three conditions can be assumed to hold in the context of our paper. First note \ref{AP-req-2} holds by fiat; it is part of our assumptions from \S\ref{sec:connection-1-forms} on the connection one-form. Second, in all of the constructions involving a non-empty boundary $\bd \Sigma$, we ensure that $\mathfrak{a}$ appears in the form $H_{\sigma,t}\d t$ where $t$ is a coordinate on the boundary, $H_{\sigma,t}$ is non-negative outside of a compact set, and moreover that $\min H_{\sigma,t}$ is uniformly bounded from below as $\sigma$ varies in $S$. It then follows easily that the integrals of $v^{*}\mathfrak{a}_{\sigma}$ will be bounded from below, i.e., \ref{AP-req-3} will be satisfied.

Morally speaking, condition \ref{AP-req-1} holds because we assume the symplectic form is exact. There is some subtlety when $\bd\Sigma\ne \emptyset$, and we will explain how to deal with this in \S\ref{sec:string-topol-floer-cohom}, when we first consider moduli spaces with boundary conditions.

\emph{Remark}. In general, condition \ref{AP-req-1} cannot be established, and the usual way one deals with this is via the introduction of Novikov coefficients in the Floer cohomology groups.

To conclude this subsection, we make the following observation:
\begin{lemma}\label{lemma:a-has-then-p-has}
  If $(\mathfrak{a},0)$ has curvature bounded above (i.e., $\mathfrak{p}=0$), then so does the perturbed data $(\mathfrak{a},\mathfrak{p})$.
\end{lemma}
\begin{proof}
  This is due to the fact that the functions appearing in $\mathfrak{p}$ are assumed to be uniformly bounded in $C^{1}$; see \cite[Lemma 2.2]{alizadeh-atallah-cant-arXiv-2023}.
\end{proof}

\subsubsection{Transversality}
\label{sec:transversality}

Each finite-energy solution $(\sigma,u)$ determines a linearized operator: $$D_{\sigma,u}:TS_{\sigma}\oplus W^{1,p}(u^{*}TW)\to L^{p}(\Hom^{0,1}(T\Sigma,u^{*}TW)),$$ where $\Hom^{0,1}$ is the bundle of $(j_{\sigma},J_{\sigma})$-antilinear homomorphisms. We always assume $p>2$ when discussing the Sobolev space $W^{1,p}$.

The details of the construction of the linearized operator are well-known; see, e.g., \cite{salamon-notes-1997,cant-thesis-2022}. Let us comment that it is obtained by differentiating the local coordinate representations of the non-linear PDE with respect to local deformations of $u$ (and $\sigma$); this yields a linear differential operator acting on the local deformations; these  can be patched together to obtain the global linearized operator.

Restricting $D_{\sigma,u}$ to variations fixing $\sigma$ yields a particular type of differential operator, namely, a Cauchy-Riemann operator with asymptotic conditions. Because we required that the asymptotic Hamiltonian systems are non-degenerate in \S\ref{sec:connection-1-forms}, the linearized operator is Fredholm operator. If the cokernel of $D_{\sigma,u}$ is zero for each solution $(\sigma,u)$, then we say that transversality holds. In this case, the space of solutions has the structure of a smooth manifold, and its dimension is equal to the dimension of the kernel of $D_{\sigma,u}$. This dimension can be computed by a general formula for the Fredholm index of a Cauchy-Riemann operator with asymptotic conditions. This is all a quite standard part of Floer theory.

We will appeal to the following general transversality lemma:
\begin{lemma}
  Given a family of domains, almost complex structures, Lagrangian boundary conditions, and the connection one-form $\mathfrak{a}$ on $S\times \Sigma\times W$, transversality can be achieved for all solutions $(\sigma,u)$ to the perturbed equation \S\ref{sec:floers-equation-general} by choosing $\mathfrak{p}$ to be a generic perturbation one-form.
\end{lemma}
\begin{proof}
  This is a standard application of the usual Sard-Smale argument; see \cite[\S2.3.10]{alizadeh-atallah-cant-arXiv-2023} and the references therein.
\end{proof}

As dicussed in \S\ref{sec:energy-integral} and \S\ref{sec:compactness}, a special role is played by solutions of the $s$-independent equation (Floer's differential equation):
\begin{equation}\label{eq:floer-differential-equation}
  \bd_{s}u+J_{\zeta}(u)(\bd_{t}u-X_{\zeta,t}(u))=0,
\end{equation}
where $u$ is defined on a cylinder and $J_{\zeta},X_{\zeta,t}$ are the asymptotic almost complex structure and Hamiltonian vector field associated to a puncture $\zeta$.

This equation also has a linearized operator, and it is also important to achieve transversality for this. However, the perturbation term can no longer be used, since we require $\mathfrak{p}$ to be supported away from the punctures. In this case, we achieve transversality by assuming the asymptotic vector fields $X_{\zeta,t}$ are sufficiently generic. As shown in \cite{floer-hofer-salamon-duke-1995}, transversality can be achieved by perturbing $X_{\zeta,t}$ in a small, compactly supported way, away from its non-degenerate orbits; see also \cite[\S4.1]{brocic-cant-JFPTA-2024}. This requires us to slightly amend condition \ref{c1f2} from \S\ref{sec:connection-1-forms} by requiring that $X_{\zeta,t}$ is sufficiently generic (in addition to having non-degenerate $1$-periodic orbits). Henceforth, we will assume asymptotic data is sufficiently generic to achieve transversality for Floer's differential equation.

\subsection{Floer cohomology}
\label{sec:floer-cohomology}

In this section we define the Floer cohomology persistence module $c\in \R\mapsto V_{c}\in \mathrm{Vect}(\Z/2\Z)$, with its three additional structures: the product, the BV-operator, and the map $H^{*}(W)\to V_{0}$.

The reader who is comfortable with these structures should feel free to skip to \S\ref{sec:evaluation-maps-ball}.

\subsubsection{The Floer complex}
\label{sec:floer-complex}

Let $H_{t}$, $t\in \R/\Z$, be a smooth family in $\mathscr{H}$, and suppose that the associated Hamiltonian vector field $X_{t}$ has non-degenerate 1-periodic orbits. Let $J$ be an almost complex structure on $W$ which is $\omega$-tame and Liouville invariant outside of $\Omega$. Moreover, suppose that $X_{t}$ is sufficiently generic that the space of finite energy solutions to Floer's differential equation is cut transversally, as explained in \S\ref{sec:transversality}. Let us call such data $(H_{t},J)$ \emph{admissible} for defining the Floer complex.

The requirements in \S\ref{sec:hamilt-conn-surf} and \S\ref{sec:general-form-floers} ensure that the asymptotics $(H_{\zeta,t},J_{\zeta})$ which appear in the general form of Floer's equation \S\ref{sec:floers-equation-general} are admissible for defining the Floer complex.

Given admissible data, the \emph{Floer complex} $\mathrm{CF}(H_{t},J)$ is defined to be the $\Z/2\Z$ vector space generated by the $1$-periodic orbits of $X_{t}$ with the differential $d$ given by counting the one-dimensional components of the moduli space of finite energy solutions to Floer's differential equation. The right asymptotic of the solution is considered as the input of the differential and the left asymptotic is considered as the output.

The homology of the complex is denoted $\mathrm{HF}(H_{t},J)$, and is called the Floer cohomology of the pair $(H_{t},J)$.

The details of the definition of the Floer differential, and why it squares to zero, are well-known; see, e.g., \cite{floer-CMP-1989,hofer-salamon-floermem-1995,salamon-notes-1997}. The requisite regularity and compactness results follow from our general discussion in \S\ref{sec:general-form-floers}. One aspect of the argument we do not explain is the Floer theory gluing argument used the relate the coefficients appearing in $d\circ d$ to the two-dimensional components of the moduli space of Floer differential cylinders; we refer the reader to, e.g., \cite[\S3.3]{salamon-notes-1997} for an exposition of this gluing theory.

Let us comment on one slightly non-standard aspect of the systems that we consider. We do not assume that $H_{t}=cr$ holds in the end for a fixed slope $c$; rather, our definition allows $H_{t}=c_{t}r$ where $c_{t}$ varies with $t$. We define the \emph{slope} of the family $H_{t}$ to be the real number:
\begin{equation*}
  c(H_{t})=\int_{0}^{1}c_{t}dt
\end{equation*}
This number plays a special role in the persistence module: $V_{c}$ is defined as a formal limit over all the Floer cohomology groups $\mathrm{HF}(H_{t},J)$ with $c(H_{t})=c$.

The slope $c(H_{t})$ has a geometric interpretation: $X_{t}=c_{t}R$ holds in the region $r\ge r_{0}$ for $r_{0}$ large enough (how large depends on the smooth family $H_{t}\in \mathscr{H}$; see \S\ref{sec:class-hamilt-funct}), where $R=X_{r}$ is the so-called Reeb vector field. Thus the isotopy obtained by integrating $X_{t}$ agrees with the Reeb flow for time $c(H_{t})$ in the region $r\ge r_{0}$, up to a time reparametrization.

\begin{figure}[h]
  \centering
  \begin{tikzpicture}[xscale=.7,yscale=.4]
    \draw (0,0) circle (0.3 and 1) (0,1)--+(10,0) (0,-1)--+(10,0) (10,0) circle (0.3 and 1);
    \path (-0.3,0)node[left]{output}--node{$\bd_{s}u+J(u)(\bd_{t}u-X_{t}(u))=0$}(10.3,0)node[right]{input};
  \end{tikzpicture}
  \caption{The Floer differential}
  \label{fig:floer-differential}
\end{figure}

\subsubsection{Continuation maps}
\label{sec:cont-maps-floer}

Continuation maps are chain maps which relate the Floer complexes for different choices of admissible data. Because we use Hamiltonian functions $H_{t}$ with varying slope, a bit of care is needed to define continuation maps in the generality we will require.

We define \emph{continuation data} to be a connection 1-form $\mathfrak{a}$ and almost complex structure $J$ on $\Sigma\times W$, where $\Sigma=\R\times \R/\Z$, such that:
\begin{equation*}
  \mathfrak{a}=K_{s,t}\d s+H_{s,t}\d t,
\end{equation*}
where $H_{s,t},K_{s,t}$ are smooth families in $\mathscr{H}$ satisfying:
\begin{enumerate}
\item\label{cdata1} $K_{s,t}=b_{s,t}r$ and $H_{s,t}=c_{s,t}r$ for $r\ge r_{0}$, for an $r_{0}$ independent of $s,t$,
\item\label{cdata2} $\bd_{s}c_{s,t}\le \bd_{t}b_{s,t}$,
\item $K_{s,t}=\bd_{s}H_{s,t}=0$ for $\abs{s}\ge s_{0}$,
\item\label{cdata4} $H_{s,t}=H_{\zeta_{1},t}$ for $s\le -s_{0}$ and $H_{s,t}=H_{\zeta_{0},t}$ for $s\ge s_{0}$,
\item\label{cdata5} $J=J_{s,t}$ satisfies $J_{s,t}=J_{\zeta_{1}}$ for $s\le -s_{0}$ and $J_{s,t}=J_{\zeta_{0}}$ for $s\ge s_{0}$,
\item\label{cdata6} $(H_{\zeta_{0},t},J_{\zeta_{0}})$ and $(H_{\zeta_{1},t},J_{\zeta_{1}})$ are admissible for defining $\mathrm{CF}$.
\end{enumerate}
Such data is called continuation data from $(H_{\zeta_{0},t},J_{\zeta_{0}})$ to $(H_{\zeta_{1},t},J_{\zeta_{1}})$, i.e., as in the Floer differential, we consider the right asymptotic to be the input.

A \emph{homotopy} of continuation data is a connection one-form and almost complex structure on $\R\times \Sigma\times W$ such that the restriction to $\set{\tau}\times \Sigma\times W$ satisfies \ref{cdata1} through \ref{cdata6} for each $\tau$, with fixed asymptotic data $(H_{\zeta_{0},t},J_{\zeta_{0}})$ and $(H_{\zeta_{1},t},J_{\zeta_{1}})$, and such that the numbers $r_{0},s_{0}$ can be taken to be independent of $\tau,s,t$. One says that the continuation data obtained by restricting to $\set{0}\times \Sigma\times W$ and $\set{1}\times \Sigma\times W$ are homotopic.

\begin{lemma}\label{lemma:when-does-cdata-exist}
  Given asymptotic data $(H_{\zeta_{0},t},J_{\zeta_{0}})$ and $(H_{\zeta_{1},t},J_{\zeta_{1}})$ satisfying \ref{cdata6}, there exists continuation data between them if and only if the slope of $H_{\zeta_{0},t}$ is at most the slope of $H_{\zeta_{1},t}$. In this case, there is a homotopy between any two choices of continuation data.
\end{lemma}
\begin{proof}
  First suppose there exists continuation data. Introduce the averages:
  \begin{equation*}
    c(s)=\int_{0}^{1}c(s,t)\d t.
  \end{equation*}
  Then one uses the requirement \ref{cdata2} to show:
  \begin{equation*}
    \bd_{s}c(s)\le \int_{0}^{1}\partial_{t}b(s,t)\d t=0.
  \end{equation*}
  In particular, the slope at the left end is greater than the slope at the right end. This proves the ``only if'' part of the first assertion.

  For the ``if'' part, we fix a standard cut-off function $\beta(s)$ and define:
  \begin{equation*}
    \begin{aligned}
      H_{s,t}&=\beta(s)H_{\zeta_{0},t}+(1-\beta(s))H_{\zeta_{1},t}\\
      K_{s,t}&=\int_{0}^{t}\pd{}{s}H_{s,\tau}d\tau - t\pd{}{s}\int_{0}^{1}H_{s,\tau}d \tau.
    \end{aligned}
  \end{equation*}
  It is straightforward to check that $H_{s,t},K_{s,t}$ are 1-periodic in the $t$-variable, and that properties \ref{cdata1} through \ref{cdata4} are satisfied. It is important that the slope of $H_{\zeta_{0},t}$ is less than the slope of $H_{\zeta_{1},t}$ in order for \ref{cdata2} to be satisfied, as it ensures that:
  \begin{equation*}
    \pd{}{s}\int_{0}^{1}H_{s,\tau}d \tau\le 0\text{ on the region where }r\ge r_{0}.
  \end{equation*}
  One uses the contractibility of the space of almost complex structures to extend the asymptotic complex structures $J_{\zeta_{0}}$ and $J_{\zeta_{1}}$ such that \ref{cdata5} is satisfied.

  Finally we prove that any two continuation data between the same asymptotic systems are homotopic. This is straightforward, as one can simply take a convex combination of two one-forms $\mathfrak{a}$ and easily verify the properties \ref{cdata1} through \ref{cdata4} (which are all preserved under convex combinations), and again appeal to the contractibility of the space of almost complex structures. This completes the proof.
\end{proof}

Given continuation data $(\mathfrak{a},J)$ between $(H_{\zeta_{0},t},J_{\zeta_{0}})$ and $(H_{\zeta_{1},t},J_{\zeta_{1}})$, one interprets the counts the rigid solutions of the moduli space of finite energy solutions to \S\ref{sec:floers-equation-general} with generic perturbation term $\mathfrak{p}$ as defining the coefficients in a linear map (called a \emph{continuation map}): $$\mathrm{CF}(H_{\zeta_{0},t},J_{\zeta_{0}})\to \mathrm{CF}(H_{\zeta_{1},t},J_{\zeta_{1}})$$
Here \emph{rigid} means the counting the zero-dimensional components of the moduli space, which is assumed to be cut transversally.

Let us comment on one aspect related to why the count of rigid elements is finite: condition \ref{cdata2} ensures that the connection 1-form $\mathfrak{a}$ has curvature bounded from above. Indeed, the curvature two-form was given by:
\begin{equation*}
  \mathfrak{r}=(\bd_{s}H_{s,t}-\bd_{t}K_{s,t}+\omega(X^{K}_{s,t},X^{H}_{s,t})) \d s\wedge \d t,
\end{equation*}
and since $X^{K}_{s,t},X^{H}_{s,t}$ are proportional when $r\ge r_{0}$ (both point in the direction of the Reeb flow), we have:
\begin{equation*}
  \mathfrak{r}=(\bd_{s}c_{s,t}-\bd_{t}b_{s,t})r\d s\wedge \d t
\end{equation*}
when $r\ge r_{0}$, and this is non-positive.

The details of the construction of the continuation map are standard; these maps form a basic ingredient in the Floer cohomology TQFT which has been carefully constructed in the open case by \cite{ritter-jtopol-2013}; see also \cite{schwarz-thesis-1995,hofer-salamon-floermem-1995,abouzaid-EMS-2015}.

Standard arguments show:
\begin{enumerate}
\item\label{cmap-prop-1} the continuation map is a chain map,
\item\label{cmap-prop-2} the chain homotopy class of the map is independent of the generic perturbation $\mathfrak{p}$ or the homotopy class of continuation data,
\item\label{cmap-prop-3} the composition of two continuation maps is a continuation map,
\item\label{cmap-prop-4} the continuation map $\mathrm{HF}(H_{t},J)\to \mathrm{HF}(H_{t},J)$ is the identity map.
\end{enumerate}
We will not review these standard arguments, and refer the reader to, e.g., \cite{schwarz-thesis-1995,hofer-salamon-floermem-1995,ritter-jtopol-2013,abouzaid-EMS-2015}. Let us comment that \ref{cmap-prop-1}, \ref{cmap-prop-2}, and \ref{cmap-prop-3} are based on gluing arguments similar to those used to prove $d\circ d=0$.

A standard consequence of \ref{cmap-prop-1} through \ref{cmap-prop-4} is, if $(H_{\zeta_{0},t},J_{\zeta_{0}})$ and $(H_{\zeta_{1},t},J_{\zeta_{1}})$ are admissible for defining $\mathrm{CF}$, and $H_{\zeta_{0},t}$ and $H_{\zeta_{1},t}$ have the same slope, then their Floer homologies are canonically isomorphic. We will continue with this line of reasoning in the next subsection \S\ref{sec:floer-cohom-pers-defin}.

\subsubsection{The Floer cohomology persistence module}
\label{sec:floer-cohom-pers-defin}

Let $\mathscr{D}$ be the category whose objects are pairs $(H_{t},J)$ which are admissible for defining the Floer complex, and where there is a unique morphism $(H_{\zeta_{0},t},J_{\zeta_{0}})\to (H_{\zeta_{1},t},J_{\zeta_{1}})$ if the slope of $H_{\zeta_{0},t}$ is at most the slope of $H_{\zeta_{1},t}$, and no morphisms otherwise. The assignment:
\begin{equation*}
  (H_{t},J)\mapsto \mathrm{HF}(H_{t},J),
\end{equation*}
together with the continuation map construction induces a functor:
\begin{equation*}
  \mathrm{HF}:\mathscr{D}\to \mathrm{Vect}(\Z/2\Z),
\end{equation*}
where $\mathrm{Vect}(\Z/2\Z)$ is the category of vector spaces over the field $\Z/2\Z$. In a certain sense to made precise, this functor factorizes through the functor $\mathscr{D}\to (\R,\le)$ which sends $(H_{t},J)$ to its slope.

Summarizing, we have a diagram of the form:
\begin{equation*}
  \begin{tikzcd}
    &{\mathscr{D}}\arrow[ld,"{\mathrm{slope}}",swap]\arrow[rd,"{\mathrm{HF}}"] &\\
    {(\R,\le)}\arrow[rr,"{V}",dashed]& &{\mathrm{Vect}(\Z/2\Z)},
  \end{tikzcd}
\end{equation*}
and the claim is that there is a persistence module $V$ which makes the diagram commute in the following sense:
\begin{lemma}
  There exists a functor $V:(\R,\le)\to \mathrm{Vect}(\Z/2\Z)$ equipped with a natural isomorphism:
  \begin{equation*}
    I:V\circ (\mathrm{slope})\to \mathrm{HF}
  \end{equation*}
  of functors defined on $\mathscr{D}$. Moreover, any two such $(V,I)$, $(V',I')$ are isomorphic, in that there is a unique natural isomorphism $U:V\to V'$ making the diagram:
  \begin{equation*}
    \begin{tikzcd}
      {V\circ \mathrm{slope}}\arrow[d,"{I}"]\arrow[r,"{U}"] &{V'\circ \mathrm{slope}}\arrow[d,"{I'}"]\\
      {\mathrm{HF}}\arrow[r,"{\id}"] &{\mathrm{HF}}
    \end{tikzcd}
  \end{equation*}
  commute; the diagram is valued in the category of functors $\mathscr{D}\to \mathrm{Vect}(\Z/2\Z)$.
\end{lemma}
\begin{proof}
  The argument is essentially abstract nonsense. For each $c$, consider the full subcategory $\mathscr{D}(c)$ of objects $(H_{t},J)$ where the slope of $H_{t}$ is no less than $c$. The restriction of $\mathrm{HF}$ to this category is again a functor, and we define the category theory limit:
  \begin{equation*}
    V_{c}=\lim_{\mathscr{D}(c)}\mathrm{HF};
  \end{equation*}
  see, e.g., \cite{maclean-springer-1971,aluffi-GSM-2009} for details on category theory limits.

  Prosaically speaking, we define $V_{c}$ as the inverse limit of the Floer cohomologies of Hamiltonian systems whose slopes are at least $c$. If $c$ is not the period of a Reeb orbit of the Reeb flow associated to $\bd \Omega$, then $\mathscr{D}(c)$ has initial objects $(H_{t},J)$, and then the limit is isomorphic to $\mathrm{HF}(H_{t},J)$; this produces the natural isomorphism $I$. If $c$ is the period of a Reeb orbit, then $\mathscr{D}(c)$ does not have an initial object; in this case, it is possible that $V_{c}$ is an infinite dimensional vector space; these outputs of the persistence module can be thought of as living in the ``completion'' of the category of finite dimensional vector spaces.

  Proving that any other construction $(V',I')$ is canonically isomorphic to our $(V,I)$ is straightforward application of universal properties of limits, and other abstract nonsense.
\end{proof}

Henceforth, we will refer to the functor $V:\R\to \mathrm{Vect}(\Z/2\Z)$ as the \emph{Floer cohomology persistence module} associated to $\Omega$.

Let us comment that, if the set of periods of the Reeb flow (i.e., its spectrum) is \emph{discrete}, then $V_{c}$ is finite dimensional for all $c$ (although we will not use this fact). In the same vein, if $c\in \R$ is such that $(c,c+\epsilon)$ contains no points in the spectrum of the Reeb flow, then $V_{c}$ is finite dimensional. As a special case, $V_{0}$ is finite dimensional.

\subsubsection{The pair-of-pants product}
\label{sec:pair-pants-product}

In this section, we briefly recall the pair-of-pants product structure on Floer cohomology groups. The details of the construction will be important in the sequel, but for now, let us summarize the resulting algebraic structure.

Let $\mathscr{D}^{3}$ be the full subcategory of $\mathscr{D}\times \mathscr{D}\times \mathscr{D}$ consisting of data:
\begin{equation*}
  ((H_{\zeta_{0},t},J_{\zeta_{0}}),(H_{\zeta_{1},t},J_{\zeta_{1}}),(H_{\zeta_{\infty},t},J_{\zeta_{\infty}})),
\end{equation*}
where $H_{\zeta_{\infty},t}$ has a slope no less than the sum of the slopes of $H_{\zeta_{0},t}$ and $H_{\zeta_{1},t}$. There are two functors defined on $\mathscr{D}^{3}$:
\begin{equation*}
  \begin{aligned}
    T_{1}&=\mathrm{HF}(H_{\zeta_{0,t}},J_{\zeta_{0}})\otimes \mathrm{HF}(H_{\zeta_{1},t},J_{\zeta_{1}}),\\
    T_{2}&=\mathrm{HF}(H_{\zeta_{\infty,t}},J_{\zeta_{\infty}}),
  \end{aligned}
\end{equation*}
the ``pair-of-pants product'' is a natural transformation $\ast:T_{1}\to T_{2}$.

Via suitable abstract nonsense, this induces a natural transformation:
\begin{equation*}
  \ast:V_{c_{0}}\otimes V_{c_{1}}\to V_{c_{\infty}}
\end{equation*}
between two functors defined on the subcategory of $\R^{3}$ (a partially ordered set) consisting of those objects $(c_{0},c_{1},c_{\infty})$ where $c_{\infty}\ge c_{0}+c_{1}$. The partial order on $\R^{3}$ is the one where $(c_{0},c_{1},c_{\infty})\le (c_{0}',c_{1}',c_{\infty}')$ if and only if $c_{i}\le c_{i}'$ for each $i=0,1,\infty$.

The construction of $\ast$ is formally similar to the construction of continuation maps: one defines a class of \emph{pair-of-pants data} which consists of connection one-forms $\mathfrak{a}$ and almost complex structures $J$, and then uses perturbations $\mathfrak{p}$ to define a transversally cut-out moduli space whose rigid counts are packaged into a chain map. The resulting map on homology is independent of the perturbation $\mathfrak{p}$ and homotopy class of pair-of-pants data $(\mathfrak{a},J)$. One shows the map on homology commutes with continuation maps (in a suitable sense) and therefore induces the aforementioned natural transformation $\ast$.

In the rest of this section, we describe the construction of the pair-of-pants product, with focus on the details relevant to the proof of Theorem \ref{theorem:main-floer}.

We begin by defining pair-of-pants data, which is a connection one-form $\mathfrak{a}$ and almost complex structure on $\Sigma\times W$ where $\Sigma=\C\setminus \set{0,1}$. We require the following properties:
\begin{enumerate}
\item $\mathfrak{a}=H_{\zeta_{i},t}\d t$ and $J=J_{\zeta_{i}}$ holds in standard cylindrical ends around the $i=0,1,\infty$ punctures, and $(H_{\zeta_{i},t},J_{\zeta_{i}})$ are admissible for defining the Floer complex
\item $\mathfrak{a}=f_{x,y}r\d x+g_{x,y}r\d y$ holds outside of $r\ge r_{0}$, where $f,g\in \R$ vary smoothly with $x,y$,
\item\label{popdata3} $\bd_{x}g_{x,y}-\bd_{y}f_{x,y}\le 0$.
\end{enumerate}
Here a standard cylindrical end around $z=0,1$ is obtained by parametrizing a disk around $z$ by $[0,\infty)\times \R/\Z$ using the exponential map, and a standard cylindrical end around $z=\infty$ is obtained by parametrizing the complement of a disk around $0$ by $(-\infty,0]\times \R/\Z$, again using the exponential map. One breaks the rotational symmetry by requiring that the line $t=0$ is aligned with the positive real axis.

We say such data goes from $(H_{\zeta_{0},t},J_{\zeta_{0}})$, $(H_{\zeta_{1},t},J_{\zeta_{1}})$ to $(H_{\zeta_{\infty},t},J_{\zeta_{\infty}})$.

Similarly to the case of the continuation data in \S\ref{sec:cont-maps-floer}, one can speak of homotopies of pair-of-pants data. The analog of Lemma \ref{lemma:when-does-cdata-exist} is:
\begin{lemma}\label{lemma:when-does-pants-data-exist}
  Given $(H_{\zeta_{i},t},J_{\zeta_{i}})$ for $i=0,1,\infty$, there exists pair-of-pants data if and only if the slope of $H_{\zeta_{\infty},t}$ is no less than the sum of the slopes of $H_{\zeta_{0},t}$ and $H_{\zeta_{1},t}$. In this case, there is a unique homotopy class of pair-of-pants data.
\end{lemma}
\begin{proof}
  The key idea in the proof is to consider $\mathfrak{a}=f_{x,y}\d x+g_{x,y}\d y$ as a one-form on $\Sigma$. The integral of $\d\mathfrak{a}$ (which is compactly supported) over $\Sigma$ is the difference in slopes: $$\int_{\Sigma}\d\mathfrak{a}=c(H_{\zeta_{0},t})+c(H_{\zeta_{1}},t)-c(H_{\zeta_{\infty},t});$$
  such an observation appears in, e.g., \cite{ritter-jtopol-2013}. This proves the ``only if'' part of the first assertion. The ``if'' part is a straightforward construction. Finally, the uniqueness of the homotopy class follows from the convexity of the space of pair-of-pants data (and the contractibility of the space of almost complex structures).
\end{proof}

Given pair-of-pants data $(\mathfrak{a},J)$, and a generic perturbation one-form $\mathfrak{p}$, one packages the counts of the rigid finite energy solutions to \S\ref{sec:general-form-floers} into a map:
\begin{equation}\label{eq:chain-level-product}
  \ast:\mathrm{CF}(H_{\zeta_{0},t},J_{\zeta_{0}})\otimes \mathrm{CF}(H_{\zeta_{1},t},J_{\zeta_{1}})\to \mathrm{CF}(H_{\zeta_{\infty},t},J_{\zeta_{\infty}}).
\end{equation}
The counts of rigid elements are finite because $\mathfrak{a}$ has curvature bounded above --- this is a consequence of requirement \ref{popdata3}, and also our assumption that $\omega$ is exact; see the discussion in \S\ref{sec:priori-energy-estim}.

As explained in \cite{schwarz-thesis-1995,ritter-jtopol-2013}, this map is a chain map, and its chain homotopy class is independent of the perturbation term $\mathfrak{p}$ and the homotopy class of the pair-of-pants data. The resulting map on homology is the product structure explained at the start of this section. The claimed naturality of the product operation follows from the fact that \eqref{eq:chain-level-product} commutes with continuation maps, up to chain homotopy; this standard fact is proved in, e.g., \cite{schwarz-thesis-1995,ritter-jtopol-2013}; one indication of why this holds is that if one ``glues'' continuation data to pair-of-pants data, one obtains new pair-of-pants data. The detailed argument involves Floer theoretic gluing and we defer the precise argument to the aforementioned references.

We end this section by commenting on one detail which will be important in the sequel. The product operation is defined by counting finite energy solutions to \S\ref{sec:general-form-floers}, and the fact that solutions have non-negative energy integral implies an inequality involving the values of $\int_{0}^{1}H_{\zeta_{i},t}$ at the three asymptotics (see the energy estimate Lemma \ref{lemma:energy-estimate}). For well-chosen pair-of-pants data, this inequality essentially says the product ``respects action filtrations;'' this will be used in a crucial way in the proof of Theorem \ref{theorem:main-floer} in \S\ref{sec:proof-of-main-floer}. We will return to this point and discuss the relevant action filtrations in \S\ref{sec:family-acti-filtr}. For other results on the interaction between the product structure and action filtration, we refer the reader to \cite{schwarz-PJM-2000,entov-poltero-IMRN-2003,kislev-shelukhin-GT-2021,alizadeh-atallah-cant-arXiv-2023}.

\subsubsection{The BV-operator on Floer cohomology}
\label{sec:bv-operator-on-HF}

The BV-operator is a natural endomorphism of functors $\mathscr{D}\to \mathrm{Vect}(\Z/2\Z)$:
\begin{equation*}
  \Delta:\mathrm{HF}\to \mathrm{HF};
\end{equation*}
abstract nonsense yields an endomorphism $\Delta$ of the persistence module $V$. The goal of this section is to briefly explain its construction. We refer the reader to \cite[pp.\,326]{abouzaid-EMS-2015} for a detailed exposition.

We define \emph{BV-data} to be a connection one-form $\mathfrak{a}=K_{\theta,s,t}\d s+ H_{\theta,s,t}\d t$ and almost complex structure $J$ on the family of domains $\R/\Z\times \Sigma\times W$, where $\Sigma=\R\times \R/\Z$ is the cylinder, satisfying:
\begin{enumerate}
\item $\mathfrak{a}=H_{t+\theta}\d t$ for $s\le -s_{0}$,
\item $\mathfrak{a}=H_{t}\d t$ for $s\ge s_{0}$,
\item $H_{\theta,s,t}=c_{\theta,s,t}r$ and $K_{\theta,s,t}=b_{\theta,s,t}r$ for $r\ge r_{0}$,
\item\label{bv-data-4} $\bd_{s}c_{\theta,s,t}\le \bd_{t}b_{\theta,s,t}$,
\item $J$ is fixed $\omega$-tame almost complex structure on $W$, which is Liouville equivariant outside of $\Omega(1)$,
\item $(H_{t},J)$ is admissible for defining the Floer complex,
\end{enumerate}
for some positive $r_{0},s_{0}$, and where $\theta$ is the $\R/\Z$ coordinate on $\R/\Z\times \Sigma\times W$.

As in \S\ref{sec:cont-maps-floer} and \S\ref{sec:pair-pants-product}, one can speak about homotopies of BV-data; such homotopies are required to satisfy the above properties with fixed $H_{t},J$ at all moments.

There is an analog of Lemma \ref{lemma:when-does-cdata-exist} and Lemma \ref{lemma:when-does-pants-data-exist}.
\begin{lemma}
  There exists BV-data for any data $(H_{t},J)$ which is admissible for defining $\mathrm{CF}$. Moreover, there is a unique homotopy class of BV-data.
\end{lemma}
\begin{proof}
  The existence of at most one homotopy class follows from the convexity of the space of BV-data with fixed $(H_{t},J)$. It remains only to prove there exists some BV-data.

  The construction is rather simple; one defines:
  \begin{equation*}
    H_{\theta,s,t}:=(1-\beta(s))H_{t+\theta}+\beta(s)H_{t},
  \end{equation*}
  where $\beta(s)$ is a standard cut-off function, and then defines:
  \begin{equation*}
    K_{\theta,s,t}:=\int_{0}^{t}\bd_{s}H_{\theta,s,\tau}d \tau-t\int_{0}^{1}\bd_{s}H_{\theta,s,\tau}d\tau=\int_{0}^{t}\bd_{s}H_{\theta,s,\tau}d\tau.
  \end{equation*}
  Note the similarity between this and the construction used in Lemma \ref{lemma:when-does-cdata-exist}. The formula for $K_{\theta,s,t}$ simplifies because $\bd_{s}H_{\theta,s,t}$ has zero time average --- the slope of $H_{\theta,s,t}$ is constant as $s$ varies.

  One easily verifies the enumerated conditions hold. Indeed, one verifies directly that $\bd_{s}H_{\theta,s,t}=\bd_{t}K_{\theta,s,t}$, which implies\footnote{Let us observe that, for any BV data, $\bd_{s}c_{\theta,s,t}-\bd_{t}b_{\theta,s,t}$ has zero time-average, and hence must be constant. Thus we could replace \ref{bv-data-4} by the apparently stronger condition $\bd_{s}c_{\theta,s,t}=\bd_{t}b_{\theta,s,t}$ without any loss of generality.} \ref{bv-data-4}. This completes the proof.
\end{proof}

The way BV-data is used to define a map $\Delta:\mathrm{CF}(H_{t},J)\to \mathrm{CF}(H_{t},J)$ follows the same lines as \S\ref{sec:cont-maps-floer} and \S\ref{sec:pair-pants-product}. Briefly, one picks generic perturbation one-form $\mathfrak{p}$ on $\R/\Z\times \Sigma\times W$, and defines the map whose coefficients are the counts of rigid finite-energy solutions to \S\ref{sec:floers-equation-general}.

Let us briefly explain how to interpret the count of solutions as coefficients in a matrix. If $(\theta,u)$ is a rigid finite energy solution, then $u(s,t)\to \gamma_{\mathrm{out}}(t+\theta)$ as $s\to-\infty$ and $u(s,t)\to \gamma_{\mathrm{in}}(t)$ as $s\to\infty$, where $\gamma_{\mathrm{in}}$ and $\gamma_{\mathrm{out}}$ are $1$-periodic orbits of the system generated by $H_{t}$; we consider $u$ as contributing to the coefficient in the matrix with entry $(\gamma_{\mathrm{out}},\gamma_{\mathrm{in}})$.

\subsubsection{PSS and inclusion of the constant loops}
\label{sec:incl-const-loops-in-HF}

Recall the full subcategory $\mathscr{D}(0)$ of objects $(H_{t},J)\in \mathscr{D}$ where the slope of $H_{t}$ is no less than $0$. In this section we explain how to define the PSS morphism, which can be considered as a natural transformation:
\begin{equation*}
  \mathrm{PSS}:H^{*}(W)\to \mathrm{HF}|_{\mathscr{D}(0)};
\end{equation*}
in order to interpret $\mathrm{PSS}$ as a natural transformation, the domain $H^{*}(W)$ is considered as a constant functor on $\mathscr{D}(0)$. Abstract nonsense (i.e., taking limits) shows that $\mathrm{PSS}$ descends to a homomorphism:
\begin{equation*}
  \mathrm{PSS}:H^{*}(W)\to V_{0};
\end{equation*}
see \S\ref{sec:floer-cohom-pers-defin}.

The PSS construction goes back to \cite{piunikhin-salamon-schwarz-1996}, and has been generalized to the setting of convex-at-infinity symplectic manifolds in \cite{frauenfelder-schlenk-IJM-2007,ritter-jtopol-2013}. Typically one works with a Morse theory version of $H^{*}(W)$. In our framework, recall from \S\ref{sec:remark-on-cohomology} that we prefer to work with a proxy for singular cohomology, and instead define $H^{*}(W)$ to be the group of smooth proper maps $C:S\to W$ modulo proper cobordisms.

We define $\mathrm{PSS}$-data to be a connection one-form $\mathfrak{a}$ and almost complex structure $J$ on $\Sigma\times W$ where\footnote{More precisely, in the terminology of \S\ref{sec:domains}, $\Sigma=\mathbb{C}P^{1}$ with $\Gamma_{-}=\set{\infty}$, so the punctured surface is $\C$.} $\Sigma=\C$ which, in the cylindrical coordinates $z=e^{-2\pi (s+i t)}$, satisfies:
\begin{enumerate}
\item $\mathfrak{a}=K_{s,t}\d s+H_{s,t}\d t$,
\item\label{pss-data-2} $\mathfrak{a}=0$ for $s\ge s_{0}$,
\item $\mathfrak{a}=H_{t}\d t$ for $s\le -s_{0}$,
\item\label{pss-data-4} $K_{s,t}=b_{s,t}r$, $H_{s,t}=c_{s,t}r$ for $r\ge r_{0}$,
\item $\bd_{s}c_{s,t}\le \bd_{t}b_{s,t}$,
\item $J$ is fixed as in \S\ref{sec:bv-operator-on-HF},
\item\label{pss-data-7} $(H_{t},J)\in \mathscr{D}$.
\end{enumerate}
Note that it follows from \ref{pss-data-2}, \ref{pss-data-4}, and \ref{pss-data-7} that $(H_{t},J)\in \mathscr{D}(0)$. PSS-data is similar to continuation data from $(0,J)$ to $(H_{t},J)$, the only difference being that $(0,J)$ is not actually admissible for defining the Floer complex, and so we consider the puncture at $s=\infty$ (namely $z=0$) as a removable singularity.

For a cycle $C:S\to W$, we define $\mathrm{PSS}(C)\in \mathrm{CF}(H_{t},J)$ as a cycle. Fix PSS data $(\mathfrak{a},J)$, which we pull back to the family $S\times \Sigma\times W$. Fix a generic perturbation term $\mathfrak{p}$, and count the rigid finite energy solutions $(\sigma,u)$ to \S\ref{sec:general-form-floers} satisfying the incidence constraint:
\begin{equation*}
  u(0)=C(\sigma).
\end{equation*}
One counts the asymptotic orbits of $u$ as a linear combination in $\mathrm{CF}(H_{t},J)$. The arguments from \cite{piunikhin-salamon-schwarz-1996,frauenfelder-schlenk-IJM-2007,ritter-jtopol-2013} adapt easily to the present case to show that this linear combination is a cycle in $\mathrm{CF}(H_{t},J)$. The cohomology class of this cycle is independent of $\mathfrak{p}$, the PSS-data chosen, and even the proper cobordism class of $F$. By construction, one sees that $\mathrm{PSS}$ is a linear map $H^{*}(W)\to \mathrm{HF}(H_{t},J)$. Finally, the standard Floer theory gluing arguments show that $\mathrm{PSS}$ is a natural homomorphism, i.e., it commutes with continuation maps.

This concludes the specification of the Floer theory framework used in the paper.

\subsection{Evaluation maps and ball embeddings}
\label{sec:evaluation-maps-ball}

Fix a map $f:N\to \Omega$, where $N$ is a closed manifold. The goal of this section is to explain how a lift of $f$ to $\mathfrak{B}(a,\Omega)/U(n)$, as in \S\ref{sec:introduction}, is used to define an evaluation map $\mathfrak{e}:V_{c}\to \Z/2\Z$ for $0<c<a$ which satisfies:
\begin{equation*}
  \mathfrak{e}(\mathrm{PSS}(\beta))=1
\end{equation*}
provided that $\beta=[C]$ where $C:S\to W$ is a cycle with non-zero homological intersection number with $f$.

In particular, a consequence we will deduce by the end of this section is:
\begin{proposition}\label{prop:concrete-consequence}
  If $f$ lifts to $\mathfrak{B}(a,\Omega)/U(n)$, and $\beta\in H^{*}(W)$ has non-zero homological intersection number with $f$, then $\mathrm{PSS}(\beta)\ne 0$ in $V_{c}$ for $c<a$.
\end{proposition}
\emph{Remark}. One can define a capacity by the number:
\begin{equation*}
  c(\beta,\Omega)=\inf\set{c:\mathrm{PSS}(\beta)=0\text{ in }V_{c}};
\end{equation*}
Proposition \ref{prop:concrete-consequence} implies $\mathrm{Gr}(f,\Omega)\le c(\beta,\Omega)$. In the special case when $\beta=[W]$, we conclude the estimate that $\mathrm{Gr}(f,\Omega)\le c([W],\Omega)$; see, e.g., \cite{benedetti-kang-JFPTA-2022}.

Without any true loss of generality, we assume that $f$ lifts to $\mathfrak{B}(a',\Omega)/U(n)$ for some number $a'>a$. We will appeal to such an extension in some of the subsequent arguments.

\subsubsection{Family of Hamiltonian functions associated to a family of balls}
\label{sec:system-of-hamiltonians}

Introduce a convex function $\mu:\R\to \R$ such that:
\begin{enumerate}
\item $\mu(x)=x$ for $x\ge 1$,
\item $\mu(x)=1/2$ for $x\le 0$,
\end{enumerate}
and introduce the autonomous Hamiltonian function: $$G_{c,\delta}=c\delta\mu(\delta^{-1}(r-1))+c$$ for a fixed non-negative constant $c$ and a very small parameter $\delta>0$. Then:
\begin{enumerate}
\item $G_{c,\delta}$ agrees with $cr$ outside of a neighbourhood of $\Omega$ (indeed, the neighborhood is $\set{r\ge 1+\delta}$),
\item $G_{c,\delta}$ equals the constant $c+c\delta/2$ on the domain $\Omega$.
\end{enumerate}
One should consider this Hamiltonian as the background Hamiltonian system.

Recall that the parametric Gromov width concerns maps $f:N\to \Omega$ together with:
\begin{enumerate}
\item an open cover $U_{\alpha}$ of $N$,
\item smooth maps $g_{\alpha\beta}:U_{\alpha}\cap U_{\beta}\to U(n)$,  
\item extensions $F_{\alpha}:U_{\alpha}\times B(a)\to \Omega$ satisfying $F_{\beta}(\eta,z)=F_{\beta}(\eta,g_{\alpha\beta}(\eta)z)$ for $\eta\in U_{\alpha}\cap U_{\beta}$.
\end{enumerate}
We will use the extensions of $F_{\alpha}$ to define an $\eta$-parametric family of perturbations of the Hamiltonian $G_{c,\delta}$.

Let us denote by $B_{\eta}$ the image of of $z\mapsto F_{\alpha}(\eta,z)$, for any $\alpha$ such that $\eta\in U_{\alpha}$. The image of this ball, and the radial coordinate $\abs{z}$, is independent of the choice of $\alpha$.

Let $D_{\eta}:W\to \R$ be a smooth family of functions, parametrized by $\eta\in N$, with compact support in $\Omega$ and such that:
\begin{enumerate}
\item $\set{D_{\eta}<0}$ is the interior of $B_{\eta}$,
\item $D_{\eta}\circ B_{\eta}(z)=\pi \abs{z}^{2} - a$.
\end{enumerate}
Such a family will be constructed in Lemma \ref{lemma:construction-of-D} below using slight extensions of $B_{\eta}$. Note that the constants are chosen so that $D_{\eta}$ vanishes on the boundary of $B_{\eta}$.

The $\eta$-parametric family of perturbations of $G_{c,\delta}$ we will use in our construction is:
\begin{equation}\label{eq:special-hamiltonian}
  H_{c,\delta,\epsilon,\eta}=G_{c,\delta}+\epsilon D_{\eta},
\end{equation}
which depends on parameters $c,\delta,\epsilon>0$, and $\eta\in N$. We will denote the Hamiltonian vector field of this function by $X_{c,\delta,\epsilon,\eta}$.

\emph{Remark}. In the following lemma, we discuss the \emph{action} of an orbit $\gamma(t)$ of a system generated by Hamiltonian function $H_{t}$, which is defined as:
\begin{equation*}
  \textstyle \int H_{t}(\gamma(t))dt-\int \gamma^{*}\lambda,
\end{equation*}
as is typical in Floer theory in exact symplectic manifolds.

\emph{Remark}. We will also refer to the period spectrum of the Reeb vector field $X_{r}$, which we denote by $\mathrm{Per}(X_{r})$.

\begin{lemma}\label{lemma:construction-of-D}
  The family $D_{\eta}$ can be chosen so that, for $\epsilon<1$ and $c$ not a period of $X_{r}$, the contractible 1-periodic orbits of $X_{c,\delta,\epsilon,\eta}$ are of three types:
  \begin{enumerate}
  \item constant orbits in $\Omega$ lying outside $B_{\eta}$, with action at least $c+c\delta/2$,
  \item a single constant orbit at the center of $B_{\eta}$, with action $c+c\delta/2-\epsilon a$,
  \item non-constant periodic orbits of $X_{r}$ lying in the hypersurfaces defined by $c\mu'(\delta^{-1}(r_{0}-1))=b\in \mathrm{Per}(X_{r})$, with action at least $c-b$.
  \end{enumerate}
  In particular, if $\epsilon a>b+c\delta/2$ for all $b\in \mathrm{Per}(X_{r})\cap [0,c]$, there are no non-constant finite energy Floer cylinders for $X_{c,\delta,\epsilon,\eta}$ whose negative asymptotic is in (2). If $a>c$, then we can achieve this conclusion for $\epsilon$ close enough to $1$ and $\delta$ close enough to $0$.
\end{lemma}
\begin{proof}
  As mentioned in \S\ref{sec:evaluation-maps-ball}, we will appeal to a small extension $B(a)\subset B(a')$. Fix a smooth function $\rho:\R\to \R$ such that $\rho(x)=x$ for $x\le a$, $\rho(x)\ge a$ for $x\in [a,a']$, and $\rho(x)=a$ for $x\ge a'$, and such that $\abs{\rho'}\le 1$ holds at all points.

  Then we define $D(z)=\rho(\pi \abs{z}^{2})-a$ which can be pushed forward using the ball embedding $B_{\eta}(a')\to \Omega$ to define the desired family of functions $D_{\eta}$.

  Then it is easily seen that the only orbits inside $\Omega$ are the orbits of types (1) and (2) with the claimed actions. The only other orbits are orbits of $G_{c,\delta}$ appearing outside of $\Omega$. A standard computation shows that these orbits are orbits lying in a fixed radius $r=r_{0}$ satisfying $c\mu'(\delta^{-1}(r_{0}-1))=b\in \mathrm{Per}(X_{r})$.

  In this case we use the fact that $\lambda(X_{r})=r$ to compute their action as:
  \begin{equation*}
    \text{action}=c\delta\mu(\delta^{-1}(r_{0}-1))+c-c\mu'(\delta^{-1}(r_{0}-1))r_{0}.
  \end{equation*}
  Simplifying we see that:
  \begin{equation*}
    \text{action}=c-b+cf(r_{0}-1)
  \end{equation*}
  where:
  \begin{equation*}
    f(x)=-\mu'(\delta^{-1}x)x+\delta\mu(\delta^{-1}x).
  \end{equation*}
  Using the convexity of $\mu$, one sees that $f'(x)\le 0$, and since $f(x)=0$ for $x\ge \delta$ (since $\mu(s)=s$ for $s\ge 1$), we conclude that $f(x)\ge 0$. This gives the desired result.
\end{proof}

\emph{Remark.} In the proof we implicitly appealed to the well-known fact that the action of the left end of a Floer cylinder is at least the action of the right end; this is a special case of the general energy identity proved in Lemma \ref{lemma:energy-identity}.

\begin{figure}[h]
  \centering
  \begin{tikzpicture}
    \draw (0,0)--(2,2)coordinate(A)to[out=45,in=180](2.3,2.1)to[out=0,in=180](2.8,2)coordinate(AP)--(3,2);
    \draw[dashed] (A)--+(0,-2.5)node[above,fill=white]{$a$};
    \draw[dashed] (AP)--+(0,-2.5)node[above,fill=white]{$a'$} (AP)--+(-3,0)node[left]{$a$};
  \end{tikzpicture}
  \caption{Graph of $\rho:\R\to \R$}
  \label{fig:graph-of-rho}
\end{figure}

\subsubsection{Definition of the evaluation-map}
\label{sec:defin-evaluation-map}

We continue in the context of the previous subsection. Let $(H_{t},J)\in \mathscr{D}$ where $H_{t}$ has slope at most $a>0$, where $a$ is the capacity appearing in the lift of $f$ to $\mathfrak{B}(a,\Omega)/U(n)$.

In this subsection, we explain how a family of continuation data from $(H_{t},J)$ to $(H_{c,\delta,\epsilon,\eta},J_{\eta})$ parametrized by $\eta\in N$ can be used to define an evaluation map provided $\epsilon a>c(1+\delta/2)$ and where $c$ is no less than the slope of $H_{t}$.

Let us therefore define \emph{evaluation data} to be a connection one-form $\mathfrak{a}$ and almost complex structure $J$ on the family on $N\times \Sigma\times W$ where $\Sigma=\R\times \R/\Z$, satisfying:
\begin{enumerate}
\item $\mathfrak{a}=K_{\eta,s,t}\d s+H_{\eta,s,t}\d t$,
\item $K_{s,t}=b_{\eta,s,t}r$ and $H_{s,t}=a_{\eta,s,t}r$ for $r\ge r_{0}$,
\item $\bd_{s}c_{\eta,s,t}\le \bd_{t}b_{\eta,s,t}$,
\item $\mathfrak{a}=H_{t}\d t$ for $s\ge s_{0}$,
\item $\mathfrak{a}=H_{c,\delta,\epsilon,\eta}\d t$ for $s\le -s_{0}$, where $\epsilon a>c(1+\delta/2)$,
\item $J$ is fixed, i.e., independent of $\eta,s,t$, for $s\ge s_{0}$,
\item $J=J_{\eta}$ for some fixed family $J_{\eta}$ for $s\le -s_{0}$.
\end{enumerate}
We comment that such data deviates slightly from the framework established in \S\ref{sec:hamilt-conn-surf} and \S\ref{sec:general-form-floers}, in that we allow the asymptotic data to depend on $\eta\in N$ at the left end. We will explain shortly why this does not cause any issues.

Introducing a generic perturbation one-form $\mathfrak{p}$, one can still consider the finite energy solutions to Floer's equation \S\ref{sec:floers-equation-general}.

For data $(\mathfrak{a},\mathfrak{p},J)$, introduce $\mathscr{M}$ as the moduli space of finite energy solutions $(\eta,u)$ such that the left asymptotic of $u$ is the constant orbit located at the center of the ball $B_{\eta}$. Since this central orbit is non-degenerate for each $\eta$, the space of solutions $\mathscr{M}$ behaves as if the asymptotic data at the left end were independent of $\eta$, at least as far as the results in \S\ref{sec:general-form-floers} are concerned.

\emph{Remark}. Some care is needed to properly define the linearization framework when the asymptotic data changes; see, e.g., the proof of Lemma \ref{lemma:glued-evaluation-PSS}.

The key result about $\mathscr{M}$ is the following:
\begin{lemma}\label{lemma:e-map-only-one-breaking}
  Suppose the perturbation term $\mathfrak{p}$ is such that $\mathscr{M}$ is cut transversally. If $(\eta_{n},u_{n})$ is a sequence of rigid solutions $\mathscr{M}$, then $(\eta_{n},u_{n})$ has a convergent subsequence in $\mathscr{M}$. If $(\eta_{n},u_{n})$ is a sequence of solutions in the $1$-dimensional component of $\mathscr{M}$, then $(\eta_{n},u_{n})$ has a subsequence which either converges in $\mathscr{M}$, or which breaks into configuration of a rigid solution in $\mathscr{M}$ on the left and a Floer differential cylinder for $(H_{t},J)$ on the right (in the sense discussed in \S\ref{sec:compactness}).
\end{lemma}
\begin{proof}
  This is mostly standard Floer theory; the only non-standard thing we need to check is that a sequence of solutions in the 1-dimensional component of $\mathscr{M}$ does not break into a configuration of a non-stationary Floer differential cylinder for $(H_{c,\delta,\epsilon,\eta},J_{\eta})$ on the left and a rigid solution in $\mathscr{M}$ on the right. However, such a breaking can be precluded by action considerations, since we have shown in Lemma \ref{lemma:construction-of-D} that there are no non-constant Floer differential cylinders for $(H_{c,\delta,\epsilon,\eta},J_{\eta})$ whose left asymptotic is the central orbit.
\end{proof}

To define the map $\mathfrak{e}:\mathrm{CF}(H_{t},J)\to \Z/2$, we define:
\begin{equation*}
  \mathfrak{e}(\gamma)=\#\set{u\in \mathscr{M}:\lim_{s\to\infty}u(s,t)=\gamma(t)}\text{ mod }2.
\end{equation*}
Lemma \ref{lemma:e-map-only-one-breaking} and standard Floer theory gluing arguments imply that $\mathfrak{e}$ is a chain map, i.e., $\mathfrak{e}\circ d=0$.

Since the space of evaluation data is weakly contractible, for fixed input system $(H_{t},J)$, the standard arguments show that the chain homotopy class of $\mathfrak{e}$ is independent of the choice of evaluation data. The resulting map on homology is denoted $\mathfrak{e}:\mathrm{HF}(H_{t},J)\to \Z/2$. Summarizing, we have constructed a map:
\begin{equation*}
  \mathfrak{e}:\mathrm{HF}(H_{t},J)\to \Z/2\Z
\end{equation*}
on the full subcategory $\mathscr{D}_{<a}\subset \mathscr{D}$ consisting of objects with slope at most $a$. Finally,
the usual gluing arguments prove that $\mathfrak{e}$ is a natural transformation, where $\Z/2\Z$ is interpreted as a constant functor.

In the next section, we prove that $\mathfrak{e}\circ \mathrm{PSS}:H^{*}(W)\to \Z/2\Z$ sends $\beta$ to $1$ provided that $\beta$ has non-zero intersection with the cycle $\eta\mapsto F(\eta,0)$.

\subsubsection{Non-triviality of the evaluation map}
\label{sec:non-triv-eval}

We continue in the context of the previous two subsections. Let $\beta\in H^{*}(W)$ have non-zero homological intersection with the cycle $\eta\mapsto F(\eta,0)$. Pick a representative $C:S\to W$ of $\beta$ which is transverse to the cycle $\eta\mapsto F(\eta,0)$, so that the set of pairs $(\sigma,\eta)$ solving $C(\sigma)=F(\eta,0)$ is a finite odd set of points.

Define \emph{glued evaluation PSS data} to be a connection one-form $\mathfrak{a}$ and almost complex structure $J$ on the family $N\times \C\times W$ such that, in the cylindrical coordinates $z=e^{-2\pi (s+it)}$,
\begin{enumerate}
\item $\mathfrak{a}=K_{\eta,s,t}\d s+H_{\eta,s,t}\d t$,
\item $K_{s,t}=b_{\eta,s,t}r$ and $H_{s,t}=c_{\eta,s,t}r$ for $r\ge r_{0}$,
\item $\bd_{s}c_{\eta,s,t}\le \bd_{t}b_{\eta,s,t}$,
\item $\mathfrak{a}=0$ for $s\ge s_{0}$,
\item $\mathfrak{a}=H_{c,\delta,\epsilon,\eta}\d t$ for $s\le -s_{0}$, where $\epsilon a>c(1+\delta/2)$, $c\ge 0$,
\item $J=J_{+,\eta}$ for $s\le -s_{0}$ and $J=J_{-,\eta}$ for $s\ge s_{0}$, for two fixed families $J_{-,\eta}$ and $J_{+,\eta}$.
\end{enumerate}
Roughly speaking, glued evaluation PSS data is the data one obtains by gluing the evaluation data from \S\ref{sec:defin-evaluation-map} to the PSS-data from \S\ref{sec:incl-const-loops-in-HF}.

Similarly to \S\ref{sec:incl-const-loops-in-HF} and \S\ref{sec:defin-evaluation-map}, we pull back glued evaluation PSS data to the family $S\times N\times \C\times W$, and for a generic perturbation term $\mathfrak{p}$, we consider the moduli space $\mathscr{M}$ of finite energy solutions $(\sigma,\eta,u)$ to \S\ref{sec:floers-equation-general} satisfying the incidence condition:
\begin{equation*}
  u(0)=C(\sigma),
\end{equation*}
and such that the asymptotic orbit of $u$ is the constant orbit located at the center of the ball $B_{\eta}$.

\begin{lemma}\label{lemma:glued-evaluation-PSS}
  We have the following:
  \begin{enumerate}[label=(\alph*)]
  \item\label{lemma-GE-PSS-a} all components of $\mathscr{M}$ are compact,
  \item\label{lemma-GE-PSS-b} the count of points in the rigid component of $\mathscr{M}$ is independent of the choice of glued evaluation PSS data and perturbation one-form $\mathfrak{p}$,
  \item\label{lemma-GE-PSS-c} the count of points in the rigid component of $\mathscr{M}$ equals $\mathfrak{e}(\mathrm{PSS}(\beta))$,
  \item\label{lemma-GE-PSS-d} the count of points in the rigid component of $\mathscr{M}$ is odd;
  \end{enumerate}
  all counts are taken mod 2. Thus $\mathfrak{e}(\mathrm{PSS}(\beta))=1$, proving Proposition \ref{prop:concrete-consequence}.
\end{lemma}
\begin{proof}
  Part \ref{lemma-GE-PSS-a} follows from the same compactness-up-to-breaking argument used in the proof of Lemma \ref{lemma:e-map-only-one-breaking}. A sequence of solutions will, in general, have a subsequence which converges or breaks into a configuration of a non-stationary Floer differential cylinder and another solution in $\mathscr{M}$. There is only one end where the breaking can happen (the left end). Since we require that solutions in $\mathscr{M}$ are asymptotic to the central orbit at the left end, and the only Floer differential cylinders whose left end is the central orbit are constant (by Lemma \ref{lemma:construction-of-D}), we conclude there can be no breaking.

  Part \ref{lemma-GE-PSS-b} follows from a cobordism argument: because the space of data is path connected, the rigid counts for two choices of data are cobordant finite sets of points, and hence have the same cardinality mod 2; here we note that the parametric moduli space used in such a cobordism argument will be compact by the exact same argument used for \ref{lemma-GE-PSS-a}.

  Part \ref{lemma-GE-PSS-c} follows from Floer theoretic gluing. By picking the glued evaluation PSS data as a concatenation of evaluation data (\S\ref{sec:defin-evaluation-map}) and PSS-data (\S\ref{sec:incl-const-loops-in-HF}), a standard gluing argument shows the count of rigid points in $\mathscr{M}$ equals the count obtained by applying the map $\mathfrak{e}$ to the cycle $\mathrm{PSS}(C)$.

  The non-standard part of the argument is establishing \ref{lemma-GE-PSS-d}. For this part of the argument, we use an explicit choice of data, and directly show the count equals the (odd) count of pairs $(\sigma,\eta)$ satisfying $C(\sigma)=F(\eta,0)$.

  Let us pick the data such that $c=0$, so that $\delta$ becomes irrelevant, and:
  \begin{equation*}
    H_{c,\delta,\epsilon,\eta}=\epsilon D_{\eta}.
  \end{equation*}
  Pick $\mathfrak{a}=(1-\beta(s))\epsilon D_{\eta}\d t,$ and pick $J=J_{\eta}$ (so $J_{+}=J_{-}$), where $J_{\eta}$ agrees with the standard almost complex structure when pulled back to the ball $B_{\eta}$. We will show that all solutions in $\mathscr{M}$ are regular with $\mathfrak{p}=0$, and therefore the count without perturbation equals the count with perturbation.

  \emph{Claim.} For $\epsilon$ sufficiently small, any solution $(\sigma,\eta,u)$ with this specific data must be such that $u$ is constant and equal to the center of the ball $B_{\eta}$.

  This claim follows from a simple adiabatic compactness argument: restricting to a half-infinite cylinder $(-\infty,R]\times \R/\Z$ we can pull $u$ back to a map valued in $B(a)$ satisfying:
  \begin{equation}\label{eq:equation-ev-PSS-in-ball}
    \bd_{s}u+J(\bd_{t}u-\epsilon (1-\beta(s))X(u))=0
  \end{equation}
  where $X$ is the Hamiltonian vector field for $\pi \abs{z}^{2}$, and $J$ is the standard almost complex structure. Either we can take $R=+\infty$, or we can take $R$ to be maximal in which case $u(R,t)$ hits the boundary $\bd B(a)$ for some $t$.

  We can take $\epsilon$ small enough so that $\abs{\bd_{t}u}$ is everywhere less than $\sqrt{a/\pi}$, by a simple compactness argument.

  Consider the center of mass $\xi(s)=\int_{0}^{1}u(s,t)d t$, so $\xi:(-\infty,R]\to B(a)$ is a smooth curve. Using the linearity of the above equation (bearing in mind that $X(u)$ is actually a linear function of $u$) one concludes that:
  \begin{equation*}
    \bd_{s}\xi(t)=\epsilon(1-\beta(s))JX(\xi).
  \end{equation*}
  It is well-known that $JX$ is a vector field which points radially inwards. Thus, since $\xi(s)$ converges to $0$ as $s\to-\infty$, we must have that $\xi(s)=0$ holds identically. It therefore follows that $u(s,t)$ has mean zero for $s\le R$.

  Thus we have $R=+\infty$, since $u$ cannot touch the boundary $\bd B(a)$ because $\abs{\bd_{t}u(R,t)}$ is smaller than the radius of $\bd B(a)$ and $u(R,-)$ has mean zero. Then one estimates the energy integral of $u$ by:
  \begin{equation*}
    \text{energy of $u$}=\int \abs{\bd_{s}u }^{2}d sd t\le -a+\int \beta'(s)a=0,
  \end{equation*}
  which implies $u$ must be constant, and thus equal to the center of the ball.

  Thus $\mathscr{M}$ is in bijection with the set of pairs $(\sigma,\eta)$ satisfying $C(\sigma)=F(\eta,0)$; the bijection is simply the projection map $(\sigma,\eta,u)\mapsto (\sigma,\eta)$. It remains only to prove that these solutions are in fact regular.

  To analyze the regularity of these solutions, we need to open the ``black box'' of the linearization framework. Let us fix a solution $(\sigma_{0},\eta_{0},u_{0})$, and use the coordinate system induced by the embedded ball $B_{\eta_{0}}$.

  If $(\sigma,\eta,u)$ is nearby $(\sigma_{0},\eta_{0},u_{0})$, then it solves the equation provided that:
  \begin{equation}\label{eq:lemma-ev-PSS-non-linear-framework}
    \left\{
      \begin{aligned}
        &u=w+f(\eta)\text{ for }w\in W^{1,p}(\Sigma)\\
        &\bd_{s}w+J_{\eta}(w+f(\eta))(\bd_{t}w-V_{\eta}(w))=0\\
        &w(0)+f(\eta)-c(\sigma)=0,\\
      \end{aligned}
    \right.
  \end{equation}
  where $f(\eta)$, $c(\sigma)=C(\sigma)$, and $V_{\eta}(z)=\epsilon X_{\eta}(z+f(\eta))$ are represented in the coordinate system. Note that the maps $c$ and $f$ are transverse at their intersection $(\sigma_{0},\eta_{0})$.

  \emph{Remark.} When linearizing with a varying asymptotic, it is important to use such an auxiliary map $w$, as it is defined on a fixed Banach space of maps.

  The second two equations in \eqref{eq:lemma-ev-PSS-non-linear-framework} can be considered as a non-linear map defined on a neighborhood of $0$ between Banach manifolds:
  \begin{equation*}
    W^{1,p}\times S\times N\to L^{1,p}\times \R^{2n}.
  \end{equation*}
  Differentiating this non-linear map gives the linearized operator.

  If $(\sigma',\eta',w')$ is a tangent vector at $(\sigma_{0},\eta_{0},0)$, the linearized operator is:
  \begin{equation*}
    \left[
      \begin{aligned}
        &\bd_{s}w'+J\bd_{t}w'-\epsilon JX(w')-J\partial V_{\eta}/\partial \eta(0)\eta'\\
        &w'(0)+\d f(\eta_{0})\eta'-\d c(\sigma_{0})\sigma'
      \end{aligned}\right]\in L^{p}\oplus \R^{2n}.
  \end{equation*}
  The fact that $V_{\eta}(0)=0$ holds for all $\eta$ implies $\partial V_{\eta}/\partial \eta(0)=0$. Thus the first component is a Cauchy-Riemann operator $W^{1,p}\to L^{p}$. Because $\epsilon X$ has a non-degenerate orbit at the origin, this operator is an isomorphism $W^{1,p}\to L^{p}$ (by standard Fredholm theory for Floer's equation). On the other hand, since $f$ and $c$ are transverse and have complementary dimensions, the second equation is an isomorphism $TS_{\sigma_{0}}\oplus TN_{\eta_{0}}\to \R^{2n}$. Thus the linearized operator is an isomorphism, and so the constant solution $(\sigma_{0},\eta_{0},u_{0})$ is regular. This completes the proof.
\end{proof}

\subsection{Family Floer cohomology}
\label{sec:family-floer-cohom}

In this section we develop a version of family Floer cohomology following the scheme in \cite{hutchings-agt-2008}.

Let us briefly comment on the strategy. The identity $\mathfrak{e}(\mathrm{PSS}(\beta))=1$ proved in \S\ref{sec:evaluation-maps-ball} implies that any cycle $\sum \gamma_{i}$ in $\mathrm{CF}(H_{t},J)$ representing $\mathrm{PSS}(\beta)$ has at least one orbit $\gamma_{i}$ which appears as the right end of one of the cylinders used to define $\mathfrak{e}$; see Figure \ref{fig:consequence-evaluation-map}.

\begin{figure}[h]
  \centering
  \begin{tikzpicture}[xscale=.9,yscale=.4]
    \draw (0,0) circle (0.2 and 1) (0,1)--+(10,0) (0,-1)--+(10,0) (10,0) circle (0.2 and 1);
    \path (-0.3,0)node[left]{center of $B_{\eta}$}--node{$\bd_{s}u-Y_{\eta,s,t}(u)+J_{\eta,s,t}(u)(\bd_{t}u-X_{\eta,s,t}(u))=0$}(10.3,0)node[right]{$\gamma_{i}$};
\end{tikzpicture}
\caption{A consequence of $\mathfrak{e}(\sum \gamma_{i})=1$ is the existence of a cylinder $(\eta,u)$ joining $\gamma_{i}$ to the center of the ball $B_{\eta}$. In the figure $Y,X$ are the Hamiltonian generators of $K,H$ where $(\mathfrak{a}=K_{\eta,s,t}\d s+H_{\eta,s,t}\d t,J_{\eta,s,t})$ is evaluation data as in \S\ref{sec:defin-evaluation-map}.}
\label{fig:consequence-evaluation-map}
\end{figure}

It is a general principle that the action of the left asymptotic (the center of the ball) is at least the action of the right asymptotic (the orbit $\gamma_{i}$), up to an error depending on the curvature of the Hamiltonian connection determined by $\mathfrak{a}$; see Lemma \ref{lemma:energy-estimate}.

In Lemma \ref{lemma:construction-of-D} it was shown that the action of the center of the ball is equal to $(c+c\delta/2)-a$; thus if we can prove that:
\begin{enumerate}
\item\label{3p6strat1} the action of $\gamma_{i}$ is greater than $c+c\delta/2-a$, and
\item\label{3p6strat2} the error coming from the curvature can be made arbitrarily small,
\end{enumerate}
then we have successfully obstructed the family of balls.

Item \ref{3p6strat1} will ultimately be a consequence of Theorem \ref{theorem:main-floer}'s hypothesis:
\begin{equation*}
  \mathrm{PSS}(\beta)=\Delta(\zeta_{1})\ast\zeta_{2}\text{ holds in }V_{c_{1}+c_{2}},
\end{equation*}
where $\zeta_{i}\in V_{c_{i}}$, with $c_{i}>0$, provided $c:=c_{1}+c_{2}<a$.

However, because the cylinder in Figure \ref{fig:consequence-evaluation-map} appears at an unknown parameter value $\eta\in N$, it is not possible to uniformly control the curvature because the input system $(H_{t},J)$ is independent of $\eta$.

For this reason, we instead work with a \emph{family of input systems} $(H_{\eta,t},J_{\eta})$, and upgrade the map $\mathfrak{e}$ to be defined on a suitable family version of Floer cohomology. As we will show, with the family version, one can control the curvature and simultaneously achieve \ref{3p6strat1} and \ref{3p6strat2}.

The family Floer cohomology essentially contains no new algebraic information (it is the regular Floer cohomology tensored with the cohomology of $N$); however, family Floer cohomology grants one access to new action filtrations on the Floer complex.

The rest of this section is dedicated to developing the necessary theory, and, in \S\ref{sec:proof-of-main-floer}, completing the proof of Theorem \ref{theorem:main-floer}.

\subsubsection{Definition of family Floer cohomology}
\label{sec:defin-family-floer}

Let $N$ be a compact manifold. By definition, \emph{family Floer data} is:
\begin{enumerate}
\item a Morse-Smale pseudogradient $P$ on $N$ (e.g., a Morse-Smale gradient-like vector field in the terminology of \cite{milnor-book-PUP-1965}),
\item a smooth family $H_{\eta,t}$ in $\mathscr{H}$ parametrized by $N\times \R/\Z$,
\item a fixed almost complex structure $J$, which is, as usual, $\omega$-tame and Liouville invariant in the convex end,
\end{enumerate}
which satisfies:
\begin{enumerate}[resume]
\item\label{FF-data-4} for each zero $\eta_{0}$ of $P$, there is a neighborhood $\mathrm{Op}(\eta_{0})$ of $\eta_{0}$ such that $H_{\eta,t}=H_{\eta_{0},t}$, and $(H_{\eta_{0},t},J)$ is admissible for defining $\mathrm{CF}(H_{\eta_{0},t},J)$.
\item $H_{\eta,t}=cr$ for $r\ge r_{0}$.
\end{enumerate}
\emph{Note.} For simplicity, we do not allow $H_{\eta,t}$ to have a time-dependent slope, unlike the data for the non-family Floer complex. The reason for this is that we have no need to appeal to time-dependent slopes in the family Floer complex; on the other hand, time-dependent slopes will be used for the regular Floer complex in \S\ref{sec:string-topol-floer-cohom}.

For each zero $\eta_{0}$ of $P$, associate the one-dimensional vector space $\Z/2\Z \eta_{0}$; for family Floer data $(P,H_{\eta,t},J)$ define the \emph{family Floer complex}:
\begin{equation*}
  \mathrm{CFF}(P,H_{\eta,t},J)=\bigoplus\mathrm{CF}(H_{\eta_{0},t},J)\otimes \Z/2\Z\eta_{0},
\end{equation*}
where the direct sum is over the zeros of $P$. Thus it makes sense to refer to a generator $\gamma\otimes \eta_{0}$ whenever $\gamma(t)$ is a 1-periodic orbit of $H_{\eta_{0},t}$.

In the rest of this section we will explain how to define the differential $d_{\mathrm{FF}}$ on the family Floer complex. As part of the definition, we will need to require that certain auxiliary moduli spaces are cut transversally, and this will impose a genericity condition on the family $H_{\eta,t}$.

For each pair $(\eta_{1},\eta_{0})$ of zeros of $P$, one can consider the (open) manifold $\mathscr{P}(\eta_{1},\eta_{0})$ of parametrized flow lines $\pi(s)$ from $\eta_{1}$ to $\eta_{0}$ for the negative pseudogradient $-P$; here $\eta_{1}$ is the left asymptotic and $\eta_{0}$ is the right asymptotic.

Fixing $(\eta_{1},\eta_{0})$, there is an induced connection and perturbation one-form, and almost complex structure, on the family $\mathscr{P}(\eta_{1},\eta_{0})\times \Sigma\times W$, where $\Sigma$ is the cylinder; precisely, one defines the family of data by:
\begin{equation*}
  \begin{aligned}
    \mathfrak{a}_{\pi,s,t}&=H_{\pi(s),t}\d t&J_{\pi,s,t}&=J.
  \end{aligned}
\end{equation*}
It is important to note that, since $H_{\eta,t}=cr$ for $r\ge r_{0}$, the curvature of $\mathfrak{a}$ vanishes outside of a compact set.

One can consider the finite energy solutions to \S\ref{sec:floers-equation-general} for this family of data. Unwinding the definitions, one sees that a solution is a pair $(\pi,u)$ solving:
\begin{equation*}
  \bd_{s}u+J(u)(\bd_{t}u-X_{\pi(s),t}(u))=0.
\end{equation*}
One defines $\mathscr{M}(\eta_{1},\eta_{0})$ to be the moduli space of such finite energy solutions.
\begin{lemma}
  For generic perturbation of $H_{\eta,t}$ away from the zeroes of $P$, all the moduli spaces $\mathscr{M}(\eta_{1},\eta_{0})$ are cut transversally.
\end{lemma}
\begin{proof}
  Note that, since $\pi(s)$ lies in the neighborhood $\mathrm{Op}(\eta_{0})\cup \mathrm{Op}(\eta_{1})$ outside of a compact set, $X_{\pi(s),t}$ is $s$-independent outside of a compact set.

  If $\eta_{1}\ne \eta_{0}$, then there is a non-empty interval where $X_{\pi(s),t}$ is actually $s$-dependent. In that case we can perturb $X_{\eta,t}$ generically away from the zeroes in such a way that $X_{\pi(s),t}$ is modified in a generic fashion only a compact part of the cylinder; such variations are sufficient to ensure transversality (compare with \cite[pp.\,971]{seidel-GAFA-2015}).

  In the case when $\eta_{1}=\eta_{0}$, then $u$ solves the Floer differential equation for $(H_{\eta_{0},t},J)$, and so requirement \ref{FF-data-4} implies $u$ is regular. In other words, the moduli space $\mathscr{M}(\eta_{0},\eta_{0})$ is simply the moduli space considered in \S\ref{sec:floer-complex} for the admissible data $(H_{\eta_{0},t},J)$.
\end{proof}
We add one additional requirement to our family Floer data:
\begin{enumerate}[resume]
\item the data $H_{\eta,t}$ is chosen generically away from the zeros of $P$ so that all moduli spaces $\mathscr{M}(\eta_{1},\eta_{0})$ are cut transversally.
\end{enumerate}
Data $(P,H_{\eta,t},J)$ which satisfies all of the requirements is said to be admissible for defining the family Floer complex.

The moduli space $\mathscr{M}(\eta_{1},\eta_{0})$ carries an $\R$-action given by translation:
\begin{equation*}
  (\pi(s),u(s,t))\mapsto (\pi(s+s_{0}),u(s+s_{0},t)).
\end{equation*}
Finally, one defines the differential by the formula:
\begin{equation*}
  d_{\mathrm{FF}}(\gamma_{0}\otimes \eta_{0}):=\sum_{\gamma_{1}\otimes \eta_{1}}\#\set{[\pi,u]\in \mathscr{M}(\eta_{1},\eta_{0})/\R:u\text{ joins }\gamma_{1},\gamma_{0}}\cdot \gamma_{1}\otimes \eta_{1};
\end{equation*}
more precisely, we count the one-dimensional components of $\mathscr{M}(\eta_{1},\eta_{0})$ with the advertised asymptotics.

It is perhaps interesting to note that one can decompose $d_{\mathrm{FF}}$ into a sum:
\begin{equation*}
  d_{\mathrm{FF}}=d_{0}+d_{1}+d_{2}+\dots,
\end{equation*}
where $d_{i}$ maps a generator $\gamma_{0}\otimes \eta_{0}$ into the piece of $\mathrm{CFF}$ generated by terms $\gamma_{1}\otimes \eta_{1}$ where $\mathrm{Index}(\eta_{1})=\mathrm{Index}(\eta_{0})+i$. A simple inspection proves that $d_{0}$ preserves the summand $\mathrm{CF}(H_{\eta_{0},t},J)\otimes \Z/2\Z\eta_{0}$ and acts as $d\otimes \id$ where $d$ is the usual Floer differential.

The first key lemma of family Floer cohomology is:
\begin{lemma}
  The family Floer differential squares to zero: $d_{\mathrm{FF}}^{2}=0$.
\end{lemma}
\begin{proof}
  This is outlined in \cite[Proposition 3.9]{hutchings-agt-2008} in the context of family Morse homology. We also refer the reader to \cite[pp.\,970]{seidel-GAFA-2015} for a Floer cohomology set-up closer to the present context.

  Briefly, one shows that the number of ends of the one-dimensional component of $\mathscr{M}(\eta_{1},\eta_{0})/\R$ equals the matrix entry for $d_{\mathrm{FF}}^{2}$ of type: $$\mathrm{CF}(H_{\eta_{0},t},J)\otimes \Z/2\Z\eta_{0}\to \mathrm{CF}(H_{\eta_{1},t},J)\otimes \Z/2\Z\eta_{1}.$$ Since the number of ends is even, one concludes $d_{\mathrm{FF}}^{2}=0$. We refer the reader who wishes for additional details to the proof of Lemma \ref{lemma:i-chain-map}.
\end{proof}

We denote by $\mathrm{HFF}(P,H_{\eta,t},J)$ the homology of $\mathrm{CFF}(P,H_{\eta,t},J)$ with respect to $d_{\mathrm{FF}}$.

\subsubsection{From Floer cohomology to family Floer cohomology}
\label{sec:from-HF-to-HFF}

In this section we explain how to define a map $i:\mathrm{HF}(H_{t},J)\to \mathrm{HFF}(P,H_{\eta,t},J)$ whenever the slope $c\ge 0$ of $H_{\eta,t}$ equals the slope of $H_{t}$. This map will be referred to as the \emph{comparison map}.

To keep things as simple as necessary, we assume that $H_{t}$ satisfies $H_{t}=cr$ for $r\ge r_{0}$, i.e., we disallow time-dependent slopes in $H_{t}$, and that the same fixed almost complex structure is used for $(H_{t},J)$ and $(P,H_{\eta,t},J)$.

We will construct on $\mathrm{HFF}(P,H_{\eta,t},J)$:
\begin{enumerate}
\item a product $\ast_{\mathrm{FF}}$ in \S\ref{sec:family-pair-of-pants},
\item a BV operator $\Delta_{\mathrm{FF}}$ in \S\ref{sec:family-BV-operator},
\item a map $\mathfrak{e}_{\mathrm{FF}}$ associated to a family of ball embeddings in \S\ref{sec:eval-map-assoc}.
\end{enumerate}
The comparison map $i$ will be shown to respect these structures.

To define $i$, we follow the usual strategy of defining \emph{comparison data}, and then counting the rigid elements in an associated moduli space.

Fix $(H_{t},J)$ and $(H_{\eta,t},J,P)$ which are admissible for defining the Floer complex and family Floer complex. Define \emph{comparison data} to be a connection one-form $\mathfrak{a}$ on the family $N\times \Sigma\times W$ where $\Sigma$ is the cylinder, satisfying:
\begin{enumerate}
\item $\mathfrak{a}=H_{\eta,s,t}\d t$,
\item $H_{\eta,s,t}=cr$ for $r\ge r_{0}$,
\item $H_{\eta,s,t}=H_{t}$ for $s\ge s_{0}$ and $H_{\eta,s,t}=H_{\eta,t}$ for $s\le -s_{0}$.
\end{enumerate}
One easily shows that the space of comparison data is convex and non-empty; to see it is non-empty, one can simply take the linear interpolation (bearing in mind that we assume that $H_{\eta,t},H_{t}$ both equal $cr$ for $r\ge r_{0}$).

As in the definition of $d_{\mathrm{FF}}$, we do not directly consider the solutions of the moduli space associated to this family (indeed, this family has asymptotics which depend on $\eta$, which would require special treatment). Rather, we will pull back this data using flow lines of the pseudogradient $P$.

For each zero $\eta_{1}$, introduce $\mathscr{P}(\eta_{1})$ as the space of flow lines $\pi(s)$ of $-\beta(-s)P$ which converge to $\eta_{1}$ as $s\to-\infty$. Notice that $\pi(s)$ is constant for $s\ge 0$. The space of points appearing as $\pi(0)$ for $\pi\in \mathscr{P}(\eta_{1})$ is simply the unstable manifold of $\eta_{1}$.

Comparison data induces a connection one-form $\mathfrak{a}$ on $\mathscr{P}(\eta_{1})\times \Sigma\times W$ given by the formula:
\begin{equation*}
  \mathfrak{a}_{\pi,s,t}=\mathfrak{a}_{\pi(s),s,t},
\end{equation*}
and it is important to note that this family has constant asymptotics, namely:
\begin{enumerate}
\item $\mathfrak{a}_{\pi,s,t}=H_{\eta_{1},t}$ for $s\le -s_{1}(\pi)$,
\item $\mathfrak{a}_{\pi,s,t}=H_{t}$ for $s\ge s_{0}$.
\end{enumerate}
The threshold $s_{1}(\pi)$ is not constant throughout the family (note that $\mathscr{P}(\eta_{1})$ is an open manifold), but it is locally constant which is sufficient for the framework established in \S\ref{sec:connection-1-forms} and \S\ref{sec:general-form-floers}.

For a generic perturbation term $\mathfrak{p}$ on $\mathscr{P}(\eta_{1})\times \Sigma\times W$, and using the almost complex structure $J$, one consider the moduli space $\mathscr{M}(\eta_{1})$ of finite energy solutions $(\pi,u)$ of \S\ref{sec:floers-equation-general}. To ground the discussion with a concrete formula, here is the equation we are counting when $\mathfrak{p}=0$:
\begin{equation*}
  \bd_{s}u+J(u)(\bd_{t}u-X_{\pi(s),s,t}(u))=0;
\end{equation*}
the perturbation $\mathfrak{p}$ simply changes the right hand side from $0$ to a small vector field.

By counting the rigid elements in $\mathscr{M}(\eta_{1})$ whose right asymptotic is equal to $\gamma\in \mathrm{CF}(H_{t},J)$, we obtain chains $i_{\eta_{1}}(\gamma)\in \mathrm{CF}(H_{\eta_{1},t},J)$. The comparison map $i$ is the sum over all zeros $\eta_{1}$ of $P$:
\begin{equation*}
  i(\gamma):=\sum i_{\eta_{1}}(\gamma)\otimes \eta_{1}.
\end{equation*}
The key lemma concerning this map is:
\begin{lemma}\label{lemma:i-chain-map}
  The map $i:\mathrm{CF}(H_{t},J)\to \mathrm{CFF}(P,H_{\eta,t},J)$ is a chain map with respect to $d$ and $d_{\mathrm{FF}}$. The chain homotopy class is independent of the generic perturbation $\mathfrak{p}$ or the precise comparison data used.
\end{lemma}
\begin{proof}
  This argument has no surprising parts. However, the moduli spaces we consider are less standard, and so we attempt to give a bit more detail than we have in previous arguments.

  First we will prove that $i$ is a chain map. The key is to consider the one-dimensional component $\mathscr{M}_{1}(\eta_{1})$. It is convenient to focus on the one-dimensional components which contain solutions $(\pi,u)$ where $u$ has asymptotics $\gamma_{1},\gamma_{0}$. Let us refer to this $1$-dimensional manifold by $\mathscr{M}_{1}(\eta_{1},\gamma_{1},\gamma_{0})$.

  As usual with parametric moduli spaces, this admits a smooth map to $\mathscr{P}(\eta_{1})$ simply given by $(\pi,u)\mapsto \pi$. Because $\mathscr{P}(\eta_{1})$ is an open manifold, there are two possible failures of compactness for a sequence $(\pi_{n},u_{n})\in \mathscr{M}_{1}(\eta_{1},\gamma_{1},\gamma_{0})$:
  \begin{enumerate}
  \item\label{lemma-i-chain-map-1} $(\pi_{n},u_{n})$ has no convergent subsequence, but $\pi_{n}$ does;
  \item\label{lemma-i-chain-map-2} $\pi_{n}$ has no convergent subsequence.
  \end{enumerate}
  Each non-compact component of $\mathscr{M}_{1}(\eta_{1},\gamma_{1},\gamma_{0})$ has two non-compact ends, and each such end is either of type \ref{lemma-i-chain-map-1} or \ref{lemma-i-chain-map-2}.

  In the case of an end of type \ref{lemma-i-chain-map-1}, the usual Floer compactness-up-to-breaking arguments imply that we can pass to a subsequence so that $\pi_{n}$ converges to $\pi$ and $u_{n}$ breaks into a configuration of a rigid element in $\mathscr{M}_{0}(\eta_{1})$ connected to a non-stationary Floer differential cylinder for $(H_{\eta_{1},t},J)$ at the left end, or for $(H_{t},J)$ at the right end. By consideration of dimensions, the Floer differential cylinders which broke-off live in one-dimensional families, and hence are counted by the Floer differential. The gluing result complementary to this compactness-up-to-breaking result proves that the number of ends of type \ref{lemma-i-chain-map-1} equals the coefficient of $\gamma_{1}\otimes \eta_{1}$ appearing in:
  \begin{equation*}
    d_{0}(i(\gamma_{0}))+i(d(\gamma_{0})),
  \end{equation*}
  where we recall $d_{\mathrm{FF}}=d_{0}+d_{1}+\dots$. Thus, to complete the argument, it suffices to prove that the number of ends of type \ref{lemma-i-chain-map-2} equals the coefficient of $\gamma_{1}\otimes \eta_{1}$ appearing in:
  \begin{equation*}
    d_{1}(i(\gamma_{0}))+d_{2}(i(\gamma_{0}))+\dots.
  \end{equation*}
  Let us therefore focus on an end $(\pi_{n},u_{n})$ of type \ref{lemma-i-chain-map-2}. Because $P$ is assumed to be a Morse-Smale pseudogradient, one can pass to a subsequence so that $\pi_{n}$ breaks into a configuration in the product $(\pi_{1},\pi_{0})\in \mathscr{P}(\eta_{1},\eta')\times \mathscr{P}(\eta')$ where the index of $\eta'$ is strictly less than the index of $\eta_{1}$; see Figure \ref{fig:lemma-i-chain-map-morse-theoretic-breaking}.

  \begin{figure}[h]
    \centering
    \begin{tikzpicture}
      \draw (0,-1)node[draw,circle,fill,inner sep=1pt]{}--(0,0)node[draw,circle,fill,inner sep=1pt]{}node[above]{$\eta'$}circle(0.6)--(-3,0)node[draw,circle,fill,inner sep=1pt]{}node[left]{$\eta_{1}$};
      \draw (0,-1)to[out=115,in=-5](-0.45,-0.2)to[in=-7,out=175](-3,0);
    \end{tikzpicture}
    \caption{Morse theoretic breaking of the flow lines into two pieces. The circle around $\eta_{0}$ signifies the open set $\mathrm{Op}(\eta')$ where $H_{\eta,t}=H_{\eta',t}$.}
    \label{fig:lemma-i-chain-map-morse-theoretic-breaking}
  \end{figure}
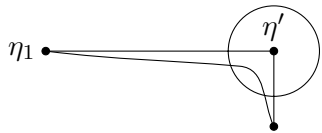

  As $\pi_{n}$ breaks into $(\pi_{1},\pi_{0})$, the equation which $u_{n}$ solves separates into two equations, in the following sense: $u_{n}(s+s_{n},t)$ will solve the equation for $d_{\mathrm{FF}}$ on compact subsets, for any sequence $s_{n}\to\infty$, while the non-translated solution $(\pi_{n},u_{n}(s,t))$ converges to a solution $(\pi_{0},u_{+})\in \mathscr{M}(\eta')$. By picking $s_{n}$ correctly, the translated solution $(\pi_{n}(s+s_{n}),u_{n}(s+s_{n}))$ will converge to a solution $(\pi_{1},u_{-})$. By consideration of dimensions, $(\pi_{0},u_{+})$ is a rigid element of $\mathscr{M}(\eta')$, and $(\pi_{1},u_{-})$ is a rigid-up-to-translation element of the moduli space used to define $d_{k}$, where $k>0$ is the index difference of $\eta_{1}$ and $\eta'$. Following similar gluing theory as in \cite[pp.\,972]{seidel-GAFA-2015}, each such configuration actually appears as a non-compact end of type (2), and thereby one shows:
  \begin{equation*}
    0=i(d(\gamma_{0}))+d_{0}(i(\gamma_{0}))+d_{1}(i(\gamma_{0}))+d_{2}(i(\gamma_{0}))+\dots\text{ modulo }2,
  \end{equation*}
  because each coefficient in the output is the count of the non-compact ends of a one-manifold. This completes the proof that $i$ is a chain map.

  The proof that the chain homotopy class is independent of the perturbation $\mathfrak{p}$ or the comparison data follows similar lines, and we omit the details.
\end{proof}

\emph{Remark}. The appeal to gluing theory, while appearing non-standard, actually follows from a general parametric gluing result for continuation cylinders; this is because the equation which $u$ solves near a breaking can be considered as a continuation cylinder for concatenated continuation data (varying in a parameter space); see Figure \ref{fig:lemma-i-chain-map-solution-near-breaking}. The arguments in \cite{salamon-notes-1997} can be employed in such a case.

\begin{figure}[h]
  \centering
  \begin{tikzpicture}
    \draw (0,0) circle (0.1 and 0.4) (2,0) circle (0.1 and 0.4) (8,0) circle (0.1 and 0.4) (10,0) circle (0.1 and 0.4);
    \draw (0,0.4)--+(10,0) (0,-0.4)--+(10,0);
    \path (-0.1,0)node[left]{$\gamma_{1}$}--node{$\bd_{s}u+J(u)(\bd_{t}u-X_{\eta',t}(u))=0$}(10.1,0)node[right]{$\gamma_{0}$};
  \end{tikzpicture}
  \caption{Solution near the breaking; on a large subcylinder (determined by $\pi(s)\in \mathrm{Op}(\eta')$), the equation appears as Floer's equation.}
  \label{fig:lemma-i-chain-map-solution-near-breaking}
\end{figure}

\subsubsection{The family pair-of-pants product}
\label{sec:family-pair-of-pants}

The product structure on family Floer cohomology is defined as a combination of the Morse cohomology product, defined using flow trees as in \cite{fukaya-AMSIP-1997}, and the pair-of-pants product from \S\ref{sec:pair-pants-product}. For details on a different Floer theoretic product combining flow lines and pairs-of-pants, we refer the reader to \cite[\S4.3]{seidel-GAFA-2015}.

First we introduce a framework for flow trees: having fixed a Morse-Smale pseudogradient $P$ on the parameter space $N$, introduce time-dependent vector fields $P_{0,s}$, $P_{1,s}$, which are defined for $s\in [-1,\infty)$, vanish when $s\in [-1,1]$, and agree with $P$ when $s\in [2,\infty)$. Let $P_{\infty,s}=\beta(-s-1)P$.

A \emph{flow tree} is a configuration $(\pi_{0},\pi_{1},\pi_{\infty})$ where $\pi_{i}$ is a flow line for $-P_{i,s}$, defined on $[-1,\infty)$ when $i=0,1$ and on $(-\infty,1]$ when $i=\infty$, and such that $\pi_{0}(0)=\pi_{1}(0)=\pi_{\infty}(0)=\eta'$.

\emph{Remark}. The fact that $\pi_{i}(s)$ is defined for $s\in [-1,1]$ and is constant on this interval will be a convenience in some of the subsequent formulas.

A flow tree has asymptotic zeros $\eta_{0},\eta_{1},\eta_{\infty}$ of $P$ at its non-compact ends; the space of flow lines with these asymptotics is denoted $\mathscr{T}(\eta_{0},\eta_{1},\eta_{\infty})$.

Notice that the junction point $\eta'$ lies in the intersection of (deformations of) the stable manifolds of $\eta_{1},\eta_{0}$ and the unstable manifold of $\eta_{\infty}$. In particular, assuming these deformed stable and unstable manifolds are transverse, then the possible choices for $\eta'$ form a (potentially open) manifold of dimension: $$\dim \mathscr{T}(\eta_{0},\eta_{1},\eta_{\infty})=\mathrm{Index}(\eta_{\infty})-\mathrm{Index}(\eta_{0})-\mathrm{Index}(\eta_{1}).$$
This dimension is also the dimension of the space of flow trees, since the junction point determines the flow tree.

Next we explain how to set-up a family of connection one-forms on the pair-of-pants parametrized by the space of flow trees. To do this, define \emph{family pair-of-pants data} to be:
\begin{enumerate}
\item a fixed almost complex structure $J$,
\item families of Hamiltonian functions $H_{i,\eta,t}$, $i=0,1,\infty$ such that $(H_{i,\eta,t},J)$ is admissible for defining the family Floer complex,
\item satisfying $H_{0,\eta,t}+H_{1,\eta,t}=H_{\infty,\eta,t}$.
\end{enumerate}
For each flow tree $(\pi_{0},\pi_{1},\pi_{\infty})$, one considers the families of Hamiltonians $H_{i,\pi_{i}(s),t}$ defined on the ``branches'' of the flow tree; at the junction point, one has three Hamiltonians related by $H_{0,\eta',t}+H_{1,\eta',t}=H_{\infty,\eta',t}$.

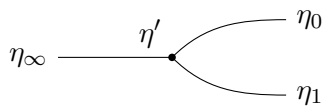
\begin{figure}[h]
  \centering
  \begin{tikzpicture}[scale=.5]
    \draw (0,0)node[above left]{$\eta'$}node[fill,circle,inner sep=1pt]{} to[out=45,in=180] (3,1)node[right]{$\eta_{0}$} (0,0) to[out=-45,in=180] (3,-1)node[right]{$\eta_{1}$} (0,0) -- (-3,0)node[left]{$\eta_{\infty}$};
  \end{tikzpicture}
  \caption{Illustration of a flow tree.}
  \label{fig:illustration-of-a-flow-tree}
\end{figure}

To lift this to the pair-of-pants surface, we follow an ad hoc recipe. Fix the pair-of-pants surface $\Sigma=\C\setminus \set{0,1}$, as in \S\ref{sec:pair-pants-product}, and consider the punctured disks $D(1/3)^{\times}$, $1+D(1/3)^{\times}$, and $\C\setminus D(2)$, as cylindrical ends, parametrized in the standard way so that the line $t=0$ is aligned with the positive real axis. The cylindrical ends around $0,1$ are parametrized by $[0,\infty)\times \R/\Z$ while the end around $\infty$ is parametrized by $(-\infty,0]\times \R/\Z$.

Fix smooth functions $t_{i}:\Sigma\to \R/\Z$ such that:
\begin{equation*}
  t_{i}=t\text{ in the $i$th cylindrical end and $\infty$th cylindrical end,}
\end{equation*}
and so that $t_{0}$ is constant in the $1$ cylindrical end while $t_{1}$ is constant in the $0$ cylindrical end. The differentials $\mathfrak{m}_{i}=\d t_{i}$ give two closed real-valued differential forms on $\Sigma$ such that:
\begin{equation*}
  \mathfrak{m}_{i}=\d t\text{ in the $i$th cylindrical end and $\infty$th cylindrical end};
\end{equation*}
note that $\mathfrak{m}_{0}$ vanishes in the $1$ cylindrical end while $\mathfrak{m}_{1}$ vanishes in the $0$ cylindrical end.

One more piece of data needed to lift the equation to the pair-of-pants is a smooth map $b:\Sigma\to \R$ satisfying $b(z)=s(z)$ when $z$ is in any of the three cylindrical ends. The values of $b(z)$ should be contained in $[-1,1]$ on the complement of the cylindrical ends.

For each flow tree $\pi=(\pi_{0},\pi_{1},\pi_{\infty})$ define:
\begin{equation*}
  \mathfrak{a}_{\pi}=\left\{
    \begin{aligned}
      &H_{\infty,\pi_{\infty}(b(z)),t}\d t\text{ in the $\infty$th cylindrical end},\\
      &H_{0,\pi_{0}(b(z)),t_{0}(z)}\mathfrak{m}_{0}+H_{1,\pi_{1}(b(z)),t_{1}(z)}\mathfrak{m}_{1}\text{ otherwise}.
    \end{aligned}
  \right.
\end{equation*}
Notice that, for $z$ outside of the cylindrical ends, $\pi_{i}(b(z))=\eta'$. This explains our choice of having $\pi_{i}$ defined on $[-1,1]$ for all branches.

The key outcome of this construction is:
\begin{lemma}\label{lemma:pop-family-curvature-bounded-above}
  The family of connection one-forms on $\mathscr{T}(\eta_{0},\eta_{1},\eta_{\infty})\times \Sigma\times W$ induced by $\mathfrak{a}$ is smooth and has curvature bounded from above. Moreover:
  \begin{equation*}
    \mathfrak{a}=H_{i,\pi_{i}(s),t}\d t
  \end{equation*}
  holds in the $i$th cylindrical end.
\end{lemma}
\begin{proof}
  To see that $\mathfrak{a}$ has curvature bounded from above, one only needs to observe that, outside of a compact set one has $H_{i,\eta,t}=c_{i}r$, and therefore the curvature of $\mathfrak{a}$ vanishes outside of a compact set because the forms $\mathfrak{m}_{i}$ are closed.

  The statement about the form of $\mathfrak{a}$ in the cylindrical ends follows immediately from the construction of $b$, $t_{i}$, and $\mathfrak{m}_{i}$.
\end{proof}

The data of $\mathfrak{a}$ on the family $\mathscr{T}(\eta_{0},\eta_{1},\eta_{\infty})\times \Sigma\times W$ together with the fixed almost complex structure $J$ and a generic perturbation term $\mathfrak{p}$ leads to a moduli space $\mathscr{M}(\eta_{0},\eta_{1},\eta_{\infty})$ of finite energy solutions to \S\ref{sec:floers-equation-general}. Counting the rigid elements asymptotic to orbits $\gamma_{0},\gamma_{1},\gamma_{\infty}$ gives a number $N_{\eta_{0},\eta_{1},\eta_{\infty}}(\gamma_{0},\gamma_{1},\gamma_{\infty})$ in $\Z/2\Z$. Define:
\begin{equation*}
  (\gamma_{0}\otimes \eta_{0})\ast_{\mathrm{FF}} (\gamma_{1}\otimes \eta_{1}):=\sum N_{\eta_{0},\eta_{1},\eta_{\infty}}(\gamma_{0},\gamma_{1},\gamma_{\infty})(\gamma_{\infty}\otimes \eta_{\infty}),
\end{equation*}
where the sum is over all $\eta_{\infty}$ and $\gamma_{\infty}$. As expected:
\begin{lemma}
  The operation $\ast_{\mathrm{FF}}$ is a chain map:
  \begin{equation*}
    \ast_{\mathrm{FF}}:\mathrm{CFF}(P,H_{0,\eta,t},J)\otimes \mathrm{CFF}(P,H_{1,\eta,t},J)\to \mathrm{CFF}(P,H_{\infty,\eta,t},J);
  \end{equation*}
  the chain homotopy class of the map is independent of the choice of $\mathfrak{m}_{i},t_{i}$, the perturbed vector fields $P_{i,s}$, and the perturbation one-form $\mathfrak{p}$.
\end{lemma}
\begin{proof}
  This follows standard lines; to see that it is a chain map, one inspects the non-compact ends of the one-dimensional components of $\mathscr{M}(\eta_{0},\eta_{1},\eta_{\infty})$ in a manner similar to the proof of Lemma \ref{lemma:i-chain-map}. Let us note that a key part of the argument is understanding the failures of compactness in the moduli space of flow trees $\mathscr{T}(\eta_{0},\eta_{1},\eta_{\infty})$; the same considerations used to prove the flow tree product on Morse cohomology is a chain map will be used here.

  To see that the chain homotopy class is independent of the auxiliary choices, one needs to find a path between two choices, then set-up a parametric moduli space and apply the usual Floer theory arguments (see, e.g., \cite[pp.\,314 and pp.\,341]{abouzaid-EMS-2015}). We only comment on why one can find a path between two such choices: clearly one can find paths between the choices of $\mathfrak{P}_{i,s}$, $b$, and $\mathfrak{p}$ (simply by a linear interpolation). To find a path between the choices of $t_{0},t_{1}$ and, say, $t_{0}',t_{1}'$, one can consider the circle-valued functions $t_{i}-t_{i}'$; by construction, these functions vanish in at least two-out-of-three cylindrical ends. Any such circle-valued function necessarily induces the zero map $\pi_{1}(\Sigma)\to \pi_{1}(\R/\Z)$, and thus lifts to $\R$. The space of $\R$-valued functions is convex, and hence the desired path can be taken to be a linear interpolation between the lifts.
\end{proof}

As mentioned in \S\ref{sec:from-HF-to-HFF}, the map $i:\mathrm{HF}(H_{t},J)\to \mathrm{HF}(P,H_{\eta,t},J)$ is compatible with the product structures, in the following sense:
\begin{lemma}
  Given Hamiltonians $H_{i,t}$ and $H_{i,\eta,t}$, $i=0,1,\infty$, such that:
  \begin{enumerate}
  \item the slope of $H_{i,t}$ equals the slope of $H_{i,\eta,t}$
  \item $(H_{i,t},J)$ and $(P,H_{i,\eta,t},J)$ are admissible for defining $\mathrm{CF}$ and $\mathrm{CFF}$
  \item $H_{\infty,t}=H_{0,t}+H_{1,t}$, and $H_{\infty,\eta,t}=H_{0,\eta,t}+H_{1,\eta,t}$,
  \end{enumerate}
  then we have an equality:
  \begin{equation*}
    \ast_{\mathrm{FF}}\circ (i\otimes i)=i\circ \ast,
  \end{equation*}
  of maps $\mathrm{HF}(H_{0,t},J)\otimes \mathrm{HF}(H_{1,t},J)\to \mathrm{HFF}(P,H_{\infty,\eta,t},J)$.
\end{lemma}
\begin{proof}
  The argument has no surprises; one simply follows their nose. The strategy is as follows: define a parametric moduli space $\mathscr{M}_{\mathrm{param}}$ admitting a map $\tau$ to $\R$. The one-dimensional components of $\mathscr{M}_{\mathrm{param}}$ have non-compact ends of three types:
  \begin{enumerate}
  \item ends containing sequences with $\tau\to\infty$; these ends will be asymptotic to the configurations composing $\ast_{\mathrm{FF}}\circ (i\otimes i)$;
  \item ends containing sequences with $\tau\to-\infty$; these ends will be asymptotic to the configurations composing $i\circ \ast$;
  \item\label{family-product-compatible-i-3} ends which project under $\tau$ to precompact sets in $\R$; these ends will be asymptotic to chain homotopy terms.
  \end{enumerate}
  The total count of ends is even, and one concludes an equation of the form:
  \begin{equation}\label{eq:ppff-and-i-chain-level}
    \ast_{\mathrm{FF}}\circ (i\otimes i)+i\circ \ast+d_{\mathrm{FF}}K+K(d\otimes \id+\id\otimes d)=0\text{ mod }2,
  \end{equation}
  as maps on the chain complexes; each of the three summands corresponds to one type of ends. The desired result follows.

  We now describe the construction of the parametric moduli space $\mathscr{M}_{\mathrm{param}}$.

  Let us define the space $\mathscr{F}$ of pairs $(\tau,\pi)$ where $\pi=(\pi_{0},\pi_{1},\pi_{\infty})$ is a flow tree of the following type:
  \begin{enumerate}
  \item $\pi_{i}:[-1,\infty)\to N$ is a flow line for $-\beta(\tau-s)P_{i,s}$, $i=0,1$,
  \item $\pi_{\infty}:(-\infty,1]\to N$  is a flow line for $-\beta(\tau-s)P_{\infty,s}$,
  \item $\pi_{0}(0)=\pi_{1}(0)=\pi_{\infty}(0)=\eta'$;
  \end{enumerate}
  here $\beta$ is the standard cut-off function. Notice that $\pi_{i}(s)$ is constant for $s\ge \tau$; consequently, each element $(\tau,\pi)\in \mathscr{F}$ is completely determined by the parameter value $\tau$ and the junction point $\eta'$ which must be a point in the unstable manifold of $\eta_{\infty}=\lim_{s\to-\infty}\pi_{\infty}(s)$. In particular, if we let $\mathscr{F}(\eta_{\infty})$ be the subset of flow trees with fixed asymptotic $\eta_{\infty}$, then $\mathscr{F}(\eta_{\infty})$ is an open manifold diffeomorphic to $\R\times (\text{unstable manifold of }\eta_{\infty})$.

  Following a similar construction used in the definition of $\ast_{\mathrm{FF}}$, we obtain a connection one form $\mathfrak{a}$ on the family $\mathscr{F}(\eta_{\infty})\times \Sigma\times W$, as follows:
  \begin{equation*}
    \mathfrak{a}_{\tau,\pi}:=\left\{
      \begin{aligned}
        &H_{\infty,b(z)-\tau,\pi_{\infty}(b(z)),t}\d t\text{ in the $\infty$th cylindrical end},\\
        &H_{0,b(z)-\tau,\pi_{0}(b(z)),t_{0}(z)}\mathfrak{m}_{0}+H_{1,b(z)-\tau,\pi_{1}(b(z)),t_{1}(z)}\mathfrak{m}_{1}\text{ otherwise},
      \end{aligned}
    \right.
  \end{equation*}
  where:
  \begin{equation*}
    H_{i,s,\eta,t}=(1-\beta(s))H_{i,\eta,t}+\beta(s)H_{i,t}.
  \end{equation*}
  As in the proof of Lemma \ref{lemma:pop-family-curvature-bounded-above}, this $\mathfrak{a}$ has curvature bounded from above.

  Morally, this $\mathfrak{a}_{\tau,\pi}$ is a sort of hybrid between a continuation from $H_{i,t}$ to $H_{i,s,\eta,t}$ and the family pair-of-pants product. The ``continuation part'' is in the region where $b(z)\approx \tau$, which occurs in the positive ends $i=0,1$ when $\tau$ is large and positive, and is in the negative end $i=\infty$ when $\tau$ is large and negative. As $\tau\to\infty$, this equation ``breaks'' into a concatenation of the equations defining the $i$-map at the $0$ and $1$ punctures and the family pair-of-pants product. When $\tau\to-\infty$, the equation breaks into a concatenation of the equation defined by:
  \begin{equation*}
    \mathfrak{a}_{-\infty}=\left\{
      \begin{aligned}
        &H_{\infty,t}\d t\text{ in the $\infty$th cylindrical end},\\
        &H_{0,t_{0}(z)}\mathfrak{m}_{0}+H_{1,t_{1}(z)}\mathfrak{m}_{1}\text{ otherwise},
      \end{aligned}\right.
  \end{equation*}
  (which is just the non-family pair-of-pants product and does not depend on any flow tree) and the equation defining the $i$-map.

  The desired moduli space $\mathscr{M}_{\mathrm{param}}$ is a union of components $\mathscr{M}_{\mathrm{param}}(\eta_{\infty})$; the ends of this component compose the terms in \eqref{eq:ppff-and-i-chain-level} with output in the summand $\mathrm{CF}(H_{\infty,\eta_{\infty},t},J)\otimes \Z/2\Z \eta_{\infty}$.

  This component $\mathscr{M}_{\mathrm{param}}(\eta_{\infty})$ is defined using the above $\mathfrak{a}$, the fixed almost complex structure $J$, and a peturbation term $\mathfrak{p}$ on $\mathscr{F}(\eta_{\infty})\times \Sigma\times W$. One constructs $\mathfrak{p}$ ``recursively,'' in the following sense: if $(\tau,\pi)$ is close to breaking (either when $\tau\to\pm\infty$, or $\pi$ approaches the boundary of the unstable manifold of $\eta_{\infty}$), $\mathfrak{p}$ should be determined by perturbation terms chosen for the equations which appear in the breaking. Such recursive choices of perturbations are a standard ingredient in Floer theory (see, e.g., \cite[pp.\,109]{seidel-book-2008}).

  The analysis of the non-compact ends of $\mathscr{M}_{\mathrm{param}}(\eta_{\infty})$ and the derivation of the desired chain-level equation \eqref{eq:ppff-and-i-chain-level} follows similar lines to the proof of Lemma \ref{lemma:i-chain-map}, and we omit further details.
\end{proof}

\subsubsection{Special Hamiltonians associated to a family of ball embeddings}
\label{sec:spec-hamilt-assoc}

In this section, we will describe a particular choice of data $H_{\eta,t}'$, $\eta\in N$, associated to a lift of $f$ to $\mathfrak{B}(a,\Omega)/U(n)$.

Recall from \eqref{eq:special-hamiltonian} in \S\ref{sec:system-of-hamiltonians} the Hamiltonian function $H_{c,\delta,\epsilon,\eta}$ which has a minimum located at the center of the ball $B_{\eta}$. The parameters $c,\delta,\epsilon$ are explained in \S\ref{sec:evaluation-maps-ball}. We assume that $\epsilon a>c(1+\delta/2)$ as this ensures the evaluation map is defined; see \S\ref{sec:defin-evaluation-map}.

Given a pseudogradient $P$ on $N$, the family $\eta\mapsto H_{c,\delta,\epsilon,\eta}$ is not valid data for the family Floer complex, since it is probably not constant in neighborhoods of the zeros of $P$. We correct for this by modifying $F$ as follows: simply precompose $F$ using a smooth map $N\to N$ which is close to the identity in the $C^{0}$ distance (associated to a Riemannian metric $g$) and is constant on neighborhoods of the zeros of $P$. Provided the $C^{0}$ distance is small enough, the modified $F$ is homotopic to the original $F$.

Fix an almost complex structure $J$. We define our family as:
\begin{equation*}
  H'_{\eta,t}=H_{c,\delta,\epsilon,\eta}+\kappa_{\eta,t},
\end{equation*}
where $\kappa_{\eta,t}$ vanishes in the ball $B_{\eta}$, is supported in $\Omega(1+\delta)$, and is locally $\eta$-independent whenever $\eta$ is in a neighborhood of the zeros of $P$. We require that $(P,H'_{\eta,t},J)$ is admissible for defining the family Floer complex (this can be achieved if $\kappa_{t}$ is chosen generically).

It will be important when considering action filtrations on $\mathrm{CFF}(P,H_{\eta,t}',J)$ to make the following quantity very small:
\begin{equation}\label{eq:important-quantity}
  \max_{\eta\in N} \norm{\partial_{\eta}{H_{\eta,t}'}}_{g}\times (\text{max $g$-length of flow lines of $P$}).
\end{equation}
This quantity can be made arbitrarily small by picking the pseudogradient $P$ to have only very short flow lines;\footnote{The construction of a pseudogradient with only very short flow lines is an exercise left for the reader. For a similar result see \cite[Lemma 2.6]{eliashberg-pancholi}.} we note the size of $\partial_{\eta}H_{\eta,t}'$ is uniformly controlled during the construction (for each $P$ one can chose a modification of $F$ so that the $\eta$ derivative of $H'_{c,\delta,\epsilon,\eta}$ is bounded independently of $P$).

\subsubsection{The family BV-operator}
\label{sec:family-BV-operator}

The goal in this subsection is to construct the operator $\Delta_{\mathrm{FF}}$. At the end, we will analyze how $\Delta_{\mathrm{FF}}$ acts on the special family $H_{\eta,t}'$ introduced in \S\ref{sec:spec-hamilt-assoc}.

Let $(P,H_{\eta,t},J)$ be admissible for defining the family Floer complex. The definition of $\Delta_{\mathrm{FF}}$ is straightforward; we define a connection one-form $\mathfrak{a}$ on the family $\mathscr{P}(\eta_{1},\eta_{0})\times \R/\Z\times \Sigma\times W$, where $\Sigma$ is the cylinder, by the formula: $$\mathfrak{a}_{\pi,\theta,s,t}=[(1-\beta(s))H_{\pi(s),t+\theta}+\beta(s)H_{\pi(s),t}]\d t.$$

Using the almost complex structure $J$, and a generic perturbation one-form $\mathfrak{p}$, we have an associated moduli space $\mathscr{M}(\eta_{1},\eta_{0})$ for each pair $\eta_{1},\eta_{0}$. Counting the rigid elements produces a map $\Delta_{\mathrm{FF},\eta_{0},\eta_{1}}:\mathrm{CF}(H_{\eta_{0},t},J)\to \mathrm{CF}(H_{\eta_{1},t},J)$, and we define:
\begin{equation*}
  \Delta_{\mathrm{FF}}(\gamma_{0}\otimes \eta_{0})=\sum \Delta_{\mathrm{FF},\eta_{0},\eta_{1}}(\gamma_{0})\otimes \eta_{1},
\end{equation*}
where the sum is over all zeros $\eta_{1}$. Similar arguments to those in \S\ref{sec:from-HF-to-HFF} and \S\ref{sec:family-pair-of-pants} prove that $\Delta_{\mathrm{FF}}$ is a chain map $\mathrm{CFF}(P,H_{\eta,t},J)\to \mathrm{CFF}(P,H_{\eta,t},J)$, and the chain homotopy class is independent of the perturbation term $\mathfrak{p}$.

As with $\ast_{\mathrm{FF}}$, one has $\Delta_{\mathrm{FF}}\circ i=i\circ \Delta$ as maps $\mathrm{HF}(H_{t},J)\to \mathrm{HFF}(P,H_{\eta,t},J)$, provided that $H_{t},H_{\eta,t}$ have the same slope $c$.

In the rest of this subsection, we will analyze how the map $\Delta_{\mathrm{FF}}$ acts on the specific family $H_{\eta,t}'$ constructed in \S\ref{sec:spec-hamilt-assoc}. We will show:

\begin{proposition}\label{proposition:BV-family}
  If the length of flow lines of $P$ are sufficiently short, and the perturbation one-form $\mathfrak{p}$ and the perturbation $\kappa_{\eta,t}$ used in the definition of $H_{\eta,t}'$ are sufficiently small, then the following holds: any chain:
  \begin{equation*}
    \sum_{i=1}^{k} \gamma_{i}\otimes \eta_{i}
  \end{equation*}
  in the output of $\Delta_{\mathrm{FF}}$ is such that $\gamma_{i}$, $i=1,\dots,k$, is not the center of $B_{\eta_{i}}$.
\end{proposition}
\begin{proof}
  First, we show that, if $\gamma_{0}$ is the center of the ball $B_{\eta_{0}}$, then $\Delta_{\mathrm{FF}}(\gamma_{0}\otimes \eta_{0})$ does not contain any term $\gamma_{i}\otimes \eta_{i}$ where $\eta_{i}$ is the center of the ball $B_{\eta_{i}}$. The key idea is to exploit the dimension of the moduli space $\mathscr{M}(\eta_{1},\eta_{0})$.

  Suppose there exists a solution $(\pi,u)\in \mathscr{M}(\eta_{1},\eta_{0})$ which joins the center $x_{1}$ of the ball $B_{\eta_{1}}$ at the left end to the center $x_{0}$ of the ball $B_{\eta_{0}}$ at the right end. Pick a generic section $\mathfrak{s}$ of $\det_{\C}TW$ which is non-vanishing at $x_{0}$ and $x_{1}$, and moreover is homotopic through non-vanishing sections to the standard trivialization of $\det_{\C}TW_{x_{0}}$ and $\det_{\C}TW_{x_{1}}$; this is trivial if $x_{0},x_{1}$ are different points, and follows by an easy construction when $x_{0}=x_{1}$, provided we assume that $\pi$ is short enough that the balls $B_{\pi(s)}$ all contain the point $x_{0}=x_{1}$.

  As is well-known (see, e.g., \cite{cant-thesis-2022}), the zero set $\mathfrak{s}^{-1}(0)$ is Poincaré dual to the first Chern class of $W$, and the dimension of $\mathscr{M}(\eta_{1},\eta_{0})$ near $(\pi,u)$ is:
  \begin{equation*}
    1+\dim \mathscr{P}(\eta_{1},\eta_{0})+2[u]\cdot \mathfrak{s}^{-1}(0);
  \end{equation*}
  the Conley-Zehnder indices do not appear because the orbits at the center of the ball have the same indices when they are computed using the homotopy class of trivializations induced by $\mathfrak{s}$. Thus, if we can prove that $u$ is null-homologous (bearing in mind that $u$ is a topologically a sphere), then $(\pi,u)$ cannot lie in a rigid component of $\mathscr{M}(\eta_{1},\eta_{0})$; this gives the desired result. The rest of the proof is dedicated to showing that $u$ must be null-homologous provided the perturbations are small enough, and the flow lines of the pseudogradient are short enough.

  To prove that the sphere $u$ is null-homologous we will argue that the diameter of each loop $t\mapsto u(s,t)$ is smaller than the injectivity radius, and hence $u$ bounds a three-dimensional ball. To show this, we will analyze the equation, and estimate the energy of $u$ in terms of $\mathfrak{p},\kappa_{t}$ and the length of flow lines.

  Unpacking the definitions, one sees that $u$ solves:
  \begin{equation*}
    \bd_{s}u+J(u)(\bd_{t}u-X_{\pi(s)}(u))=V_{\pi,s,t}(u),
  \end{equation*}
  where $V_{\pi,s,t}$ is due to the perturbation one-form $\mathfrak{p}$ and perturbation term $\kappa$, and $X_{\eta}$ is the Hamiltonian vector field for $H_{c,\delta,\epsilon,\eta}$.

  By construction, $V_{\pi,s,t}$ is compactly supported in the cylinder, and can be taken to be as small as desired. The energy integral of $u$ is equal to:
  \begin{equation*}
    \int \pd{H_{c,\delta,\epsilon,\pi(s)}}{s}(u(s,t))d sdt+\mathrm{Error},
  \end{equation*}
  where the error term depends only on $\kappa_{t},\mathfrak{p}$, and can be made as small as desired. We then estimate:
  \begin{equation*}
    \abs{\partial_{s}{H_{c,\delta,\epsilon,\pi(s)}}}\le \max_{\eta}\abs{\partial_{\eta}H_{c,\delta,\epsilon,\eta}}_{g}\abs{\pi'(s)}_{g}.
  \end{equation*}
  Integrating this over the cylinder, one concludes the energy is bounded by the error plus the length of $\pi$ times a uniform constant; see the discussion in \S\ref{sec:spec-hamilt-assoc}. Since we assume the length of $\pi$ is short, we can assume the energy of $u$ is as small as desired.

  The proof is finished by appealing to a compactness argument. Suppose we have a sequence of solutions $u_{n}$ of the above form, with perturbation terms $\mathfrak{p}_{n},\kappa_{n,\eta,t}$, lengths of flow lines of $P_{n}$, and, consequently, energies all tending to zero. Because $\omega$ is tamed by $J$, we obtain:
  \begin{equation*}
    E_{n}=\int \abs{\bd_{t}u_{n}-X_{\pi_{n}(s),t}(u_{n})}^{2}dsdt\to 0\text{ as }n\to\infty,
  \end{equation*}
  as is well-known in estimates of the Floer theory energy integrals; see, e.g., \cite[pp.\,12]{salamon-notes-1997}.

  By standard bubbling analysis, we can assume $\abs{\bd_{s}u_{n}}$ and $\abs{\bd_{t}u_{n}}$ are uniformly bounded, say by $C>0$.

  Consider the loops $\gamma_{n,s}(t)=u_{n}(s,t)$. For any $s$, there must be a nearby point $s'$ such that: $$\abs{s-s'}\le E_{n}^{1/2}\text{ and }\int_{\R/\Z} \abs{\bd_{t}\gamma_{n,s'}(t)-X_{\pi_{n}(s'),t}(\gamma_{n,s'}(t))}^{2}dt\le E_{n}^{1/2}.$$ By the gradient bound, we conclude that $\gamma_{n,s}(t)$ lies in the $CE_{n}^{1/2}$ neighborhood of $\gamma_{n,s'}(t)$. Thus it suffices to prove that, if $\gamma_{n}$ is a sequence of loops sampled from $u_{n}(s,t)$ such that:
  \begin{equation*}
    \int_{\R/\Z}\abs{\bd_{t}\gamma_{n}(t)-X_{\eta_{n},t}(\gamma_{n}(t))}^{2}dt \le E_{n}^{1/2},
  \end{equation*}
  for some $\eta_{n}\in N$, then $\gamma_{n}(t)$ has a subsequence which converges uniformly to a point; this will imply all loops sampled from $u$ have a small enough diameter (for $n$ sufficiently large).

  By standard bootstrapping for ODEs, similar to the argument in \cite[\S2.2.2]{brocic-cant-JFPTA-2024}, it follows that a subsequence $\gamma_{n}(t)$ converges in the $C^{1}$ topology to an orbit of $X_{\eta,t}$ for some $\eta\in N$. The orbits of $X_{\eta,t}$ are either constants, or have action uniformly far from the action of the center of the ball (see Lemma \ref{lemma:construction-of-D}). Assume that $\gamma_{n}$ does not converge to a point; it then follows that $\gamma_{n}(t)$ has action far from the action of the center of the ball as $n\to\infty$. Since $\gamma_{n}(t)$ was sampled from $u_{n}(s,t)$, and $u_{n}$ joins two centers of balls, we conclude that $u_{n}$ must have a minimum positive amount of energy (using the well-known principle that the energy integral governs the change in action; the presence of the perturbation terms in the energy integral will not ruin the application of this principle). This minimum amount of energy of $u_{n}$ contradicts our assumption, and the proof is complete in this case.

  The second thing to show is that $\Delta_{\mathrm{FF}}(\gamma_{0}\otimes \eta_{0})$ does not contain any term $\gamma_{i}\otimes \eta_{i}$ where $\gamma_{i}$ is the center of $B_{\eta_{i}}$ provided $\gamma_{0}$ is \emph{not} the center of $B_{\eta_{0}}$. This case is much easier (for sufficiently small perturbations and lengths of flow lines), since the action of the center has the lowest action among all orbits, and $\Delta_{\mathrm{FF}}$ increases actions (up to an error which becomes as small as desired as the perturbations and lengths of flow lines tend to zero).
\end{proof}

\subsubsection{Evaluation map associated to a family of ball embeddings}
\label{sec:eval-map-assoc}

In this section we construct a map $\mathfrak{e}_{\mathrm{FF}}:\mathrm{HFF}(P,H_{\eta,t},J)\to \Z/2\Z$ using a lift of $f$ to $\mathfrak{B}(a,\Omega)/U(n)$ in a similar way to \S\ref{sec:evaluation-maps-ball}.

We assume, as in \S\ref{sec:evaluation-maps-ball}, that $\epsilon a>c(1+\delta/2)$, and that the family $H_{\eta,t}$ has slope at most $c$.

Consider the family $\mathscr{P}'(\eta_{0})$ of flow lines $\pi$ of $-\beta(s)P$ which are asymptotic at the positive end to the zero $\eta_{0}$. Each $\pi\in \mathscr{P}'(\eta_{0})$ is determined by its terminal point $\eta'=\pi(0)$, and the set of such terminal points is the stable manifold of $\eta_{0}$.

Define a connection one-form $\mathfrak{a}$ on $\mathscr{P}'(\eta_{0})\times \Sigma\times W$ where $\Sigma$ is the cylinder by the equation:
\begin{equation*}
  \mathfrak{a}_{\pi,s,t}=(1-\beta(s))H_{c,\delta,\epsilon,\pi(s)}\d t+\beta(s)H_{\pi(s),t}\d t;
\end{equation*}
in words, $\mathfrak{a}$ is a continuation data from $H_{\eta_{0},t}$ to $H_{c,\delta,\epsilon,\eta'}$ where $\eta'=\pi(0)$. Note that, as in \S\ref{sec:defin-evaluation-map}, this connection one-form has a varying asymptotic at the left end, and so some care is needed when considering the associated moduli space.

Fixing a generic perturbation term $\mathfrak{p}$ and an almost complex structure $J$, we consider the moduli space $\mathscr{M}(\eta_{0})$ of solutions $(\pi,u)$ whose left asymptotic is the central orbit of $H_{c,\delta,\epsilon,\pi(0)}$. This moduli space has similar compactness and regularity properties as if the asymptotics of $\mathfrak{a}$ were fixed, essentially because the left asymptotic orbit is always non-degenerate (see the discussion in \S\ref{sec:defin-evaluation-map}).

Define:
\begin{equation*}
  \mathfrak{e}_{\mathrm{FF}}(\gamma_{0}\otimes \eta_{0}):=\sum (\text{rigid elements in $\mathscr{M}(\eta_{0})$ whose right asymptotic is $\gamma_{0}$}).
\end{equation*}
The main structural result is:
\begin{lemma}
  The map $\mathfrak{e}_{\mathrm{FF}}:\mathrm{CFF}(P,H_{\eta,t},J)\to \Z/2\Z$ is a chain map, and the chain homotopy class is independent of the generic perturbation term $\mathfrak{p}$.
\end{lemma}
\begin{proof}
  The argument is the same as in \S\ref{sec:defin-evaluation-map}, with the modifications needed to work with family Floer cohomology used in, e.g., \S\ref{sec:from-HF-to-HFF}.
\end{proof}

We now show that $\mathfrak{e}_{\mathrm{FF}}$ is compatible with the $i$-map.
\begin{lemma}
  Fix $H_{t}$ so that $(H_{t},J)$ is admissible for defining the Floer complex, the same slope as $H_{\eta,t}$, so that the $i$-map $\mathrm{CF}(H_{t},J)\to \mathrm{CFF}(P,H_{\eta,t},J)$ is defined. Then:
  \begin{equation*}
    \mathfrak{e}_{\mathrm{FF}}\circ i=\mathfrak{e}
  \end{equation*}
  as maps $\mathrm{HF}(H_{t},J)\to\Z/2\Z$.
\end{lemma}
\begin{proof}
  The key is to consider the family $\mathscr{P}''$ of pairs $(R,\pi)$ where $\pi$ is a flow line of the vector field: $-\beta(s)\beta(R-s)P.$ Note that each such $\pi$ is locally constant outside of the interval $[0,R]$, and moreover $\pi$ is determined by the point $\pi(0)$ which is an arbitrary point of $N$. Thus $\mathscr{P}''$ is diffeomorphic to the product $\R\times N$, via the map $(R,\pi)\mapsto (R,\pi(0))$.

  Define a connection one-form $\mathfrak{a}$ on the family $\mathscr{P}''\times \Sigma\times W$ by:
  \begin{equation*}
    a_{\pi,s,t}=[(1-\beta(s))H_{c,\delta,\epsilon,\pi(s)}+\beta(s)(\beta(R-s)H_{\pi(s),t}+(1-\beta(R-s))H_{t})]\d t;
  \end{equation*}
  It is important to notice that, for $s\le R-1$, $\mathfrak{a}_{\pi,s,t}$ agrees with the connection $1$-form used to define $\mathfrak{e}_{\mathrm{FF}}$, and on $s\ge 1$, $\mathfrak{a}_{\pi,s,t}$ agrees with the $R$-translated version of the connection $1$-form used to define $i$.

  Fixing a perturbation term $\mathfrak{p}$, we can therefore consider the one-dimensional component of the moduli space $\mathscr{M}$ consisting of triples $(R,\pi,u)$.

  \emph{Claim}: for each generic number $R_{0}$, the fiber $\mathscr{M}(R_{0})$ of triples $(R_{0},\pi,u)$ is a zero-dimensional manifold, and counting the points in $\mathscr{M}(R_{0})$ defines a chain map $\mathrm{CF}(H_{t},J)\to \Z/2\Z$. The chain homotopy class of this map is independent of $R_{0}$.

  The claim is standard Floer theory, following similar arguments used in \S\ref{sec:defin-evaluation-map}, and we omit further discussion.

  The two crucial observations are that:
  \begin{enumerate}
  \item $\mathscr{M}(R_{0})$ is precisely the moduli space used to define $\mathfrak{e}$, provided $R_{0}\le 0$,
  \item as $R_{0}\to +\infty$, $\mathscr{M}(R_{0})$ ``breaks'' into configurations of rigid solutions $((\pi_{-},u_{-}),(\pi_{+},u_{+}))$, where $\pi_{-}\in \mathscr{P}'(\eta_{0})$ and $\pi_{+}\in \mathscr{P}(\eta_{0})$, and where $(\pi_{-},u_{-})$ contributes to $\mathfrak{e}_{\mathrm{FF}}$ and $(\pi_{+},u_{+})$ contributes to $i$.
  \end{enumerate}
  In this manner, we conclude that $\mathfrak{e}=\mathfrak{e}_{\mathrm{FF}}\circ i$, up to chain homotopy. One point which merits comment is that one should pick $\mathfrak{p}$ to be compatible with the breaking as $R_{0}\to \infty$. This completes the proof.
\end{proof}

We end this subsection with an analysis of how $\mathfrak{e}_{\mathrm{FF}}$ acts on $\mathrm{CFF}(P,H_{\eta,t}',J)$ for the special family $H_{\eta,t}'$ constructed in \S\ref{sec:spec-hamilt-assoc}.

\begin{proposition}\label{proposition:eFF-vanish-on-non-central}
  If the length of flow lines of $P$ are sufficiently short, and the perturbation one-form $\mathfrak{p}$ and the perturbation $\kappa_{\eta,t}$ used in the definition of $H_{\eta,t}'$ are sufficiently small, then the following holds:
  \begin{equation*}
    \mathfrak{e}_{\mathrm{FF}}(\sum_{i=1}^{k} \gamma_{i}\otimes \eta_{i})=0
  \end{equation*}
  provided each $\gamma_{i}$ is not the center of $B_{\eta_{i}}$. In the definition of $\mathfrak{e}_{\mathrm{FF}}$ we use the same family of ball embeddings as is used in the family $H_{\eta,t}'$.
\end{proposition}
\begin{proof}
  The argument is similar (and easier) than Proposition \ref{proposition:BV-family}. One estimates that the energy of any solution $u$ appearing in the moduli space used to define $\mathfrak{e}_{\mathrm{FF}}$ is as small as desired; for this, it is important that the same family $H_{c,\delta,\epsilon,\eta}$ is used in the construction of $H_{\eta,t}'$ and in $\mathfrak{e}_{\mathrm{FF}}$. Because the energy governs the action difference, any solution $u$ joins asymptotic orbits $\gamma_{-},\gamma_{+}$ where the action of $\gamma_{-}$ is larger than the action of $\gamma_{+}$ up to an arbitrarily small error.

  This error can be taken to be smaller than the distance between the action at the center of the ball and the action of any other orbit (see Lemma \ref{lemma:construction-of-D}). In particular, since the center of the ball has the lowest action, we conclude that the only possible solutions contributing to $\mathfrak{e}_{\mathrm{FF}}$ must be asymptotic at their positive end to $\gamma_{0}\otimes \eta_{0}$ where $\gamma_{0}$ is the center of the ball $B_{\eta_{0}}$, as desired.
\end{proof}

\subsubsection{The family action filtration}
\label{sec:family-acti-filtr}

In Propositions \ref{proposition:BV-family} and \ref{proposition:eFF-vanish-on-non-central} we appealed to the actions of orbits in the context of family Floer cohomology, in the context when the flow lines of the pseudogradient $P$ are short. In this subsection, we will formalize such considerations by introducing special action filtrations on the family Floer complex depending on a parameter $\varepsilon>0$.

For $\varepsilon>0$, define the $\varepsilon$-\emph{action} of $\gamma_{0}\otimes \eta_{0}\in \mathrm{CFF}(P,H_{\eta,t},J)$ to be the number:
\begin{equation*}
  A_{\varepsilon}(P,H_{\eta,t};\gamma_{0}\otimes \eta_{0}):=\varepsilon \mathrm{Index}(P;\eta_{0})+\int H_{\eta_{0},t}(\gamma_{0}(t))dt -\int \gamma_{0}^{*}\lambda,
\end{equation*}
where, recall, $\lambda$ is the Liouville form on $W$. One defines the $\varepsilon$-action of a chain in $\mathrm{CFF}(P,H_{\eta,t},J)$ to be the minimum action of a generator which appears in the chain.

\begin{lemma}\label{lemma:dff-filtration}
  The differential $d_{\mathrm{FF}}$ increases the $\varepsilon$-action provided:
  \begin{equation*}
    \max_{\eta,t,u}\abs{\bd_{\eta}H_{\eta,t}(u)}_{g}dt\times \max \set{g\text{-length of a flow line of $P$}}<\varepsilon,
  \end{equation*}
  for some Riemannian metric $g$ on $N$.
\end{lemma}
\begin{proof}
  Recall that the differential counts $\R$-families of finite-energy solutions $(\pi,u)$ where $\pi\in \mathscr{P}(\eta_{1},\eta_{0})$ is a flow line of $-P$ and $u:\R\times \R/\Z\to W$ solves:
  \begin{equation*}
    \bd_{s}u+J(u)(\bd_{t}u-X_{\pi(s),t}(u))=0,
  \end{equation*}
  where $X_{\eta,t}$ is the Hamiltonian vector field for $H_{\eta,t}$. The space of such solutions admits an $\R$-action by translating $\pi$ and $u$, and the differential counts the rigid elements in the quotient by this $\R$-action.

  Unpacking the definition from \S\ref{sec:energy-integral}, one sees the energy of $u$ is:
  \begin{equation*}
    E(u):=\int \omega(\bd_{s}u,\bd_{t}u-X_{\pi(s),t}(u))dsdt,
  \end{equation*}
  which is non-negative and can be computed as:
  \begin{equation*}
    E(u)=\omega(u)+\int_{0}^{1}H_{\eta_{1},t}(\gamma_{1}(t))dt-\int_{0}^{1}H_{\eta_{0},t}(\gamma_{0}(t))dt+\int\bd_{s}H_{\pi(s),t}(u)dsdt,
  \end{equation*}
  where $\gamma_{i}(t)$ are the asymptotic orbits of $u$. If $\eta_{1}=\eta_{0}$, then $\pi'(s)=0$ and $\bd_{s}H_{\pi(s),t}=0$. Otherwise, the last term can be estimated:
  \begin{equation*}
    \int \bd_{s}H_{\pi(s),t}(u)dsdt\le \max_{\eta,t,u}\abs{\bd_{\eta}H_{\eta,t}(u)}_{g}\int_{\R}\abs{\pi'(s)}_{g}ds\le \varepsilon.
  \end{equation*}

  Using $\omega=\d\lambda$, one obtains:
  \begin{equation*}
    0\le A_{\varepsilon}( \gamma_{1}\otimes\eta_{1})-A_{\varepsilon}(\gamma_{0}\otimes \eta_{0});
  \end{equation*}
  this is proved in two cases: first, if $\eta_{1}=\eta_{0}$, and, second, if $\eta_{1}\ne \eta_{0}$, in which case their index difference is at least $1$.
\end{proof}

Thus, for $\varepsilon$ satisfying the hypotheses of Lemma \ref{lemma:dff-filtration}, $(\mathrm{CFF}(P,H_{t},J),d_{\mathrm{FF}},A_{\varepsilon})$ is a (cohomologically) filtered complex, and so it makes sense to speak about the filtration level of a homology class (as the maximum action of all representative cycles).

Next we will show that $A_{\varepsilon}$ is compatible with $\Delta_{\mathrm{FF}}$ up to a bounded error.

\begin{lemma}\label{lemma:bvff-filtration}
  For any homology class $\alpha\in \mathrm{HFF}(P,H_{\eta,t},J)$, it holds that:
  \begin{equation*}
    -\int \max_{\eta,u}\abs{H_{\eta,t}(u)-H_{\eta,t+\theta}(u)}\d t\le A_{\varepsilon}(\Delta_{\mathrm{FF}}(\alpha))-A_{\varepsilon}(\alpha),
  \end{equation*}
  provided that $\varepsilon$ satisfies the hypotheses of Lemma \ref{lemma:dff-filtration}.
\end{lemma}
\emph{Remark}. In particular, if $H_{\eta,t}$ is close to being autonomous (for each $\eta$), then $\Delta_{\mathrm{FF}}$ is close to being action non-decreasing.
\begin{proof}
  The argument is similar to Lemma \ref{lemma:dff-filtration}. Recall that $\Delta_{\mathrm{FF}}$ counts the rigid solutions $(\pi,u)$ where $\pi\in \mathscr{P}(\eta_{1},\eta_{0})$ and $u$ solves a perturbed version of the equation:
  \begin{equation}\label{eq:deltaFF-filtration-cmap-type}
    \bd_{s}u+J(u)(\bd_{t}u-X_{s,t}(u))=0
  \end{equation}
  where $X_{s,t}$ is the Hamiltonian vector field for $(1-\beta(s))H_{\pi(s),t+\theta}+\beta(s)H_{\pi(s),t}$. The energy of solutions for the unperturbed equation can be estimated as:
  \begin{equation*}
    \begin{aligned}
      E(u)&\le \int_{0}^{1}\max_{\eta,u}\abs{H_{\eta,t}(u)-H_{\eta,t+\theta}(u)}\d t+\max_{u,\eta,t}\abs{\bd_{\eta}H_{\eta,t}}_{g}\int_{\R}\abs{\pi'(s)}_{g}\\
          &+\omega(u)+\int_{0}^{1}H_{\eta_{1},t+\theta}(\gamma_{1}(t+\theta))d t-\int_{0}^{1}H_{\eta_{0},t}(\gamma_{0}(t))dt.
    \end{aligned}
  \end{equation*}
  The proof uses the standard energy estimates for continuation map type equation like \eqref{eq:deltaFF-filtration-cmap-type}, similarly to Lemma \ref{lemma:dff-filtration}.

  If $\eta_{1}=\eta_{0}$, then $\pi'(s)=0$; otherwise the index difference is at least $1$. In both cases, one rearranges to obtain the desired result.

  In general, one needs to work with the perturbed equation $\mathfrak{p}$; however, since the chain homotopy class is independent of the perturbation, one can prove the estimate with small error terms due to the perturbation $\mathfrak{p}$, and taking a limit as $\mathfrak{p}\to 0$ will reduce to the above unperturbed analysis; see also the proof of Lemma \ref{lemma:popff-filtration} for further details on this step of the argument.
\end{proof}

To conclude this section, we show that $\ast_{\mathrm{FF}}$ respects the $\varepsilon$-action filtration, up to an error depending on $\varepsilon$.

\begin{lemma}\label{lemma:popff-filtration}
  Suppose that $\varepsilon$ satisfies the hypotheses of Lemma \ref{lemma:dff-filtration}. Then, for any two homology classes $\alpha_{i}\in \mathrm{HFF}(P,H_{i,\eta,t},J)$, $i=0,1$, we have:
  \begin{equation*}
    -(3+\dim N)\varepsilon-C\max_{\eta,t_{0},t_{1}}\abs{\set{H_{0,\eta,t_{0}},H_{1,\eta,t_{1}}}}\le A_{\varepsilon}(\alpha_{0}\ast_{\mathrm{FF}}\alpha_{1})-A_{\varepsilon}(\alpha_{0})-A_{\varepsilon}(\alpha_{1}),
  \end{equation*}
  where $C$ is a uniform constant depending only on the pair-of-pants surface; here $\alpha_{0}\ast_{\mathrm{FF}}\alpha_{1}\in \mathrm{HFF}(P,H_{\infty,\eta,t},J)$ and $H_{\infty,\eta,t}=H_{0,\eta,t}+H_{1,\eta,t}$ is assumed to be admissible for defining the family Floer complex as required in the definition of $\ast_{\mathrm{FF}}$, and $\set{-,-}$ is the Poisson bracket.
\end{lemma}
\begin{proof}
  Recall from \S\ref{sec:family-pair-of-pants} that the product is defined by counting rigid solutions to the perturbed equation determined by the connection one-form:
  \begin{equation*}
    \mathfrak{a}=\left\{
      \begin{aligned}
        &H_{i,\pi_{i}(s),t}\d t\text{ in the $i$ cylindrical end, $i=0,1,\infty$},\\
        &H_{0,\eta',t_{0}(z)}\mathfrak{m}_{0}+H_{1,\eta',t_{1}(z)}\mathfrak{m}_{1}\text{ otherwise},
      \end{aligned}
    \right.
  \end{equation*}
  where $(\pi_{0},\pi_{1},\pi_{\infty})$ is a flow tree, $\eta'$ is the junction point of the flow tree, and the circle valued functions $t_{i}$, and their differentials $\mathfrak{m}_{i}=\d t_{i}$ are as in \S\ref{sec:family-pair-of-pants}. The branches of the flow tree are flow lines for perturbations of the pseudogradient.

  Thus there are two relevant perturbations which play a role in the definition of the pair-of-pants product:
  \begin{enumerate}
  \item\label{pop-perturbation-1} the usual perturbation term $\mathfrak{p}$,
  \item\label{pop-perturbation-2} the perturbations used to define the flow tree.
  \end{enumerate}
  For each solution $(\pi,u)$ to the family pair-of-pants equation, ideally, we would like to show that:
  \begin{equation}\label{eq:pop-estimate-need-to-show}
    -(3+\dim N)\varepsilon\le A_{\varepsilon}(\gamma_{\infty}\otimes \eta_{\infty})-A_{\varepsilon}(\gamma_{0}\otimes \eta_{0})-A_{\varepsilon}(\gamma_{1}\otimes \eta_{1}),
  \end{equation}
  where $\gamma_{i},\eta_{i}$, $i=0,1,\infty$ are the asymptotics. We simplify our task as follows: it suffices to prove \eqref{eq:pop-estimate-need-to-show} when the perturbations \ref{pop-perturbation-1} and \ref{pop-perturbation-2} are turned off. To see why this is sufficient, recall that the chain homotopy class of the pair-of-pants is independent of the perturbation used. Thus we can take a sequence of solutions $(\pi_{n},u_{n})$ solving the equation with perturbations \ref{pop-perturbation-1} and \ref{pop-perturbation-2} depending on $n$ which tend to zero.

  By the usual compactness theory for solutions to Floer's equation (see \S\ref{sec:compactness}) and standard compactness results for ODEs, one concludes that $(\pi_{n},u_{n})$ has a subsequence which converges to a configuration consisting of a central limit $(\pi_{\infty},u_{\infty})$ solving the unperturbed pair-of-pants equation, together with some number of cylinders at the punctures solving the equation for $d_{\mathrm{FF}}$; see the illustration in Figure \ref{fig:pop-turn-off-perturb}. Since we have already shown that $d_{\mathrm{FF}}$ increases the $A_{\varepsilon}$-action in Lemma \ref{lemma:dff-filtration}, it is sufficient to prove \eqref{eq:pop-estimate-need-to-show} holds for solutions of the unperturbed equation; this completes the explanation of why we can work only with the unperturbed equation.

  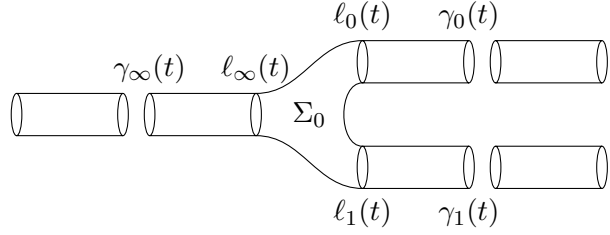
\begin{figure}[h]
    \centering
    \begin{tikzpicture}[scale=.7]
      \draw (0,0) circle (.1 and .4) (0,0.4)--+(2,0)coordinate(A1) (0,-0.4)--+(2,0)coordinate(B1) (2,0) circle (.1 and .4);
      \draw[shift={(4,1)}] (0,0) circle (.1 and .4) (0,0.4) coordinate(A2)--+(2,0) (0,-0.4)coordinate(B2)--+(2,0) (2,0) circle (.1 and .4);
      \draw[shift={(4,-1)}] (0,0) circle (.1 and .4) (0,0.4) coordinate(A3)--+(2,0) (0,-0.4)coordinate(B3)--+(2,0) (2,0) circle (.1 and .4);
      \node at (3,0) {$\Sigma_{0}$};

      \draw (A1)to[out=0,in=180](A2) (B2)to[out=180,in=180](A3) (B1)to[out=0,in=180](B3);

      \draw[shift={(6.5,1)}] (0,0) circle (.1 and .4) (0,0.4) coordinate(A4)--+(2,0) (0,-0.4)coordinate(B4)--+(2,0) (2,0) circle (.1 and .4);
      \draw[shift={(6.5,-1)}] (0,0) circle (.1 and .4) (0,0.4) coordinate(A5)--+(2,0) (0,-0.4)coordinate(B5)--+(2,0) (2,0) circle (.1 and .4);
      \draw[shift={(-2.5,0)}] (0,0) circle (.1 and .4) (0,0.4) coordinate(A6)--+(2,0) (0,-0.4)coordinate(B6)--+(2,0) (2,0) circle (.1 and .4);

      \path (A2) node [above] {$\ell_{0}(t)$}--+(2,0) node[above] {$\gamma_{0}(t)$};
      \path (B3) node [below] {$\ell_{1}(t)$}--+(2,0) node[below] {$\gamma_{1}(t)$};
      \path (A1) node [above] {$\ell_{\infty}(t)$}--+(-2,0) node[above] {$\gamma_{\infty}(t)$};
    \end{tikzpicture}
    \caption{Illustration of the pair-of-pants and limit (cylinders for $d_{\mathrm{FF}}$ break off) when perturbations are turned off}
    \label{fig:pop-turn-off-perturb}
  \end{figure}

  Referring to the notation in Figure \ref{fig:pop-turn-off-perturb}, we claim the following:
  \begin{equation}\label{eq:many-inequalities-pop-filtration}
    \begin{aligned}
      A_{0}( \gamma_{\infty} \otimes \eta_{\infty})&\ge A(H_{\infty,\eta',t};\ell_{\infty}(t))-\varepsilon\\
      A(H_{0,\eta',t};\ell_{0}(t))&\ge A_{0}( \gamma_{0}\otimes\eta_{0})-\varepsilon\\
      A(H_{1,\eta',t};\ell_{1}(t))&\ge A_{0}(\gamma_{1}\otimes \eta_{1})-\varepsilon\\
      A(H_{\infty,\eta',t};\ell_{\infty}(t))&\ge A(H_{0,\eta',t};\ell_{0}(t))+A(H_{1,\eta',t};\ell_{1}(t)),
    \end{aligned}
  \end{equation}
  where $A_{0}(\eta\otimes \gamma)$ is simply the action functional considered in \S\ref{sec:family-acti-filtr} with $\varepsilon$ set to zero, and: $$A(H_{i,\eta',t};\ell(t))=\int H_{i,\eta',t}(\ell(t))dt-\int \ell^{*}\lambda,$$ where $\eta'$ is the junction point of the limit flow tree.

  Combining everything, we conclude that:
  \begin{equation*}
    A_{0}(\gamma_{\infty}\otimes \eta_{\infty})+3\varepsilon\ge A_{0}(\gamma_{0}\otimes \eta_{0})+A_{0}(\gamma_{1}\otimes \eta_{1}).
  \end{equation*}
  Finally we observe that, along any unperturbed flow tree we have: $$\mathrm{Index}(\eta_{\infty})\ge \max\set{\mathrm{Index}(\eta_{1}),\mathrm{Index}(\eta_{0})},$$
  and hence $\varepsilon\mathrm{Index}(\eta_{\infty})+\varepsilon\dim N\ge \varepsilon \mathrm{Index}(\eta_{0})+\varepsilon\mathrm{Index}(\eta_{1});$ this implies \eqref{eq:pop-estimate-need-to-show}, as desired.

  It remains only to verify \eqref{eq:many-inequalities-pop-filtration}. The first three lines of \eqref{eq:many-inequalities-pop-filtration} follow from the same exact argument as Lemma \ref{lemma:dff-filtration}. The rest of the proof is showing the last line; it will ultimately follow from a computation of the curvature of $\mathfrak{a}$ on the central region $\Sigma_{0}$ of the pair-of-pants shown in Figure \ref{fig:pop-turn-off-perturb}.

  By Lemma \ref{lemma:energy-identity}, it is sufficient to bound the intergral of the connection two-form $\mathfrak{r}$ associated to $\mathfrak{a}$ over the central region $\Sigma_{0}$. Recall that:
  \begin{equation*}
    \mathfrak{a}=H_{x,y}\d x+K_{x,y}\d y\implies \mathfrak{r}=\left(\bd_{x}K_{x,y}-\bd_{y}H_{x,y}+\omega(X_{x,y},Y_{x,y})\right)\d x\wedge \d y,
  \end{equation*}
  where $X_{s,t},Y_{s,t}$ are the Hamiltonian vector fields for $H_{s,t},K_{s,t}$, respectively. In our case, above the central region, we have:
  \begin{equation*}
    \mathfrak{a}=H_{0,\eta',t_{0}(z)}\mathfrak{m}_{0}+H_{1,\eta',t_{1}(z)}\mathfrak{m}_{1}.
  \end{equation*}
  We claim that:
  \begin{equation}\label{eq:desired-formula-curvature-popff-filtration}
    \mathfrak{r}=\omega(X_{0,\eta',t_{0}(z)},X_{1,\eta',t_{1}(z)})\mathfrak{m}_{0}\wedge \mathfrak{m}_{1}.
  \end{equation}
  To compute this, we first observe that $\bd_{x}K_{x,y}-\bd_{y}H_{x,y}$ is linear, and hence it suffices to prove it vanishes for each $H_{i,\eta',t_{i}(z)}\mathfrak{m}_{i}$ separately. Write $t_{i}=f_{i}$, $H_{i,\eta',t_{i}(z)}=G_{i,f_{i}(z)}$ and $\mathfrak{m}_{i}=\d f_{i}$. Then, in conformal coordinates $x+iy$,
  \begin{equation*}
    G_{_{i},f_{i}(z)}\d f_{i}=G_{i,f_{i}(z)}\bd_{x}f_{i}\d x+G_{i,f_{i}(z)}\bd_{y}f_{i} \d y,
  \end{equation*}
  and a short computation shows the $\bd_{x}K_{x,y}-\bd_{y}H_{x,y}$ term vanishes. Thus the only term which contributes to $\mathfrak{r}$ is $\omega(X_{x,y},Y_{x,y})\d x\wedge \d y$, which expands to:
  \begin{equation*}
    \omega(V_{0,f_{0}(z)}\bd_{x}f_{0}+V_{1,f_{1}(z)}\bd_{x}f_{1},V_{0,f_{0}}\bd_{y}f_{0}+V_{1,f_{1}(z)}\bd_{y}f_{1})\d x\wedge \d y,
  \end{equation*}
  where $V_{i,f}$ is the Hamiltonian vector field of $G_{i,f}$. Simplifying, one obtains:
  \begin{equation*}
    \omega(V_{0,f_{0}(z)},V_{1,f_{1}(z)})(\bd_{x}f_{0}\bd_{y}f_{1}-\bd_{y}f_{0}\bd_{x}f_{1})\d x\wedge \d y,
  \end{equation*}
  which equals \eqref{eq:desired-formula-curvature-popff-filtration}, as desired. Finally, using the notation $v(z)=(z,u(z))$ from Lemma \ref{lemma:energy-identity}, we estimate:
  \begin{equation*}
    \int v^{*}\mathfrak{r}_{\sigma}\le \max_{\eta,t_{0},t_{1}}\abs{\set{H_{0,\eta,t_{0}},H_{1,\eta,t_{1}}}}\int_{\Sigma_{0}} \abs{\mathfrak{m}_{0}\wedge \mathfrak{m}_{1}},
  \end{equation*}
  which gives the desired result.
\end{proof}

\subsubsection{Proof of Theorem \ref{theorem:main-floer}}
\label{sec:proof-of-main-floer}

Suppose that $H_{c,\delta,\epsilon,\eta}$ is the family as constructed in \S\ref{sec:system-of-hamiltonians}, using the lift of $f$ to $\mathfrak{B}(a,\Omega)/U(n)$. The proof of Theorem \ref{theorem:main-floer} is an argument by contradiction: we assume $a>c$, and then derive a contradiction.

As in the hypotheses of Theorem \ref{theorem:main-floer}, suppose that there are classes $\zeta_{i}\in V_{c_{i}}$, $i=1,2$, where $c_{i}>0$, such that:
\begin{equation*}
  \mathrm{PSS}(\beta)=\Delta(\zeta_{1})\ast\zeta_{2}\text{ in }V_{c},
\end{equation*}
where $c=c_{1}+c_{2}$, and where $\beta$ has non-zero homological intersection with $f$. Without loss of generality, we can suppose that $c_{2}>0$. Since $a>c$, we can use the evaluation map $\mathfrak{e}$ and \S\ref{sec:non-triv-eval} to conclude:
\begin{equation*}
  1=\mathfrak{e}(\Delta(\zeta_{1})\ast\zeta_{2}).
\end{equation*}
Let $\theta_{i}=c_{i}/c$, so that $\theta_{1}+\theta_{2}=1$. Given a pseudogradient $P$ on $N$, construct the family $H_{\eta,t}'$ as a small perturbation of $H_{c,\delta,\epsilon,\eta}$ as in \S\ref{sec:spec-hamilt-assoc}. Let:
\begin{equation*}
  H_{i,\eta,t}=\theta_{i}H_{\eta,t}',\text{ for }i=1,2,
\end{equation*}
Fix a small constant $\varepsilon>0$, and pick the perturbation term used in the construction of $H_{\eta,t}'$ so that:
\begin{enumerate}
\item\label{setup-main-floer-1} $H_{1,\eta,t}$, $H_{2,\eta,t}$, and $H_{1,\eta,t}+H_{2,\eta,t}$ are all admissible for defining the family Floer complex,
\item\label{setup-main-floer-2} $\max_{\eta,t_{0},t_{1}}\abs{\set{H_{\eta,t_{0}}',H_{\eta,t_{1}}'}}<\varepsilon/C$, where $C$ is from Lemma \ref{lemma:popff-filtration}.
\item\label{setup-main-floer-3} the action of any non-central orbit appearing in $H_{1,\eta,t}$ is non-negative.
\end{enumerate}
Lemma \ref{lemma:construction-of-D} implies that we can assume that:
\begin{enumerate}[resume]
\item\label{setup-main-floer-4} the action of any orbit for $H_{2,\eta,t}$ is at least $\theta_{2}(c-a)-\varepsilon$.
\end{enumerate}
By picking the pseudogradient $P$ to have short enough flow lines, we can also suppose that:
\begin{enumerate}[resume]
\item\label{setup-main-floer-5} $\max_{\eta,t}\abs{\bd_{\eta}H_{\eta,t}'}_{g}\times \max\set{g\text{-length of flow lines for $P$}}<\varepsilon.$
\item\label{setup-main-floer-6} the conclusion of Proposition \ref{proposition:BV-family} holds for $H_{1,\eta,t}$; here we note that $H_{1,\eta,t}$ is a perturbation of $\theta_{1}H_{c,\delta,\epsilon,\eta}=H_{\theta_{1}c,\delta,\theta_{1}\epsilon,\eta}$, and Proposition \ref{proposition:BV-family} applies.
\end{enumerate}

With this established, we claim:
\begin{lemma}\label{lemma:proof-main-floer-action-estimate}
  For any classes $\zeta_{i}\in \mathrm{HFF}(P,H_{i,\eta,t},J)$, we have that $\varepsilon$ satisfies the hypotheses of Lemma \ref{lemma:dff-filtration} for each $H_{i,\eta,t}$, so $\mathscr{A}_{\varepsilon}$ is valid cohomological filtration, and:
  \begin{equation*}
    \mathscr{A}_{\varepsilon}(\Delta_{\mathrm{FF}}(\zeta_{1})\ast_{\mathrm{FF}}\zeta_{2})\ge \theta_{2}(c-a)-(5+\dim N)\varepsilon.
  \end{equation*}
\end{lemma}
\begin{proof}
  The first statement follows from \ref{setup-main-floer-5}. The next step is to use Proposition \ref{proposition:BV-family} and \ref{setup-main-floer-6} to conclude that:
  \begin{equation*}
    \Delta_{\mathrm{FF}}(\zeta_{1})\text{ is represented by cycles not containing any central orbit.}
  \end{equation*}
  In particular, using \ref{setup-main-floer-3}, we conclude:
  \begin{equation*}
    \mathscr{A}_{\varepsilon}(\Delta_{\mathrm{FF}}(\zeta_{1}))\ge 0.
  \end{equation*}
  Use this, Lemma \ref{lemma:popff-filtration}, and \ref{setup-main-floer-2}, to conclude:
  \begin{equation*}
    \mathscr{A}_{\varepsilon}(\Delta_{\mathrm{FF}}(\zeta_{1})\ast_{\mathrm{FF}}\zeta_{2})\ge \mathscr{A}_{\varepsilon}(\Delta_{\mathrm{FF}}(\zeta_{1}))+\mathscr{A}_{\varepsilon}(\zeta_{2})-(4+\dim N)\varepsilon,
  \end{equation*}
  for $j=1,\dots,k-1$. Then use \ref{setup-main-floer-3} and \ref{setup-main-floer-4} to conclude:
  \begin{equation*}
    \mathscr{A}_{\varepsilon}(\Delta_{\mathrm{FF}}(\zeta_{1})\ast_{\mathrm{FF}}\zeta_{2})\ge \theta_{2}(c-a)-(5+\dim N)\varepsilon,
  \end{equation*}
  as desired.
\end{proof}

If the lengths of the flow lines are short enough, then Proposition \ref{proposition:eFF-vanish-on-non-central} says $\mathfrak{e}_{\mathrm{FF}}:\mathrm{HFF}(P,H_{\eta,t}',J)\to \Z/2\Z$ vanishes on classes which are represented by cycles which contain only non-central orbits. The discussion at the start of this subsection, together with the compatibility of $\mathfrak{e}_{\mathrm{FF}},\Delta_{\mathrm{FF}},\ast_{\mathrm{FF}}$ and their non-family analogues implies:
\begin{equation*}
  \mathfrak{e}_{\mathrm{FF}}(\Delta_{\mathrm{FF}}(\zeta_{1})\ast_{\mathrm{FF}}\zeta_{2})=1.
\end{equation*}
The contradiction leading to the proof of Theorem \ref{theorem:main-floer} will therefore be completed provided we can prove:
\begin{equation}\label{eq:main-cycle-proof-main-floer}
  \Delta_{\mathrm{FF}}(\zeta_{1})\ast_{\mathrm{FF}}\zeta_{2}\in\mathrm{HFF}(P,H_{\eta,t}',J)
\end{equation}
is represented by a cycle which contains only non-central orbits. This fact follows from the action estimate in Lemma \ref{lemma:proof-main-floer-action-estimate}, provided $\varepsilon$ is small enough. Indeed, Lemma \ref{lemma:construction-of-D} implies that the central orbit has action:
\begin{equation*}
  c+c\delta/2-\epsilon a.
\end{equation*}
Picking $\epsilon$ close enough to $1$, $\delta$ small enough, and picking $\varepsilon$ small enough, we can use the fixed negative number $c-a$ to ensure:
\begin{equation*}
  \mathscr{A}_{\varepsilon}(\text{any central orbit})\le c+c\delta/2-\epsilon a+\varepsilon \dim N< \theta_{2}(c-a)-(5+\dim N)\varepsilon,
\end{equation*}
where we use $\theta_{2}<1$ (since $\theta_{1}>0$ and $\theta_{1}+\theta_{2}=1$). It therefore follows that \eqref{eq:main-cycle-proof-main-floer} must be represented by a cycle which does not contain central orbits. Thus $\mathfrak{e}_{\mathrm{FF}}$ must vanish on it, providing the desired contradiction, proving Theorem \ref{theorem:main-floer}.\hfill$\square$

\subsubsection{Proof of Theorem \ref{theorem:hofer-zehnder}}
\label{sec:proof-theorem-hofer-zehnder}

The argument is essentially exactly the same. The first difference is that one defines:
\begin{equation*}
  H_{c,\delta,\epsilon,\eta}=G_{c,\delta}+K+\epsilon D_{\eta},
\end{equation*}
as a minor replacement of the functions defined in \S\ref{sec:system-of-hamiltonians}.

The second difference $f$ is now assumed to lift to $\mathfrak{B}(a,U)/U(n)$, where $U$ is the well of the Hofer-Zehnder admissible function $K$.

Everything else follows the same argument. This works because $K$ is independent of $c,\delta,\epsilon,\eta$ (so many terms involving derivatives of $H_{c,\delta,\epsilon,\eta}$ with respect to its parameters vanish), and $K$ only has constant orbits and $U$ is contained in the set where $K=-A=\min K$ (this is used in Lemma~\ref{lemma:construction-of-D} which estimates the actions). Because of the Hofer-Zehnder admissible function $K$, the action of the center of the ball is $c+c\delta/2-A-\epsilon a$. Thereby one concludes the evaluation map is defined when $c<A+a$.\hfill$\square$

\subsubsection{Proof of Theorem \ref{theorem:non-contractible}}
\label{sec:proof-theorem-non-contractible}

The argument is very similar to the proof of Theorem \ref{theorem:main-floer}, and in fact is a bit easier. Since $\zeta_{1}$ is represented by a cycle whose orbits are all non-contractible, $\zeta_{1}$ is represented by a cycle which does not contain the central orbit. The proof then follows the same lines as \S\ref{sec:proof-of-main-floer}. The details are left to the reader. \hfill$\square$


\section{From string topology to Floer cohomology}
\label{sec:string-topol-floer-cohom}

The goal of this section is to prove Theorem \ref{theorem:main-comparison} on the comparison between string topology and Floer cohomology. We specialize to the case when $W=T^{*}M$ and $\Omega$ is a fiberwise starshaped domain.

Much of the arguments in this section are straightforward modifications of the arguments of \cite{abbondandolo-schwarz-GT-2010,abouzaid-EMS-2015} from their Morse theory approach to our bordism approach. The exception is the proof that the product structures are identified, as we instead give a seemingly novel adiabatic gluing argument.

Recall from \S\ref{sec:comp-betw-string} the monoid $Z(\Lambda_{c})$ of smooth families of loops:
\begin{equation*}
  A:P\times \R/\Z\to M,
\end{equation*}
such that $P$ is a compact finite dimensional manifold, and $\ell_{\Omega}(A(x,-))<c$ holds for all $x\in P$, where:
\begin{equation*}
  \ell_{\Omega}(q)=\int_{0}^{1}\max\set{\ip{p,q'(t)}:p\in \Omega\cap T^{*}M_{q(t)}}dt
\end{equation*}
is the length of the loop $q$ as measured by $\Omega$. As in \S\ref{sec:comp-betw-string}, this leads to the bordism group $H(\Lambda_{c})$ as a quotient of $Z(\Lambda_{c})$. Then $c\mapsto H(\Lambda_{c})$ has the structure of a persistence module on the positive real line.

As explained in \S\ref{sec:three-struct-string}, there are three relevant structures on this persistence module:
\begin{enumerate}
\item the BV-operator $\Delta:H(\Lambda_{c})\to H(\Lambda_{c})$,
\item the inclusion of the constant loops $\mathfrak{i}:H^{*}(W)\to H(\Lambda_{c})$, and,
\item the Chas-Sullivan product $\ast:H(\Lambda_{c_{1}})\otimes H(\Lambda_{c_{2}})\to H(\Lambda_{c_{1}+c_{2}})$.
\end{enumerate}

In the first subsection \S\ref{sec:definition-morphism}, we will explain the definition of a morphism of persistence modules $\Theta:H(\Lambda_{c})\to V_{c}$; afterwards, we show $\Theta$ intertwines $\Delta,\mathfrak{i},\ast$ with the analogous structures on the Floer cohomology persistence module.

\subsection{Definition of the $\Theta$-morphism}
\label{sec:definition-morphism}

The definition of the $\Theta$-morphism follows the strategy of \cite{abbondandolo-schwarz-CPAM-2006,abbondandolo-portaluri-schwarz-JFPTA-2008,abbondandolo-schwarz-GT-2010,abouzaid-EMS-2015,cieliebak-hingston-oancea-JFPTA-2023,brocic-cant-shelukhin-math-ann-2025}, and uses moduli spaces with moving Lagrangian boundary conditions similar to the ones considered in \cite{cieliebak-JMPA-1994,brocic-cant-JFPTA-2024}.

\subsubsection{$\Theta$-data.}
\label{sec:data-theta-morphism}

Let $(H_{t},J)$ be a Hamiltonian system admissible for defining the Floer complex, and let $A:P\times \R/\Z\to M$ be a family of loops in $Z(\Lambda_{c})$.

For such inputs $(H_{t},J)$ and $A$, we define \emph{$\Theta$-data} to be:
\begin{enumerate}
\item\label{theta-data-1} a Hamiltonian connection $\mathfrak{a}$ on $P\times \Sigma\times W$ where $\Sigma$ is the half-infinite cylinder $(-\infty,0]\times\R/\Z$, so $\mathfrak{a}_{x,s,t}=K_{x,s,t}\d s+H_{x,s,t}$, where $K_{x,s,t}=0$ and $H_{x,s,t}=H_{t}$ for $s\le s_{0}$,
\item\label{theta-data-2} a smooth family of $\omega$-tame and Liouville equivariant almost complex structures $J_{p,s,t}$ on $P\times \Sigma \times W$ so $J_{p,s,t}=J$ when $s\le -s_{0}$,
\end{enumerate}
which satisfies the following properties:
\begin{enumerate}[resume]
\item\label{theta-data-3} $H_{x,s,t}=c_{x,s,t}r$ and $K_{x,s,t}=b_{x,s,t}r$ for $r\ge r_{0}$,
\item\label{theta-data-4} $\bd_{s}c_{x,s,t}\le \bd_{t}b_{x,s,t}$,
\item\label{theta-data-5} $c_{x,0,t}\ge \max\set{\ip{p,q'(t)}:p\in \Omega\cap T^{*}M_{q(t)}\text{ where }q(t)=A(x,t)}$, outside of $r\ge r_{0}$.
\end{enumerate}
Similarly to \S\ref{sec:cont-maps-floer}, condition \ref{theta-data-4} is used when showing that the curvature of $\mathfrak{a}$ is bounded from above, which is used in proving a priori energy estimates. Condition \ref{theta-data-5} is also used in the a priori energy estimate; see \S\ref{sec:a-priori-energy-estim-theta-map}.

Let us note that \ref{theta-data-5} implies:
\begin{equation*}
  \int c_{x,0,t}dt\ge \ell_{\Omega}(q)\text{ for each }q=A(x,-),
\end{equation*}
and so it is necessary that the slope of $H_{t}$ is at least the $\ell_{\Omega}$-length of all loops appearing in the family $A$.

\begin{lemma}
  If the slope of $H_{t}$ is at least the $c$, where $A\in Z(\Lambda_{c})$, then the space of $\Theta$-data for $(H_{t},J)$ and $A$ is weakly contractible.
\end{lemma}
\begin{proof}
  The argument is exactly as in Lemma \ref{lemma:when-does-cdata-exist}; one can take convex combinations between any the connection one forms, and use the contractibility of the space of almost complex structures, to prove the space of $\Theta$-data is either empty or weakly contractible. To prove it is non-empty, one picks a family of smooth functions $H_{x,0,t}$ so that $H_{x,0,t}=c_{x,0,t}r$ for $r\ge r_{0}$, and: $$c_{x,0,t}\ge \max \set{\ip{p,q'(t)}:p\in \Omega\cap T^{*}M_{q(t)}\text{ where }q(t)=A(x,t)}.$$ Moreover, we can assume that:
  \begin{equation}\label{eq:estimate-theta-data-slope-decrease}
    \int c_{x,0,t}\d t\le c\le \text{slope of $H_{t}$}\text{ for each }x\in P.
  \end{equation}
  Similarly to the proof of Lemma \ref{lemma:when-does-cdata-exist}, define:
  \begin{equation*}
    \left\{
      \begin{aligned}
        H_{x,s,t}&=(1-\beta(s+1))H_{t}+\beta(s+1)H_{x,0,t},\\
        K_{x,s,t}&=\int_{0}^{t}\partial_{s}H_{x,s,\tau}d\tau-t\partial_{s}\int_{0}^{1}H_{x,s,\tau}d\tau.
      \end{aligned}
    \right.
  \end{equation*}
  Then \eqref{eq:estimate-theta-data-slope-decrease} implies condition \ref{theta-data-4}. Condition \ref{theta-data-5} holds by construction, and the other properties are obvious.
\end{proof}

\subsubsection{A priori energy estimate}
\label{sec:a-priori-energy-estim-theta-map}

The $\Theta$-map will be defined as follows: given $\Theta$-data $(\mathfrak{a},J)$, we will pick a generic perturbation one-form $\mathfrak{p}$ on $P\times \Sigma\times W$, which we assume vanishes above $s=0$ and above $s\le -s_{0}$. Then we will count the rigid solutions $(x,u)$ to \S\ref{sec:floers-equation-general} satisfying the boundary conditions:
\begin{equation}\label{eq:theta-boundary-equations}
  u(0,t)\in T^{*}M_{q(t)}\text{ where }q(t)=A(x,t).
\end{equation}
In this subsection, we show such solutions satisfy a priori energy bound. Similar energy estimates are proved in \cite{brocic-cant-JFPTA-2024,brocic-cant-shelukhin-math-ann-2025}.

We will use the general energy estimate Lemma \ref{lemma:energy-estimate}. First, \ref{theta-data-4} implies $\mathfrak{a}$ has curvature bounded from above. Therefore Lemma \ref{lemma:a-has-then-p-has} implies $\mathfrak{a},\mathfrak{p}$ also has curvature bounded from above. Thus, the a priori energy bound follows from:
\begin{lemma}\label{lemma:energy-estimate-theta}
  Let $\mathfrak{a},J$ be $\Theta$-data, and let $\mathfrak{p}$ be a perturbation one-form as above. There is a uniform bound:
  \begin{equation*}
    \sup \set{\omega(u)-\int_{\bd\Sigma}v^{*}\mathfrak{a}_{x}}<\infty
  \end{equation*}
  where the supremum is over $v(z)=(z,u(z))$ where $(x,u)$ is a finite energy solution of \S\ref{sec:floers-equation-general} for $(P\times \Sigma\times W,\mathfrak{a},\mathfrak{p},J)$ with boundary conditions \eqref{eq:theta-boundary-equations}.
\end{lemma}
\begin{proof}
  Unpacking the definitions, one sees that:
  \begin{equation*}
    \int_{\bd\Sigma}v^{*}\mathfrak{a}_{x}=\int_{0}^{1} H_{x,0,t}(u(0,t))d t,
  \end{equation*}
  and:
  \begin{equation*}
    \omega(u)=\int_{0}^{1}\lambda_{u(0,t)}(\bd_{t}u(0,t))dt-\int \gamma^{*}\lambda,
  \end{equation*}
  where $\gamma$ is the asymptotic orbit of $u$. Since there are only finitely many orbits, it is sufficient to bound:
  \begin{equation*}
    I=\int_{0}^{1}\lambda_{u(0,t)}(\bd_{t}u(0,t))-H_{x,0,t}(u(0,t))dt,
  \end{equation*}
  as can be verified by inspecting the terms in the general estimate Lemma \ref{lemma:energy-estimate}.

  As we are working with the cotangent bundle $W=T^{*}M$, we use the canonical Liouville form $\lambda=p\d q$; substituting this, and using the boundary conditions for $u$, the above integral becomes:
  \begin{equation*}
    I=\int_{0}^{1}\ip{p(t),q'(t)}-H_{x,0,t}(u(0,t))dt,
  \end{equation*}
  where $q(t)=A(x,t)$ and $u(0,t)=(p(t),q(t))$ considered as section of $T^{*}M$.

  Write $r(t)=r(u(0,t))$. Then:
  \begin{equation*}
    \ip{p(t),q'(t)}\le \max \set{\ip{p,q'(t)}:p\in \Omega\cap T^{*}M_{q(t)}}r(t),
  \end{equation*}
  as can be proved by rescaling in the fiber direction until $p\in \bd\Omega$, and then using the fact the above estimate is invariant under rescaling in the fiber direction. Thus:
  \begin{equation*}
    I\le \int_{0}^{1}\max \set{\ip{p,q'(t)}:p\in \Omega\cap T^{*}M_{q(t)}}r(t)-H_{x,0,t}(u(0,t))d t.
  \end{equation*}
  Because $\Theta$-data satisfies property \ref{theta-data-5}, the function:
  \begin{equation*}
    F(t,w)=\max \set{\ip{p,q'(t)}:p\in \Omega\cap T^{*}M_{q(t)}}r(w)-H_{x,0,t}(w)
  \end{equation*}
  is non-positive for $r(w)\ge r_{0}$. Thus $F$ attains a finite maximum $F_{\mathrm{max}}$ over the space of $t,w,x$, independently of $u$. Then $I\le F_{\mathrm{max}}$ holds independently of the solution $(x,u)$, completing the proof.
\end{proof}

\subsubsection{Definition of $\Theta$}
\label{sec:definition-theta}

Let $(H_{t},J)$ be admissible for defining $\mathrm{CF}(H_{t},J)$, let $A:P\times \R/\Z\to M$ be a family of loops in $Z(\Lambda_{c})$, and let $(\mathfrak{a},J)$ be $\Theta$-data. We assume throughout this subsection that the slope of $H_{t}$ is at least $c$.

For a generic perturbation term $\mathfrak{p}$ on $P\times \Sigma\times W$, consider the moduli space $\mathscr{M}$ of solutions to \S\ref{sec:floers-equation-general}, satisfying the boundary conditions \eqref{eq:theta-boundary-equations}. Let $\mathscr{M}_{0}$ be the rigid component of $\mathscr{M}$.

We define:
\begin{equation*}
  \Theta(A):=\sum_{\gamma}\#\set{(x,u)\in \mathscr{M}_{0}:u\text{ is asymptotic to }\gamma}\gamma.
\end{equation*}
By the above a priori energy estimate, this sum is finite. The key structural result is:
\begin{lemma}\label{lemma:theta-structural}
  For each $A\in Z(\Lambda_{c})$, $\Theta(A)$ is a cycle in the Floer complex. Moreover, the homology class of $\Theta(A)$ in $\mathrm{HF}(H_{t},J)$ is independent of the choice of $\Theta$-data, perturbation one-form $\mathfrak{p}$, and is independent of the bordism class of $A$. The resulting map:
  \begin{equation*}
    \Theta:H(\Lambda_{c})\to \mathrm{HF}(H_{t},J)
  \end{equation*}
  is a map of $\Z/2\Z$-vector spaces.
\end{lemma}
\begin{proof}
  To see that $\Theta(A)$ is a cycle, one considers the ends of the 1-dimensional component $\mathscr{M}_{1}\subset \mathscr{M}$; one shows that the number of ends of the component of $\mathscr{M}_{1}$ asymptotic to $\gamma$ equals the coefficient in front of $\gamma$ appearing in the composition $d(\Theta(A))$. Thus $d(\Theta(A))=0$ as we count modulo two.

  To see that the chain homotopy class of $\Theta(A)$ is independent of the $\Theta$-data, or perturbation one-form $\mathfrak{p}$, one appeals to the usual parametric moduli space argument; see, e.g., the discussion at the end of \S\ref{sec:cont-maps-floer}.

  A bit care is needed to prove the chain homotopy class of $\Theta(A)$ is independent of the bordism class of $A$. Nonetheless, the argument is still based on a parametric moduli space. Given a cobordism $C:Q\times \R/\Z\to M$ between $A$ and $A'$, one can consider $\Theta$-data for $C$; this entails a connection one-form $\mathfrak{a}$ and almost complex structure on $Q\times \Sigma\times W$ satisfying properties \ref{theta-data-1} through \ref{theta-data-5}, the only difference being that $Q$ is a manifold with boundary rather than a closed manifold. Invoking another generic perturbation term $\mathfrak{p}$, this data leads to a parametric moduli space $\mathscr{M}^{\mathrm{param}}$ of solutions $(y,u)$ where $y\in Q$. We note that $\Theta$-data for $C$ exists since we assume the cobordism happens within $Z(\Lambda_{c})$, and the slope of $H_{t}$ is at least $c$, by assumption.

  For generic perturbation term, the $1$-dimensional component $\mathscr{M}_{1}^{\mathrm{param}}$ has two boundary components:
  \begin{equation*}
    \bd\mathscr{M}_{1}^{\mathrm{param}}=\set{(y,u):y\in P\text{ or }y\in P'},
  \end{equation*}
  where $P\sqcup P'=\bd Q$. These solutions on the boundary are rigid when $y$ is restricted to variations tangent to $\bd Q$. A moment's reflection therefore reveals that:
  \begin{equation*}
    \Theta(A)-\Theta(A')=\sum \#\set{(y,u)\in \bd\mathscr{M}^{\mathrm{param}}_{1}:u\text{ is asymptotic to $\gamma$}}\gamma.
  \end{equation*}
  If $\mathscr{M}_{1}^{\mathrm{param}}$ were compact, then we would be done, as the number of boundary points of a compact manifold with boundary is even. However, $\mathscr{M}_{1}^{\mathrm{param}}$ can have non-compact components. Nonetheless, the non-compact ends of $\mathscr{M}_{1}^{\mathrm{param}}$ can be understood via Floer breaking/gluing; the usual theory shows that:
  \begin{equation*}
    d(\sum \#\set{(y,u)\in \mathscr{M}_{0}^{\mathrm{param}}:u\text{ is asymptotic to $\gamma$}}\gamma)=dK=\sum N_{\gamma'}\gamma'
  \end{equation*}
  where $N_{\gamma'}$ is the number of non-compact ends of the component of $\mathscr{M}_{1}^{\mathrm{param}}$ consisting of those $(y,u)$ where $u$ is asymptotic to $\gamma'$. Thus one proves:
  \begin{equation*}
    \Theta(A)-\Theta(A')=dK,
  \end{equation*}
  as desired.

  Finally, to see that $\Theta$ is a vector space map, it suffices to prove that $\Theta$ respects addition. Since addition in $H(\Lambda_{c})$ is given by disjoint union, and we can pick the $\Theta$-data independently for each component of the parameter space $P$, it follows easily that $\Theta(A+A')=\Theta(A)+\Theta(A')$.
\end{proof}

\subsubsection{Compatibility with continuation maps}
\label{sec:comp-with-cont}

To prove that $\Theta$ induces a map of persistence modules $H(\Lambda_{c})\to V_{c}$, it is necessary to show that $\Theta$ is compatible with continuation maps between the Floer cohomologies.

\begin{lemma}\label{lemma:theta-continuation}
  If $\mathfrak{c}:\mathrm{HF}(H_{0,t},J_{0})\to \mathrm{HF}(H_{1,t},J_{1})$ is a continuation map, as in \S\ref{sec:cont-maps-floer}, and the slope of $(H_{0,t},J_{0})$ at at least $c$, then:
  \begin{equation*}
    \mathfrak{c}\circ \Theta_{0}=\Theta_{1}
  \end{equation*}
  as maps $H(\Lambda_{c})\to \mathrm{HF}(H_{1,t},J_{1})$, where $\Theta_{i}:H(\Lambda_{c})\to \mathrm{HF}(H_{i,t},J_{i})$ is the $\Theta$-map constructed in \S\ref{sec:definition-theta}.
\end{lemma}
\begin{proof}
  This follows easily from the construction of $\mathfrak{c}$ and $\Theta_{i}$, and is left to the reader (see also the arguments in \cite{abbondandolo-schwarz-GT-2010,abouzaid-EMS-2015}).
\end{proof}

\subsection{BV-operators}
\label{sec:bv-operators}

This subsection is concerned with showing that $\Theta$ is compatible with the BV-operator on $H(\Lambda_{c})$ constructed in \S\ref{sec:bv-operator} and the BV-operator on $\mathrm{HF}(H_{t},J)$ constructed in \S\ref{sec:bv-operator-on-HF}. This is part of the content of Theorem \ref{theorem:main-comparison}.

\begin{lemma}
  If $(H_{t},J)$ is admissible for defining the Floer complex, and has slope at least $c$, then:
  \begin{equation*}
    \Delta\circ \Theta=\Theta\circ \Delta,
  \end{equation*}
  as maps $H(\Lambda_{c})\to \mathrm{HF}(H_{t},J)$.
\end{lemma}
\begin{proof}
  The argument has no surprises, and follows similar lines to the arguments in \cite{abbondandolo-schwarz-GT-2010,abouzaid-EMS-2015}. We leave the details to the reader.
\end{proof}

\subsection{Inclusion of the constant loops and PSS}
\label{sec:incl-const-loops}

In this subsection, we argue that $\Theta$ intertwines the inclusion of constant loops map $\mathfrak{i}:H^{*}(W)\to H(\Lambda_{c})$ with $\mathrm{PSS}:H^{*}(W)\to \mathrm{HF}(H_{t},J)$. We assume that $H_{t},J$ has slope at least $c>0$.

\begin{lemma}
  Suppose that $\beta\in H^{*}(W)$. Then $\mathrm{PSS}(\beta)=\Theta(\mathfrak{i}(\beta))$ as homology classes in $\mathrm{HF}(H_{t},J)$, where $\mathfrak{i}(\beta)$ is considered as an element of $H(\Lambda_{c})$.
\end{lemma}
\begin{proof}
  The first step in the proof is to appeal to the proof of Lemma \ref{lemma:isomorphism_constant}; there it is shown that $\beta$ is represented by the map $C:S\to T^{*}M$ in the fiber product diagram:
  \begin{equation*}
    \begin{tikzcd}
      {S}\arrow[d,"{}"]\arrow[r,"{C}"] &{T^{*}M}\arrow[d,"{}"]\\
      {P}\arrow[r,"{f}"] &{M},
    \end{tikzcd}
  \end{equation*}
  where $f:P\to M$ is a smooth map and $P$ is a compact manifold. One can think of $C$ as the collection of cotangent fibers living over the map $f$.

  As in Lemma \ref{lemma:isomorphism_constant}, it holds that $\mathfrak{i}(\beta)$ is represented by:
  \begin{equation*}
    A:P\times \R/\Z\to M\text{ given by }A(x,t)=f(x),
  \end{equation*}
  i.e., $\mathfrak{i}(\beta)$ is represented by a cycle of constant loops.

  Let $(\mathfrak{a}^{\mathrm{PSS}},J)$ be PSS-data, so it is defined on $\R\times \R/\Z\times W$. Let us suppose that $\mathfrak{a}_{s,t}^{\mathrm{PSS}}=0$ for $s\ge s_{0}$. For each $R>s_{0}$, define:
  \begin{equation*}
    \mathfrak{a}_{R,x,s,t}=\mathfrak{a}_{s+R,t}^{\mathrm{PSS}};
  \end{equation*}
  note that $\mathfrak{a}_{R,x,s,t}$ is $\Theta$-data for $A(x,t)=f(x)$, since $\mathfrak{a}_{R,x,0,t}=0$, for each $R$.

  One can think of $\mathfrak{a}_{R,x,s,t}$ as a connection one-form on the family: $$(s_{0},\infty)\times P\times (-\infty,0]\times \R/\Z,$$
  and one can introduce a compatible perturbation one-form $\mathfrak{p}_{R,x,s,t}$ supported in the region where $s\le -R$, which is sufficient to ensure transversality.

  This leads to the parametric moduli space $\mathscr{M}$ of solutions $(R,x,u)$, satisfying the boundary conditions:
  \begin{equation*}
    u(0,t)\in T^{*}M_{f(x)},
  \end{equation*}
  which admits a smooth map $(R,x,u)\mapsto R\in (s_{0},\infty)$. Let us denote the fiber over $R_{0}\in (s_{0},\infty)$ by $\mathscr{M}(R_{0})$. Since $\mathfrak{a}_{R_0,x,s,t}$ is $\Theta$-data for each $R_{0}$, the count of elements rigid elements of $\mathscr{M}(R_{0})$ represents $\Theta(\mathfrak{i}(\beta))$.

  The rigid elements of $\mathscr{M}(R_{0})$ live in the one-dimensional component of $\mathscr{M}$. Let us note that $\mathscr{M}$ has three kinds of non-compact ends:
  \begin{enumerate}
  \item ends containing sequences $(R_{n},x_{n},u_{n})$ where $R_{n}\to s_{0}$; we will ignore these ends,
  \item ends containing sequences $(R_{n},x_{n},u_{n})$ where $R_{n}$ converges in $(s_{0},\infty)$; as is well-understood, these ends converge to configurations contributing to $dK$ where $K$ counts the rigid elements of $\mathscr{M}$; we will also ignore these ends,
  \item\label{i-theta-PSS-3} ends containing sequences $(R_{n},x_{n},u_{n})$ where $R_{n}$ converges to $\infty$.
  \end{enumerate}
  A bit of thought shows that, the count of ends of type \ref{i-theta-PSS-3}, where the left asymptotic of $u_{n}$ is $\gamma$, defines a coefficient $N_{\gamma}$, and $\sum N_{\gamma}\gamma$ represents $\Theta(\mathfrak{i}(\beta))$; the argument is the same as the one given in \cite[\S5.3]{cant-arXiv-2024} --- briefly, the count of rigid elements $\mathscr{M}(R_{0})$, for $R_{0}$ large enough, equals the count of ends of type \ref{i-theta-PSS-3}.

  Thus it remains to show the count of ends of type \ref{i-theta-PSS-3} represents $\mathrm{PSS}(\beta)$. To analyze this, consider the change of coordinates $w(s,t)=u(s-R,t)$, so that $w$ is defined on $(-\infty,R]\times \R/\Z$. Then, along any sequence $(R_{n},x_{n},u_{n})$, the shifted map $(x_{n},w_{n})$ has a subsequence which converges to a solution $(x,w)$ of the PSS-equation, where $\lim_{s\to\infty}w(s,t)=C(x)$. For this step to work properly, we should assume that $\mathfrak{p}_{R,x,s-R,t}$ converges as $R\to \infty$ to a limiting perturbation term $\mathfrak{p}_{x,s,t}$ compatible with the PSS equation.

  On the other hand, for any solution $(x,w)$ of the PSS equation, a standard gluing argument for solutions to Floer's equation with Lagrangian boundary conditions shows that each $(x,w)$ arises as the limit of an end of type \ref{i-theta-PSS-3} in the above sense. Thus we conclude that $\mathrm{PSS}(\beta)=\Theta(\mathfrak{i}(\beta))$, as desired.
\end{proof}

\subsection{Product structures}
\label{sec:product-structures}

In this final subsection, we explain why $\Theta$ is compatible with the product structures. Such a result was originally proved in \cite{abbondandolo-schwarz-GT-2010}, and is a cornerstone in the relationship between string topology and Floer cohomology.

\emph{Remark}. Let us comment that there is another way in the literature which relates string topology with Floer cohomology, where one defines a map from $\mathrm{HF}(H_{t},J)$ to a suitable Morse homology of the free loop space. In this context, \cite[pp.\,398]{abouzaid-EMS-2015} proves the product structures are identified, and the argument is simpler than the argument in \cite{abbondandolo-schwarz-GT-2010} (and the argument we will explain in this section). That this direction of morphism provides simpler argument for the identification of ring structures was observed also in \cite[pp.\,500]{abbondandolo-schwarz-proceedings-2012}. We should note that \cite[\S5.4]{cieliebak-hingston-oancea-JFPTA-2023} show that the map from Floer cohomology to Morse homology considered by \cite{abouzaid-EMS-2015} does respect filtrations (the argument relies on special choice of $(H_{t},J)$ and a novel action versus length estimate); moreover this argument can be applied to a suitable construction of the pair-of-pants product to show the product structures are preserved when the map is restricted to each filtration level.

Unfortunately, this direction of the morphism (going from Floer cohomology to string topology) does not seem to work well for our purposes. It does seem likely that, if we were to restrict to domains $\Omega$ which appear as the unit disk bundle associated to a Riemannian metric, then we could argue instead using Morse homology and using existing results in the literature, and reverse the direction of the $\Theta$ morphism, and obtain a proof (albeit a slightly convoluted one) that the product structures are identified.\footnote{Note that the length function $\ell_{\Omega}$ is the same as the length function associated to the fiberwise convex hull of $\Omega$. In particular, there is the potential of a Morse theoretical argument, using the energy functional for (irreversible) Finsler metrics.} We prefer to stick with the present direction of the $\Theta$ morphism, and give a direct proof.

\subsubsection{Set-up}
\label{sec:set-up}

Throughout, we fix classes $\alpha_{i}\in H(\Lambda_{c_{i}})$, $i=0,1$.

First, pick two Hamiltonian system $H_{i,t}$, $i=0,1$, defined on $[0,1]$, and an almost complex structure $J$ so that:
\begin{enumerate}
\item $H_{i,t}=c_{i}(t)r$ holds outside of $r\ge r_{0}$, 
\item $H_{i,t}=0$ unless $\abs{t-1/2}<1/4$, 
\item $(H_{i,t},J)$ are admissible for defining $\mathrm{CF}$, when $H_{i,t}$ is extended to $\R/\Z$ by $1$-periodicity,
\item the slope of $H_{i,t}$ is at least $c_{i}$.
\end{enumerate}
Define:
\begin{equation*}
  H_{\infty,t}:=\left\{
    \begin{aligned}
      &2H_{0,2t}\text{ for }t\in [0,1/2],\\
      &2H_{1,2t-1}\text{ for }t\in [1/2,1],
    \end{aligned}
  \right.
\end{equation*}
and extend $H_{\infty,t}$ to $\R/\Z$ by $1$-periodicity. Finally, suppose that:
\begin{enumerate}[resume]
\item $(H_{\infty,t},J)$ is admissible for defining $\mathrm{CF}$.
\end{enumerate}

Since the slope of $H_{\infty,t}$ is the sum of the slopes of $H_{0,t}$ and $H_{1,t}$, there is a pair-of-pants map:
\begin{equation*}
  \ast:\mathrm{HF}(H_{0,t},J)\otimes \mathrm{HF}(H_{1,t},J)\to \mathrm{HF}(H_{\infty,t},J).
\end{equation*}

Second, pick representatives $A_{i}:P\times \R/\Z\to M$ for $\alpha_{i}$ so that:
\begin{enumerate}
\item\label{product-theta-1} $x\in P_{i}\mapsto A_{i}(x,0)$, $i=0,1,$ are mutually transverse,
\item\label{product-theta-2} $\max\set{\ip{p,q'(t)}:p\in \Omega\cap T^{*}M_{q(t)}}\le c_{i}(t)$ for each $q(t)=A_{i}(x,t)$.
\end{enumerate}
the second condition can achieved by a $x$-dependent smooth family of time-reparametrizations, since the integral of $c_{i}(t)$ is strictly larger than $\ell_{\Omega}(q(t))$. As a consequence, it follows that $A_{i}(x,0)=A_{i}(x,t)$ unless $\abs{t-1/2}<1/4$.

Define $P_{\infty}$ to be the fiber product of the transverse maps in \ref{product-theta-1}, and define:
\begin{equation*}
  A_{\infty}:P_{\infty}\times \R/\Z \to M \text{ given by }A_{\infty}((x_{0},x_{1}),t)=\left\{
    \begin{aligned}
      &A_{0}(x_{0},2t)\text{ for }t\in [0,1/2],\\
      &A_{1}(x_{1},2t)\text{ for }t\in [1/2,1],
    \end{aligned}
  \right.
\end{equation*}
which we extend from $[0,1]$ to all of $\R$ by $1$-periodicity. Then $A_{\infty}\in Z(\Lambda_{c_{0}+c_{1}})$, and $[A_{\infty}]=\alpha_{0}\ast \alpha_{1}$.

The goal in this section is to prove:
\begin{equation}\label{eq:goal-product-compatibility}
  \Theta(A_{\infty})=\Theta(A_{0})\ast \Theta(A_{1}),
\end{equation}
as cycles in $\mathrm{HF}(H_{\infty,t},J)$.

\subsubsection{An auxiliary map}
\label{sec:an-auxiliary-map}

Before explaining why \eqref{eq:goal-product-compatibility} holds, we will define an auxiliary cycle in $\mathrm{CF}(H_{\infty,t},J)$ which we will show represents both $\Theta(A_{\infty})$ and $\Theta(A_{0})\ast \Theta(A_{1})$.

Define a connection one-form $\mathfrak{a}$ on $\C\setminus \set{0}$ by the formula:
\begin{equation*}
  \mathfrak{a}_{s,t}=H_{\infty,t}\d t
\end{equation*}
in cylindrical coordinates $z=e^{-2\pi (s+it)}$. Note that the connection one-form $\mathfrak{a}$ vanishes outside of the segments contained between rays $\R_{+} ie^{\pm \pi i/4}$ (where its values are determined by $H_{1,t}$) and $\R_{-}ie^{\pm \pi i/4}$ (where its values are determined by $H_{0,t}$), as shown in Figure \ref{fig:illustration-rays}.

\begin{figure}[h]
  \centering
  \begin{tikzpicture}[scale=.6]
    \path[fill=black!10!white] (0,0) -- (45:2) arc (45:135:2) --cycle;
    \path[fill=black!10!white] (0,0) -- (-45:2) arc (-45:-135:2) --cycle;
    \path (90:1.1) node {$H_{1,t}$}--(90:-1.1)node{$H_{0,t}$};
    \draw (45:-2)--(45:2) (-45:-2)--(-45:2);
  \end{tikzpicture}
  \caption{Illustration of the connection one-form $\mathfrak{a}$ on the domain $\C\setminus \set{0}$.}
  \label{fig:illustration-rays}
\end{figure}
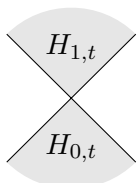

Let us say that two smoothly embedded closed disks $D_{0},D_{1}$ are \emph{shadowing} provided:
\begin{enumerate}
\item $D_{0}\subset \set{z=x+iy:y<0}$ and $D_{1}\subset \set{z=x+iy:y>0}$,
\item\label{shadowing-disk-2} in coordinates $z=e^{-2\pi (s+it)}$, each ray with fixed $t$ coordinate in $[1/8,3/8]$ intersects $\bd D_{0}$ in points $s^{0}_{-}(t)<s^{0}_{+}(t)$, where $s_{\pm}$ are smooth functions on $[1/8,3/8]$, and the radial vector $\bd_{s}$ is transverse to $\bd D_{0}$ when $t\in [1/8,3/8]$
\item\label{shadowing-disk-3} in coordinates $z=e^{-2\pi (s+it)}$, each ray with fixed $t$ coordinate in $[5/8,7/8]$ intersects $\bd D_{1}$ in points $s^{1}_{-}(t)<s^{1}_{+}(t)$, where $s_{\pm}$ are smooth functions on $[5/8,7/8]$, and the radial vector $\bd_{s}$ is transverse to $\bd D_{1}$ when $t\in [5/8,7/8]$
\end{enumerate}
A \emph{homotopy} of shadowing disks is an isotopy of embedded disks $D_{0,\tau},D_{1,\tau}$ which satisfy the above properties for all $\tau$.

The conditions admittedly look a bit strange, but it is perhaps best understood by comparing Figure \ref{fig:shadowing-disks} with Figure \ref{fig:illustration-rays}.

\begin{figure}[h]
  \centering
  \begin{tikzpicture}[scale=.7]
    \begin{scope}
      \path[fill=black!10!white] (0,0) -- (45:2) arc (45:135:2) --cycle;
      \path[fill=black!10!white] (0,0) -- (-45:2) arc (-45:-135:2) --cycle;
    \end{scope}
    \draw (45:-2)--(45:2) (-45:-2)--(-45:2);
    \path[fill=white] (1,0.5) arc (0:180:1) -- (-1,-0.5) arc (180:360:1)--cycle;
    \draw (1,0.5)to[out=-90,in=-90,looseness=0.3](-1,0.5) arc (180:0:1);
    \node at (0,.8) {$D_{1}$};
    \node at (0,-.8) {$D_{0}$};
    \draw[dashed] (0,0) --(45:1.3);
    \draw[scale=-1] (1,0.5)to[out=-90,in=-90,looseness=0.3](-1,0.5) arc (180:0:1);
    \node[fill, inner sep=1pt, circle] at (0,0) {};
  \end{tikzpicture}
  \caption{Shadowing disks block the rays from hitting the origin.}
  \label{fig:shadowing-disks}
\end{figure}
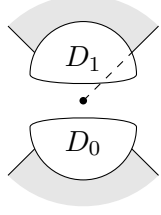

Given shadowing disks $D=(D_{0},D_{1})$, define a connection one-form $\mathfrak{a}^{D}$ on $$\Sigma^{D}:=\C\setminus \mathrm{Interior}(D_{0}\cup D_{1})$$ by the formula:
\begin{equation*}
  \mathfrak{a}_{z}^{D}=\left\{
    \begin{aligned}
      &0\text{ if the ray joining $0$ and $z$ is misses the interior of $D_{0}\cup D_{1}$},\\
      &\mathfrak{a}_{z}\text{ otherwise};
    \end{aligned}
  \right.
\end{equation*}
referring to Figure \ref{fig:shadowing-disks}, the connection one-form $\mathfrak{a}^{D}$ is supported in the shaded region of $\Sigma^{D}$. It is straightforward to show that $\mathfrak{a}^{D}$ is smooth on $\Sigma^{D}$.

Shadowing disks $D=(D_{0},D_{1})$ also determine moving Lagrangian boundary conditions for $\bd \Sigma^{D}$. Referring to the notation $s_{\pm}^{i}(t)$ in \ref{shadowing-disk-2} and \ref{shadowing-disk-3}, we define Lagrangians $L_{x_{0},x_{1},z}$ on $P_{0}\times P_{1}\times \bd \Sigma^{D}$ by the rule:
\begin{equation*}
  L^{D}_{x_{0},x_{1},z}:=\left\{
    \begin{aligned}
      &T^{*}M_{A_{0}(x_{0},2t)}\text{ if }z=e^{-2\pi (s_{-}^{0}(t)+it)}\text{ for }t\in [1/8,3/8],\\
      &T^{*}M_{A_{1}(x_{1},2t)}\text{ if }z=e^{-2\pi (s_{-}^{1}(t)+it)}\text{ for }t\in [5/8,7/8],\\
      &T^{*}M_{A_{0}(x_{0},0)}\text{ if }z\in \bd D_{0}\text{ and is not captured by previous rule},\\
      &T^{*}M_{A_{1}(x_{1},0)}\text{ if }z\in \bd D_{1}\text{ and is not captured by previous rule}.
    \end{aligned}
  \right.
\end{equation*}
In words, the Lagrangian boundary conditions are determined by the two loops $A_{i}(x_{i},-)$, $i=0,1$, which are traversed along the ``outer segment'' determined by $s=s^{i}_{-}(t)$, and are parametrized using the angular coordinate associated to $z=e^{-2\pi (s+it)}$. Outside of the outer segments, the Lagrangian boundary conditions remain fixed at $T^{*}M_{A_{i}(x_{i},0)}$, the fiber over the basepoint of the loop.

This set-up leads to data $(\mathfrak{a}^{D},L^{D})$ on the family $P_{0}\times P_{1}\times \Sigma^{D}\times W$. For an almost complex structure $J$ on the same family, so that $J_{x_{0},x_{1},z}=J$ is fixed outside of a sufficiently large disk, and a perturbation term $\mathfrak{p}$ on the same family, there is an associated moduli space $\mathscr{M}(\mathfrak{a}^{D},\mathfrak{p},J,L^{D})$ of finite energy solutions $(x_{0},x_{1},u)$ satisfying boundary conditions:
\begin{equation*}
  u(z)\in L_{x_{0},x_{1},z}\text{ for }z\in \bd \Sigma^{D},
\end{equation*}
and which solve the general form of Floer's equation described in \S\ref{sec:floers-equation-general}.

\begin{lemma}
  There is an a priori energy bound on solutions: $$(x_{0},x_{1},u)\in \mathscr{M}(\mathfrak{a}^{D},\mathfrak{p},J,L^{D})$$ which is independent of $D$ and $\mathfrak{p}$, provided $\mathfrak{p}$ is sufficiently small.
\end{lemma}
\begin{proof}
  The argument is similar to that used for Lemma \ref{lemma:energy-estimate-theta}, in that it suffices to bound the integrals:
  \begin{equation*}
    I=\int_{\bd \Sigma^{D}}u^{*}\lambda-v^{*}\mathfrak{a}^{D}
  \end{equation*}
  independently of the solution; here $v(z)=(z,u(z))$ is as in Lemma \ref{lemma:energy-estimate}. The integrand vanishes on points $z\in \bd D$ which do not lie in the outer segments $s=s_{-}^{i}(t)$ described above; thus it suffices to bound the integral over the outer segments. We parametrize the outer segments by the angular coordinate $t$ in $z_{0}(t)=e^{-2\pi (s_{-}^{0}(t)+it)}$ and $z_{1}(t)=e^{-2\pi (s_{+}^{1}(t)+it)}$. The integral decomposes as $I=I_{0}+I_{1}$ where:
  \begin{equation*}
    I_{0}=\int_{1/8}^{3/8}\ip{p(t),q'(t)}-2H_{0,2t}(u(z(t)))d t,\text{ where }q(t)=A_{0}(x_{0},2t).
  \end{equation*}
  and similarly for $I_{1}$, with $[1/8,3/8]$ replaced by $[5/8,7/8]$. Reparametrizing the integral by $\tau=2t$ for $i=0$ and $\tau=2t-1$ for $i=1$, so $\tau\in [1/4,3/4]$, and using the same estimate as in Lemma \ref{lemma:energy-estimate-theta} together with property \ref{product-theta-2} in the choice of $A_{i}(x_{i},t)$, one concludes that each $I_{i}$ is uniformly bounded independently of the solution. This completes the proof.
\end{proof}

Counting the rigid finite energy solutions $(x_{0},x_{1},u)$ in $\mathscr{M}(\mathfrak{a}^{D},\mathfrak{p},J,L^{D})$ for generic $\mathfrak{p}$, where the asymptotic orbit of $u$ equals $\gamma$ (at the negative end, as seen in the coordinates $z=e^{-2\pi(s+it)}$) gives a coefficient $N_{\gamma}$. Define:
\begin{equation*}
  \Pi(D):=\sum N_{\gamma}\gamma\in \mathrm{CF}(H_{\infty,t},J).
\end{equation*}
Then:
\begin{lemma}
  The chain $\Pi(D)$ is a cycle for each choice $\mathfrak{p},D,J$, provided $\mathfrak{p}$ is generic. The homology class of $\Pi(D)$ is independent of $\mathfrak{p},J$, and is also independent of the homotopy class of $D$ in the space of shadowing disks.
\end{lemma}
\begin{proof}
  The argument is similar to many other arguments in this paper, in particular, Lemma \ref{lemma:theta-structural}. We omit the proof.
\end{proof}

To eliminate the apparent dependency on the homotopy class of shadowing disks, we select a distinguished homotopy class. Let us fix $D^{\mathrm{std}}=(D_{0}^{\mathrm{std}},D_{1}^{\mathrm{std}})$ to be \emph{standard circular shadowing disks}, defined to be:
\begin{equation*}
  D_{0}^{\mathrm{std}}=-i+D(r)\text{ and }D_{1}^{\mathrm{std}}=i+D(r)
\end{equation*}
where $D(r)$ is the disk of radius $r\in (2^{-1/2},1)$; see Figure \ref{fig:standard-shadow}. Elementary geometry proves these are indeed shadowing disks, and the homotopy class is independent of the choice of $r$.

\begin{figure}[h]
  \centering
  \begin{tikzpicture}[scale=.7]
    \begin{scope}
      \path[fill=black!10!white] (0,0) -- (45:2) arc (45:135:2) --cycle;
      \path[fill=black!10!white] (0,0) -- (-45:2) arc (-45:-135:2) --cycle;
    \end{scope}
    \draw (45:-2)--(45:2) (-45:-2)--(-45:2);
    \path[fill=white] ({1/sqrt(2)},1) arc (0:180:{1/sqrt(2)}) -- (-{1/sqrt(2)},-1) arc (180:360:{1/sqrt(2)})--cycle;
    \path[fill=white] (0,1) circle (0.8) (0,-1) circle(0.8);
    \draw (0,1) circle (0.8) (0,-1) circle (0.8);

    \node[fill, inner sep=1pt, circle] at (0,0) {};
  \end{tikzpicture}
  \caption{Standard circular shadowing disks}
  \label{fig:standard-shadow}
\end{figure}
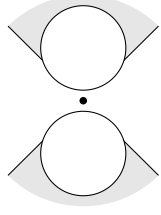

The two results whose proofs occupy the rest of the paper are:
\begin{lemma}\label{lemma:Pi-is-POP}
  If $D^{\mathrm{std}}$ comprises standard circular shadowing disks, then:
  \begin{equation*}
    \Pi(D^{\mathrm{std}})=\Theta(A_{0})\ast\Theta(A_{1}),
  \end{equation*}
  as elements of $\mathrm{HF}(H_{\infty,t},J)$.
\end{lemma}
\begin{lemma}\label{lemma:Pi-is-CS}
  If $D^{\mathrm{std}}$ comprises standard circular shadowing disks, then:
  \begin{equation*}
    \Pi(D^{\mathrm{std}})=\Theta(A_{\infty}),
  \end{equation*}
  as elements of $\mathrm{\mathrm{HF}}(H_{\infty,t},J)$.
\end{lemma}
Combining these, we conclude $\Theta(A_{\infty})=\Theta(A_{0})\ast \Theta(A_{1})$, completing the proof that $\Theta$ is compatible with the product structures (since $A_{\infty}$ represents the Chas-Sullivan product $\alpha_{0}\ast \alpha_{1}$). Together with the results in \S\ref{sec:bv-operators} and \S\ref{sec:incl-const-loops}, this completes the proof of Theorem \ref{theorem:main-comparison}.

Relatively speaking, Lemma \ref{lemma:Pi-is-POP} is easy and Lemma \ref{lemma:Pi-is-CS} is hard. We prove Lemma \ref{lemma:Pi-is-POP} in the next subsection \S\ref{sec:proof-lemma-pi-is-pop}, and the proof of Lemma \ref{lemma:Pi-is-CS} occupies the remaining parts of the paper.

\subsubsection{Proof of Lemma \ref{lemma:Pi-is-POP}}
\label{sec:proof-lemma-pi-is-pop}

The argument is fairly standard, and we only explain the main construction, omitting the Floer theory details.

For $r\in (2^{-1/2},1)$, let $D_{0}(r)=-i+D(r)$ and $D_{1}(r)=i+D(r)$, and let $D(r)=(D_{0}(r),D_{1}(r))$; these form standard circular shadowing disks.

Parametrize the punctured disk $D_{i}(r)^{\times}$ as a positive cylindrical end using cylindrical coordinates $s+it$. Then, $\mathfrak{a}^{D(r)}$ extends to a collar $s\in [0,\epsilon]$ for some small $\epsilon$, using the same formula (as slightly shrunken disks remain shadowing disks). Let us abbreviate:
\begin{equation*}
  \mathfrak{a}_{i}=\mathfrak{a}^{D_{i}(r)}|_{s\in [0,\epsilon]},
\end{equation*}
which takes the form:
\begin{equation*}
  \mathfrak{a}_{i}=K_{i,s,t}\d s+G_{i,s,t}\d t,
\end{equation*}
and a moment's thought reveals that $K_{i,s,t}=b_{i}(s,t)r$ and $G_{i,s,t}=c_{i}(s,t)r$ for $r\ge r_{0}$. Moreover, $b_{i}(s,t)\d s+c_{i}(s,t)\d t$ is a closed one-form, and:
\begin{equation*}
  \int_{0}^{1}c_{i}(s,t)\d t=\text{slope of $H_{i,t}$}.
\end{equation*}
Moreover, $\mathfrak{a}_{i}$ has zero curvature, since $\mathfrak{a}^{D_{i}(r)}$ has zero curvature. The goal is now to extend $\mathfrak{a}_{i}$ to all of $[0,\infty)\times \R/\Z$. Define:
\begin{equation*}
  \begin{aligned}
    G_{i,s,t}'&=\beta(1-s/\epsilon)G_{i,s,t}+\beta(s/\epsilon)H_{i,t}\\
    K_{i,s,t}'&=\int_{0}^{t}\bd_{s}G_{i,s,\tau}'d\tau-t\int_{0}^{1}\bd_{s}G_{i,s,\tau}'d\tau.
  \end{aligned}
\end{equation*}
Note that $G_{i,s,t}'=G_{i,s,t}$ in a neighborhood of $s=0$, by construction of the standard cut-off function $\beta$. Since $\mathfrak{a}_{i}$ has zero curvature, it holds that:
\begin{equation*}
  \bd_{s}G_{i,s,t}=\bd_{t}K_{i,s,t},
\end{equation*}
where we use the fact that $\mathfrak{a}_{i}(V_{1}),\mathfrak{a}_{i}(V_{2})$ are Poisson-commuting for any two tangent vectors $V_{1},V_{2}$ based at the same point (this holds since $\mathfrak{a}_{i}$ is obtained by restricting $\mathfrak{a}^{D(r)}$, which has this property). Thus it holds that $K_{i,s,t}'=K_{i,s,t}$ also holds in a neighborhood of $s=0$. Thus:
\begin{equation*}
  \mathfrak{a}_{i}'=K_{i,s,t}'\d s+G_{i,s,t}'\d t
\end{equation*}
is a valid extension of $\mathfrak{a}_{i}$ from $[0,\epsilon]\times \R/\Z$ to $[0,\infty)\times \R/\Z$, provided we shrink $\epsilon$. Moreover, as in \S\ref{sec:cont-maps-floer}, $\mathfrak{a}_{i}'$ has curvature bounded from above.

Denote by:
\begin{equation*}
  \mathfrak{a}'=\left\{
    \begin{aligned}
      &\mathfrak{a}_{0}'\text{ in }-i+D(r)^{\times},\\
      &\mathfrak{a}_{1}'\text{ in }i+D(r)^{\times},\\
      &\mathfrak{a}^{D(r)}\text{ elsewhere},
    \end{aligned}
  \right.
\end{equation*}
which is a smooth connection one-form on $\C\setminus \set{i,-i}$. By construction, $\mathfrak{a}'$ agrees with $H_{\infty,t}\d t$ on the cylindrical end around $z=\infty$, (since it agrees with $\mathfrak{a}^{D(r)}$), and agrees with $H_{0,t}\d t$, $H_{1,t}\d t$ in the cylindrical ends around the punctures $-i,i$. A bit of thought (and standard arguments) show that appropriately counting the rigid finite-energy solutions of the equation determined by $(\mathfrak{a}',J,\mathfrak{p})$, for some perturbation term $\mathfrak{p}$, defines a chain-level representation of the pair-of-pants operation:
\begin{equation*}
  \ast:\mathrm{CF}(H_{0,t},J)\otimes \mathrm{CF}(H_{1,t},J)\to\mathrm{CF}(H_{\infty,t},J).
\end{equation*}

Next we explain how to deform the equation defining $\Pi(D)$ to the equation defining $\ast\circ (\Theta(A_{0}),\Theta(A_{1}))$. The process is illustrated in Figure \ref{fig:Pi-is-POP}.

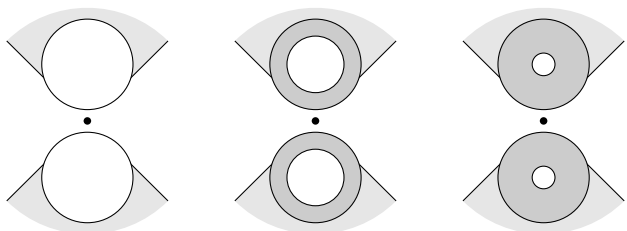
\begin{figure}[h]
  \centering
  \begin{tikzpicture}[scale=.75]
    \begin{scope}
      \path[fill=black!10!white] (0,0) -- (45:2) arc (45:135:2) --cycle;
      \path[fill=black!10!white] (0,0) -- (-45:2) arc (-45:-135:2) --cycle;
    \end{scope}
    \draw (45:-2)--(45:2) (-45:-2)--(-45:2);
    \path[fill=white] ({1/sqrt(2)},1) arc (0:180:{1/sqrt(2)}) -- (-{1/sqrt(2)},-1) arc (180:360:{1/sqrt(2)})--cycle;
    \path[fill=white] (0,1) circle (0.8) (0,-1) circle(0.8);
    \draw (0,1) circle (0.8) (0,-1) circle (0.8);

    \node[fill, inner sep=1pt, circle] at (0,0) {};
    \begin{scope}[shift={(4,0)}]
      \begin{scope}
        \path[fill=black!10!white] (0,0) -- (45:2) arc (45:135:2) --cycle;
        \path[fill=black!10!white] (0,0) -- (-45:2) arc (-45:-135:2) --cycle;
      \end{scope}
      \draw (45:-2)--(45:2) (-45:-2)--(-45:2);
      \path[fill=white] ({1/sqrt(2)},1) arc (0:180:{1/sqrt(2)}) -- (-{1/sqrt(2)},-1) arc (180:360:{1/sqrt(2)})--cycle;
      \path[fill=white] (0,1) circle (0.8) (0,-1) circle(0.8);
      \path[fill=black!20!white] (0,1) circle (0.8) (0,-1) circle(0.8);
      \path[fill=white] (0,1) circle (0.5) (0,-1) circle(0.5);
      \draw (0,1) circle (0.8) (0,-1) circle (0.8);
      \draw (0,1) circle (0.5) (0,-1) circle (0.5);

      \node[fill, inner sep=1pt, circle] at (0,0) {};
    \end{scope}
    \begin{scope}[shift={(8,0)}]
      \begin{scope}
        \path[fill=black!10!white] (0,0) -- (45:2) arc (45:135:2) --cycle;
        \path[fill=black!10!white] (0,0) -- (-45:2) arc (-45:-135:2) --cycle;
      \end{scope}
      \draw (45:-2)--(45:2) (-45:-2)--(-45:2);
      \path[fill=white] ({1/sqrt(2)},1) arc (0:180:{1/sqrt(2)}) -- (-{1/sqrt(2)},-1) arc (180:360:{1/sqrt(2)})--cycle;
      \path[fill=black!20!white] (0,1) circle (0.8) (0,-1) circle(0.8);
      \path[fill=white] (0,1) circle (0.2) (0,-1) circle(0.2);
      \draw (0,1) circle (0.8) (0,-1) circle (0.8);
      \draw (0,1) circle (0.2) (0,-1) circle (0.2);

      \node[fill, inner sep=1pt, circle] at (0,0) {};
    \end{scope}
  \end{tikzpicture}
  \caption{1-parameter family used to prove Lemma \ref{lemma:Pi-is-POP}; as the parameter $R$ increases, the boundary components of the surface contract onto the points $\pm i$. Solutions in the associated parametric moduli space converge to configurations representing $\Theta(A_{1})\ast \Theta(A_{2})$}
  \label{fig:Pi-is-POP}
\end{figure}

For each $R\in [0,\infty)$, define $\mathfrak{a}^{R}$ to be the restriction of $\mathfrak{a}'$ to the region obtained by removing from the disks $D_{0}(r)$ and $D_{1}(r)$ the cylindrical ends defined by $s>R$. Then $\mathfrak{a}^{0}$ agrees with $\mathfrak{a}^{D(r)}$, and $\mathfrak{a}^{R}$ ``converges'' to $\mathfrak{a}'$ as $R\to\infty$. Let us denote by $\Sigma(R)$ the surface with boundary obtained from this removal, so $\Sigma(0)=\Sigma^{D}$.

For each $R$ and $x_{i}\in P_{i}$, there are moving Lagrangian boundary conditions for $(\mathfrak{a}^{R},\Sigma(R))$ defined as follows: pick a reparametrization $\rho_{R,x_{i}}:\R/\Z\to \R/\Z$ so that:
\begin{equation*}
  \max\set{\ip{p,q'(t)}:p\in \Omega\cap T^{*}M_{q(t)}}<c_{i}(R,t)\text{ where }q(t)=A_{i}(x_{i},\rho_{R,x_{i}}(t)),
\end{equation*}
which induces boundary conditions by requiring that $u(R,t)\in T^{*}M_{q(t)}$ for each solution $(R,x_{0},x_{1},u)$. The functions $\rho_{R,x_{i}}$ should be chosen so that:
\begin{enumerate}
\item the boundary conditions for $R=0$ agree with the previously defined boundary conditions for $\mathfrak{a}^{D},\Sigma^{D}$,
\item $\rho_{R,x_{i}}(t)=t$ for $R$ sufficiently large.
\end{enumerate}
Because the set of reparametrizations achieving the above estimate is contractible (the condition is a convex condition $\bd_{t}\rho_{R,x_{i}}(t)$, since the slope of $H_{i,t}$ is large enough), it is possible to pick such $\rho_{R,x_{i}}$.

The data: $\mathfrak{a}^{R},\Sigma(R),J$, the above $(R,x_{0},x_{1})$-dependent moving boundary conditions, and a generic perturbation term $\mathfrak{p}$ on the family, leads to a moduli space $\mathscr{M}$ of solutions $(R,x_{0},x_{1},u)$. The 1-dimensional component of $\mathscr{M}$ maps to $[0,\infty)$ via $(R,x_{0},x_{1},u)\to R$, and generic fibers $\mathscr{M}(R_{0})$ defined by $R=R_{0}$ are zero-dimensional manifolds. By standard arguments, similar to those used in \S\ref{sec:inclusion-of-constants}, the count of:
\begin{enumerate}
\item\label{count-1-Pi-is-POP} elements in $\mathscr{M}(0)$ and,
\item\label{count-2-Pi-is-POP} non-compact ends of $\mathscr{M}$ containing sequences $(R_{n},x_{0,n},x_{1,n},u_{n})$ with $R_{n}\to\infty$,
\end{enumerate}
define homologous cycles in $\mathrm{CF}(H_{\infty,t},J)$. The count of \ref{count-1-Pi-is-POP} is exactly the count defining $\Pi(D)$, while the count of \ref{count-2-Pi-is-POP} represents $\Theta(A_{1})\ast \Theta(A_{2})$, as can be shown using standard Floer theory gluing arguments. We note that in this last part of the argument, one should pick the perturbation term $\mathfrak{p}$ so that it converges in an appropriate sense as $R\to\infty$. \hfill$\square$

\subsubsection{Geometric set-up for Lemma \ref{lemma:Pi-is-CS}}
\label{sec:geometric-set-up}

The main idea is to deform the shadowing disks to be rectangles with rounded corners, with a very thin neck between the disks, as shown on the left of Figure \ref{fig:Pi-is-CS-regions}. Roughly speaking, as the neck gets thinner, the solutions approximate those used to define $\Theta(A_{\infty})$ (in fact, the neck will degenerate to a flow line connecting two marked points on the boundary of the limiting curve). The presence of this flow line indicates that the gluing/compactness theory we will employ is similar to the adiabatic gluing/compactness theory of \cite{fukaya-oh-AJM-1997,ekholm-GT-2007,oh-zhu-JSG-2011,ekholm-etnyre-ng-sullivan-GT-2013,cant-chen-kyoto-2023}.

We now describe the surface and various regions inside the domain in more detail. The surfaces obtained by removing the approximately rectangular shadowing disks from $\C$ have a long neck region. As the space between the shadowing disks shrinks, the modulus of the neck grows. This neck region is shown in Figure \ref{fig:Pi-is-CS-regions}.

Define the domain $\Sigma(R)$, for $R\ge R_{0}$, by replacing the inner neck of modulus $2R_{0}$ with a neck of modulus $2R$. Note that $\Sigma(R)$ is defined abstractly and is not defined as a subset of $\C$. One can still speak about the various regions inside of $\Sigma(R)$; i.e., it still makes sense to speak about $\mathscr{R}_{1},\mathscr{R}_{2}$, and the necks.

In the limit $R\to\infty$, the surface $\Sigma(R)$ converges (in the moduli space of domains) to a surface $\Sigma(\infty)$ which is conformally equivalent to a disk with one interior puncture and two boundary punctures. The local model around the boundary punctures is illustrated in Figure \ref{fig:end-region}. The limit surface $\Sigma(\infty)$ has two ends (rather than one neck).

There is an obvious connection one-form $\mathfrak{a}$ on $\Sigma(R_{0})$, described by the shadowing disk construction. It has the property that $\mathfrak{a}$ vanishes on the neck region. When we define $\Sigma(R)$ by removing the inner neck and gluing in a longer neck, we obtain a connection one-form $\mathfrak{a}$ by requiring that it vanishes on the neck and agrees with the previously defined $\mathfrak{a}$ outside the neck.

The Lagrangian boundary conditions are defined for parameter $R_{0}$ by the shadowing disk construction. For $R\ge R_{0}$, one defines the Lagrangian boundary conditions to match those for $R_{0}$ away from the neck, and require solutions to map the boundary of the neck to cotangent fibers $T^{*}M_{A(x_{0},0)}$ and $T^{*}M_{A(x_{1},0)}$ (as in the boundary conditions for $R_{0}$).

\begin{figure}[h]
  \centering
  \begin{tikzpicture}[scale=0.5]
    \begin{scope}[shift={(-5,0.5)},scale=1.25]
      \begin{scope}
        \path[fill=black!10!white] (0,0) -- (45:3) arc (45:135:3) --cycle;
        \path[fill=black!10!white] (0,0) -- (-45:3) arc (-45:-135:3) --cycle;
      \end{scope}
      \draw (45:-3)--(45:3) (-45:-3)--(-45:3);

      \fill[white] (-1.5,-1) rectangle(1.5,1);
      \draw[rounded corners] (-1.5,.1)--(1.5,.1)--(1.5,1)--(-1.5,1)--cycle;
      \draw[rounded corners] (-1.5,-.1)--(1.5,-.1)--(1.5,-1)--(-1.5,-1)--cycle;
      \draw[dashed] (0,0) circle (2.5);
      \node at (0,1.5) {$\mathscr{R}_{1}$};
      \node at (0,-1.5) {$\mathscr{R}_{1}$};
      \node at (-3,0) {$\mathscr{R}_{2}$};

    \end{scope}
    \draw[rounded corners] (0,-1)--(0,0)--+(13,0)--+(13,-1) (0,2)--(0,1)--+(13,0)--+(13,1);
    \draw (0.5,0)--+(0,1) (12.5,0)--+(0,1) (11,0)--+(0,1) (2,0)--+(0,1);
    \draw[<->] (2,-0.5)--node[below]{$2R_{0}$}(10.5,-0.5);
    \draw[<->] (0.5,1.5)--node[above]{$\ell$}(2,1.5);
    \draw[<->] (12.5,1.5)--node[above]{$\ell$}(11,1.5);
  \end{tikzpicture}
  \caption{(Left) Shadowing disks defining the domain $\Sigma(R_{0})$; the region $\mathscr{R}_{1}$ is the region between the rays where the connection one-form $\mathfrak{a}_{R_{0}}$ is supported; the region $\mathscr{R}_{2}$ is the complement of a large disk centered on the origin. (Right) The neck region of $\Sigma(R_{0})$ is the neck with modulus $2R_{0}$; the \emph{total neck} is the neck with modulus $2\ell+2R_{0}$.}
  \label{fig:Pi-is-CS-regions}
\end{figure}
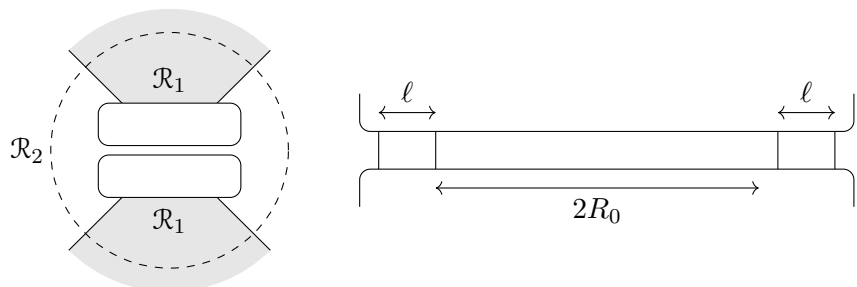

\begin{figure}[h]
  \centering
  \begin{tikzpicture}[scale=0.5]
    \draw[rounded corners] (0,-1)--(0,0)--+(10,0) (0,2)--(0,1)--+(10,0);
    \draw (0.5,0)--+(0,1) (2,0)--+(0,1);
    \draw[<->] (0.5,1.5)--node[above]{$\ell$}(2,1.5);
  \end{tikzpicture}
  \caption{The end region on the limiting disk around the left end; the picture around the right end is a reflected version.}
  \label{fig:end-region}
\end{figure}
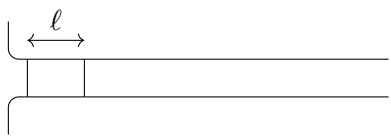

Let us pick perturbation one-forms $\mathfrak{p}_{R}$ depending on $R$ with the property that $\mathfrak{p}_{R}$ is supported in the region $\mathscr{R}_{1}\cap \mathscr{R}_{2}^{c}$; in particular, the perturbation term vanishes near the neck, and the equation appears as Floer's equation for $(H_{\infty,t},J)$ in the cylindrical end $\mathscr{R}_{2}$, using coordinates $z=e^{-2\pi (s+it)}$.

Once a complex structure is chosen (see \S\ref{sec:choice-almost-compl}), the construction then gives a parametric moduli space $\mathscr{M}$ of solutions $(R,x_{0},x_{1},u)$ where $(x_{0},x_{1},u)$ solves the equation for the domain $\Sigma(R)$. Let us denote by $\mathscr{M}$ the one-dimensional component, and by $\mathscr{M}_{0}$ the rigid component; we will not have occasion to consider higher dimensional components. We denote by $\mathscr{M}(R)$ the fiber of $\mathscr{M}$ of solutions $(x_{0},x_{1},u)$, so $\mathscr{M}(R)$ consists of rigid solutions for generic $R$.

As a special case, we define $\mathscr{M}(\infty)$ to be the moduli of rigid finite energy solutions on the limiting surface $\Sigma(\infty)$, where $(x_{0},x_{1})$ lies in the fiber product $P_{\infty}$; the finite energy condition implies any solution converges exponentially to a removable singularity lying on $T^{*}M_{A(x_{0},0)}=T^{*}M_{A(x_{1},0)}$ at the punctures.

We will require that $\mathfrak{p}_{R}$ is equal to $\mathfrak{p}_{\infty}$ for $R$ sufficiently large, where $\mathfrak{p}_{\infty}$ is a perturbation term used to achieve transversality for $\mathscr{M}(\infty)$. As part of the gluing analysis, we will show that such $\mathfrak{p}_{R}$ also achieve transversality for the parametric moduli space $\mathscr{M}$.

In \S\ref{sec:choice-ell} we will explain how to pick the parameter $\ell$; essentially it needs to be chosen large enough that the neck/ends are mapped into a small neighborhood of the cotangent fiber $T^{*}M_{A(x_{0},0)}$.

At this stage, we claim the following:
\begin{proposition}
  The count of elements $(x_{0},x_{1},u)\in \mathscr{M}(\infty)$, weighted by the asymptotic orbit at the cylindrical end, defines a cycle in $\mathrm{CF}(H_{\infty,t},J)$ which represents $\Theta(A_{\infty})$.
\end{proposition}
\begin{proof}
  The surface $\Sigma(\infty)$ is conformally equivalent to a closed disk $D(1)$ with two boundary punctures (say $-1,1$) and one interior puncture $0$. The equation and boundary conditions are what is used to define $\Theta$, with the sole exception that the cylindrical end at $0$ is not the standard cylindrical end around $0$ (the biholomorphism between $\Sigma(\infty)$ and $D(1)$ does not respect the cylindrical coordinates). However, the cycle one obtains is stable under small changes in the data, and so one can correct the cylindrical end very close to zero, without changing the resulting cycle. In this fashion, one proves the count obtained from $\mathscr{M}(\infty)$ is exactly the one used to define $\Theta(A_{\infty})$.
\end{proof}

Moreover, standard Floer theoretic arguments show:
\begin{proposition}
  The count of elements $(R_{0},x_{0},x_{1},u)\in \mathscr{M}(R_{0})$, weighted by the asymptotic orbit at the cylindrical end, and the count of non-compact ends of $\mathscr{M}$, also weighted by the asymptotic orbits, define homologous cycles in $\mathrm{CF}(H_{\infty,t},J)$. Both of these cycles represent $\Pi(D)$.
\end{proposition}
\begin{proof}
  The argument is standard and similar to other arguments in this paper. We only comment that the chain homotopy between the two cycles is given by counting the rigid elements in the parametric moduli space $\mathscr{M}_{0}$.
\end{proof}

Thus the rest of the proof is dedicated to proving the cycle obtained by counting solutions in $\mathscr{M}(\infty)$ also represents the cycle $\Pi(D)$, i.e., represents the cycle obtained by counting the non-compact ends of $\mathscr{M}$. This step of the proof is what involves the delicate gluing argument. The argument is quite technical, and we have written it with an expert audience in mind; we assume the reader is already familiar with gluing theory in Floer theory, and we focus on the details needed to convince an expert in the validity of the approach.

\subsubsection{Choice of almost complex structure}
\label{sec:choice-almost-compl}

To simplify the argument as much as possible, we will make a particular choice of almost complex structure. The construction uses Riemannian geometry. For a given Riemannian metric $g$ on $TM$ (and hence on $T^{*}M$), and a smooth function $f:[0,\infty)\to (0,\infty)$, we let $J_{g,f}$ be the almost complex structure so that:
\begin{equation*}
  J_{g,f}=\left[
    \begin{matrix}
      0&-f(\rho_{g})g_{*}\\
      (f(\rho_{g})g_{*})^{-1}&0
    \end{matrix}
  \right]\text{ with respect to splitting }\mathrm{Ver}\oplus \mathrm{Hor}_{g},
\end{equation*}
where $\rho_{g}(p)=g(p,p)$ is the fiberwise radius-squared function, and $g_{*}$ is the duality isomorphism between $TM$ and $T^{*}M$. The horizontal distribution $\mathrm{Hor}_{g}$ is defined using the Levi-Civita connection.

The key result about such almost complex structures is that:
\begin{lemma}\label{lemma:maximum-principle-levi-civita}
  Let $(\Sigma,\bd\Sigma)$ be a (not necessarily compact) Riemann surface with boundary. Let $f_{z}$, $z\in \Sigma$, be a family of functions $[0,\infty)\to (0,\infty)$, yielding a domain dependent family of almost complex structures $J_{g,f_{z}}$. For any $J_{g,f_{z}}$-holomorphic map $u:(\Sigma,\bd\Sigma)\to T^{*}M$, whose boundary is mapped onto cotangent fibers, the function $\rho_{g}(u)$ does not attain any local maxima.
\end{lemma}
\begin{proof}
  Write $u(s,t)=(p(s,t),q(s,t))$ where $q(s,t)$ is the projection to the base, and $p(s,t)$ is considered as section of $u^{*}T^{*}M$. The computations in \cite[\S2]{cant-chen-kyoto-2023} imply that:
  \begin{equation*}
    \nabla_{s}p=f_{z}(\rho_{g})g_{*}\bd_{t}q\text{ and }\nabla_{t}p=-f_{z}(\rho_{g})g_{*}\bd_{s}q.
  \end{equation*}
  A standard computation gives:
  \begin{equation*}
    \Delta \rho_{g}=\Delta g(p,p)\ge 2g(\nabla_{s}\nabla_{s}p,p)+2g(\nabla_{t}\nabla_{t}p,p),
  \end{equation*}
  and using the above holomorphic curve equation, and the well-known fact $\nabla_{s}(g_{*}\bd_{t}q)=\nabla_{t}(g_{*}\bd_{s}q)$ (see \cite[Lemma 2.4]{cant-chen-kyoto-2023}), we conclude:
  \begin{equation*}
    \Delta \rho_{g}\ge (\bd_{s}\rho_{g})f_{z}(\rho_{g})^{-1}\bd_{s}(f_{z}(\rho_{g}))-(\bd_{t}\rho_{g})f_{z}(\rho_{g})^{-1}\bd_{t}(f_{z}(\rho_{g})).
  \end{equation*}
  If $\bd\Sigma=\emptyset$, the desired result then follows from the maximum principle \cite[Chapter 3]{gilbarg-trudinger-1998}. In general, observe that $\d\rho_{g}$ vanishes on directions orthogonal to the boundary, and hence doubles to a $C^{2}$ function; the maximum principle can then be applied to this doubling.
\end{proof}

Having established this, we construct a family $J_{x_{0},z}$ of almost complex structures on the family $P_{0}\times \C$, as follows. Pick a family of Riemannian metrics $g_{x_{0}}$ so that $g_{x_{0}}$ is flat in a neighborhood of $A_{0}(x_{0},0)$. Then pick the almost complex structures $J_{x_{0},z}$ so that:
\begin{enumerate}
\item in $\mathscr{R}_{2}$, $J_{x_{0},z}=J$ is fixed,
\item inside of a disk slightly smaller than the one bounding $\mathscr{R}_{2}$, it holds that $J_{x_{0},z}=J_{g_{x_{0}},f_{z}}$,
\item $f_{z}$ is locally constant inside the total neck (where it equals $f_{0}$) and on a neighborhood of $\mathscr{R}_{1}$ (where it equals $f_{z_1}$ for some $z_1 \in \mathscr{R}_{1}$).
\item $f_{0}(\rho)=1$ holds for $\rho\le r_{1}$,
\item $f_{z}(\rho)=\rho^{1/2}$ for $\rho \ge r_{1}+1$ holds at all points $z$.
\end{enumerate}
Note that because $f_{z}(\rho)=\rho^{1/2}$ holds outside of a compact set, $J_{x_{0},z}$ is Liouville equivariant outside of a compact set.

The construction is admittedly a bit ad hoc, but the benefits of the construction is that:
\begin{lemma}
  There is a constant $C$ independent of $R_{0}$ and $f_{z}$ for $z\not\in \mathscr{R}_{1}$ (but depending on $f_{z_1}$) so that any finite energy solution $(x_{0},x_{1},u)$ in the moduli space $\mathscr{M}(R)$, where $R\in [R_{0},\infty]$, satisfies $\rho_{x_{0}}(u)\le C$. In other words, fixing $f_{z_1}$, one can alter the choice of $f_{z}$ for $z\not\in \mathscr{R}_{1}$ so the above conclusion holds.
\end{lemma}
\begin{proof}
  We consider the two regions $\mathscr{R}_{1},\mathscr{R}_{2}\subset \Sigma(R)$ illustrated in Figure \ref{fig:Pi-is-CS-regions}. By Lemma \ref{lemma:maximum-principle-levi-civita}, the supremum of $\rho_{x_{0}}(u)$ is equal to the supremum of $\rho_{x_{0}}(u)$ on the region $\mathscr{R}_{1}\cup \mathscr{R}_{2}$.

  By the earlier maximum principle Proposition \ref{proposition:maximum-principle}, the supremum of $\rho_{x_{0}}(u)$ over $\mathscr{R}_{2}$ depends only on $(H_{\infty,t},J)$, and the apriori energy bound which is independent of the choice of family of almost complex structures and $R$.

  Any point in $\mathscr{R}_{1}$ can be joined to $\mathscr{R}_{2}$ by a radial path. Standard bubbling analysis proves the derivatives along the path are bounded by a constant depending only on the Hamiltonian connection and $f_{z_1}$ (the radial path remains inside $\mathscr{R}_{1}$). Therefore the supremum of $\rho_{x_{0}}(u)$ over $\mathscr{R}_{1}$ is also bounded independently of $f_{z}$ for $z\not\in \mathscr{R}_{1}$, and $R_{0}$. This proves the statement.
\end{proof}

Therefore we can pick $r_{1}$ larger than $C$, because $C$ is independent of $r_{1}$.

The upshot of this construction is the following:

\begin{corollary}
  Any solution $(x_{0},x_{1},u)$ in $\mathscr{M}(R)$, $R\in [R_{0},\infty]$, is $J_{0}$-holomorphic on the intersection of the total neck and $u^{-1}(B(1))$ if $B(1)$ is a coordinate ball around $A(x_{0},0)$ where $g_{x_{0}}$ is flat; here $J_{0}$ is the standard almost complex structure in canonical coordinates $p,q$.
\end{corollary}
The standard almost complex structure is the one which satisfies $J_{0}\bd_{p_{i}}=\bd_{q_{i}}$, $i=1,\dots,n$. In other words, we can assume without loss of generality that $J_{x_{0},z}$ is standard on canonical coordinates on the coordinate ball $B(1)$, and if $z$ is in the total neck.

\subsubsection{On the choice of $\ell$}
\label{sec:choice-ell}

In this subsection, we prove that $\ell$ can be chosen large enough that certain properties hold for solutions in $\mathscr{M}(R)$ for $R\ge R_{0}$.
\begin{lemma}\label{lemma:choice-ell}
  For any $\epsilon>0$, the parameter $\ell$ can be chosen large enough that: for all solutions $(R,x_{0},x_{1},u)\in \mathscr{M}(R)$, $u$ maps the neck (or ends if $R=\infty$) into the preimage of $B(\epsilon)$, where $B(\epsilon)\subset B(1)$ is contained in the coordinate ball around $A(x_{0},0)$. Recall the neck/ends exclude the pieces with modulus $\ell$.
\end{lemma}
\begin{proof}
  This is a simple compactness argument, and is left to the reader.
\end{proof}

The parameter $\epsilon$ will be chosen as follows. For each of the finitely many points $(x_{0},x_{1})$ in $P_{\infty}\subset P_{0}\times P_{1}$ so that there is some solution $(x_{0},x_{1},u)\in \mathscr{M}(\infty)$, consider the non-linear map:
\begin{equation}\label{eq:short-exact-sequence-P0P1}
  (y_{0},y_{1})\in U\subset P_{0}\times P_{1}\mapsto A_{1}(y_{1},0)-A_{0}(y_{0},0)
\end{equation}
defined on the open neighborhood of points $(y_{0},y_{1})$ where $A_{0}(y_{0},0),A_{1}(y_{1},0)$ lie in the coordinate ball $B(1)$ (this difference vector is computed using the coordinate chart).

By the transversality of $A_{0}(-,0),A_{1}(-,0)$, this map is a submersion at $(x_{0},x_{1})$, and the kernel of the differential is $TP_{\infty,(x_{0},x_{1})}$. Pick a smooth open disk $V$ passing through $(x_{0},x_{1})$ so that $TV_{(x_{0},x_{1})}$ and $TP_{\infty,(x_{0},x_{1})}$ are complementary spaces. Then the restriction of \eqref{eq:short-exact-sequence-P0P1} to $V$ is a diffeomorphism between small neighborhoods of $(x_{0},x_{1})\in V$ and $0\subset \R^{n}$. We assume that $\epsilon$ is small enough that $B(\epsilon)$ is contained in this small neighborhood; thus, for each vector $q$ in $B(\epsilon)$, there is $(y_{0},y_{1})\in V$ so that:
\begin{enumerate}
\item $A_{1}(y_{1},0)-A_{0}(y_{0},0)=q$.
\end{enumerate}
Moreover, for appropriate metric on $P_{0}\times P_{1}$ it follows that:
\begin{enumerate}[resume]
\item $\mathrm{dist}((y_{0},y_{1}),(x_{0},x_{1}))\le \abs{q}.$
\end{enumerate}
For the rest of argument, we fix $\epsilon,\ell,$ and $V$, without further modifications so that the conclusion of the lemma holds.

\subsubsection{Linearization framework for solutions of $\mathscr{M}(\infty)$}
\label{sec:line-fram-solut}

Let $(x_{0},x_{1},u)$ be a solution in $\mathscr{M}(\infty)$. The goal in this section is to describe the linearization framework. As usual, the space of maps nearby $u$ is modelled on an appropriate Banach space completion of sections of $u^{*}TW$. Since the domain of $u$ is a punctured domain, we will use an appropriately weighted Sobolev space. For use as a weight, fix $\delta\in (0,\pi)$.

First we define $W^{1,p,\delta}$, $p>2$, to be the Sobolev space of sections of $u^{*}TW$ which are tangent to the Lagrangian boundary conditions with an exponential weight in the ends. On the end regions, there is a canonical trivialization of $u^{*}TW\simeq \C^{n}$ induced by the canonical coordinates associated to the coordinate ball $B(1)$ around $A(x_{0},0)=A(x_{1},0)$, and the linear deformations $\xi$ are required to point in the real space $\R^{n}$.

There are two ends: the left end is identified with $[0,\infty)\times [0,1]$ and the right end with $(-\infty,0]\times [0,1]$. In these ends, we require that:
\begin{equation*}
  e^{\pm \delta s}\xi(s,t)\text{ is in }W^{1,p},
\end{equation*}
where $s$-coordinate on the end is such that the line $s=0$ corresponds to the innermost boundary of the strip of modulus $\ell$.

Of course, if one only uses variations in $W^{1,p,\delta}$, then one could not change the position of the removable singularity of $u$ at the two punctures. Thus it is necessary to stabilize to allow this point to vary. We therefore pick two linear maps $\Phi_{\pm}$ from $\R^{n}$ to the space of smooth variations of $u$ so that:
\begin{enumerate}
\item $\Phi_{-}(v)=\beta(s)v$, $\Phi_{+}(v)=0$ holds in the left end,
\item $\Phi_{-}(v)=0$, $\Phi_{+}(v)=\beta(-s)v$ holds in the right end,
\end{enumerate}
where $\beta$ cuts off in the cut-off region shown in Figure \ref{fig:Pi-is-CS-regions}.

This gives a family of variations: $$(v_{-},v_{+},\xi)\in \R^{n}\oplus \R^{n}\oplus W^{1,p,\delta} \mapsto \Phi_{-}(v_{-})+\Phi_{+}(v_{+})+\xi,$$ which is sufficient to capture all nearby maps with the same boundary conditions. For further information on the use of stablized exponentially weighted Sobolev spaces in the context of maps with Lagrangian boundary conditions with boundary punctures, we refer the reader to \cite{biran-cornea-arXiv-2007}.

Because $\mathscr{M}(\infty)$ is a parametric moduli space of triples $(x_{0},x_{1},u)$, where $(x_{0},x_{1})\in P_{\infty}$, there are also variations which move the point $(x_{0},x_{1})$. To account for this, we introduce a linear map $\Psi$ taking $w=(w_{0},w_{1})\in TP_{\infty}$ into the space of smooth variations of $u$ so that:
\begin{enumerate}[resume]
\item the projection of $\Psi(w)$ at $z\in \bd \Sigma(\infty)$ onto the zero section matches the variation of $A_{i}(x_{i},t(z))$ in the direction of $w_{i}\in TP_{i,x_{i}}$; here the $t$ coordinate is determined by the relation that $u(z)\in T^{*}M_{A_{i}(x_{i},t(z))}$ and is described in greater detail in the shadowing disks construction;
\item $\Psi(w)$ is constant and imaginary in the ends (i.e., projects to a constant imaginary vector).
\end{enumerate}
The total space of variations of $(x_{0},x_{1},u)$ is then identified with: $$\R^{n}\oplus \R^{n}\oplus TP_{\infty}\oplus W^{1,p,\delta}$$ and each tuple $(v_{-},v_{+},w,\xi)$ produces a variation of $u$ given by: $$\Phi_{-}(v_{-})+\Phi_{+}(v_{+})+\Psi(w)+\xi.$$

By the usual procedure (for concreteness, we use a Riemannian exponential map), one can differentiate the differential equation in the directions of these these variations. This produces a linear differential operator $D_{x_{0},x_{1},u}$.

By using a Riemannian exponential map for a metric which agrees with the standard metric in the canonical coordinate system above the coordinate ball $B(1)$, it is arranged that:
\begin{equation}\label{eq:linearized-operator}
  D_{x_{0},x_{1},u}(\Phi_{-}(v_{-})+\Phi_{+}(v_{+})+\Psi(w)+\xi)=\bar{\bd}\xi\text{ on the ends},
\end{equation}
where $\bar{\bd}=\bd_{s}+i\bd_{t}$, using the aforementioned trivialization. Then it is clear that the left hand side of \eqref{eq:linearized-operator} is valued in the exponentially weighted $L^{p}$ space.

To obtain this simple formula for the linearized operator on the ends, one argues that the nearby map associated to the variation $(v_{-},v_{+},w,\xi)$ equals:
\begin{equation*}
  u+\Phi_{-}(v_{-})+\Phi_{+}(v_{+})+\Psi(w)+\xi,
\end{equation*}
on the ends (in the canonical coordinate system) provided $v_{-},v_{+},w,\xi$ are not too big (this is possible since $u$ lies above $B(\epsilon)$, and so there is lots of room before one leaves $B(1)$, the domain of the coordinates).

Because $(x_{0},x_{1},u)$ is supposed to be rigid, the Fredholm index of \eqref{eq:linearized-operator} as a map $\R^{n}\oplus \R^{n}\oplus TP_{\infty}\oplus W^{1,p,\delta}\to L^{p,\delta}$ is zero. By the standard Sard-Smale argument, as in \cite{mcduff-salamon-book-2012}, the generic perturbation term $\mathfrak{p}_{\infty}$ can be chosen so the operator is an isomorphism for all solutions $(x_{0},x_{1},u)$. In particular, it has a bounded inverse. This will be used in \S\ref{sec:line-fram-pregl}.

\subsubsection{Pregluing}
\label{sec:pregluing}

The strategy is now to:
\begin{enumerate}
\item\label{gluing-step-1} For each rigid solution $(x_{0},x_{1},u)\in \mathscr{M}(\infty)$, construct preglued solutions which are supposed to approximate solutions in $\mathscr{M}(R)$ for $R$ sufficiently large. Prove they approximately solve the equation, as measured with an appropriately weighted Sobolev space norm.
\item\label{gluing-step-2} Prove that each preglued solution approximates exactly one of the non-compact ends of $\mathscr{M}$ containing sequences $(R_{n},x_{0}^{n},x_{1}^{n},u_{n})$ with $R_{n}\to\infty$. This involves analyzing a linearized operator for each preglued solution (gluing).
\item\label{gluing-step-3} Prove that genuine solutions in $\mathscr{M}(R)$ are close to the preglued solutions as $R\to\infty$. Together with \ref{gluing-step-2}, this establishes a bijection between the non-compact ends and solutions in $\mathscr{M}(\infty)$.
\end{enumerate}
This is a standard strategy for Floer gluing, and the precise details are quite close to the ``adiabatic'' gluing results cited in \S\ref{sec:geometric-set-up}. In this subsection, we are concerned with step \ref{gluing-step-1} of the strategy.

First of all, each solution $(x_{0},x_{1},u)\in \mathscr{M}(\infty)$ determines two points $p_{-},p_{+}$ in the cotangent fiber $T^{*}M_{A_{0}(x_{0},0)}=T^{*}M_{A_{1}(x_{1},0)}\simeq \R^{n}$ --- the identification with $\R^{n}$ uses the canonical coordinates above the ball $B(1)$. The point $p_{-}$ is the asymptotic at the left end, and $p_{+}$ is the asymptotic at the right end. These two points vary smoothly with the solution, and there are two constants $C_{\pm}$ so that:
\begin{enumerate}
\item\label{Cminus-gluing} on the left end, $\abs{u(s,t)-p_{-}}\le C_{-}e^{-\pi s}$, where $s\in [0,\infty)$,
\item\label{Cplus-gluing} on the right end, $\abs{u(s,t)-p_{+}}\le C_{+}e^{\pi s}$, where $s\in (-\infty,0]$.
\end{enumerate}
This follows from the removable singularity theorem for $J_{0}$-holomorphic maps.

The argument will require us splitting the neck of modulus $2R$ into specific regions; this is illustrated in Figure \ref{fig:pregluing-regions}.
\begin{figure}[h]
  \centering
  \begin{tikzpicture}[scale=.5]
    \draw[rounded corners] (0,-1)--(0,0)--+(13,0)--+(13,-1) (0,2)--(0,1)--+(13,0)--+(13,1);
    \fill[pattern=north west lines] (4,0) rectangle +(1,1) (8,0) rectangle +(1,1);
    \draw (2,0)--+(0,1) (11,0)--+(0,1);
    \draw[<->] (2,-0.5)--node[below]{$\mathfrak{s}$}(4,-0.5);
    \draw[<->] (4,1.5)--node[above]{$1$}(5,1.5);
    \draw[<->] (8,1.5)--node[above]{$1$}(9,1.5);
    \draw[<->] (9,-0.5)--node[below]{$\mathfrak{s}$}(11,-0.5);

    \draw (4,0)--+(0,1) (5,0)--+(0,1) (8,0)--+(0,1) (9,0)--+(0,1);
    \draw[<->] (5,-0.5)--node[below]{$2\mathfrak{r}$}(8,-0.5);

  \end{tikzpicture}
  \caption{Decomposing $R=\mathfrak{s}+1+\mathfrak{r}$. The shaded regions are the \emph{cut-off regions}.}
  \label{fig:pregluing-regions}
\end{figure}
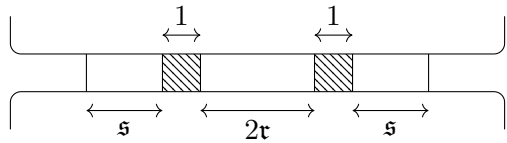

For each $(x_{0},x_{1},u)$, $q_{0}\in B(\epsilon/2)$, define:
\begin{equation*}
  \left\{
  \begin{aligned}
    &N_{\mathfrak{r}}:[-\mathfrak{r}-1,\mathfrak{r}+1]\times [0,1]\to \R^{n}\times \R^{n},\\
    &N_{\mathfrak{r}}(z):=iq_{0}+2\mathfrak{r}^{-1}((\mathfrak{r}-z)p_{-}+(\mathfrak{r}+z)p_{+}),
  \end{aligned}
  \right.
\end{equation*}
where $p_{\pm}$ are considered as real-vectors (in $\R^{n}\times \set{0}$). Let us observe that, the imaginary part of $N_{\mathfrak{r}}(z)$ equals $q_{0}+2\mathfrak{r}^{-1}(p_{+}-p_{-})t$, and so, provided $\mathfrak{r}$ is sufficiently large, the size of the imaginary part can be bounded by $\epsilon$, and hence $N_{\mathfrak{r}}$ can be considered as a holomorphic map in $\R^{n}\times B(1)\subset T^{*}M$.

Let $y_{0},y_{1}\in V$ be such that:
\begin{enumerate}[resume]
\item $A_{1}(y_{1},0)-A_{0}(y_{0},0)=2\mathfrak{r}^{-1}(p_{+}-p_{-})$, and
\item pick $q_{0}=A_{0}(y_{0},0)$.
\end{enumerate}
Then $N_{\mathfrak{r}}$ takes boundary conditions for $t=0,1$ on $T^{*}M_{A_{0}(y_{0},0)}$ and $T^{*}M_{A_{1}(y_{1},0)}$.

The idea is to glue this holomorphic neck to the existing solution $(x_{0},x_{1},u)$ to form a preglued solution.

Since $V$ is contained in a small neighborhood of $(x_{0},x_{1})$, we can use the Riemannian exponential map to define a variation $\Gamma(y_{0},y_{1})$ of $u$ which pushes the boundary conditions from those for $(x_{0},x_{1})$ to those for $(y_{0},y_{1})$, in a controlled way: we suppose that the $C^{1}$ size of $\Gamma(y_{0},y_{1})$ is controlled by the distance $\mathrm{dist}((y_{0},y_{1}),(x_{0},x_{1}))$. Moreover, we suppose that $\Gamma$ is supported away from the region $\mathscr{R}_{2}$. We denote this variation by $u+\Gamma(y_{0},y_{1})$, where the ``$+$'' symbol is to be interpreted using the Riemannian exponential map. Since the size of $\Gamma(y_{0},y_{1})$ is controlled by $\mathfrak{r}^{-1}\abs{p_{+}-p_{-}}$, we can assume that $u+\Gamma(y_{0},y_{1})$ maps the neck into $\R^{n}\times B(2\epsilon)$ provided $\mathfrak{r}$ is large enough.

Finally, we define the preglued solution $\mathrm{PG}_{\mathfrak{s},\mathfrak{r}}(x_{0},x_{1},u)$ as a piecewise function gluing together $u+\Gamma(y_{0},y_{1})$ and $N_{\mathfrak{r}}$, as illustrated in Figure \ref{fig:preglued-solution}.
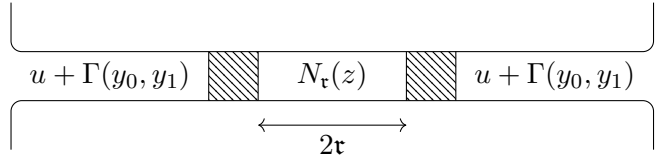
\begin{figure}[h]
  \centering
  \begin{tikzpicture}[scale=.65]
    \draw[rounded corners] (0,-1)--(0,0)--+(13,0)--+(13,-1) (0,2)--(0,1)--+(13,0)--+(13,1);
    \fill[pattern=north west lines] (4,0) rectangle +(1,1) (8,0) rectangle +(1,1);
    \draw (4,0)--+(0,1) (5,0)--+(0,1) (8,0)--+(0,1) (9,0)--+(0,1);
    \draw[<->] (5,-0.5)--node[below]{$2\mathfrak{r}$}(8,-0.5);
    \node at (6.5,0.5) {$N_{\mathfrak{r}}(z)$};
    \node at (11,0.5) {$u+\Gamma(y_{0},y_{1})$};
    \node at (2,0.5) {$u+\Gamma(y_{0},y_{1})$};
  \end{tikzpicture}
  \caption{Preglued solution; compare with Figure \ref{fig:Pi-is-CS-regions}. On the shaded cut-off regions (which are mapped into $B(2\epsilon)$), one should interpolate between the solutions using standard cut-off functions.}
  \label{fig:preglued-solution}
\end{figure}

In the left cut-off region, parametrized by $s,t\in [0,1]^{2}$, the interpolation is given by:
\begin{equation*}
  \beta(s)N_{\mathfrak{r}}(z)+(1-\beta(s))(u+\Gamma(y_{0},y_{1})),
\end{equation*}
and a similar (reflected) formula is used in the right cut-off region.

The decomposition of $R$ into $\mathfrak{s}+1+\mathfrak{r}$ affects how this preglued solution is constructed, so different choices of $\mathfrak{s},\mathfrak{r}$ yield different preglued solutions.

It will be important to measure sizes of variations of $\mathrm{PG}_{\mathfrak{s},\mathfrak{r}}(x_{0},x_{1},u)$ using a weighted Sobolev norm. We introduce the weight:
\begin{equation*}
  \mathfrak{w}=\min\set{e^{\delta(R+s)},e^{\delta (-s+R)}}
\end{equation*}
supported on the neck of length $2R$ parametrized so $s=0$ is the middle of the neck (see Figure \ref{fig:Pi-is-CS-regions} for illustration of the region). Then we define $L^{p,\mathfrak{w}}$ as the set of $L^{p}_{\mathrm{loc}}$ sections $\eta$ which are $L^{p}$ integrable on the asymptotic cylindrical end, and so that $\mathfrak{w}\eta$ is $L^{p}$ integrable on the neck; $W^{1,p,\mathfrak{w}}$ is defined analogously.

We then have:
\begin{lemma}\label{lemma:estimating-size-of-error}
  If $F(y_{0},y_{1},w)$ is the non-linear map encoding the PDE for solutions of $\mathscr{M}(R)$, i.e., in local coordinates:
  \begin{equation*}
    F(w)=\bd_{s}w+J_{x_{0},z}(w)\bd_{t}w-a(w),
  \end{equation*}
  then the $L^{p,\mathfrak{w}}$ size is bounded by:
  \begin{equation*}
    \norm{F(y_{0},y_{1},\mathrm{PG}_{\mathfrak{s},\mathfrak{r}}(x_{0},x_{1},u))}_{L^{p,\mathfrak{w}}}\le (C_{+}+C_{-})e^{-(\pi-\delta)\mathfrak{s}}+C_{\mathfrak{s}}\mathfrak{r}^{-1},
  \end{equation*}
  where $C_{\pm}$ are defined in \ref{Cminus-gluing} and \ref{Cplus-gluing}, and $C_{\mathfrak{s}}$ is independent of $x_{0},x_{1},u,\delta,\mathfrak{r}$, but depends on $\mathfrak{s}$.
\end{lemma}
\begin{proof}
  Since $u$ solves the equation, $F(y_{0},y_{1},u+\Gamma(y_{0},y_{1}))$ will approximately solve the equation up to an error whose $C^{0}$ size is controlled by $$\mathfrak{r}^{-1}\abs{p_{+}-p_{-}},$$ since that size governs the $C^{0}$ size of $\Gamma(y_{0},y_{1})$.

  The integral of this error over the region disjoint from the neck is then bounded by $C_{1}\mathfrak{r}^{-1}$. Moreover, the integral of this error over the necks of length $\mathfrak{s}+1$ produces errors of size:
  \begin{equation*}
    \mathfrak{r}^{-1}\abs{p_{+}-p_{-}}(\mathfrak{s}+1)^{1/p}e^{\delta (\mathfrak{s}+1)}\le C_{2}\mathfrak{r}^{-1}.
  \end{equation*}
  In total, these errors can be bounded by $C_{3}\mathfrak{r}^{-1}$.

  There is an additional error arising due to the interpolation (we focus only on the left end as the right end follows from a reflected estimate):
  \begin{equation*}
    \beta(s)N_{\mathfrak{r}}(z)+(1-\beta(s))(u(z)+\Gamma(y_{0},y_{1})),
  \end{equation*}
  Since $N_{\mathfrak{r}}(z)$ and $u(z)$ both solve the equation (which is the standard $\bar{\bd}$ equation in the cut-off region), this produces an error of $C^{0}$ size:
  \begin{equation*}
    \abs{u(z)+\Gamma(y_{0},y_{1})-N_{\mathfrak{r}}(z)}+C_{4}\mathfrak{r}^{-1},
  \end{equation*}
  where $C_{4}\mathfrak{r}^{-1}$ is due to the fact that $\Gamma(y_{0},y_{1})$ has $C^{1}$ size bounded by $\mathfrak{r}^{-1}$. We use that $\beta'(s)$ is approximately $1$ to avoid introducing another constant.

  Then we estimate $\abs{u(z)+\Gamma(y_{0},y_{1})-N_{\mathfrak{r}}(z)}$ in the region $s \in [\mathfrak{-r}-1, \mathfrak{-r}]$ by:
  \begin{equation*}
    \abs{u(z)-p_{-}}+\abs{p_{-}-N_{\mathfrak{r}}(z)}+\abs{\Gamma(y_{0},y_{1})}\le C_{-}e^{-\pi \mathfrak{s}}+C_{5}\mathfrak{r}^{-1}.
  \end{equation*}
  The bound on the first and the last term is clear from the cosntruction. Let us expand a bit why is $\abs{p_{-}-N_{\mathfrak{r}}(z)}$ bounded in terms of  $\mathfrak{r}^{-1}$. First note that $\abs{i q_0}$ is bounded in terms of $\mathfrak{r}^{-1}$ since the distance between $y_0$ and $x_0$ is controled by $\mathfrak{r}^{-1}$, and $A_0(\cdot, 0) $ is a smoot map, hence $q_0 - 0= A_0(y_0,0) - A_0(x_0,0)$ is controled by $\mathfrak{r}^{-1}$ as well. Secondly, since $s \in [\mathfrak{-r}-1, \mathfrak{-r}]$, we have that $\abs{\mathfrak{r}+z}$ is bounded by $\sqrt{2}$ and $\abs{ p_- -2\mathfrak{r}^{-1}((\mathfrak{r}-z)p_- + (\mathfrak{r}+z)p_+)} = \abs{2\mathfrak{r}^{-1}(\mathfrak{r} +z)(p_+ - p_-)}$.

  The contribution to the $L^{p,\mathfrak{w}}$ size can be estimated by:
  \begin{equation*}
    e^{+\delta\mathfrak{s}}C_{-}e^{-\pi \mathfrak{s}}+e^{+\delta\mathfrak{s}}C_{5}\mathfrak{r}^{-1}.
  \end{equation*}
  Combining with the earlier estimate, we conclude:
  \begin{equation*}
    \norm{F(y_{0},y_{1},\mathrm{PG}_{\mathfrak{s},\mathfrak{r}}(x_{0},x_{1},u))}_{L^{p,\mathfrak{w}}}\le (C_{-}+C_{+})e^{-(\pi-\delta)\mathfrak{s}}+C_{\mathfrak{s}}\mathfrak{r}^{-1},
  \end{equation*}
  where $C_{\mathfrak{s}}=(C_{3}+2e^{+\delta\mathfrak{s}}(C_{4}+C_{5}))$, where we take into account the reflected estimate at the right end. This completes the proof.
\end{proof}

In the rest of the argument, one should imagine that first $\mathfrak{s}$ is chosen large enough, and then, $\mathfrak{r}$ is chosen in terms of $\mathfrak{s}$.

\subsubsection{Linearization framework for the preglued solutions}
\label{sec:line-fram-pregl}

In this section we are concerned with step \ref{gluing-step-2} of the gluing argument outlined above. The ideas are similar to those in \S\ref{sec:line-fram-solut}; we will linearize the equation at the preglued solution producing a linearized operator. Then we will show this linearized operator is uniformly surjective as $R\to\infty$, provided the parameters $\mathfrak{s},\mathfrak{r}$ are chosen sufficiently large.

Let us abbreviate $(y_{0},y_{1},u_{\mathfrak{s},\mathfrak{r}})=\mathrm{PG}_{\mathfrak{s},\mathfrak{r}}(x_{0},x_{1},u)$. Using the Riemannian metric $g_{x_{0}}$ as in \S\ref{sec:line-fram-solut}, one can associate to each variation of $u_{\mathfrak{s},\mathfrak{r}}$ a nearby map, and by the usual procedure of differentiating the equation, obtain a linearized operator $D_{y_{0},y_{1},u_{\mathfrak{s},\mathfrak{r}}}$. As in \S\ref{sec:line-fram-solut}, $D_{y_{0},y_{1},u_{\mathfrak{s},\mathfrak{r}}}$ agrees with $\bar{\bd}$ on the neck region.

The space of variations of $(y_{0},y_{1},u_{\mathfrak{s},\mathfrak{r}})$ we will work with is:
\begin{equation*}
  \R^{n}\oplus TP_{0}\oplus TP_{1}\oplus W^{1,p,\mathfrak{w}},
\end{equation*}
where:
\begin{enumerate}
\item $TP_{0},TP_{1}$ account for variations $w_{0},w_{1}$ based at $y_{0},y_{1}$; these potentially move the boundary conditions,
\item $\R^{n}$ is a space of variations $v$ which are real-valued and constant on the neck, similarly to the space of variations $\R^{n}\oplus \R^{n}$ used in \S\ref{sec:line-fram-solut},
\item $W^{1,p,\mathfrak{w}}$ are variations of $u_{R}$ which are tangent to the boundary conditions and lie in the weighted Sobolev space.
\end{enumerate}
It is convenient to split the short-exact sequence:
\begin{equation*}
  0\to TP_{\infty}\to TP_{0}\oplus TP_{1}\to \R^{n}\to 0,
\end{equation*}
using the slice $V$, where the right map is the difference of the derivatives $y_{i}\mapsto A_{i}(y_{i},0)$; see \S\ref{sec:choice-ell} for the definition of $V$. With this splitting chosen, the space of variations is identified with:
\begin{equation*}
  \R^{n}\oplus TP_{\infty}\oplus \R^{n}\oplus W^{1,p,\mathfrak{w}}.
\end{equation*}

We will reuse the notation from \S\ref{sec:line-fram-solut} as much as possible. First, for each $v\in \R^{n}$, let $\Phi(v)$ be a function which is constant and equal to $v\in \R^{n}\times \set{0}$ in the neck region and ``agrees'' with $\Phi_{-}(v)$ and $\Phi_{+}(v)$ outside of the neck region. Here ``agrees'' should be interpreted up to the identification of variations of $u$ with variations of $u_{R}$ outside of the neck region (which appeals to a parallel transport map).

Second, for each $w\in TP_{\infty}$, let $\Psi(w)$ be a variation which is constant and imaginary in the neck region, and ``agrees'' with the old $\Psi(w)$ outside of the neck region.

Next, for $\mu_{0},\mu_{1}\in \R^{n}$, let $k(\mu_{0},\mu_{1}):[-\mathfrak{r}-1,\mathfrak{r}+1]\times [0,1]\to \C^{n}$ be given by:
\begin{equation*}
  k(\mu_{0},\mu_{1})=2\mathfrak{r}^{-1}i\mu_{0}+2\mathfrak{r}^{-1}((\mathfrak{r}-z)\mu_{0}+(\mathfrak{r}+z)\mu_{1}).
\end{equation*}
This has boundary conditions on $\R^{n}\times 2\mathfrak{r}^{-1}\mu_{0}$ and $\R^{n}\times 2\mathfrak{r}^{-1}\mu_{1}$.

For $\mathfrak{r}$ large enough, $(2\mathfrak{r}^{-1}\mu_{0},2\mathfrak{r}^{-1}\mu_{1})$ lies in the image of the derivatives of $(A_{0}(-,0),A_{1}(-,0))$ restricted to the slice $V$. Such pairs are uniquely determined by the difference vector $\mu=\mu_{1}-\mu_{0}\in \R^{n}$, and given such a pair we can find a variation $\gamma(\mu_{0},\mu_{1})$ by differentiating $\Gamma(y_{0},y_{1})$ in the direction of a vector tangent to $V$. Then $\gamma(\mu_{0},\mu_{1})$ has $C^{1}$ size controlled by $2\mathfrak{r}^{-1}\abs{\mu_{1}-\mu_{0}}$, and, by construction, $k(\mu_{0},\mu_{1})$ and $\gamma(\mu_{0},\mu_{1})$ have the same imaginary parts along the boundary. Then we define $\mathrm{K}(\mu)$ to be  $k(\mu_{0},\mu_{1})$ on the neck of modulus $2\mathfrak{r}$, the linear interpolation between $k(\mu_{0},\mu_{1})$ and $\gamma(\mu_{0},\mu_{1})$ on the the cut-off region (using the cut-off functions $\beta(s),\beta(-s)$), and $\gamma(\mu_{0},\mu_{1})$ everywhere else.

It suffices to say that the linearized operator takes the form:
\begin{equation}\label{eq:linearized-op-preglued}
  (v,w,\mu,\xi)\mapsto D_{y_{0},y_{1},u_{\mathfrak{s},\mathfrak{r}}}(\Phi(v)+\Psi(w)+\mathrm{K}(\mu)+\xi)\in L^{p,\mathfrak{w}}.
\end{equation}
We will now argue that this linearized operator is uniformly surjective.

\begin{lemma}
  For any $\eta\in L^{p,\mathfrak{w}}$, there exists $(v,w,\mu,\xi)$ so that:
  \begin{equation*}
    D_{y_{0},y_{1},u_{\mathfrak{s},\mathfrak{r}}}(\Phi(v)+\Psi(w)+\mathrm{K}(\mu)+\xi)=\eta,
  \end{equation*}
  and:
  \begin{equation*}
    \abs{v}+\abs{w}+\abs{\mu}+\norm{\xi}_{W^{1,p,\mathfrak{w}}}\le C_{\mathrm{lin}}\norm{\eta},
  \end{equation*}
  where $C_{\mathrm{lin}}$ is independent of $\eta,\rho,\mathfrak{s},\mathfrak{r}$ or the original solution $(x_{0},x_{1},u)$, provided $\mathfrak{r}$ is large enough.
\end{lemma}
\begin{proof}
  Divide the neck $[-\mathfrak{r},\mathfrak{r}]$ into two regions $[-\mathfrak{r},0]\cup [0,\mathfrak{r}]$ and introduce two cut-off functions:
  \begin{enumerate}
  \item $f_{-}(s)=\beta(1-s/\mathfrak{r})$,
  \item $f_{+}(s)=\beta(1+s/\mathfrak{r})$,
  \end{enumerate}
  The key is that $f_{-}$ is $1$ on the left region, and cuts off in the right region, while $f_{+}$ cuts off in the left region, and is $1$ on the right region. Another key is that the derivative of $f_{\pm}$ is $O(\mathfrak{r}^{-1})$.

  We define two maps. First: $\mathfrak{R}:L^{p,\mathfrak{w}}\to L^{p,\delta}$ as an extension by zero map; this is illustrated in Figure \ref{fig:extension-by-zero}. Modulo the small change in sizes due to parallel transport (taking variations of $u_{\mathfrak{s},\mathfrak{r}}$ to $u$), this function is norm preserving by virtue of how the weights are defined. Second we define $\mathfrak{L}:L^{p,\delta}\to L^{p,\mathfrak{w}}$ using the cut-off functions $f_{-},f_{+}$ as shown in Figure \ref{fig:partition-of-unity-map}. The $\mathfrak{L}$ map is almost norm preserving (due to how the weights are defined), and only increases norms slightly.

  Let us observe that $\mathfrak{L}\circ \mathfrak{R}=\id$, and moreover $\mathfrak{L}$ maps $W^{1,p,\delta}$ into $W^{1,p,\mathfrak{w}}$ in an approximately norm preserving way.

  We argue as follows: given $\eta\in L^{p,\mathfrak{w}}$, we can find $v_{-},v_{+},w,\xi$ so that:
  \begin{equation*}
    D_{x_{0},x_{1},u}(\Phi_{-}(v_{-})+\Phi_{+}(v_{+})+\Psi(w)+\xi)=\mathfrak{R}(\eta),
  \end{equation*}
  where $\abs{v_{-}}+\abs{v_{+}}+\abs{w}+\norm{\xi}_{W^{1,p,\delta}}\le C_{\infty}\norm{\eta}_{L^{p,\mathfrak{w}}}$, using the assumed surjectivity of the linearized operator for the genuine solution $(x_{0},x_{1},u)$.

  Pick $\mu$ and $v$ so that:
  \begin{equation}\label{eq:linear-equation}
    v+\mu_{0}=v_{-}\text{ and }v+\mu_{1}=v_{+},
  \end{equation}
  so that $\mu=v_{+}-v_{-}$. We then claim that:
  \begin{equation*}
    D_{y_{0},y_{1},u_{\mathfrak{s},\mathfrak{r}}}(\Phi(v)+\Psi(w)+\mathrm{K}(\mu)+\mathfrak{L}(\xi))=\eta+O(\mathfrak{r}^{-1})\norm{\eta}.
  \end{equation*}
  Indeed, $\mathrm{K}(\mu)+\Phi(v)$ is close to $\Phi_{-}(v_{-})$ and $\Phi_{+}(v_{+})$ on the cut-off regions, and the difference between $\bar{\bd}(\mathrm{K}(\mu)+\Phi(v))$ and $\bar{\bd}(\Phi_{\pm}(v_{\pm}))$ is a small error of order $\mathfrak{r}^{-1}\abs{\mu}$, using equation \eqref{eq:linear-equation} and the fact $\gamma(\mu_{0},\mu_{1})$ has $C^{1}$ size controlled by $\mathfrak{r}^{-1}\abs{\mu}$.

  On the inner neck, the only contribution is $\bar{\bd}\mathfrak{L}(\xi)$ which is approximately $\eta=\mathfrak{L}(\mathfrak{R}(\eta))$ up to an error $O(\mathfrak{r}^{-1})\norm{\eta}$, due to derivatives of $f_{+},f_{-}$.

  The only other errors are due to the parallel transport, but this is of order $O(\mathfrak{r}^{-1})\norm{\eta}$, since the distance of curves we need to parallel transport along is $O(\mathfrak{r}^{-1})$ (noting that parallel transport acts identically in the neck).

  Thus we can solve the equation up to an error of $O(\mathfrak{r}^{-1})\norm{\eta}$. Provided $O(\mathfrak{r}^{-1})<1/2$, we can then solve the equation by an iterative process, as is common in solving linear equations in Banach spaces. The result gives an inverse image whose norm is uniformly bounded in terms of the input norm $\norm{\eta}$ by some constant $C_{\mathrm{lin}}$.
\end{proof}

\begin{figure}[h]
  \centering
  \begin{tikzpicture}[scale=.4]
    \draw[rounded corners] (0,-1)--(0,0)--+(13,0)--+(13,-1) (0,2)--(0,1)--+(13,0)--+(13,1);
    \draw (3,0)--+(0,1) (10,0)--+(0,1) (6.5,0)--+(0,1);
    \draw[<->] (3,-0.5)--node[below]{$2R$}(10,-0.5);
    \node at (4.75,0.5) {$\eta$};
    \node at (8.25,0.5) {$\eta$};
    \node at (11.5,0.5) {$\eta$};
    \node at (1.5,0.5) {$\eta$};
    \begin{scope}[shift={(15,2)}]
      \draw[rounded corners] (0,-1)--(0,0)--+(13,0) (0,2)--(0,1)--+(13,0);
      \draw (3,0)--+(0,1) (10,0)--+(0,1) (6.5,0)--+(0,1);
      \node at (4.75,0.5) {$\eta$};
      \node at (8.25,0.5) {$0$};
      \node at (11.5,0.5) {$0$};
      \node at (1.5,0.5) {$\eta$};
    \end{scope}

    \begin{scope}[shift={(15,-2)}]
      \draw[rounded corners] (0,0)--+(13,0)--+(13,-1) (0,1)--+(13,0)--+(13,1);
      \draw (3,0)--+(0,1) (10,0)--+(0,1) (6.5,0)--+(0,1);
      \node at (4.75,0.5) {$0$};
      \node at (8.25,0.5) {$\eta$};
      \node at (11.5,0.5) {$\eta$};
      \node at (1.5,0.5) {$0$};
    \end{scope}
  \end{tikzpicture}
  \caption{Extension by zero map $\mathfrak{R}$; this goes from deformations of the preglued map $u_{R}$ to deformations of the original map $u$.}
  \label{fig:extension-by-zero}
\end{figure}
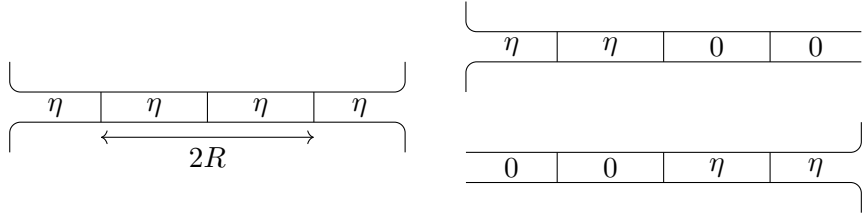

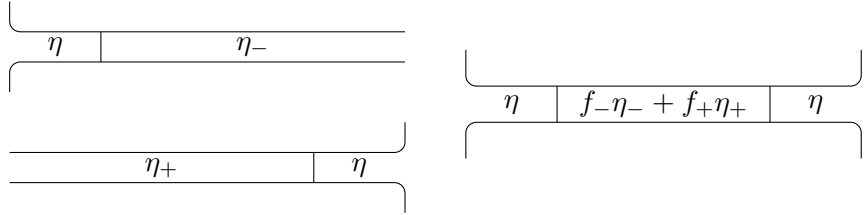
\begin{figure}[h]
  \centering
  \begin{tikzpicture}[scale=.4]
    \begin{scope}[yscale=1.2]
      \draw[rounded corners] (0,-1)--(0,0)--+(13,0)--+(13,-1) (0,2)--(0,1)--+(13,0)--+(13,1);
      \draw (3,0)--+(0,1) (10,0)--+(0,1);
      \node at (6.5,0.5) {$f_{-}\eta_{-}+f_{+}\eta_{+}$};
      \node at (11.5,0.5) {$\eta$};
      \node at (1.5,0.5) {$\eta$};
    \end{scope}
    \begin{scope}[shift={(-15,2)}]
      \draw[rounded corners] (0,-1)--(0,0)--+(13,0) (0,2)--(0,1)--+(13,0);
      \draw (3,0)--+(0,1);
      \node at (8,0.5) {$\eta_{-}$};
      \node at (1.5,0.5) {$\eta$};
    \end{scope}

    \begin{scope}[shift={(-15,-2)}]
      \draw[rounded corners] (0,0)--+(13,0)--+(13,-1) (0,1)--+(13,0)--+(13,1);
      \draw (10,0)--+(0,1);
      \node at (5,0.5) {$\eta_{+}$};
      \node at (11.5,0.5) {$\eta$};
    \end{scope}
  \end{tikzpicture}
  \caption{Partition of unity map $\mathfrak{L}$; this goes from deformations of $u$ to deformations of $u_{R}$.}
  \label{fig:partition-of-unity-map}
\end{figure}

As a consequence of this result, and the previous result Lemma \ref{lemma:estimating-size-of-error}, one concludes by the usual application of the inverse function theorem in Banach spaces that there are indeed genuine solutions of $\mathscr{M}(R)$ close to the preglued solutions, provided the parameter $\mathfrak{s}$ is large, and $\mathfrak{r}$ is chosen much larger than $\mathfrak{s}$, so that the failure of the preglued solution to solve the equation is small enough (see Lemma \ref{lemma:estimating-size-of-error}).

Moreover, these genuine solutions are cut transversally, because the linearization at the preglued solutions is uniformly surjective. The solutions are rigid in $\mathscr{M}(R)$ as can be deduced by the index formula together with the above uniform surjectivity. Let us denote this genuine solution by $\mathrm{G}_{\mathfrak{s},\mathfrak{r}}(x_{0},x_{1},u)$. This rigidity leads to the following conclusion:

\begin{lemma}
  Let $\mathfrak{s}_{0}\ge 0$ be a constant, and let $P(\mathfrak{s})$ be such that $\mathfrak{r}\ge P(\mathfrak{s})$ and $\mathfrak{s}\ge \mathfrak{s}_{0}$ implies that the gluing argument applied to $\mathrm{PG}_{\mathfrak{s},\mathfrak{r}}(x_{0},x_{1},u)$ converges to a genuine solution $\mathrm{G}_{\mathfrak{s},\mathfrak{r}}(x_{0},x_{1},u)$ (this requires $\mathfrak{r}$ to be large enough for the linearized operator to be uniformly surjective, and then the error in Lemma \ref{lemma:estimating-size-of-error} to be small enough). Now suppose that $\mathfrak{s}(\tau),\mathfrak{r}(\tau)$ varies continuously and so that:
  \begin{enumerate}
  \item $\mathfrak{s}(\tau)\ge \mathfrak{s}_{0}$ and $\mathfrak{r}(\tau)\ge P(\mathfrak{s}(\tau))$ for all $\tau$,
  \item $\mathfrak{s}(\tau)+\mathfrak{r}(\tau)=\mathfrak{s}(0)+\mathfrak{r}(0)$,
  \end{enumerate}
  then $G_{\mathfrak{s}(\tau),\mathfrak{r}(\tau)}(x_{0},x_{1},u)=\mathrm{G}_{\mathfrak{s}(0),\mathfrak{r}(0)}(x_{0},x_{1},u)$. In particular, the glued solution in $\mathscr{M}(R)$ is $G_{\mathfrak{s},R-\mathfrak{s}}$, provided $\mathfrak{s}\ge \mathfrak{s}_{0}$ and $R\ge \mathfrak{s}+P(\mathfrak{s})$.
\end{lemma}
\begin{proof}
  This follows from the rigidity of the solutions, and the fact the pregluing construction is continuous with respect to variations of $\mathfrak{s}$ and $\mathfrak{r}$.
\end{proof}

This completes part \ref{gluing-step-2} of the gluing argument. In the next and final subsection, we will complete step \ref{gluing-step-3} of the gluing argument.

\subsubsection{Compactness for ends of $\mathscr{M}$.}
\label{sec:comp-theor-ends}

Suppose that $(R_{n},x_{0}^{n},x_{1}^{n},u_{n})$ is a sequence of genuine solutions in $\mathscr{M}$, with $R_{n}\to\infty$. To complete the gluing argument, we need to show that (after passing to a subsequence) $(R_{n},x_{0}^{n},x_{1}^{n},u_{n})$ eventually equals the glued solution $\mathrm{G}_{\mathfrak{s},\mathfrak{r}_{n}}(x_{0},x_{1},u)$, where $\mathfrak{r}_{n}=R_{n}-\mathfrak{s}$, and $(x_{0},x_{1},u)$ is an appropriate Gromov limit of $(x_{0}^{n},x_{1}^{n},u_{n})$.

First, pick the subsequence so that $(x_{0}^{n},x_{1}^{n})$ converge to a limit $(x_{0},x_{1})$. For later use, let us suppose the subsequence is such that $x_{0}^{n}$ is in the ball $B(\epsilon)$ around $x_{0}$. After passing to a further subsequence, the restriction of $u_{n}$ to the neck $[-R_{n},R_{n}]$ converges on compact subsets to a limiting holomorphic strip with boundary on cotangent fibers $T^{*}M_{A_{0}(x_{0},0)}$ and $T^{*}M_{A_{1}(x_{1},0)}$. Since this limiting holomorphic strip has bounded energy (by Fatou's lemma), it must be that $T^{*}M_{A_{0}(x_{0},0)}=T^{*}M_{A_{1}(x_{1},0)}$, otherwise no such finite energy holomorphic strip exists. Thus we conclude $A_{0}(x_{0},0)=A_{1}(x_{1},0)$, so $(x_{0},x_{1})\in P_{\infty}$.

By passing to yet a further sequence, standard elliptic regularity and Floer theory compactness results imply that $u_{n}$ converges on compact subsets of the limiting domain $\Sigma(\infty)$ to a smooth map $u$ --- here we use the identification of the surfaces with sufficiently deep ends/necks removed, so that any compact subset of $\Sigma(\infty)$ eventually is identified with a compact subset of $\Sigma(R_{n})$.

Let $(\bar{x}_{0}^{n},\bar{x}_{1}^{n},\bar{u}_{n})=\mathrm{G}_{\mathfrak{s},\mathfrak{r}_{n}}(x_{0},x_{1},u)$. It is sufficient to prove there is a variation of $(\bar{x}_{0}^{n},\bar{x}_{1}^{n},\bar{u}_{n})$ of small norm which equals $(x_{0}^{n},x_{1}^{n},u_{n})$ after applying the Riemannian exponential map, as then we can appeal to the rigidity of the glued solution to conclude $(\bar{x}_{0}^{n},\bar{x}_{1}^{n},\bar{u}_{n})=(x_{0}^{n},x_{1}^{n},u_{n})$.

It is clear from the construction that $(\bar{x}_{0}^{n},\bar{x}_{1}^{n},\bar{u}_{n})$ also converges to $(x_{0},x_{1},u)$ as $n\to\infty$, in the same manner that $(x_{0}^{n},x_{1}^{n},u_{n})$ does.

In particular, the restrictions of $\bar{u}_{n},u_{n}$ to $[-\mathfrak{r}_{n},\mathfrak{r}_{n}]\times [0,1]$ are both holomorphic strips whose endpoints are eventually within the $C_{\pm}e^{-\pi \mathfrak{s}}$ neighborhoods of $p_{\pm}$, since the limit $(x_{0},x_{1},u)$ has its endpoints within this neighborhood, as explained in \S\ref{sec:pregluing}.

Let us abbreviate by $q_{0}^{n},q_{1}^{n}$ and $\bar{q}_{0}^{n},\bar{q}_{1}^{n}$ the imaginary parts of the boundary components of $u_{n}$ and $\bar{u}_{n}$. Then we estimate:
\begin{lemma}\label{lemma:basepoint-control}
  It holds that:
  \begin{equation*}
    \mathfrak{r}_{n}\abs{(q_{1}^{n}-q_{0}^{n})-(\bar{q}_{1}^{n}-\bar{q}_{0}^{n})}\le (C_{-}+C_{+})e^{-\pi\mathfrak{s}}.
  \end{equation*}
\end{lemma}
\begin{proof}
  Let $U=u_{n}-\bar{u}_{n}$, restricted to $[-\mathfrak{r}_{n},\mathfrak{r}_{n}]\times [0,1]$. Let $q(s,t),p(s,t)$ be the coordinate projections of $U$ to $\R^{n}\times B(1)$. Then:
  \begin{equation*}
    \bd_{s}p=\bd_{t}q\implies 2\mathfrak{r}_{n}((q_{1}^{n}-q_{0}^{n})-(\bar{q}_{1}^{n}-\bar{q}_{0}^{n}))=\int_{0}^{1}p(\mathfrak{r}_{n},t)-p(-\mathfrak{r}_{n},t)dt
  \end{equation*}
  Using that $U$ has its $s=\pm\mathfrak{r}_{n}$ boundary component in the $2C_{\pm}e^{-\pi \mathfrak{s}}$ neighborhood of $0$, we conclude the desired result.
\end{proof}

Next, we recall from \S\ref{sec:line-fram-pregl} that there is a variation $\mathrm{K}(\mu_{n})$ associated to $\mu_{n}\in \R^{n}$ whose restriction to $[-\mathfrak{r}_{n},\mathfrak{r}_{n}]\times [0,1]$ equals:
\begin{equation*}
  2\mathfrak{r}_{n}^{-1}i\mu_{n,0}+2\mathfrak{r}_{n}^{-1}(\mu_{n,0}(z+\mathfrak{r}_{n})+\mu_{n,1}(\mathfrak{r}_{n}-z)),
\end{equation*}
and which has a $C^{1}$ distance controlled by $\mathfrak{r}^{-1}_{n}(\mu_{n,1}-\mu_{n,0})$ outside of the neck $[-\mathfrak{r}_{n},\mathfrak{r}_{n}]\times [0,1]$. Let us ``add'' this variation to $(\bar{x}_{0}^{n},\bar{x}_{1}^{n},\bar{u}_{n})$ (using the Riemannian exponential map) so that: $$2\mathfrak{r}_{n}^{-1}(\mu_{n,1}-\mu_{n,0})=q_{1}^{n}-q_{0}^{n}-(\bar{q}_{1}^{n}-\bar{q}_{0}^{n}).$$ Then it holds that:
\begin{equation*}
  u_{n}-(\bar{u}_{n}+\mathrm{K}(\mu_{n}))
\end{equation*}
has boundary conditions on a single cotangent fiber. Moreover, Lemma \ref{lemma:basepoint-control} $\abs{\mu_{n}}$ is controlled by $2^{-1}(C_{-}+C_{+})e^{-\pi\mathfrak{s}}$. It then follows that:
\begin{equation*}
  \norm{u_{n}-(\bar{u}_{n}+\mathrm{K}(\mu_{n}))}_{W^{1,p,\mathfrak{w}}}\le O(\mathfrak{r}^{-1}_{n}e^{\delta\mathfrak{s}})+O(\mathfrak{e}^{\delta\mathfrak{s}}\mathrm{dist}_{\mathfrak{s}}(u_{n},\bar{u}_{n}))
\end{equation*}
where $\mathrm{dist}_{\mathfrak{s}}$ is the $C^{1}$ distance computed on the complement of $[-\mathfrak{r}_{n},\mathfrak{r}_{n}]\times [0,1]$. To see this, one uses:
\begin{enumerate}
\item the aforementioned bound of the $C^{1}$ size $K(\mu_{n})$ on the complement of $[-\mathfrak{r}_{n},\mathfrak{r}_{n}]\times [0,1]$,
\item the fact the $W^{1,p,\mathfrak{w}}$ distance on the complement of $[-\mathfrak{r}_{n},\mathfrak{r}_{n}]\times [0,1]$ is bounded by $e^{\delta\mathfrak{s}}$ times the usual $C^{1}$ distance, and
\item $u_{n}-(\bar{u}_{n}+\mathrm{K}(\mu))$ is holomorphic on $[-\mathfrak{r}_{n},\mathfrak{r}_{n}]\times [0,1]$, with boundary on a single cotangent fiber, and therefore the $W^{1,p,\mathfrak{w}}$ size on this neck is bounded by the $C^{1}$ size at the endpoints.
\end{enumerate}
Thus we conclude that:
\begin{equation*}
  u_{n}=\bar{u}_{n}+\mathrm{K}(\mu_{n})+\xi_{n}
\end{equation*}
where $\norm{\xi_{n}}_{W^{1,p,\mathfrak{w}}}=O(\mathfrak{r}^{-1}_{n}e^{\delta\mathfrak{s}})+O(\mathfrak{e}^{\delta\mathfrak{s}}\mathrm{dist}_{\mathfrak{s}}(u_{n},\bar{u}_{n}))$. In particular, by taking $n$ large enough, we can make $\norm{\xi_{n}}_{W^{1,p,\mathfrak{w}}}+\abs{\mu_{n}}$ as small as desired. In particular, $(x_{0},x_{1},u_{n})$ enters arbitrarily small neighborhoods of $(\bar{x}_{0}^{n},\bar{x}_{1}^{n},\bar{u}_{n})$ in the correct topology for the uniform surjectivity of the linearized operator to be applied. By the rigidity of the glued solution, it follows that $(x_{0},x_{1},u_{n})=(\bar{x}_{0}^{n},\bar{x}_{1}^{n},\bar{u}_{n})$, as desired.

This completes the proof that the count of non-compact ends of $\mathscr{M}$ equals the count of rigid solutions of $\mathscr{M}(\infty)$.\hfill$\square$

\bibliographystyle{./amsalpha-doi}
\bibliography{citations}

\providecommand{\bysame}{\leavevmode\hbox to3em{\hrulefill}\thinspace}
\providecommand{\MR}{\relax\ifhmode\unskip\space\fi MR }
\providecommand{\MRhref}[2]{%
  \href{http://www.ams.org/mathscinet-getitem?mr=#1}{#2}
}
\providecommand{\href}[2]{#2}
\begin{thebibliography}{{Mac}71}

\bibitem[AAC23]{alizadeh-atallah-cant-arXiv-2023}
H.~Alizadeh, M.~S. Atallah, and D.~Cant, \emph{Lagrangian intersections and the
  spectral norm in convex-at-infinity symplectic manifolds}, arXiv:2312.14752,
  12 2023, Accepted to Math. Z.,
  \href{https://doi.org/10.48550/arXiv.2312.14752}{doi:10.48550/arXiv.2312.14752}.

\bibitem[Abo15]{abouzaid-EMS-2015}
M.~Abouzaid, \emph{Symplectic cohomology and {V}iterbo's theorem}, Free Loop
  Spaces in Geometry and Topology, European Mathematical Society, 2015,
  pp.~271--486, \href{https://doi.org/10.4171/153}{doi:10.4171/153}.

\bibitem[Alu09]{aluffi-GSM-2009}
P.~Aluffi, \emph{Algebra: Chapter 0}, Graduate Studies in Mathematics, vol.
  104, AMS, 2009, \href{https://doi.org/10.1090/gsm/104}{doi:10.1090/gsm/104}.

\bibitem[APS08]{abbondandolo-portaluri-schwarz-JFPTA-2008}
A.~Abbondandolo, A.~Portaluri, and M.~Schwarz, \emph{The homology of path
  spaces and {F}loer homology with conormal boundary conditions}, J. Fixed
  Point Theory Appl. \textbf{4} (2008), 263--293,
  \href{https://doi.org/10.1007/S11784-008-0097-Y}{doi:10.1007/S11784-008-0097-Y}.

\bibitem[AS06]{abbondandolo-schwarz-CPAM-2006}
A.~Abbondandolo and M.~Schwarz, \emph{On the {F}loer homology of cotangent
  bundles}, Comm. Pure Appl. Math. \textbf{59} (2006), 254--316,
  \href{https://doi.org/10.1002/cpa.20090}{doi:10.1002/cpa.20090}.

\bibitem[AS10]{abbondandolo-schwarz-GT-2010}
\bysame, \emph{Floer homology of cotangent bundles and the loop product}, Geom.
  Topol. \textbf{14} (2010), 1569--1722,
  \href{https://doi.org/10.2140/GT.2010.14.1569}{doi:10.2140/GT.2010.14.1569}.

\bibitem[AS12]{abbondandolo-schwarz-proceedings-2012}
\bysame, \emph{On product structures in {F}loer homology of cotangent bundles},
  Global Differential Geometry, Springer Berlin Heidelberg, 2012, pp.~491--521.

\bibitem[BC07]{biran-cornea-arXiv-2007}
P.~Biran and O.~Cornea, \emph{Quantum structures for lagrangian submanifolds},
  arXiv:0708.4221, 2007,
  \href{https://doi.org/10.48550/arXiv.0708.4221}{doi:10.48550/arXiv.0708.4221}.

\bibitem[BC24]{brocic-cant-JFPTA-2024}
F.~Bro\'ci\'c and D.~Cant, \emph{Bordism classes of loops and {F}loer's
  equation in cotangent bundles}, J. Fixed Point Theory Appl. \textbf{26}
  (2024), 1--29,
  \href{https://doi.org/10.1007/s11784-024-01114-x}{doi:10.1007/s11784-024-01114-x}.

\bibitem[BCS25]{brocic-cant-shelukhin-math-ann-2025}
F.~Bro{\'c}i{\'c}, D.~Cant, and E.~Shelukhin, \emph{The chord conjecture for
  conormal bundles}, Math. Ann. (2025), 63,
  \href{https://doi.org/10.1007/s00208-025-03171-0}{doi:10.1007/s00208-025-03171-0}.

\bibitem[Bim24]{bimmermann-arch-math-2024}
J.~Bimmermann, \emph{{H}ofer–{Z}ehnder capacity of magnetic disc tangent
  bundles over constant curvature surfaces}, Arch. Math. \textbf{123} (2024),
  103--111,
  \href{https://doi.org/10.1007/s00013-024-02003-y}{doi:10.1007/s00013-024-02003-y}.

\bibitem[BK22]{benedetti-kang-JFPTA-2022}
G.~Benedetti and J.~Kang, \emph{Relative {H}ofer-{Z}ehnder capacity and
  positive symplectic homology}, J. Fixed Point Theory Appl. \textbf{24}
  (2022), no.~44, 1--32,
  \href{https://doi.org/10.1007/s11784-022-00963-8}{doi:10.1007/s11784-022-00963-8}.

\bibitem[Bro25]{brocic-CCM-2025}
F.~Bro\'{c}i\'{c}, \emph{Riemannian distance and symplectic embeddings in
  cotangent bundle}, Comm. Cont. Math. \textbf{27} (2025), no.~03, 1--28,
  \href{https://doi.org/10.1142/S021919972450024X}{doi:10.1142/S021919972450024X}.

\bibitem[Can22]{cant-thesis-2022}
D.~Cant, \emph{A dimension formula for relative symplectic field theory},
  Stanford University PhD Thesis. Available at:
  {\url{https://dylancant.ca/thesis.pdf}}, 7 2022.

\bibitem[Can24]{cant-arXiv-2024}
\bysame, \emph{Remarks on eternal classes in symplectic cohomology},
  arXiv:2410.03914, 10 2024, Submitted,
  \href{https://doi.org/10.48550/arXiv.2410.03914}{doi:10.48550/arXiv.2410.03914}.

\bibitem[CC23]{cant-chen-kyoto-2023}
D.~Cant and D.~Chen, \emph{Adiabatic compactness for holomorphic curves with
  boundary on nearby {L}agrangians}, arXiv:2302.13391, 02 2023, 77 pages.
  Accepted to Kyoto J. Math.,
  \href{https://doi.org/10.48550/arXiv.2302.13391}{doi:10.48550/arXiv.2302.13391}.

\bibitem[CHO23]{cieliebak-hingston-oancea-JFPTA-2023}
K.~Cieliebak, N.~Hingston, and A.~Oancea, \emph{Loop coproduct in {M}orse and
  {F}loer homology}, J. Fixed Point Theory Appl. \textbf{25} (2023), no.~59,
  1--84,
  \href{https://doi.org/10.1007/s11784-023-01061-z}{doi:10.1007/s11784-023-01061-z}.

\bibitem[Cie94]{cieliebak-JMPA-1994}
K.~Cieliebak, \emph{Pseudo-holomorphic curves and periodic orbits on cotangent
  bundles}, J. Math. Pure Appl. \textbf{9} (1994), no.~73, 251--278.

\bibitem[Cie02]{cieliebak-JEMS-2002}
\bysame, \emph{Handle attaching in symplectic homology and the chord
  conjecture}, J. European Math. Soc. \textbf{4} (2002), no.~2, 115--142,
  \href{https://doi.org/10.1007/s100970100036}{doi:10.1007/s100970100036}.

\bibitem[CS99]{chas-sullivan-arXiv-1999}
M.~Chas and D.~Sullivan, \emph{String topology}, arXiv:math/9911159v1, 11 1999,
  \href{https://doi.org/10.48550/arXiv.math/9911159}{doi:10.48550/arXiv.math/9911159}.

\bibitem[EENS13]{ekholm-etnyre-ng-sullivan-GT-2013}
T.~Ekholm, J.~Etnyre, L.~Ng, and M.~G. Sullivan, \emph{Knot contact homology},
  Geom. Topol. \textbf{17} (2013), 975--1112,
  \href{https://doi.org/10.2140/gt.2013.17.975}{doi:10.2140/gt.2013.17.975}.

\bibitem[EG91]{eliashberg-gromov-AMS-1991}
Y.~Eliashberg and M.~Gromov, \emph{Convex symplectic manifolds}, Several
  complex variables and complex geometry, part 2, Proc. Sympos. Pure Math.,
  vol.~52, AMS, 1991, pp.~135--162,
  \href{https://doi.org/10.1090/pspum/052.2}{doi:10.1090/pspum/052.2}.

\bibitem[Ekh07]{ekholm-GT-2007}
T.~Ekholm, \emph{Morse flow trees and {L}egendrian contact homology in 1-jet
  spaces}, Geom. Topol. \textbf{11} (2007), 1083--1224,
  \href{https://doi.org/10.2140/gt.2007.11.1083}{doi:10.2140/gt.2007.11.1083}.

\bibitem[EP03]{entov-poltero-IMRN-2003}
M.~Entov and L.~Polterovich, \emph{Calabi quasimorphism and quantum homology},
  Int. Math. Res. Not. \textbf{2003} (2003), 1635--1676,
  \href{https://doi.org/10.1155/S1073792803210011}{doi:10.1155/S1073792803210011}.

\bibitem[EP22]{eliashberg-pancholi}
Y.~Eliashberg and D.~Pancholi, \emph{{H}onda-{H}uang’s work on contact
  convexity revisited}, arXiv:2207.07185, 2022,
  \href{https://doi.org/10.48550/arXiv.2207.07185}{doi:10.48550/arXiv.2207.07185}.

\bibitem[Fau20]{fauck-IJM-2020}
A.~Fauck, \emph{On manifolds with infinitely many fillable contact structures},
  Int. J. Math. \textbf{31} (2020), no.~13, 1--71,
  \href{https://doi.org/10.1142/S0129167X20501086}{doi:10.1142/S0129167X20501086}.

\bibitem[FHS95]{floer-hofer-salamon-duke-1995}
A.~Floer, H.~Hofer, and D.~Salamon, \emph{Transversality in elliptic {M}orse
  theory for the symplectic action}, Duke Math. J. \textbf{80} (1995), no.~1,
  251--292,
  \href{https://doi.org/10.1215/S0012-7094-95-08010-7}{doi:10.1215/S0012-7094-95-08010-7}.

\bibitem[Flo89]{floer-CMP-1989}
A.~Floer, \emph{Symplectic fixed points and holomorphic spheres}, Comm. Math.
  Phys. \textbf{120} (1989), 575--611,
  \href{https://doi.org/10.1007/BF01260388}{doi:10.1007/BF01260388}.

\bibitem[FO97]{fukaya-oh-AJM-1997}
K.~Fukaya and Y-G. Oh, \emph{Zero-loop open strings in the cotangent bundle and
  {M}orse homotopy}, Asian J. Math. \textbf{1} (1997), no.~1, 96--180,
  \href{https://doi.org/10.4310/AJM.1997.V1.N1.A5}{doi:10.4310/AJM.1997.V1.N1.A5}.

\bibitem[FRV23]{ferreira-ramos-vicente-arXiv-2023}
B.~Ferreira, V.~G.~B. Ramos, and A.~Vicente, \emph{Gromov width of the disk
  cotangent bundle of spheres of revolution}, arXiv:2301.08528, 2023,
  \href{https://doi.org/10.48550/arXiv.2301.08528}{doi:10.48550/arXiv.2301.08528}.

\bibitem[FS07]{frauenfelder-schlenk-IJM-2007}
U.~Frauenfelder and F.~Schlenk, \emph{Hamiltonian dynamics on convex symplectic
  manifolds}, Isr. J. Math. \textbf{159} (2007), 1--56,
  \href{https://doi.org/10.1007/S11856-007-0037-3}{doi:10.1007/S11856-007-0037-3}.

\bibitem[Fuk97]{fukaya-AMSIP-1997}
K.~Fukaya, \emph{Morse homotopy and its quantization}, Geometric topology: 1993
  Georgia international topology conference (William~H Kazez, ed.), Studies in
  Advanced Mathematics, vol.~2, American Mathematical Society and International
  Press of Boston, 1997, pp.~409--440,
  \href{https://doi.org/10.1090/amsip/002.1}{doi:10.1090/amsip/002.1}.

\bibitem[GT98]{gilbarg-trudinger-1998}
D.~Gilbarg and N.~S. Trudinger, \emph{Elliptic partial differential equations
  of second order}, revised 3rd printing ed., Grundlehren der mathematischen
  Wissenschaften, vol. 224, Springer, 1998,
  \href{https://doi.org/10.1007/978-3-642-61798-0}{doi:10.1007/978-3-642-61798-0}.

\bibitem[HS95]{hofer-salamon-floermem-1995}
H.~Hofer and D.~Salamon, \emph{Floer homology and {N}ovikov rings}, The Floer
  memorial volume, Progr. Math., vol. 133, Birkh{\"a}user, 1995, pp.~483--524,
  \href{https://doi.org/10.1007/978-3-0348-9217-9}{doi:10.1007/978-3-0348-9217-9}.

\bibitem[Hut08]{hutchings-agt-2008}
M.~Hutchings, \emph{Floer homology of families {I}}, Alg. Geom. Topol.
  \textbf{8} (2008), 435--492,
  \href{https://doi.org/10.2140/agt.2008.8.435}{doi:10.2140/agt.2008.8.435}.

\bibitem[Iri14]{irie-JEMS-2014}
K.~Irie, \emph{{H}ofer-{Z}ehnder capacity of unit disk cotangent bundles and
  the loop product}, J. Eur. Math. Soc. \textbf{16} (2014), 2477--2497,
  \href{https://doi.org/10.4171/JEMS/491}{doi:10.4171/JEMS/491}.

\bibitem[KS21]{kislev-shelukhin-GT-2021}
A.~Kislev and E.~Shelukhin, \emph{Bounds on spectral norms and barcodes},
  Geometry {\&} Topology \textbf{25} (2021), 3257–3350,
  \href{https://doi.org/10.2140/gt.2021.25.3257}{doi:10.2140/gt.2021.25.3257}.

\bibitem[{Mac}71]{maclean-springer-1971}
S.~{Mac Lane}, \emph{Categories for the working mathematician}, 2nd edition
  ed., Graduate texts in mathematics, vol.~5, Springer, 1971,
  \href{https://doi.org/10.1007/978-1-4757-4721-8}{doi:10.1007/978-1-4757-4721-8}.

\bibitem[Mil65a]{milnor-book-PUP-1965}
J.~Milnor, \emph{Lectures on the h-cobordism theorem}, Princeton University
  Press, 1965,
  \href{https://doi.org/10.1515/9781400878055}{doi:10.1515/9781400878055}.

\bibitem[Mil65b]{milnor-book-1965}
\bysame, \emph{Topology from the differentiable viewpoint}, The University
  Press of Virginia, 1965,
  \href{https://doi.org/10.2307/2314613}{doi:10.2307/2314613}.

\bibitem[MS12]{mcduff-salamon-book-2012}
D.~McDuff and D.~Salamon, \emph{$j$-holomorphic curves and symplectic
  topology}, 2nd edition ed., American Mathematical Society, Colloquium
  Publications, 2012,
  \href{https://doi.org/10.1090/coll/052}{doi:10.1090/coll/052}.

\bibitem[MS17]{mcduffsalamon-alt}
\bysame, \emph{Introduction to symplectic topology}, 3rd ed., Oxford University
  Press, 2017,
  \href{https://doi.org/10.1093/oso/9780198794899.001.0001}{doi:10.1093/oso/9780198794899.001.0001}.

\bibitem[MT94]{mcduff-traynor-LMS-1994}
D.~McDuff and L.~Traynor, \emph{The 4-dimensional symplectic camel and related
  results}, Symplectic Geometry (D.~Salamon, ed.), London Mathematical Society
  Lecture Note Series, Cambridge University Press, 1994, pp.~169--182,
  \href{https://doi.org/10.1017/CBO9780511526343.010}{doi:10.1017/CBO9780511526343.010}.

\bibitem[OZ11]{oh-zhu-JSG-2011}
Y-G. Oh and K.~Zhu, \emph{Floer trajectories with immersed nodes and
  scale-dependent gluing}, J. Symplectic Geom. \textbf{9} (2011), no.~4,
  483--636,
  \href{https://doi.org/10.4310/JSG.2011.v9.n4.a4}{doi:10.4310/JSG.2011.v9.n4.a4}.

\bibitem[PSS96]{piunikhin-salamon-schwarz-1996}
S.~Piunikhin, D.~Salamon, and M.~Schwarz, \emph{Symplectic {F}loer-{D}onaldson
  theory and quantum cohomology}, Contact and symplectic geometry, Publications
  of the Newton Institute, vol.~8, Cambridge University Press, 1996,
  pp.~171--200.

\bibitem[Rit13]{ritter-jtopol-2013}
A.~Ritter, \emph{Topological quantum field theory structure on symplectic
  cohomology}, J. Topol. \textbf{6} (2013), 391--489,
  \href{https://doi.org/10.1112/jtopol/jts038}{doi:10.1112/jtopol/jts038}.

\bibitem[Sal97]{salamon-notes-1997}
D.~Salamon, \emph{Lectures on {F}loer homology}, IAS/Park City Graduate Summer
  School on Symplectic Geometry and Topology, 1997,
  \href{https://doi.org/10.1090/pcms/007/05}{doi:10.1090/pcms/007/05}.

\bibitem[Sch95]{schwarz-thesis-1995}
M.~Schwarz, \emph{Cohomology operations from {$S^{1}$}-cobordisms in {F}loer
  homology}, \url{https://www.math.uni-leipzig.de/~schwarz/diss.pdf}, 1995, PhD
  Dissertation.

\bibitem[Sch00]{schwarz-PJM-2000}
\bysame, \emph{On the action spectrum for closed symplectically aspherical
  manifolds}, Pacific J. Math. \textbf{193} (2000), 419--461,
  \href{https://doi.org/10.2140/PJM.2000.193.419}{doi:10.2140/PJM.2000.193.419}.

\bibitem[Sei08]{seidel-book-2008}
P.~Seidel, \emph{Fukaya categories and {P}icard-{L}efshetz theory}, Zurich
  Lectures in Advanced Mathematics, European Mathematical Society, 2008,
  \href{https://doi.org/10.4171/063}{doi:10.4171/063}.

\bibitem[Sei14]{seidel-inventiones-2014}
\bysame, \emph{Disjoinable {L}agrangian spheres and dilations}, Invent. Math.
  \textbf{197} (2014), 299--359,
  \href{https://doi.org/10.1007/s00222-013-0484-x}{doi:10.1007/s00222-013-0484-x}.

\bibitem[Sei15]{seidel-GAFA-2015}
\bysame, \emph{The equivariant pair-of-pants product in fixed point {F}loer
  cohomology}, Geom. Funct. Anal. \textbf{25} (2015), 942--1007,
  \href{https://doi.org/10.1007/s00039-015-0331-x}{doi:10.1007/s00039-015-0331-x}.

\bibitem[She22a]{shelukhin-GAFA-2022}
E.~Shelukhin, \emph{Symplectic cohomology and a conjecture of {V}iterbo}, Geom.
  Funct. Anal. \textbf{32} (2022), 1514--1543,
  \href{https://doi.org/10.1007/s00039-022-00619-2}{doi:10.1007/s00039-022-00619-2}.

\bibitem[She22b]{shelukhin-invetiones-2022}
\bysame, \emph{Viterbo conjecture for {Z}oll symmetric spaces}, Inventiones
  mathematicae \textbf{230} (2022), no.~1, 321--373,
  \href{https://doi.org/10.1007/s00222-022-01124-x}{doi:10.1007/s00222-022-01124-x}.

\bibitem[SS12]{seidel-solomon-GAFA-2012}
P.~Seidel and J.~P. Solomon, \emph{Symplectic cohomology and {$q$}-intersection
  numbers}, Geom. Funct. Anal. \textbf{22} (2012), 443--477,
  \href{https://doi.org/10.1007/s00039-012-0159-6}{doi:10.1007/s00039-012-0159-6}.

\bibitem[SW06]{salamon-weber-GAFA-2006}
D.~Salamon and J.~Weber, \emph{Floer homology and the heat flow}, Geom. Funct.
  Anal. \textbf{16} (2006), 1050--1138,
  \href{https://doi.org/10.1007/s00039-006-0577-4}{doi:10.1007/s00039-006-0577-4}.

\bibitem[Vit92]{viterbo-mathann-1992}
C.~Viterbo, \emph{Symplectic topology as the geometry of generating functions},
  Math. Ann. \textbf{292} (1992), 685--710,
  \href{https://doi.org/10.1007/BF01444643}{doi:10.1007/BF01444643}.

\bibitem[Vit99]{viterbo-GAFA-1999}
\bysame, \emph{Functors and computations in {F}loer homology with applications
  {I}}, Geom. Funct. Anal. \textbf{9} (1999), 985--1033,
  \href{https://doi.org/10.1007/s000390050106}{doi:10.1007/s000390050106}.

\bibitem[Wei91]{weinstein-hokkaido-91}
A.~Weinstein, \emph{{Contact surgery and symplectic handlebodies}}, Hokkaido
  Mathematical Journal \textbf{20} (1991), no.~2, 241 -- 251,
  \href{https://doi.org/10.14492/hokmj/1381413841}{doi:10.14492/hokmj/1381413841}.

\bibitem[Zho21]{zhou-topology-2021}
Z.~Zhou, \emph{Symplectic fillings of asymptotically dynamically convex
  manifolds i}, Journal of Topology \textbf{14} (2021), no.~1, 112--182,
  \href{https://doi.org/10.1112/topo.12177}{doi:10.1112/topo.12177}.

\end{thebibliography}
\end{document}